\documentclass{amsart}
\usepackage[fontsize=12]{scrextend}
\usepackage[french,english]{babel}

\usepackage[T1]{fontenc}

\usepackage{fancyheadings}

\usepackage{amsthm}
\usepackage{amsfonts}
\usepackage{soul}

\usepackage{color}
\usepackage{hyperref}
\usepackage{mathrsfs}

\numberwithin{equation}{section}

\newtheorem{theorem}{Theorem}[section]
\newtheorem{proposition}[theorem]{Proposition}
\newtheorem{corollary}[theorem]{Corollary}
\newtheorem{lemma}[theorem]{Lemma}

\newtheorem{definition}[theorem]{Definition}

\theoremstyle{definition}
\newtheorem{remark}[theorem]{Remark}

\def\[#1\]{\begin{align*}#1\end{align*}}

\def\hn{\hat{n}}
\def\hg{\hat{g}}
\def\hx{\hat{x}}
\def\hy{\hat{y}}

\def\hE{\hat{E}}
\def\hX{\hat{X}}
\def\hY{\hat{Y}}
\def\hZ{\hat{Z}}
\def\hR{\hat{R}}
\def\hh{\hat{h}}
\def\hI{\hat{I}}
\def\hH{\hat{H}}
\def\hf{\hat{f}}
\def\hr{\hat{r}}
\def\hF{\hat{F}}
\def\hU{\hat{U}}
\def\hV{\hat{V}}

\def\hT{\hat{T}}

\def\hN{\hat{N}}
\def\hM{\hat{M}}
\newcommand{\dif}[1]{\frac{d}{d{#1}}}
\newcommand{\grad}{\mathrm{grad}}
\newcommand{\hW}{\hat{W}}
\def\hGamma{\hat{\Gamma}}
\def\hsigma{\hat{\sigma}}
\def\qmhm{Q(M,\hat{M})}
\def\tsmthm{T^* M \otimes T\hat{M}}

\def\dr{\mathscr{D}_{R}}

\def\odr{\mathcal{O}_{\mathscr{D}_{R}}}

\def\lns{\mathscr{L}_{NS}}
\def\lr{\mathscr{L}_{R}}

\def\pqm{\pi_{Q,M}}

\def\rarrow{\rightarrow}
\def\tbar{\overline{T}}
\def\nablabar{\overline{\nabla}}
\def\hnabla{\hat{\nabla}}
\def\hlambda{\hat{\lambda}}
\def\hH{\hat{H}}
\def\hgamma{\hat{\gamma}}
\def\Xbar{\overline{X}}

\def\sbar{\overline{S}}
\def\hT{\hat{T}}
\def\hS{\hat{S}}
\def\Oh{\mathcal{O}}
\def\onq{|_{q}}

\def\q{q=(x,\hat{x};A)}
\def\qz{q_0=(x_0,\hat{x}_0;A_0)}

\def\Xtilde{\tilde{X}}
\def\Ytilde{\tilde{Y}}
\def\XtildeA{\tilde{X}_A}
\def\YtildeA{\tilde{Y}_A}

\def\cphi{c_{\phi}}
\def\sphi{s_{\phi}}

\def\Rolbar{\overline{\Rol}}

\newcommand{\Rol}{\mathsf{Rol}}
\def\R{\mathbb{R}}

\def\R{\mathbb{R}}

\newcommand{\spn}{\mathrm{span}}

\newcommand{\VF}{\mathrm{VF}}

\newcommand{\sso}{\mathfrak{so}}
\newcommand{\mc}{\mathcal}
\newcommand{\qmatrix}[1]{\left(\begin{matrix}#1\end{matrix}\right)}

\newcommand{\ol}[1]{\overline{#1}}
\newcommand{\pa}[1]{\frac{\partial}{\partial{#1}}}
\newcommand{\p}[2]{\frac{\partial #1}{\partial{#2}}}
\newcommand{\n}[1]{\left\lVert#1\right\rVert}

\begin{document}

\title[ \fontsize{6}{7}\selectfont Controllability results for the rolling of $2$-dim. against $3$-dim. Riemannian Manifolds]{Controllability results for the rolling of $2$-dimensional against $3$-dimensional Riemannian Manifolds}

 \author{Amina Mortada}
	\address{Lebanese University\\
		Faculty of sciences 1, Hadath, Beirut, Lebanon}
	\email{amina\_mortada2010@hotmail.com}
 \author{Yacine Chitour}
	\address{L2S\\
		Paris-Saclay University, 3 Rue Joliot Curie, Gif-sur-Yvette, France}
	\email{yacine.chitour@l2s.centralesupelec.fr}
	\author{Petri Kokkonen}
	\address{Helsinki\\
	Finland}
	\email{pvkokkon@gmail.com}
    \author{Ali Wehbe}
	\address{Lebanese University\\
		Faculty of sciences 1, Khawarizmi Laboratory of Mathematics and Applications-KALMA, Hadath, Beirut, Lebanon}
	\email{ali.wehbe@ul.edu.lb}
	

\date{}
\maketitle

\begin{abstract}
In this article, we consider the rolling (or development) of two Riemannian connected manifolds $(M,g)$ and $(\hM,\hg)$ of dimensions $2$ and $3$ respectively, with the constraints of no-spinning and no-slipping. 
The present work is a continuation of \cite{MortadaKokkonenChitour}, which modelled the general setting of the rolling of two Riemannian connected manifolds with different dimensions as a driftless control affine system on a fibered space $Q$, 
with an emphasis on understanding the local structure of the rolling orbits, i.e., the reachable sets in $Q$.
In this paper, the state space $Q$ has dimension eight and we show that the possible dimensions of non open rolling orbits belong to the set $\{2,5,6,7\}$. We describe the structures of orbits of dimension $2$, the possible local structures of rolling orbits of dimension $5$ and some of dimension $7$.
\end{abstract}


\tableofcontents

\section{Introduction}

This article studies the model of rolling of two connected and oriented Riemannian manifolds $(M,g)$ and $(\hM,\hg)$ of dimensions $n=2$ and $\hn=3$ respectively, where the rolling is assumed to be without spinning nor slipping.
In the papers \cite{ChitourGodoyMolinaKokkonen,ChitourGodoyMolinaKokkonen1,ChitourGodoyMolinaKokkonen2, ChitourKokkonen1}, such a rolling model is defined intrinsically in the case where the manifolds have equal dimensions, i.e., $n=\hn\geq 2$ as a driftless  affine control system: the state space $Q$ (of the rolling of two Riemannian manifolds)
is a bundle space with the typical fiber diffeomorphic to the set of orthogonal maps $A$ between the tangent spaces of the respective manifolds,
and the set of controls corresponding to the set of absolutely continuous curves on $M$.
The non-spinning and non-slipping conditions translate respectively into the facts that the image of a vector field parallel to a curve on $M$ by $A$ is also a vector field parallel to a curve on $\hM$ and the associated vector fields are the tangent vectors at the point of contact of manifolds respectively. The (locally) absolutely continuous curves $q(\cdot)$ in $Q$ that verify both the no-slipping and no-spinning conditions
are referred to as the \emph{rolling curves} and it is shown that they are the tangent curves to 
a distribution $\dr$ on $Q$ called the rolling distribution.

The main purpose in these studies consists in understanding the controllability of the control system using geometric tools. More precisely, one seeks necessary and/or sufficient conditions controllability of the rolling system expressed in terms of the geometries of $M$ and $\hM$. Here controllability means that, for any pair $(q_{init},q_{final})$ of points in the state space $Q$, there exists a rolling curve $q(\cdot) \in Q $ which steers $q_{init}$ to $q_{final}$. Fixing $q_{init}$ in $Q$, the set of the points $q_{final}$ is called the reachable set or the orbits from $q_{init}$. In other words, the rolling system is said to be (completely) controllable if the orbits of such points by the control system are all equal to the state space $Q$.

When $M$ and $\hM$ are two-dimensional, the rolling system is completely controllable if and only if the manifolds are not isometric, and, if they are, then the dimension of the orbits is in general equal to 2 or 5 (cf. \cite{AgrachevSachkov}). As regards the motion planning problem for 
two-dimensional manifolds (i.e., finding an effective procedure for the controllability issue), it has been addressed in \cite{CC03,ACL10}.
Then, \cite{ChitourKokkonen1} gave complete answers for the controllability question in case of 3-dimensional manifolds. The authors also established the necessary and sufficient conditions for the controllability of rolling against manifold of constant curvature (cf. \cite{ChitourKokkonen2}). Furthermore, the rolling of affine manifolds with not necessarily zero torsion tensors is explained in \cite{Kokkonen} and \cite{Kokkonen2}.

In \cite{HafassaMortadaChitourKokkonen,CGJK19}, another case of rolling manifolds is yet addressed. Let $M$ be an affine manifold of dimension $n$ and $\Delta$ a constant rank distribution on $M$, i.e., a subbundle of the tangent bundle $T(M)$ of $M$. If one uses $H(M)$ to denote the holonomy group of $M$,
then one can define $H_\Delta(M)$, the holonomy group with respect to $\Delta$ as the subset of $H(M)$ obtained by parallel transporting frames of $M$ along a restricted set of absolutely continuous $\Delta$-horizontal loops, namely along loops which are tangent (almost everywhere) to the distribution $\Delta$. One of the results of \cite{HafassaMortadaChitourKokkonen} says that  $H_\Delta(M)$ is a Lie group strictly included in $H(M)$, even if $\Delta$ is completely controllable, i.e., every pair of points in $M$ can be connected by an absolutely continuous $\Delta$-horizontal curve. On the other hand, \cite{CGJK19}
 provides explicit means of computing these holonomy groups by deriving analogues of Ambrose-Singer and Ozeki theorems.

The present paper deals with the case where $M$ and $\hM$
have different (low) dimensions $n=2$ and $\hn=3$, respectively.
The first reference on the subject for general non-equal dimensions $n$ and $\hn$
(at least in the context of geometric control) is \cite{MortadaKokkonenChitour}, where this problem is recasted as a 
control system: definitions of the appropriate state space, rolling distributions, computation of the main Lie 
brackets of vector fields tangent to the rolling distribution. In particular, it was shown in \cite{MortadaKokkonenChitour} that these Lie brackets can be expressed using the Riemannian curvature tensors of the 
considered manifolds. Moreover, if $n \neq \hn$, then the rolling problem is not symmetric anymore with respect to the order of manifolds. It turns out that several controllability results are available 
in the case $n>\hn$ and $ \hn - n = 1$. In particular, when $(n,\hn)=(3,2)$, one shows that  the system is not completely controllable if and only if $M$ is locally isometric to a warped product of a real interval and 2-
dimensional manifold. However, in the case $(n,\hn)=(2,3)$, the situation is much more involved and this is the subject of the present paper. We obtain only partial results on the controllability issue, most of 
them of local nature. We prove that the dimension of a non open rolling orbit belongs to the set $\{2,5,6,7\}$. It is equal $2$ if and only if $M$ contains an open neighborhood isometric to 
a 2-dimensional embedded totally geodesic submanifold of $\hM$.
If the dimension of a rolling orbit is equal to $5$, then either $\hM$ contains a totally geodesic embedded submanifold, or an open neighborhood of $\hM$ is isometric to a warped product of a real interval with a two-dimensional 
Riemannian manifold respectively.
The case where the rolling orbit has dimension $7$ may occur when $M$ has constant curvature and an open neighbourhood of $\hM$ is isometric to a Riemannian product of a real interval with a two-dimensional Riemannian manifold.
The main open questions remaining are the following: does there exist examples of rolling orbits of dimension $6$ and are there other examples of $7$-dimensional rolling orbits?

The paper is structured as follows. We gather the general notations in Section \ref{sec:notations} and we provide
the control theoretic framework of the rolling problem in Section \ref{s10} as well as the computation of the main Lie brackets tangent to the distribution.
The main result of this paper, Theorem~\ref{main-theorem}, along with
notations and conventions specific to the formulation and proof of it,
are given in Section \ref{s1}.
The proof of the main result is produced in Sections \ref{sec:5} and \ref{ss1},
and finally, the Appendix lists several computational results useful for the proof.

\vspace{2\baselineskip}\section{Notations}\label{sec:notations}

In this section, we provide notations and some concepts that will be used throughout this text.
All manifolds are assumed to be finite dimensional and they, along with any maps between manifolds (such as vector fields) are assumed to be $C^\infty$-smooth unless otherwise specified.

A smooth distribution $\Delta$ of constant rank $m$ over a smooth manifold $M$ is a smooth assignment $x\mapsto \Delta |_x$,
where $\Delta |_x$ is a linear subspace of $T_x M$ and $\dim( \Delta |_x)=m$ for every $x\in M$.
Since all the distributions we encounter are smooth and have constant rank, we will simply call them distributions in the sequel.

An absolutely continuous (a.c. for short) curve $\gamma:I \rarrow M$ defined on a nonempty interval $I \subset \mathbb{R}$ is said to be tangent to $\Delta$ if $\dot{\gamma}(t)\in \Delta|_{\gamma(t)}$ for almost every $t\in I$.
For $x_0 \in M$, the \emph{$\Delta$-orbit}
(or the orbit of $\Delta$)
passing through $x_0$, denoted by $\Oh_{\Delta} (x_0)$, is the set of endpoints
of all a.c. curves on $M$ defined on $I=[0,1]$, tangent to $\Delta$ and starting at $x_0$, i.e.,
\begin{align*}
\Oh_{\Delta} (x_0) := \{ \gamma(1) \mid \gamma : [0,1] \rarrow M, \text{ a.c. curve tangent to $\Delta$}, \; \gamma(0) = x_0 \}.
\end{align*}
A smooth vector field is tangent to a distribution $\Delta$ on $M$ if $X|_x\in \Delta|_x$ for every $x\in M$.

By the Orbit Theorem (\cite{sussmann73}), any $\Delta$-orbit $\Oh_{\Delta} (x_0)$ is an immersed smooth (and in fact initial\footnote{ 
An  immersed submanifold $N\subset M$ is an \emph{initial submanifold} if for any smooth manifold $Z$ and any smooth 
(or continuous) map $f:Z\to M$ such that $f(Z)\subset N$, the map $f$ viewed as having values in $N$ equipped with 
its intrinsic differentiable structure (or topology), $f:Z\to N$, is smooth (or continuous). Note that any embedded submanifold is initial.
}) 
submanifold of $M$ containing $x_0$ such that the tangent space $T_x \Oh_{\Delta} (x_0)$, for every $x \in \Oh_{\Delta} (x_0)$,
contains $\mathrm{Lie}_x (\Delta)$, the evaluation at $x$ of the Lie algebra generated by the vector fields tangent to $\Delta$.
Furthermore, if a  smooth distribution $\Delta'$ on $M$ is a subdistribution of $\Delta$, i.e., $\Delta' \subset \Delta$, then $\Oh_{\Delta'} (x_0) \subset \Oh_{\Delta} (x_0)$ for all $x_0 \in M$.
We say that $\Delta$ is \emph{completely controllable} if for every $x\in M$, we have $\Oh_{\Delta} (x)=M$, i.e., if 
any two points of $M$ can be joined by an a.c. curve tangent to $\Delta$.

For any smooth bundle $\pi: E \rarrow M$,
where $E$ is the total space, and $M$ the base space of the bundle,
the \emph{$\pi$-vertical distribution} $V(\pi)$ on $E$ is defined by setting
$V |_{y} (\pi)=\{Y \in T |_{y} E\ |\ \pi_* (Y) = 0\}$ for all $y \in E$.
Moreover, $\pi_{TE} |_{V(\pi)}:V(\pi)\to M$ defines a vector subbundle of $\pi_{TE} : TE \rarrow E$.
If $\pi: E \rarrow M$ and $\eta: F \rarrow M$ are two smooth bundles,
we write $C^{\infty} (\pi ,\eta)$ for the set of smooth bundle morphisms,
i.e., maps $f : E \rarrow F$ such that $\eta \circ f = \pi$.

If $\pi:E\to M$ is a \emph{vector} bundle, $f\in C^\infty(E)$
and if $u, w \in \pi^{-1} (x)$,
then define the \emph{$\pi$-vertical derivative} of $f$ at $u$ in the direction $w$ by
\begin{align}\label{eq:def:vert}
\nu(w) |_{u} (f) = \dif{t} \big|_0 f(u+tw),\quad f\in C^\infty(M).
\end{align}
It follows from this definition that
$\nu(w)|_u$ is a tangent vector of $E$ at $u$,
that $\nu(w)|_u\in V|_u(\pi)$
and that
$w \rarrow \nu(w) |_{u}$ is an $\R$-linear isomorphism from
$\pi^{-1} (x)$ onto $V |_{u} (\pi)$ for all $u\in E$ with $x=\pi (u)$.

For a manifold $M$, we let $T_{m}^{k} M$ be the space of $(k,m)$-tensors on $M$. Its fiber over $x\in M$
is $(T_m^k)_x M$. Tangent and cotangent space are $TM=T_0^1 M$, $T^*M=T_1^0 M$.
The set of all vector fields on a manifold $M$ is denoted by $\VF(M)$.
If $M$ and $\hM$ are two manifolds,
then the set of all linear maps $T_x M\to T_{\hx} \hM$
is linearly isomorphic to $T^*_x M\otimes T_{\hx} \hM$ (tensor product over $\R$).
Hence the space of all linear maps $T_x M\to T_{\hx} \hM$ for all $x\in M$, $\hx\in \hM$ 
can be written as
$T^* M\otimes T\hM=\bigcup_{(x,\hx)\in M\times\hM} \big(T^*_x M\otimes T_{\hx} \hM\big)$.

We will frequently be writing elements $A\in T^*_x M\otimes T_{\hx} \hM$ as $(x,\hx;A)$,
so that by writing a point $q\in T^*M\otimes T\hM$ as $\q$ we know that $q=A$ belongs to $T^*_x M\otimes T_{\hx} \hM$.
It should be mentioned that $T^*M\otimes T\hM\to M\times\hM$; $(x,\hx;A)\mapsto (x,\hx)$ defines a smooth vector bundle,
while the maps $T^*M\otimes T\hM\to M$; $(x,\hx;A)\mapsto x$ and $T^*M\otimes T\hM\to \hM$; $(x,\hx;A)\mapsto \hx$
define just smooth bundles.

If $(M,\nabla)$ is an \emph{affine manifold}, i.e., a manifold $M$ equipped with a linear connection $\nabla$,
then $(P^{\nabla})_0^{t} (\gamma) S$
is used to denote the $\nabla$-parallel transport of any tensor $S\in (T^k_m)_{\gamma(0)} M$
for $\gamma(0)$ to $\gamma(t)$
along an a.c. curve $\gamma: I \rarrow M$ such that $0,t\in I$.
Notice also that $(P^{\nabla})_0^{t} (\gamma) S\in (T^k_m)_{\gamma(t)} M$.
In case $(M,g)$ is a Riemannian manifold, we shall always view it as an affine manifold
equipped with its Levi-Civita connection.

Finally, given any Riemannian manifold $(M,g)$ and two vectors $X,Y\in T_x M$ on it,
we shall identify the two-vector $X\wedge Y\in \wedge^2 T|_x M$ with the linear map $X\wedge Y:T_x M\to T_x M$
that is defined by
\begin{align}\label{eq:two_vector}
(X\wedge Y)Z:=g(Z,X)Y-g(Z,Y)X.
\end{align}

\vspace{2\baselineskip}\section{Rolling Motion as a Control System}\label{s10}

In this section, we recall the main definitions introduced first in  \cite{MortadaKokkonenChitour} relative to the rolling of two smooth connected complete oriented Riemannian manifolds $(M,g)$ and $(\hM,\hg)$ of (not necessarily equal) dimensions $n$ and $\hn$ respectively. The two constraints of non-spinning and non-slipping are also considered.

The first definition concerns the state (or configuration) space for rolling of a lower dimensional Riemannian manifold
on a higher dimensional one.

\vspace{2\baselineskip}\subsection{The State Space $Q$}

\begin{definition}
Let $(M,g)$ and $(\hM,\hg)$ be two Riemannian manifolds of dimensions $n$ and $\hn$, respectively,
such that $n \leq \hn$.
The \emph{state space} $Q=\qmhm$ for the problem of \emph{rolling of $M$ against $\hM$} is defined as
\begin{align}\label{eq:def:Q}
\qmhm:=\{A \in \tsmthm\ |\ \hg(AX,AY) = g (X,Y),\; X,Y \in T_{x} M,\; x \in M \},
\end{align}
and it is the total space of the smooth bundles over $M\times\hM$, $M$ and $\hM$, respectively,
\[ 
\pi_Q:Q\to M\times\hM;\quad {}& (x,\hx;A)\mapsto (x,\hx) \\
\pi_{Q,M}:Q\to M\times\hM;\quad {}& (x,\hx;A)\mapsto x \\
\pi_{Q,\hM}:Q\to M\times\hM;\quad {}& (x,\hx;A)\mapsto \hx.
\]

\end{definition}

It is straightforward to verify the following assertion.

\begin{proposition}\label{p1.1}
The space $Q=\qmhm$ is a smooth closed submanifold of $\tsmthm$ of dimension
\[
\dim Q=n + \hat{n} + n \hn - \frac{n (n +1)}{2}.
\]
\end{proposition}

The particular case that we will be concerned with in this paper
has $n=2$ and $\hn=3$ and therefore $\dim Q=8$.

\vspace{2\baselineskip}\subsection{Rolling Lift, Distribution and Orbit}\ \newline

For the remainder of this text, unless otherwise mentioned,
$(M,g)$ and $(\hM,\hg)$ are Riemannian manifolds of dimensions $\dim M=n$, $\dim\hM=\hn$.
They come equipped with their Levi-Civita connections $\nabla$, $\hnabla$, respectively,
and the parallel transports with respect to them are denoted by $P_0^t(\gamma)$, $P_0^t(\hgamma)$,
when $\gamma$ and $\hgamma$ are a.c. curves on $M$ and $\hM$, respectively.
The state space for rolling $Q(M,\hM)$ is often written shortly as $Q$.

The dynamics of the rolling motion without spinning or slipping is formulated as follows.

\begin{definition}\label{def:rolling}
An absolutely continuous curve  $q: [a,b] \rarrow Q$; $t \mapsto (\gamma(t), \hgamma(t); A(t))$ in $Q$ is said to describe
\begin{itemize}
\item[(i)] a rolling motion without \emph{spinning} of $M$ against $\hM$ if
\begin{align}\label{e1.2}
\nablabar_{(\dot{\gamma} (t), \dot{\hgamma} (t))} A (t) = 0 \textrm{ for a.e. } t \in [a,b],
\end{align}
where $\nablabar$ is the product connection of $\nabla$ and $\hnabla$ on $M\times\hM$;
\item[(ii)]  a rolling motion without \emph{slipping} of $M$ against $\hM$ if
\begin{align}\label{e1.3}
A(t) \dot{\gamma} (t) = \dot{\hgamma} (t) \text{ for a.e. } t \in [a,b].
\end{align}
\item[(iii)]  a rolling motion without spinning or slipping of $M$ against $\hM$ if both conditions $(i)$ and $(ii)$ hold true.
\end{itemize}

The rolling motion without spinning or slipping (item (iii)) will simply be called \emph{rolling},
and the associated curve $q(t)$ a \emph{rolling curve}.
\end{definition}

The next result gives a characterization of the no-spinning condition \eqref{e1.2},
and shows that the formulation of this condition is well defined on $Q=Q(M,\hM)$,
i.e., that a curve in $T^*M\otimes T\hM$ that passes through a point of $Q$
and satisfies \eqref{e1.2} stays in the space $Q$.
The proof is straightforward and hence omitted (see \cite{ChitourKokkonen}).

\begin{proposition}\label{p1.3}
Let $\qz\in T^*M\otimes T\hM$
and let $q:[a,b]\to T^*M\otimes T\hM$; $t\mapsto (\gamma(t), \hgamma (t); A(t))$
be an a.c. curve in $T^*M\otimes T\hat{M}$
such that $0\in [a,b]$ and $q(0)=q_0$.
Then we have
\begin{align*}
\nablabar_{(\dot{\gamma} (t), \dot{\hgamma} (t))} A (t) = 0\ \textrm{ for a.e. } t \in [a,b]
\quad \Leftrightarrow \quad
& A(t) = P_0^{t} (\hgamma) \circ A_0 \circ P_{t}^0 (\gamma),\quad \forall t\in [a,b].
\end{align*}
Therefore, if $\nablabar_{(\dot{\gamma} (t), \dot{\hgamma} (t))} A (t) = 0$ for a.e. $t \in [a,b]$,
then
\[
\qz \in Q \ \Longrightarrow \ q(t)=(x(t),\hx(t);A(t))\in Q\quad \forall t\in [a,b].
\]
\end{proposition}

Next we give an infinitesimal characterization of the no-spinning condition \eqref{e1.2}
by defining a lift operator associated with it.

\begin{definition}\label{d1.1}
Given $q=(x,\hx;A) \in Q$ and $X\in T_{x}M$, $\hX \in T_{\hx} \hM$,
one defines the \emph{no-spinning lift} of $(X,\hat{X})$ as the vector $\lns(X, \hat{X}) |_{q}$ in $T_q Q$
given by
\[
\lns(X, \hX)|_q = \frac{d}{dt} \big|_0 \big(P_0^{t} (\hgamma) \circ A \circ P_{t}^0 (\gamma)\big),
\]
where $\gamma$ (resp. $\hat{\gamma}$) is any smooth curve on $M$ (resp. $\hM$)
such that $\gamma(0)=x$, $\dot{\gamma}(0)=X$ (resp. $\hat{\gamma}(0)=\hx$, $\dot{\hat{\gamma}}(0)=\hX$).

Moreover, if $X,\hat{X}$ are (locally defined) vector fields on $M,\hat{M}$, respectively, one writes $\lns(X,\hat{X})$
for the (locally defined) vector field on $Q$ whose value at $q$ is $\lns(X,\hat{X})|_q$.
\end{definition}

A basic characterization of the $\lns$-lift is formulated in the next proposition,
whose easy proof we will omit (see \cite{ChitourKokkonen}).
Recall that $\nablabar$ is the product (Levi-Civita) connection on $\overline{M}=M\times\hat{M}$.

\begin{proposition}\label{p1.4}
If $X\in T_x M$, $\hX\in T_{\hx}\hM$ are vectors and $A$ is a local section of $\pi_{Q}$ then
\begin{align}\label{e1.5}
\lns(X,\hX)|_{A(x,\hx)} = A_{*} (X,\hX) - \nu (\nablabar_{(X,\hX)} A) |_{A(x, \hx)},
\end{align}
where $A_{*}$ is the push-forward of $A$.
\end{proposition}

Finally, we are in position to define an infinitesimal characterization of the rolling motion
as defined in item (iii) of Definition \ref{def:rolling}.

\begin{definition}\label{d1.2}
\begin{itemize}
\item[(i)]
For any $q=(x,\hx;A) \in Q$, the \emph{rolling lift} of $X \in T_{x} M$
is the vector $\lr(X)|_q$ in $T_q Q$ defined by
\begin{align} \label{e1.6}
\lr(X)|_q:= \lns (X, AX) |_{q}.
\end{align}
Moreover, if $X$ is a (locally defined) vector field on $M$, one writes $\lns(X)$
for the (locally defined) vector field on $Q$ whose value at $q$ is $\lns(X)|_q$.
		
\item[(ii)]
The \emph{rolling distribution} $\dr$ is the $n$-dimensional smooth distribution on $Q$
whose plane at every $q=(x,\hx;A) \in Q$ is given by
\begin{align}\label{e1.7}
\dr |_{q} := \lr (T_{x} M) |_{q}.
\end{align}
\end{itemize}
\end{definition}

Some elementary properties of the rolling distribution $\dr$ and rolling curves
are given the following proposition.
Its proof is straightforward and will not be produced here (see \cite{ChitourKokkonen}).

\begin{proposition}\label{pr:prelim:1}
\begin{itemize}
\item[(i)] $(\pqm)_{*}$ maps $\dr |_{q}$ isomorphically onto $T_x M$
for every $q=(x,\hat{x};A)\in Q$.
		
\item[(ii)] An a.c. curve $q:[a,b]\to Q$; $t\mapsto (\gamma(t), \hgamma(t); A(t))$ on $Q$
is a rolling curve if and only if it is tangent to $\dr$ for a.e. $t\in [a,b]$,
i.e., if and only if $\dot{q} (t) = \lr (\dot{\gamma} (t)) |_{q(t)}$
for a.e. $t\in [a,b]$.

\item[(iii)]
If $\gamma:[0,1]\to M$ is an a.c. curve, $\gamma(0)=x_0$
and if $\qz\in Q$,
then there exist $a>0$ and a unique rolling curve $q:[0,a]\to Q$
such that $\pi_{Q,M}(q(t))=\gamma(t)$ for all $t\in [0,a]$.

\item[(iv)] 
If $q(t)=(\gamma(t),\hgamma(t);A(t))$, $t\in [a,b]$, is a rolling curve
and $\gamma$ is a geodesic on $M$,
then $\hgamma$ is a geodesic on $\hM$.
\end{itemize}
\end{proposition}

At last, we define the key concept of this paper, namely that of the rolling orbit.

\begin{definition}\label{def:rol_orbit}
The \emph{rolling orbit} $\odr (q_0)$
corresponding to rolling of $(M,g)$ against $(\hM,\hg)$
is the $\dr$-orbit in $Q$ passing through $q_0$.

We say that the rolling problem is completely controllable if $\odr(q_0)=Q$ for any (and hence all) $q_0\in Q$.
\end{definition}

\vspace{2\baselineskip}\subsection{Lie Brackets on $Q$}\ \newline

Let $\Oh$ be an immersed submanifold of $\tsmthm$ and write $\pi_{\Oh} := \pi_{\tsmthm} |_{\Oh}$,
where $\pi_{\tsmthm}$ is the projection $\tsmthm\to M\times\hM$; $(x,\hx;A)\mapsto (x,\hx)$.
In cases we will be concerned with, $\mc{O}$ will be either $Q$,
or an appropriate submanifold of the rolling orbit $\mc{O}_{\dr}(q_0)$.

If  $\tbar \in C^{\infty} (\pi_{\Oh}, \pi_{T_{m}^{k} (M \times \hM)})$
(i.e., $\tbar$ is a smooth map $\tbar : \Oh \rarrow T_{m}^{k} (M \times \hM) $
satisfying $\pi_{T_{m}^{k} (M \times \hM)} \circ \tbar = \pi_{\Oh}$)
and if $q =(x, \hx; A) \in \Oh$ and $\Xbar=(X,\hX) \in T_{(x,\hx)} (M \times \hM)$ are such that $\lns (\Xbar) |_{q} \in T_{q} \Oh$,
then one can define the derivative $\lns(\Xbar) |_{q} \tbar$
of $\ol{T}$ with respect to $\lns(\Xbar) |_{q}$ in the following (tensorial) manner.

If $\overline{\omega} \in \Gamma (\pi_{T_{k}^{m} (M \times \hM)})$ (i.e., an $(m,k)$-tensor field on $M\times\hM$)
and if we write $(\tbar \overline{\omega}) (q) := \tbar (q) \overline{\omega} |_{(x,\hx)}$ as the full contraction,
then one defines
$\lns (\Xbar) |_{q} \tbar$ as the unique element of $(T^k_m)_{(x,\hx)}(M\times\hM)$
whose full contraction with $\overline{\omega} |_{(x,\hx)}$,
for all $\overline{\omega} \in \Gamma (\pi_{T_{k}^{m} (M \times \hM)})$, is
\begin{align}\label{e1.8}
(\lns (\Xbar) |_{q} \tbar) \overline{\omega}|_{(x,\hx)} := \lns (\Xbar) |_{q} (\tbar \overline{\omega} )  - \tbar (q) \nablabar_{\Xbar} \overline{\omega},
\end{align}
where on the right hand side
$\lns (\Xbar) |_{q} (\tbar \overline{\omega} )$ is defined using \eqref{e1.5}
(note that $\tbar \overline{\omega}\in C^\infty(\mc{O})$),
and the connection $\ol{\nabla}$ on $M\times\hM$ is the product of the (Levi-Civita) connections $\nabla$ and $\hnabla$.
It is readily checked that this definition for $\lns (\Xbar) |_{q} \tbar$ is well posed,
and that it is compatible with the formula \eqref{e1.5}
(when applied in the special case of $(k,m)=(0,0)$ i.e., $\ol{T}\in C^\infty(\mc{O})$).

The derivatives $\lns(\ol{X})\onq \ol{T}$ act as certain kinds of horizontal derivatives,
and we would like to complement them with the concept of vertical derivatives $\nu(U)\onq \ol{T}$,
that we now define.

If again $\tbar \in C^{\infty} (\pi_{\Oh}, \pi_{T_{m}^{k} (M \times \hM)})$,
and if $q =(x, \hx; A) \in \Oh$ and $U\in (T^*M\otimes T\hM)_{(x,\hx)}$
are such that $\nu(U)|_{q} \in T_{q} \Oh$ (see \eqref{eq:def:vert}),
then the derivative of $\nu(U)|_{q}\ol{T}$ of $\ol{T}$ with respect to $\nu(U)|_{q}$
can be defined as
\[
\nu(U)|_q \ol{T}:=\dif{t}\big|_0 \widetilde{\ol{T}}(\tilde{q}(t)),
\]
where $\tilde{q}(t)$, for $t\in ]-a,a[$, $a>0$,
is any smooth curve in $\mc{O}$ such that $\tilde{q}(0)=q$,
$\dot{\tilde{q}}(0)=\nu(U)|_q$,
and $\nu(U)|_q$ is defined by \eqref{eq:def:vert}.
Here the definition of $\nu(U)|_q$ by \eqref{eq:def:vert}
makes sense if we take $\pi_{\tsmthm}:\tsmthm\to M\times\hM$
as the vector bundle $\pi:E\to M$ in there.

To conclude this section, we present the general Lie-bracket formulas with respect to
vector fields of the form $\lns(\ol{T})$ and $\nu(U)$
involving mappings $\ol{T}$ and $U$ that we will encounter throughout this paper.
The proof of the following result is presented e.g. in \cite{ChitourKokkonen}.

\begin{proposition}\label{p1.7}
Let $\Oh \subset \tsmthm$ be an immersed submanifold,
$\tbar = (T , \hT)$, $\sbar = (S, \hS) \in C^{\infty} (\pi_{\Oh}, \pi_{T (M \times \hM)})$
be such that $\lns (\tbar (q)) |_{q}, \lns (\sbar (q)) |_{q} \in T_{q} \Oh$  for all $\q \in \Oh$,
and $U$, $V \in C^{\infty} (\pi_{\Oh}, \pi_{\tsmthm})$
be such that $\nu (U(q)) \onq$, $\nu(V(q)) \onq \in T_{q} \Oh$ for all $\q \in \Oh$.
Then one has at every $\q\in \mc{O}$,
\[
[\lns(\tbar (\cdot)), \lns(\sbar(\cdot))] \onq  ={}& \lns \big(\lns(\tbar (q)) \onq \sbar - \lns(\sbar (q)) \onq \tbar\big) \onq \\[2mm]
{}&
+ \nu \big(A R(T(q), S(q)) - \hR (\hT(q) , \hS (q)) A\big) \onq,\\[2mm]
[\lns(\tbar (\cdot)), \nu(U(\cdot))] \onq  ={}& - \lns \big(\nu(U (q)) \onq \tbar\big) \onq + \nu \big(\lns( \tbar (q)) \onq U \big) \onq, \\[2mm]
[\nu(U(\cdot)) , \nu(V(\cdot))] \onq ={}& \nu \big(\nu(U(q)) \onq V - \nu(V(q)) \onq U\big)\onq.
\]
Furthermore, both sides of these three equalities are tangent to $\Oh$.
\end{proposition}

\vspace{2\baselineskip}\subsection{Rolling Curvature}\label{s11}\ \newline

The following concept that measures the difference of the curvature tensors of the
two spaces $(M,g)$, $(\hM,\hg)$
will appear in several occasions in what follows.
The Riemannian curvature of $(M,g)$ will be written as $R$, while that of $(\hM,\hg)$ as $\hR$.

\begin{definition}
For $q=(x,\hat{x};A)\in Q$, we define the \emph{rolling curvature} $\Rol_q$ at $q$ by
\[
\Rol_q(X,Y):= AR(X,Y) - \hR (AX,AY) A,\quad X,Y\in T_xM.
\]
If $X,Y\in\VF(M)$, we write $\Rol(X,Y)$ for the map $Q\to \tsmthm$; $q\mapsto \Rol_q(X,Y)$.
\end{definition}

This quantity will make appearance right away in the Lie-bracket
formulas of first order of vector fields tangent to $\dr$.

\begin{proposition}
If $X$, $Y \in \VF(M)$, $\q \in Q$, then
\begin{align}\label{e1.11}
[\lr(X), \lr(Y)] \onq = \lr ([X,Y]) \onq + \nu (\Rol_q(X,Y))\onq.
\end{align}
\end{proposition}

\begin{proof}
Consequence of the first Lie-bracket equality in Proposition \ref{p1.7}.
Alternatively, see \cite{ChitourKokkonen} for a proof.
\end{proof}

\begin{proposition}\label{p1.12}
Suppose that $(M,g)$ and $(\hM,\hg)$ are complete Riemannian manifolds of dimensions $n=\dim M$, $\hn=\dim\hM$,
and let $\qz\in Q$.
Then the following assertions are equivalent:
\begin{enumerate}
\item[(i)] The orbit $\odr(q_0) $ is an integral manifold of $\dr$.
\item[(ii)] $\Rol_q (X,Y) =0$ for all $ \q \in \odr(q_0)$ and $X,Y \in T_{x} M$.
\item[(iii)]
When $n \leq \hn$, there is a complete Riemannian manifold $(N,h)$, a Riemannian covering map $F: N \rarrow M$ and a Riemannian immersion $G: N \rarrow \hM$ that maps $h$-geodesics to $\hg$-geodesics.
In addition, $A_0=G_*|_{y_0}\circ (F_*|_{y_0})^{-1}$
where $y_0$ is any point in $N$ such that $x_0=F(y_0)$.

\end{enumerate}
\end{proposition}

\vspace{2\baselineskip}\section{Rolling of 2-dimensional against 3-dimensional Riemannian manifolds}\label{s1}

\vspace{\baselineskip}\subsection{Preliminaries}\label{sec:preliminaries}
\ \newline

For the rest of this paper, we shall assume that $(M,g)$ and $(\hM,\hg)$ are oriented Riemannian manifolds of dimension $2$ and $3$, respectively. Suppose $(X,Y)$ is an oriented local $g$-orthonormal frame on $M$,
defined on some open set $V\subset M$.
Usually we will want, for convenience, this $V$ to be the open set $\pi_{Q,M}(O(q_0))$,
where $O(q_0)$ will be soon introduced in Definition \eqref{def:rolN} below,
but for the time being it can be any open subset of $M$.
For $\q\in Q$ such that $x\in V$ we set $\hZ_{A} := \star (AX|_x \wedge AY|_x) \in T_{\hx} M$,
with $\star$ denoting the Hodge dual (in $(\hM,\hg$)).
Clearly then $AX|_x$, $AY|_x$, $\hZ_{A}$ is an orthonormal basis of $T_{\hx} \hM$.

If $\theta_{X}:=g(X,\cdot)$, $\theta_{Y}:=g(Y,\cdot)$ are the $g$-duals of $X$, $Y$, respectively, then we remark that
\[
\begin{array}{l}
(\star \hZ_{A}) A = (AX \wedge AY ) A = A (X \wedge Y),\\
(\star AY) A = (\hZ_{A} \wedge AX ) A = - \theta_{X} \otimes \hZ_{A},\\
(\star AX) A = (AY \wedge \hZ_{A}) A = \theta_{Y} \otimes \hZ_{A},
\end{array}
\]
where use of the identification \eqref{eq:two_vector} has been made.

The \emph{connection table} $\Gamma$ of the $2$-dimensional manifold $(M,g)$ (or of $\nabla$)
with respect to the frame $X,Y$
has the matrix form (on the set $V$)
\[
\Gamma=\qmatrix{
\Gamma^1_{(1,2)} & \Gamma^2_{(1,2)}
}
\]
where the connection coefficients are $\Gamma^i_{(j,k)}:=g(\nabla_{E_i} E_j,E_k)$ with $(E_1,E_2)=(X,Y)$,
that is $\Gamma_{(1,2)}^1 = g (\nabla_{X} X , Y)$ and $\Gamma_{(1,2)}^2 = g (\nabla_{Y} X , Y)$.
In other words, the information of this table can be encoded into the following vector, which we also call $\Gamma$,
\begin{align}\label{eq:Gamma}
\Gamma := \Gamma_{(1,2)}^1 X + \Gamma_{(1,2)}^2 Y.
\end{align}
Consequently, the Gaussian curvature $K$ of $M$ reads
\begin{eqnarray}\label{e2.2}
\begin{array}{rl}
K (x) := g (R(X,Y)Y, X) & = g((X(\Gamma_{(1,2)}^2) - Y (\Gamma_{(1,2)}^1)  + ((\Gamma_{(1,2)}^1)^2 + (\Gamma_{(1,2)}^2)^2 ) X \wedge Y) Y, X)\\
& = Y (\Gamma_{(1,2)}^1) -  X(\Gamma_{(1,2)}^2) - ((\Gamma_{(1,2)}^1)^2 + (\Gamma_{(1,2)}^2)^2 )\\
& = g (\nabla_Y \Gamma, X) - g (\nabla_X \Gamma, Y).
\end{array}
\end{eqnarray}

We also mention here that if $\hE_1$, $\hE_2$, $\hE_3$ is a local (oriented) orthonormal frame of $(\hM,\hg)$ defined on some open subset $\hV$ of $\hM$,
then by the connection table $\hGamma$ of $(\hM,\hg)$ (or of $\hnabla$) with respect to $\hE_1$, $\hE_2$, $\hE_3$
we mean the matrix
\[
\hGamma =
\qmatrix{
\hGamma_{(2,3)}^1 & \hGamma_{(2,3)}^2 & \hGamma_{(2,3)}^3 \\
\hGamma_{(3,1)}^1 & \hGamma_{(3,1)}^2 & \hGamma_{(3,1)}^3 \\
\hGamma_{(1,2)}^1 & \hGamma_{(1,2)}^2 & \hGamma_{(1,2)}^3 
},
\]
where $\hGamma^i_{(j,k)}=\hg(\hnabla_{\hE_i} \hE_j,\hE_k)$.

The locally defined curvature functions on $Q$ essential to our analysis throughout this paper are
the following:
\begin{align}\label{eq:hsigma_Pi}
\hsigma_{A}^1 : =&{} \hg (\hR (AY, \hZ_{A}) \hZ_{A}, AY) = - \hg (\hR (\star A X), \star A X), \nonumber \\
\hsigma_{A}^2 : =&{} \hg (\hR (AX, \hZ_{A}) \hZ_{A}, AX) = - \hg (\hR (\star A Y), \star A Y), \nonumber \\
\hsigma_{A}^3 : =&{} \hg (\hR (AX, AY) AY, AX) = - \hg (\hR (\star \hZ_{A}), \star \hZ_{A}), \nonumber \\
\Pi_{X} (q) : =&{} \hg (\hR (\star \hZ_{A}), \star AX), \nonumber \\
\Pi_{Y} (q) : =&{} \hg (\hR (\star \hZ_{A}), \star AY), \nonumber \\
\Pi_{Z} (q) : =&{} \hg (\hR (\star A X), \star AY),
\end{align}
where $\q$ is in the open set $(\pi_{Q,M})^{-1}(V)$ of $Q$, and $V\subset M$ is the open domain of definition
of the local orthonormal frame $(X,Y)$ of $M$.

One notices that $\hsigma_{A}^i$, $i=1,2,3$, are the sectional curvatures on $\hM$ at $\hx$
along planes defined by pairs of vectors in $AX$, $AY$, $\hZ_A$.
For notational convenience we shall be omitting the superscript $3$ from $\hsigma_{A}^3$
and therefore write 
\[
\hsigma_A:=\hsigma_{A}^3.
\]

We also point out that the function $\sigma_{(\cdot)}$ i.e., $\q\mapsto \sigma_A$ on $Q$ 
as well as the curvature $K$ on $M$
are globally defined on these respective spaces,
and do not depend on the choice of the oriented orthonormal frame $X,Y$ on $M$,
i.e., they are intrinsically defined.
Similarly the function $q\mapsto \Pi_X(q)^2+\Pi_Y(q)^2$ is globally and intrinsically defined on $Q$.
Indeed, thanks to the orientability of $(M,g)$ and $(\hM,\hg)$ (by assumption),
the map $Q\to T\hM$; $\q\mapsto \hZ_A$ is intrinsically defined, and therefore so is
\[
\q\mapsto \hg(\hR(\star \hZ_A),\hR(\star \hZ_A))-\hg(\hR(\star \hZ_A),\star \hZ_A)^2=\Pi_X(q)^2+\Pi_Y(q)^2.
\]
The moral of these observations is that the assumptions $\sigma_{(\cdot)}-K=0$ or $\neq 0$,
and $(\Pi_X,\Pi_Y)=(0,0)$ or $\neq (0,0)$ that are present in all the main results of this paper,
actually do not depend on the choice of a particular frame $(X,Y)$.

Even though the curvature $K$ of $(M,g)$ is a function on $M$, not in $Q$,
we will often identify it implicitly with the function $K\circ \pi_{Q,M}$ on $Q$.
More generally, if $f:M\to N$ (resp. $\hf:\hM\to N$) is a map from $M$ to a some space $N$
(resp. form $\hM$ to $N$), we will frequently write $f$ for the associated map $f\circ \pi_{Q,M}$
(resp. $\hf$ for $\hf\circ \pi_{Q,\hM}$) from $Q$ to $N$.

Given the above concepts and notations, we find that
\begin{align}\label{eq:RolXY}
\Rol_{q} (X,Y) & = A R(X,Y) - \hR(AX,AY)A \nonumber \\
& = A ((- K) (X \wedge Y)) - (\Pi_X \star AX + \Pi_Y \star AY + (-\hsigma_{A}) \star \hZ_{A} )A \nonumber \\
& = - ( K - \hsigma_{A} ) A (X \wedge Y)  + \Pi_Y \theta_{X} \otimes \hZ_{A}  - \Pi_X \theta_{Y} \otimes \hZ_{A} \nonumber \\
& = 
\qmatrix{
0 & - ( K - \hsigma_{A} ) \\
K - \hsigma_{A} & 0\\
\Pi_Y & - \Pi_X,
}
\end{align}
for $\q\in (\pi_{Q,M})^{-1}(V)$.
In the rest of the text, if there is no risk of confusion, we will write $\Rol_q (X,Y)$ simply as $\Rol_q$.

Virtually all the results in this paper that follow are local in character,
and we will systematically rely on the following localization procedure, which involves some notations.

\begin{definition}\label{def:rolN}
For $q_0=(x_0,\hx_0; A_0) \in Q$, we say that a subset  $O(q_0)$ of $Q$ is a \emph{rolling neighborhood} 
of $q_0$ if there exists a connected open neighborhood $V=V(q_0)$ of $x_0\in M$
such that $O(q_0)$ consists of the points $q(t)$ traced by all
rolling curves $t \mapsto q(t)=(\gamma(t), \hgamma(t); A(t))$, $t\in [0,1]$, in $Q$ starting at $q_0$ (i.e., $q(0)=q_0$) and 
such that $\gamma(t)\in V$ for all $t\in [0,1]$.
In particular, one has $O(q_0)\subset \odr (q_0)$
and $V(q_0)=\pi_{Q,M}(O(q_0))$.
\end{definition}

It should be emphasized that a rolling neighbourhood $O(q_0)$ of $q_0$ need not be an open subset of $\odr(q_0)$
(hence nor of $Q$).
Stated in another way, $O(q_0)$ is nothing else than the orbit through $q_0$ of the distribution $\dr$ restricted
onto the open subset $(\pi_{Q,M})^{-1}(V(q_0))$ of $Q=Q(M,\hM)$.
In that sense, $O(q_0)$ is a \emph{local orbit} of $\dr$ passing through $q_0$.
That said, it is clear that $O(q_0)$ is a smooth submanifold of $\odr(q_0)$ and hence of $Q$.
Moreover, on several occasions in the text below,
we shall say that we shrink $O(q_0)$ or make $O(q_0)$ smaller around $q_0$, or that we take small enough $O(q_0)$ around $q_0$, when necessary,
since we are mainly concerned with local results. This shrinking (or localization) procedure simply means that, if we write $\qz$,
then we are actually choosing first an appropriately small, connected neighbourhood $V$ of $x_0$,
and then let that shrunk $O(q_0)$
to be the orbit through $q_0$ of the restriction of $\dr$ onto $(\pi_{Q,M})^{-1}(V)$.

\vspace{2\baselineskip}\subsection{Main Results}
\ \newline

The main results of this paper are summarized in the following theorem.

\begin{theorem}\label{main-theorem}
Let $(M,g)$ and $(\hM,\hg)$ be two connected oriented Riemannian manifolds of dimensions 2 and 3, respectively.
Given $\qz\in Q$ and a small enough rolling neighbourhood $O(q_0)$ around $q_0$,
and writing $V=\pi_{Q,M}(O(q_0))$, $\hV=\pi_{Q,\hM}(O(q_0))$,
the following assertions hold:
\begin{itemize}

\item[(i)]
Suppose that $(\Pi_X, \Pi_Y) = (0,0)$ and $\hsigma_{(\cdot)} - K=0$ everywhere on $O(q_0)$.
Then
\begin{enumerate}
\item $\dim O(q_0)=2$;

\item the map $\pi_{Q, M}|_{O(q_0)}$ is a diffeomorphism onto its open image $V=V(q_0)$ in $M$,
and $H:=(\pi_{Q, \hM}|_{O(q_0)}) \circ (\pi_{Q, M}|_{O(q_0)})^{-1}$
maps $(V,g|_V)$ isometrically onto a 2-dimensional \emph{totally geodesic} submanifold $\hN$ of $(\hat{M},\hat{g})$;

\item $A_0 T_{x_0} M=T_{\hx_0} \hat{N}$;

\item $A_0=H_*|_{x_0}$.
\end{enumerate}

\item[(ii)]
Suppose that $(\Pi_X , \Pi_Y) = (0,0)$ and $\hsigma_{(\cdot)} - K\neq 0$ everywhere on $O(q_0)$.
Then
\begin{enumerate}
\item $\dim O(q_0) = 5$;

\item the map $\pi_{Q,\hat{M}}|_{O(q_0)}$ has constant rank $2$,
and its image $\hN$ is a 2-dimensional embedded, \emph{totally geodesic} submanifold of $\hat{M}$.

\item $A_0 T_{x_0} M=T_{\hx_0} \hat{N}$;

\item There does not exist (local) isometric isomorphism $H:(V,g|_V)\to (\hN, \hg|_{\hN})$ such that $H_*|_{x_0}=A_0$.
\end{enumerate}

\item[(iii)]
Suppose that $(\Pi_X,\Pi_Y)\neq (0,0)$ and $\hsigma_{(\cdot)} - K=0$ everywhere on $O(q_0)$.
Then
\begin{enumerate}
\item $\dim O(q_0)=7$;

\item the curvature $K$ of $(M,g)$ is \emph{constant} on $V$ and $K\neq 0$;

\item $\hV$ is an open neighbourhood of $\hx_0$ in $\hM$;

\item there exists an isometric isomorphism $\hF:(\hI \times \hN, s_1\oplus \hat{h})\to (\hV,\hg|_{\hV})$
from a \emph{Riemannian product onto} $\hV$,
where $\hI\subset\R$ is a non-empty open interval equipped with its usual metric $s_1$
and $(\hN,\hh)$ is a 2-dimensional connected Riemannian manifold
whose curvature $K^{\hN}$ is \emph{constant} and $K^{\hN}\neq 0$;

\item $A_0 T_{x_0} M\neq T_{\hat{x}_0} \hN$,
if $\hN$ is identified with the embedded submanifold $\hF(\{\hr_0\}\times \hN)$ of $\hM$,
where $\hr_0\in \hI$ is determined from $(\hr_0,\hy_0)=\hF^{-1}(\hx_0)$;

\item the space $A_0 T_{x_0} M$ does not contain the normal space of $\hN$ at $\hx_0$ in $(\hM,\hg)$;

\item If $\alpha_0:=\hg\big(\hZ_{A_0}, \pa{\hr}\big|_{\hr_0}\big)$, then $0<|\alpha_0|<1$ and
\[
\alpha_0^2 K^{\hN}=K,
\]
where
$\pa{\hr}$ is the canonical vector field on $\hI$,
and
$\hZ_{A_0}$ is any unit vector in $T_{\hx_0} \hM$ orthogonal to the subspace $A_0 T_{x_0} M$.
In particular, the curvatures $K$ and $K^{\hN}$ have the same sign.
\end{enumerate}

\item[(iv)]
Suppose that $(\Pi_X,\Pi_Y)\neq (0,0)$ and $\hsigma_{(\cdot)} - K\neq 0$ everywhere on $O(q_0)$.
Consider the following functions which are defined on an open neighbourhood $U$ of $O(q_0)$ in $Q$,
\begin{align}\label{eq:main-theorem:c4:1}
G_{\Xtilde},\ \omega G_{\Ytilde}-H_{\Xtilde},\ H_{\Ytilde},\ 
\omega \nu (\theta_{\YtildeA} \otimes \hZ_A) \onq \phi-1,\ 
\nu (\theta_{\YtildeA} \otimes \hZ_A) \onq \omega,
\end{align}
where for $\q\in U$ we write
$\cphi(q)=\cos(\phi(q))$, $\sphi(q)=\sin(\phi(q))$,
and define
\begin{align*}
& \omega(q) \cphi(q) := \frac{\Pi_X(q)}{K(x)-\hsigma_A},\quad
&& \omega(q) \sphi(q) := \frac{\Pi_Y(q)}{K(x)-\hsigma_A},\\
& \XtildeA := \cphi(q) X|_x + \sphi(q) Y|_x, \quad
&& \YtildeA := - \sphi(q) X|_x + \cphi(q) Y|_x, \\
& G_{\Xtilde} : = \lr (\XtildeA) \onq \phi + g (\Gamma , \XtildeA),\quad
&& G_{\Ytilde} : = \lr (\YtildeA) \onq \phi + g (\Gamma , \YtildeA),\\
& H_{\Xtilde} : = \lr (\XtildeA) \onq \omega, \quad
&& H_{\Ytilde} : = \lr (\YtildeA) \onq \omega.
\end{align*}

%

Assume first that all the functions in \eqref{eq:main-theorem:c4:1} are identically equal to zero on $O(q_0)$. Then,
\begin{enumerate}
\item $\dim O(q_0)=5$;

\item $\hV$ is an open neighbourhood of $\hx_0$ in $\hM$;

\item there exists an isometry $F:(I\times N, s_1\oplus_{f} h)\to (V,g|_V)$
from a \emph{warped product} onto $V$,
where $I\subset\R$ is a non-empty open interval equipped with its usual metric $s_1$
$(N,h)$ is some 1-dimensional connected Riemannian manifold,
and $f\in C^\infty(I)$ is a warping function;

\item there exists an isometry $\hF:(\hI\times \hN, s_1\oplus_{\hf} \hh)\to (\hV,\hg|_{\hV})$
from a \emph{warped product} onto $\hV$,
where $\hI\subset\R$ is a non-empty open interval equipped with its usual metric $s_1$
$(\hN,\hh)$ is some 2-dimensional connected Riemannian manifold,
and $\hf\in C^\infty(\hI)$ is a warping function;
 
\item
the warping functions $f$, $\hf$ are implicitly related by
\[
\quad
\frac{\hf''(\hr)}{\hf(\hr)}=\frac{f''(r)}{f(r)},
\]
holding at every $r\in I$, $\hr\in\hI$ for which there are $y\in N$ and $\hy\in\hN$
such that $(F(r,y),\hF(\hr,\hy))\in \pi_Q(O(q_0))$;

\item $A_0 T_{x_0} M\neq T_{\hx_0} \hN$, if $\hN$ is identified with the embedded submanifold $\hF(\{\hr_0\}\times \hN)$ of $\hM$,
where $\hr_0\in I$ is determined from $(\hr_0,\hy_0)=\hF^{-1}(\hx_0)$;

\item the space $A_0 T_{x_0} M$ does not contain the normal space of $\hN$ at $\hx_0$ in $(\hM,\hg)$;

\item if $\hf'$ is constant on all of $\hI$, then $(V,g|_V)$ is \emph{flat};

\item if $\hf'\neq 0$ on all of $\hI$, then $(\hN,\hh)$ is \emph{flat} and
for $a>0$ small enough, the warping function $f$ is determined from the system
of relations
\[
{}& \frac{\hf(\hr(t))^2}{\hf(\hr_0)^2}-1=P_0^2\Big(\frac{f(t+r_0)^2}{f(r_0)^2}-1\Big), \quad t\in ]-a,a[, \\
{}& (\hr'(t))^2+(1-P_0^2)\frac{\hf(\hr_0)^2}{\hf(\hr(t))^2}=1,\quad \hr(0)=\hr_0,\quad \hr'(0)=P_0,
\]
where
$P_0:=\hg(A_0\pa{r}\big|_{r_0},\pa{\hr}\big|_{\hr_0})$,
$(r_0,y_0)=F^{-1}(x_0)$, $(\hr_0,\hy_0)=\hF^{-1}(\hx_0)$,
and $\pa{r}$, $\pa{\hr}$ are the canonical vector fields on $I$, $\hI$, respectively.
Moreover, $0<\vert P_0\vert<1$ and $\hr'(t)\neq 0$ for all $t\in ]-a,a[$.

\end{enumerate}
On the other hand, assume that one of the functions defined in \eqref{eq:main-theorem:c4:1} is not equal to zero at $q_0$. Then, $\dim O(q_0)\geq 6$.
 
\end{itemize}
\end{theorem}

The proof of this theorem will take effectively the rest of this paper,
and arguments employed to showing the assertions of items (i)-(iv) turn out to be increasingly elaborate.

Simplest of the cases, item (i), is quickly taken care of in subsection \ref{sec:case_I}, see Theorem \ref{th:p1}.
The item (ii) is the topic of subsection \ref{sec:case_II}, whose main results are Proposition \ref{p2.1} and Theorem \ref{th:2.2}.
Subsection \ref{ss10} is dedicated to the proof of item (iii), which already requires a number of steps to complete.
The key results there are Proposition \ref{p2.2}, Proposition \ref{pr:5.14}, Corollary \ref{cor:ss10:const_alpha} and Theorem \ref{th:5.17}.

Finally, the arguments for proving the assertions of item (iv) will be given in subsection \ref{ss11}.
This particular case, while being under the constraints imposed by the five remarkable relations given in \eqref{eq:main-theorem:c4:1},
requires by far more work than the other cases (i)-(iii) combined
(at least with the methods employed in the present text).
Discovering all the properties listed in item (iv) requires several steps to be taken,
the central results being stated in
Proposition \ref{pf1}, Corollary \ref{cor:ss11:warped_hM}, Propositions \ref{pr:ss11:hN_hh_flat} and \ref{pr:ss11:V_g_flat},
Corollary \ref{cor:ss11:warped_M},
Lemma \ref{le:ss11:2.5},
and they are ultimately summarized in Theorem \ref{th:ss11:main}. 

We deduce immediately from the previous theorem the following proposition, which actually rephrases some of the results of Theorem~\ref{main-theorem}.
\begin{proposition}\label{p0.0}
Let $(M,g)$ and $(\hM,\hg)$ be two connected oriented Riemannian manifolds of dimensions 2 and 3, respectively.
Given $\qz\in Q$ and a small enough rolling neighbourhood $O(q_0)$ around $q_0$,
and writing $V=\pi_{Q,M}(O(q_0))$, $\hV=\pi_{Q,\hM}(O(q_0))$,
the following assertions hold:
\begin{itemize}
\item[(a)] $\dim O(q_0)$ takes values in $\{2,5,6,7,8\}$ and is open only if it is equal to $8$;
\item[(b)] $\dim O(q_0)=2$ if and only if Sub-items (2), (3) and (4) of Item (ii) in Theorem~\ref{main-theorem} hold true;

\item[(b)] assume that $\dim O(q_0)=5$. Then either $\hM$ contains a 2-dimensional embedded, totally geodesic submanifold $\hN$ or Sub-items (2) to (8) of Item (iv) in Theorem~\ref{main-theorem} hold true.
\end{itemize}
\end{proposition}

\vspace{2\baselineskip}\section{Case $(\Pi_X , \Pi_Y)\equiv  (0,0)$ on $O(q_0)$}\label{sec:5}\ \newline

This section is itself divided into two parts,
the first one addressing the case where $\hsigma_A \equiv K (x) $ for any $\q$ in $ O(q_0)$,
while in the second part we assume $\hsigma_A \neq K (x) $ for any $\q$ in $O(q_0)$.
In addition, in both cases we assume that $(\Pi_X,\Pi_Y)=(0,0)$ on $O(q_0)$.
The set $O(q_0)$ is some rolling neighbourhood of $q_0$, which we keep making smaller around $q_0$ whenever appropriate.

In the following $O(q_0)$ is a (small enough) rolling neighbourhood of the given point $\qz\in Q$,
while $V=\pi_{Q,M}(O(q_0))$ as in Definition \eqref{def:rolN}.
It is assumed that the local orthonormal frame $(X,Y)$ of $M$ is defined on the open set $V$.

\vspace{2\baselineskip}\subsection{Subcase $\hsigma_A \equiv K (x) $ on $O(q_0)$}\label{sec:case_I}
\ \newline

In this case, one can essentially say that the open set $V$ of $(M,g)$ is isometric
to a totally geodesic submanifold of $(\hat{M},\hat{g})$.
More precisely we have the following:
 
\begin{theorem}\label{th:p1}
Let $\qz \in Q$ and assume that 
$(\Pi_X(q) , \Pi_Y(q)) = (0,0)$ and $\hsigma_A = K (x)$ for all $\q\in O(q_0)$
where $O(q_0)$ is some rolling neighbourhood of $q_0$.
Then, after shrinking $O(q_0)$ around $q_0$ if necessary,
the map $\pi_{Q, M}|_{O(q_0)}$ is a diffeomorphism onto its open image $V=V(q_0)$ in $M$,
and $H:=(\pi_{Q, \hM}|_{O(q_0)}) \circ (\pi_{Q, M}|_{O(q_0)})^{-1}$
maps $(V,g|_V)$ isometrically onto a 2-dimensional embedded, totally geodesic submanifold $\hN$ of $(\hat{M},\hat{g})$.
Furthermore, $A_0=H_*|_{x_0}$, $A_0 T_{x_0} M=T_{\hx_0} \hat{N}$ and $\dim O(q_0)=2$.
\end{theorem}

\begin{proof}
By \eqref{eq:RolXY} the condition that
$(\Pi_X , \Pi_Y) = (0,0)$ and $\hsigma_A \equiv K (x)$
on $O(q_0)$ mean that
$\Rol_q(X,Y)=0$ for all $q\in O(q_0)$.
The result then follows from a local version of Proposition~\ref{p1.12}.
\end{proof}

\vspace{2\baselineskip}\subsection{Subcase $\hsigma_A \neq K (x) $ on $O(q_0)$}\label{sec:case_II}
\ \newline

In this case, it follows that there is a totally geodesic submanifold $N$ of $(\hat{M},\hat{g})$,
that is not isometric to $(M,g)$ but $A_0$ maps $T_{x_0} M$ to a tangent space of $N$.
The aim of this subsection and results in it is to formulate and prove rigorously that statement.

\begin{proposition}\label{p2.1}
Let $\qz \in Q$ and assume that
$(\Pi_X , \Pi_Y) = (0,0)$ on $O(q_0)$, while $\hsigma_A - K (x)$ does not vanish on $O(q_0)$.
Then, after shrinking $O(q_0)$ around $q_0$ if need be,
we have $\dim O(q_0) = 5$,
the map $\pi_{Q,\hat{M}}|_{O(q_0)}$ has constant rank $2$,
and its image $\hN$ is a 2-dimensional embedded submanifold of $\hat{M}$.
\end{proposition}

\begin{proof}
According to the stated assumptions and \eqref{eq:RolXY} one has
\[
\nu(\Rol_q (X,Y))\onq = (K(x) - \hsigma_A) \nu (A (X \wedge Y)) \onq,\quad \forall \q \in O(q_0),
\]
where $K(x)-\hsigma_A\neq 0$ by assumption.
Since this vector field is tangent to $O(q_0)$,
it follows that so are
$[\lr (X),\nu ((\cdot) (X \wedge Y))] \onq$ and $[\lr (Y),\nu ((\cdot)(X \wedge Y))] \onq$,
and hence Lemma \ref{l2.1}
implies that $\lns (A X) \onq$ and $\lns (A Y) \onq$ are tangent to $O(q_0)$ as well
at every $\q\in O(q_0)$.
Thus the linearly independent vector fields
\begin{align}\label{eq:involutive_dim5}
\lr (X) \onq,\ \lr (Y) \onq,\ \nu (A (X \wedge Y)) \onq,\ \lns (A X) \onq,\ \lns (A Y) \onq,
\end{align}
are tangent to $O(q_0)$ at all points $\q\in O(q_0)$.
Moreover, by Lie-bracket formulas listed Lemma \ref{l2.1},
and taking into account that $\Pi_X$ and $\Pi_Y$ vanish on $O(q_0)$ by assumption,
one concludes that the 5 vector fields in \eqref{eq:involutive_dim5} tangent to $O(q_0)$
form an involutive system on
the local orbit manifold $O(q_0)$ of $\dr$.
From this observation it follows that $O(q_0)$ has dimension 5.
Finally, it is easily seen that $(\pi_{Q,\hat{M}}|_{O(q_0)})_*$ maps vector fields \eqref{eq:involutive_dim5}
onto $\spn\{AX, AY\}$ at each point $\q\in O(q_0)$,
so that the map $\pi_{Q,\hat{M}}|_{O(q_0)}$ has indeed constant rank 2 as claimed.
This completes the proof.
\end{proof}

\begin{proposition}\label{pr:c2.1}
Under the assumptions and notations of the statement of Proposition \ref{p2.1},
and after shrinking $O(q_0)$ around $q_0$ if necessary,
there is a local oriented orthonormal frame $(\hE_1, \hE_2, \hE_3)$ of $(\hM,\hg)$ defined on a neighbourhood
of $\hx_0$ in $\hM$
containing the smooth 2-dimensional submanifold $\hN:=\pi_{Q,\hat{M}}(O(q_0))$ of $\hM$
such that
\begin{itemize}
\item[(i)]
$\hE_1$, $\hE_2$ are tangent to $\hN$, $\hE_3$ is a unit normal vector of $\hN$;
\item[(ii)] $\spn\{\hE_1|_{\hx},\hE_2|_{\hx}\}=\spn\{AX,AY\}=T_{\hx} \hN$ and $\hE_3=\hZ_A$ for all $\q\in O(q_0)$;
\item[(iii)] at every point of $\hN$,
the connection table of $(\hM, \hg)$ with respect to $\hE_1, \hE_2, \hE_3$ has the form
\begin{align}\label{e2.7}
\hGamma =
\qmatrix{
0 & 0 & \hGamma_{(2,3)}^3 \\
0 & 0 & \hGamma_{(3,1)}^3 \\
\hGamma_{(1,2)}^1 & \hGamma_{(1,2)}^2 & \hGamma_{(1,2)}^3
}
\end{align}
\end{itemize}	
\end{proposition}

\begin{proof}
By Proposition \ref{p2.1} the image $\hN$ of $O(q_0)$ by $\pi_{Q,\hM}|_{O(q_0)}$
is a 2-dimensional smooth submanifold of $\hM$,
whose tangent space $T_{\hx}\hN$ at every point $\hx$ such that $\q$ for some $q\in O(q_0)$
is spanned by $\{AX, AY\}$,
and whose normal space at $\hx$ is spanned by $\hZ_A$.

If needed, we can shrink $O(q_0)$ around the point $q_0$ again,
resulting in the shrinking of the 2-dimensional submanifold $\hN$ of $\hM$ containing $\hx_0$
to the point where we can assume
the existence of a local orthonormal frame $\{ \hE_1,  \hE_2 , \hE_3\}$ of $\hat{M}$
defined on an open neighbourhood of $\hx_0$ in $\hat{M}$ that contains $\hN$ as a subset,
and such that $\hE_3$ is normal to $\hN$ at every point of $\hN$.

Up to multiplying $\hE_3$ with $-1$ if necessary, we have $\hE_3|_{\hx}=\hZ_A$
for every $\q\in O(q_0)$,
and $\spn\{\hE_1|_{\hx},\hE_2|_{\hx}\}=\spn\{AX,AY\}$,
which is $=T_{\hx}\hN$ by the proof of Proposition \ref{p2.1}.
One can also swap $\hE_1$ with $\hE_2$ if needed to ensure that the frame $(\hE_1,\hE_2,\hE_3)$
has the positive orientation on $\hM$.
This proves the assertions given in items (i) and (ii).

Knowing from the proof of Proposition \ref{p2.1} that
the vector fields in \eqref{eq:involutive_dim5} form a frame of $O(q_0)$, we can therefore
conclude that the following five vector fields on $Q$ span
form a frame of $O(q_0)$ as well:
\begin{align}\label{eq:involutive_dim5:2}
\lr (X) \onq ,\quad \lr (Y) \onq,\quad \nu (\star \hE_3 A) \onq,\quad \lns (\hE_1) \onq,\quad \lns (\hE_2) \onq.
\end{align}

Consequently, vector field $F_1$ given by the Lie-bracket (at $\q\in\odr(q_0)$),
\[
F_1 \onq:=[\lns (\hE_1) , \nu (\star \hE_3 (\cdot) )] \onq
=
\nu(\star (\hnabla_{\hE_1} \hE_3)A)|_q
=
\nu (\star (\hGamma_{(3,1)}^1 \hE_1 - \hGamma_{(2,3)}^1 \hE_2) A) \onq,
\]
(see Proposition \ref{p1.7})
must be tangent to $O(q_0)$ as well.

We shall now show that $F_1$ (and later $F_2$) must actually vanishes at every point of $O(q_0)$,
which is equivalent to saying that $\hGamma_{(3,1)}^1$, $\hGamma_{(2,3)}^1$
$\hGamma_{(2,3)}^2$ and $\hGamma_{(3,1)}^2$
vanish at every point of $\hN=\pi_{Q,\hM}(O(q_0))$.

Indeed, the map $\pi_Q:Q\to M\times\hM$ restricted to $O(q_0)$ has constant rank 4 since it maps
the vectors at $\q\in O(q_0)$ listed in \eqref{eq:involutive_dim5:2} to
the following vectors of $M\times\hM$, in the respective order,
$(X|_x,AX|_x)$, $(Y|_x,AY|_x)$, $(0,0)$, $(0,\hE_1|_{\hx})$, $(0,\hE_2|_{\hx})$,
and these vectors span a 4-dimensional subspace of $T_{(x,\hx)} (M\times\hM)$,
since $\spn\{AX|_x,AY|_x\}=\spn\{\hE_1|_{\hx},\hE_2|_{\hx}\}$ as we have already noted.

Because $\dim O(q_0)=5$, we can infer from these observations that
the 1-dimensional kernel of $(\pi_Q|_{O(q_0)})_*$ at any $\q\in O(q_0)$
is spanned by $\nu(\star \hE_3 A)|_q$.
But $F_1|_q$, which we know to be a vector tangent to $O(q_0)$,
also lies in the kernel of $(\pi_Q|_{O(q_0)})_*$ at $\q$,
and therefore it must be parallel to $\nu(\star \hE_3 A)|_q$.
However, the only way
$F_1|_q=\nu (\star (\hGamma_{(3,1)}^1(\hx) \hE_1 - \hGamma_{(2,3)}^1(\hx) \hE_2) A)$
could be parallel to the (nowhere vanishing) vector $\nu(\star \hE_3 A)|_q$
is that
$\hGamma_{(3,1)}^1(\hx) \hE_1|_{\hx} - \hGamma_{(2,3)}^1(\hx) \hE_2|_{\hx}$ be parallel to $\hE_3|_{\hx}$,
and that is only possible if $\hGamma_{(3,1)}^1(\hx)=0$ and $\hGamma_{(2,3)}^1(\hx)=0$.
In conclusion,
\begin{align}\label{eq:gamma1_zero}
\hGamma_{(2,3)}^1=0\quad \mathrm{and}\quad \hGamma_{(3,1)}^1=0\quad \mathrm{everywhere\ on}\ \hN,
\end{align}
i.e., $F_1=0$ everywhere on $O(q_0)$.

By the same reasoning,
one can show that vector field $F_2$ defined by
(see Proposition \ref{p1.7})
\[
F_2|_q:=[\lns (\hE_2) , \nu (\star \hE_3 (\cdot) )] \onq
=\nu(\star (\hnabla_{\hE_2} \hE_3)A)\onq
=\nu(\star (\hGamma^2_{(3,1)}\hE_1 - \hGamma^2_{(2,3)}\hE_2)A)\onq,
\]
is tangent to $O(q_0)$, and must therefore vanish everywhere on $O(q_0)$,
i.e.,
\[
\hGamma_{(2,3)}^2=0\quad \mathrm{and}\quad \hGamma_{(3,1)}^2=0\quad \mathrm{everywhere\ on}\ \hN.
\]
These identities together with those in \eqref{eq:gamma1_zero} implies \eqref{e2.7},
i.e., proves our last case (iii).
\end{proof}

We next derive some results about the geometrical aspect of $\hM$.
Recall that $V=\pi_{Q,M}(O(q_0))$ and that it is an open neighbourhood of $x_0$ in $M$.

\begin{theorem}\label{th:2.2}
Under the assumptions and notations of Proposition \ref{pr:c2.1},
$(\hM, \hg)$ contains a 2-dimensional totally geodesic submanifold $\hN$
such that
\begin{itemize}
\item[(i)] $\hat{N}$ is totally geodesic in $(\hM,\hg)$;
\item[(ii)] $A_0 T_{x_0} M=T_{\hx_0} \hat{N}$;
\item[(iii)] There does not exist (local) isometry $H:(V,g|_V)\to (\hN, \hg|_{\hN})$ such that $H_*|_{x_0}=A_0$.
\end{itemize}
\end{theorem}

\begin{proof}
Write the Levi-Civita connection of $(\hN, \hg|_{\hN})$ as $\hat{D}$,
so that the shape tensor (second fundamental form)
of $\hN$ in $(\hM,\hg)$ is $\mathrm{II}(U,V)=\hat{D}_U V-\hnabla_U V$
for all smooth vector fields $U,V$ on $\hN$.
Then $\mathrm{II}(U,V)$ is normal to $\hN$,
and since
$\hE_1,\hE_2$ span $T\hN$ everywhere on $\hN$ (Proposition \ref{pr:c2.1}, item (ii)),
while $\hE_3$ is everywhere (unit) normal to $\hN$,
if follows that
$\mathrm{II}$ is uniquely determined by (scalar) functions
$\hg(\mathrm{II}(\hE_i,\hE_j),\hE_3)$, $i,j=1,2$.
But for any $i,j=1,2$,
\[
\hg(\mathrm{II}(\hE_i,\hE_j),\hE_3)=-\hg(\hnabla_{\hE_i} \hE_j, \hE_3)=\hg(\hnabla_{\hE_i} \hE_3, \hE_j)=\hGamma^i_{(3,j)},
\]
which vanishes on $\hN$ by item (iii) of Proposition \ref{pr:c2.1},
and therefore the shape tensor $\mathrm{II}$ vanishes,
which means exactly that $\hN$ is totally geodesic submanifold of $(\hM,\hg)$
(cf. \cite{petersen06}).
This completes the demonstration of item (i).

Case (ii) is a direct consequence of the fact that
$T_{x_0} M=\spn\{X|_{x_0},Y|_{x_0}\}$,
and therefore
$A_0T_{x_0} M=\spn\{AX|_{x_0},AY|_{x_0}\}=\spn\{\hE_1|_{\hx_0}, \hE_2|_{\hx_0}\}=T_{\hx_0} \hN$
by items (i) and (ii) in Proposition \ref{pr:c2.1}.

At last, assertion (iii) is implied by the following facts:
\begin{itemize}
\item[(a)] By (ii), the point $\qz$ belongs to $Q(V,\hN)=:Q'$, the state space of rolling of $(V,g|_V)$ rolling against $(\hN, \hg|_{\hN})$. Write this embedding as $\iota$.

\item[(b)] The orthonormal frame $\hE_1,\hE_2,\hE_3$ on $\hM$
allows construction of an embedding $\iota$
of $Q(V,\hN)$ into $Q=Q(V,\hM)$
that simply maps $(x,\hx,A')\in Q'$ to $(x,\hx,A')\in Q$.

\item[(c)] Since the shape tensor $\mathrm{II}$ of $(\hN,\hg|_{\hN})$ in $(\hM,\hg)$
vanishes by case (i) (see \cite{petersen06}),
$\iota_*$ maps the plane $\mc{D}'|_q$ of the rolling distribution $\mc{D}'$ in $Q'$
linearly isomorphically onto the plane $\dr|_q$ of $\dr$ in $Q$ for every $q\in Q'$.

\item[(d)]
Let $O'(q_0)$ be a rolling neighbourhood of $q_0$ in $Q'$ 
that is defined in same manner as in Definition \ref{def:rolN},
the rolling there referring to rolling of $V$ against $\hN$, with the respective state space and rolling distribution being $Q'$ and $\mc{D}'$, respectively.

As a consequence of (c), $\iota$ maps a rolling neighbourhood $O'(q_0)$ of $q_0$
diffeomorphically onto an open subset of $O(q_0)$.
Hence the manifolds $O'(q_0)$ and $O(q_0)$ have he same dimension.

\item[(e)] The existence of an isometry $H$ as in case (iii) of the statement of this theorem would imply that $O'(q_0)$ be 2-dimensional, and hence $O(q_0)$ would also be 2-dimensional by (d).
However, that would contradict the fact that in our current case as described by Proposition \ref{p2.1},
we have $\dim O(q_0)=5$.
\end{itemize}
\end{proof}

\begin{remark}
Note that the existence of the (local) isometry $H$ could be obtained directly from items (i) and (iii) of a local version of Proposition \ref{p1.12} above, in case $n=\hat{n}$ and under an understanding that the manifold $N$ in case (iii) can actually be taken to be
$\mc{O'}(q_0)$, the map $F$ to be $\pi_{\mc{O'}(q_0),M}$ and $G$ to be $\pi_{\mc{O'}(q_0),\hN}$ so that $H$ could be taken (locally) to be $G\circ F^{-1}$. Note also that $G_*|_{q_0}\circ (F_*|_{q_0})^{-1}=A_0$ in this case.
\end{remark}

In the other direction, the following proposition establishes sufficient conditions associated to the result of Theorem \ref{th:2.2}.
Note that we formulate this result in terms of the actual rolling orbit $\odr(q_0)$ instead of restricting (only) to some rolling neighbourhood $O(q_0)$.

\begin{proposition}\label{pr:5.10}
Suppose that $(\hat{M},\hat{g})$ contains a totally geodesic submanifold $\hN$ of dimension 2,
and assume that that $q_0=(x_0,\hx_0;A_0)\in Q$ with $\hx_0\in\hN$
and $A_0T_{x_0} M=T_{\hx_0} \hat{N}$.
Then if there does not exist (local) isometry $H:(M,g)\to (\hN, \hg|_{\hN})$ such that $H_*|_{x_0}=A_0$,
we have $\dim \odr(q_0)=5$.
\end{proposition}

\begin{proof}
By shrinking $M$, $\hN$ if need be around $x_0$, $\hx_0$,
respectively, we may assume an existence of an oriented orthonormal frame $\hE_1,\hE_2,\hE_3$
as described in the statement of Proposition \ref{pr:c2.1},
and used in the proof of Theorem \ref{th:2.2}.
 
The assumption that $\hN$ is totally geodesic on $\hM$,
and since $q_0=(x_0,\hx_0;A_0)\in Q$
is such that $\hx_0\in\hN$ and $A_0T_{x_0} M=T_{\hx_0} \hN$,
along with the introduction of the frame $\hE_1,\hE_2,\hE_3$ above,
imply that the assertions in items (a)-(d) of the proof of Theorem \ref{th:2.2} hold true.

It is well known (cf. \cite{AgrachevSachkov}) that since $M$ and $\hN$ are both 2-dimensional,
then the rolling orbit $\mc{O}'(q_0)$ of $(M,g)$ rolling against $(\hN,\hg|_{\hN})$
through $q_0$ has either dimension 2 or 5.
Also, recall that by item (d) of the proof of Theorem \ref{th:2.2}
one has $\dim \odr(q_0)=\dim \mc{O}'(q_0)$.

But if $\dim \mc{O}'(q_0)=2$, then $\mc{O}'(q_0)$ would be an integral manifold of the rolling distribution $\mc{D}'$
of $(M,g)$ rolling against $(\hN,\hg|_{\hN})$, because $\dim\mc{D}'=2$,
and it would then follow (cf. \cite{AgrachevSachkov})
that there exist a local isometry $H:(M,g)\to (\hN,\hg|_{\hN})$
such that $H_*|_{x_0}=A_0$,
which contradicts our assumption that no such isometry exists.
\end{proof}

\vspace{2\baselineskip}\section{Case $(\Pi_X , \Pi_Y) \neq (0,0)$ on $O(q_0)$}\label{ss1}\ \newline

In this section, in addition to establishing some notations,
we will introduce a new frame (as in \cite{ChitourKokkonen,ChitourKokkonen1}), which turns out to be useful in subsequent computations,
by rotating the given orthonormal frame $X$, $Y$ of $M$ by an angle $\phi$.
Indeed, since we are here assuming that $\qz$
has a rolling neighbourhood $O(q_0)$ on which $(\Pi_X , \Pi_Y)$ (is defined and) does not vanish,
one can (after shrinking $O(q_0)$ around $q_0$ if necessary)
define smooth functions $r$, $\phi:O'\rarrow \mathbb{R}$
defined on some open set $O'$ of $Q$ containing $O(q_0)$
and such that
\begin{equation}\label{e2.9}
\Pi_X = r \cos \phi, \quad \Pi_Y = r \sin \phi,\quad r>0.
\end{equation}

In order to keep some of the formulas that appear below more compact,
we shall write $\cphi := \cos \phi$ and $\sphi := \sin \phi $.
Then we define the following maps $O'\to TM$ (i.e., ``$Q$-dependent'' vector fields of $M$),
\begin{equation}\label{e2.10}
\XtildeA = \cphi X + \sphi Y, \quad \YtildeA = - \sphi X + \cphi Y,
\end{equation}
and note that at each point $\q\in O'$
the vectors $\XtildeA$ and $\YtildeA$ are orthonormal in $T_x M$.

With respect to this new frame,
we have according to Eqs. \eqref{eq:hsigma_Pi}, \eqref{e2.9}, \eqref{e2.10},
\begin{align}\label{e2.11}
&\Pi_{\Xtilde} := \hg (\hR (\star \hZ_A) , \star A \XtildeA) =  \cphi \Pi_X + \sphi \Pi_Y = r, \nonumber \\
&\Pi_{\Ytilde} := \hg (\hR (\star \hZ_A) , \star A \YtildeA) =  - \sphi \Pi_X + \cphi \Pi_Y = 0,
\end{align}
and
\begin{align}\label{eq:def:sigma_Pi_tilde}
\tilde{\hsigma}_A^1 :={}& - \hg (\hR (\star A \tilde{X}_A) , \star A \XtildeA ) =  \cphi^2 \hsigma_A^1 + \sphi^2 \hsigma_A^2 - 2 \cphi \sphi \Pi_Z, \nonumber \\
\tilde{\hsigma}_A^2 :={}& - \hg (\hR (\star A \tilde{Y}_A) , \star A \YtildeA ) =   \sphi^2 \hsigma_A^1 + \cphi^2 \hsigma_A^2 + 2 \cphi \sphi \Pi_Z,\nonumber \\
\tilde{\Pi}_{\hZ} :={}& \hg (\hR (\star A \tilde{X}_A) , \star A \YtildeA ) =  \sphi \cphi (\hsigma_A^1 - \hsigma_A^2) + (\cphi^2 - \sphi^2) \Pi_Z,
\end{align}
these definitions and relations being valid on the open set $O'$ of $Q$ that contains $O(q_0)$.

\vspace{2\baselineskip}\subsection{Subcase $\hsigma_A \equiv K (x) $ on $O(q_0)$}\label{ss10}
\ \newline

We assume in this subsection that
$(\Pi_X , \Pi_Y)$ does not vanish on $O(q_0)$,
while $K (x) - \hsigma_A=0$ on $O(q_0)$.
It follows by Eqs. \eqref{eq:RolXY}, \eqref{e2.9} and \eqref{e2.10} that 
\begin{align}\label{eq:Rol:reduced}
\Rol_q = (\Pi_Y \theta_X - \Pi_X \theta_Y) \otimes \hZ_A=- r \theta_{\YtildeA} \otimes \hZ_A,
\end{align}
where $\theta_X$ and $\theta_Y$ are the $g$-dual 1-forms of $X$, $Y$, respectively.
We shall also write along this subsection
\begin{align}\label{eq:def:beta}
\beta (q) :=  \hg\big((\hnabla_{\hZ_A} \hR) (\star \hZ_A), \star \hZ_A\big),
\end{align}
which is a smooth function that is defined on an open subset of $Q$, which we assume to be $O'$
(see section \ref{ss1}).

The expression for $\Rol_q$ given just above
implies that the following Lie brackets,
which one can compute e.g. using Lemma \ref{l2.2} and Eq. \eqref{e2.10},
will be relevant to our study of the structure of the rolling neighbourhood $O(q_0)$.
Notice that $\nu(\Rol_q)$ being tangent to $O(q_0)$ (as implied by Proposition \ref{e1.11})
is now equivalent to $\nu(\theta_{\YtildeA} \otimes \hZ_A)$
being tangent to $O(q_0)$, because $r\neq 0$.
Consequently, the following vector fields are defined and tangent to $O(q_0)$
\begin{align}\label{e2.13}
\tilde{L}_{\Xtilde} \onq : ={}& [\lr (\Xtilde) , \nu(\theta_{\Ytilde} \otimes \hZ)] \onq \nonumber\\
={}& - ( \nu (\theta_{\YtildeA} \otimes \hZ_A ) \onq \phi ) \lr (\YtildeA) \onq - ( \lr(\XtildeA) \onq \phi + g (\Gamma, \XtildeA) ) \nu (\theta_{\XtildeA} \otimes \hZ_{A}) \onq, \nonumber\\[2mm]
\tilde{L}_{\Ytilde} \onq : ={}& [\lr (\Ytilde) , \nu(\theta_{\Ytilde} \otimes \hZ)] \onq\nonumber\\
={}& ( \nu (\theta_{\YtildeA} \otimes \hZ_A ) \onq \phi ) \lr (\XtildeA) \onq- ( \lr(\YtildeA) \onq \phi + g (\Gamma, \YtildeA)) \nu (\theta_{\XtildeA} \otimes \hZ_{A}) \onq
- \lns (\hZ_A) \onq, \nonumber\\
\end{align}
where $\Gamma$ is as defined in \eqref{eq:Gamma}

The quantities $\lr (\Xtilde_A) \onq \phi $ and $\lr (\Ytilde_A) \onq \phi $
appearing in the above equation
can further be expressed
in terms of the derivatives of $\hsigma_A$ and $K$ with respect to the two brackets in \eqref{e2.13}.
Indeed,
$$
\begin{array}{l}
\tilde{L}_{\Xtilde} \onq \hsigma_{(\cdot)}  = - ( \nu (\theta_{\YtildeA} \otimes \hZ_A ) \onq \phi ) \lr (\YtildeA) \onq \hsigma_{(\cdot)} - (\lr (\Xtilde_A) \onq \phi + g (\Gamma, \XtildeA)) 2 r,\\[2mm]
\tilde{L}_{\Ytilde} \onq \hsigma_{(\cdot)} = + ( \nu (\theta_{\YtildeA} \otimes \hZ_A ) \onq \phi ) \lr (\XtildeA) \onq \hsigma_{(\cdot)} + \beta (q) - (\lr (\Ytilde_A) \onq \phi + g (\Gamma, \YtildeA)) 2 r,
\end{array}
$$
and,
\[
\tilde{L}_{\Xtilde} \onq K ={}& - ( \nu (\theta_{\YtildeA} \otimes \hZ_A ) \onq \phi ) \lr (\YtildeA) \onq K, \\
\tilde{L}_{\Ytilde} \onq K ={}& + ( \nu (\theta_{\YtildeA} \otimes \hZ_A ) \onq \phi ) \lr (\XtildeA) \onq K,
\]
which combined with our assumptions that $K \equiv \hsigma_A$ and $r\neq 0$ on $O(q_0)$,
allow us to conclude that
\begin{align}\label{e2.14}
&\lr (\Xtilde_A) \onq \phi = - g (\Gamma, \XtildeA) ,\quad \lr (\Ytilde_A) \onq \phi =  - g (\Gamma, \YtildeA) + \frac{1}{2 r} \beta (q),
\end{align}
at every point $\q\in O(q_0)$.

In addition, the vector fields in \eqref{e2.13} are now simplified into the form
\begin{align}\label{eq:LX_LY}
\tilde{L}_{\Xtilde} \onq ={}& [\lr (\Xtilde) , \nu(\theta_{\Ytilde} \otimes \hZ)] \onq
= - ( \nu (\theta_{\YtildeA} \otimes \hZ_A ) \onq \phi ) \lr (\YtildeA) \onq, \nonumber\\[2mm]
\tilde{L}_{\Ytilde} \onq ={}& [\lr (\Ytilde) , \nu(\theta_{\Ytilde} \otimes \hZ)] \onq
= ( \nu (\theta_{\YtildeA} \otimes \hZ_A ) \onq \phi ) \lr (\XtildeA) \onq - \frac{\beta (q)}{2 r} \nu (\theta_{\XtildeA} \otimes \hZ_{A}) \onq
- \lns (\hZ_A) \onq. \nonumber\\
\end{align}

The above identities for $\tilde{L}_{\Xtilde}$ and $\tilde{L}_{\Ytilde}$ imply that, so far,
there is only one new tangent vector field on $O(q_0)$ coming out of the previous Lie bracket
computations, and that is
\begin{equation}\label{e2.18}
F \onq : = \lns (\hZ_{A} ) \onq + \frac{\beta(q)}{2r} \nu (\theta_{\XtildeA} \otimes \hZ_{A} ) \onq.
\end{equation}

The Lie bracket between $\lr (\Ytilde)$ and $F$ is equal to
\begin{align*}
[\lr (\Ytilde) ,  F] \onq  = & (F \onq \phi) \lr (\XtildeA) \onq + \big(\tilde{\Pi}_Z + \lr(\YtildeA) \onq (\frac{\beta}{2r} ) \big) \nu (\theta_{\XtildeA} \otimes \hZ_A) \onq\\
& - r \nu (A(X \wedge Y)) \onq + \big(\tilde{\sigma}_A^1 + (\frac{\beta(q)}{2r})^2\big) \nu (\theta_{\YtildeA} \otimes \hZ_A) \onq,
\end{align*}
from where we can extract a tangent vector field $H$ to $O(q_0)$
that is (pointwise) linearly independent to the ones previously encountered,
\[
H|_q:=\nu (A(X \wedge Y))|_q - \frac{1}{r}\big(\tilde{\Pi}_Z + \lr(\YtildeA) \onq (\frac{\beta}{2r} ) \big) \nu (\theta_{\XtildeA} \otimes \hZ_A) \onq.
\]
We will now proceed to show that the second term on the right hand side of this expression for $H$ vanishes.

Indeed, since $\hsigma_{(\cdot)} - K=0$ on $O(q_0)$ by assumption,
and because $H$ is tangent to $O(q_0)$,
we have $H|_q(\hsigma_{(\cdot)} - K)=0$ for all $\q\in O(q_0)$.
But clearly $\mathrm{V}(K)=0$ for all $\pi_Q$-vertical vectors $\mathrm{V}\in V|_q(\pi_Q)$ (which $H\onq$ is entirely composed of),
while by Lemma \ref{le:6.3} it holds $\nu (A ( X \wedge Y)) \onq \hsigma_{(\cdot)} = 0$,
and therefore we are left with
\[
0=H|_q(\hsigma_{(\cdot)} - K)=- \frac{1}{r}\big(\tilde{\Pi}_Z + \lr(\YtildeA) \onq (\frac{\beta}{2r} ) \big) \nu (\theta_{\XtildeA} \otimes \hZ_A) \onq \hat{\sigma}_{(\cdot)}.
\]
By Lemma \ref{l2.3}, the last factor on the right is given by
$\nu (\theta_{\XtildeA} \otimes \hZ_A) \onq \hat{\sigma}_{(\cdot)}=2r$,
and since $r\neq 0$ on $O(q_0)$ (by assumption), we can conclude that
\[
\lr(\YtildeA) \onq (\frac{\beta}{2r}\big)=-\tilde{\Pi}_Z
\]
and hence
\[
H|_q=\nu (A(X \wedge Y))|_q,\quad \forall q\in O(q_0),
\]
i.e., $\nu (A(X \wedge Y))|_q$ is tangent to $O(q_0)$ for all $\q\in O(q_0)$.

By Lemma \ref{l2.2} the Lie brackets of $\lr (\Xtilde)$ and $\lr (\Ytilde)$
with $\nu ((\cdot) (X \wedge Y))$ are
\[
{}& [\lr (\Xtilde) , \nu \big((\cdot) (X \wedge Y)\big)] \onq = \lr (\YtildeA) \onq - \lns (A \YtildeA) \onq \\
{}& [\lr (\Ytilde) , \nu ((\cdot) (X \wedge Y))] \onq = - \lr (\XtildeA) \onq + \lns (A \XtildeA) \onq
\]
implying that $\lns (A \Xtilde_A) \onq$ and $\lns (A \Ytilde_A) \onq$ are also tangent to $O(q_0)$ at every point $\q\in O(q_0)$.
Therefore, there are at least 7 pointwise linearly independent vector fields tangent to $O(q_0)$
\begin{align}\label{eq:sec_5_3:VFs_on_O}
\lr (\XtildeA) \onq,\, \lr (\YtildeA) \onq,\, \lns (A \XtildeA) \onq,\, \lns (A \YtildeA) \onq,\, F \onq,\, \nu (\theta_{\YtildeA} \otimes \hZ_A) \onq,\, \nu (A (X \wedge Y) ) \onq,
\end{align}
where $F \onq$ is the vector field given by \eqref{e2.18}.

At any $\q\in O(q_0)$, the map $(\pi_{Q,\hM}|_{O(q_0)})_*|_q$ sends
the tangent vectors
$\lns (A \XtildeA) \onq$, $\lns (A \YtildeA)$ and $F\onq$ of $O(q_0)$
to the vectors $A\Xtilde_A$, $A\Ytilde_A$ and $\hZ_A$ of $T_{\hx}\hM$, respectively.
Since $A\Xtilde_A$, $A\Ytilde_A$ and $\hZ_A$ span $T_{\hx}\hM$,
we can conclude that $\pi_{Q,\hM}|_{O(q_0)}$ is a submersion, and hence an open map.
We record these observation into a proposition.

\begin{proposition}\label{pr:ss10:1}
The map $\pi_{Q,\hM}|_{O(q_0)}:O(q_0)\to \hM$ is a submersion.
Consequently, the set $\hV:=\pi_{Q,\hM}(O(q_0))$ is an open neighbourhood of $\hx_0$ in $\hM$.
\end{proposition}

With the above preliminary work done, we can formulate our first important result of this section.
We remind that $V=\pi_{Q,M}(O(q_0))$ and that it is an open neighbourhood of $x_0$ in $M$.

\begin{proposition}\label{p2.2}
If $K \equiv \hsigma_A$ on a rolling neighbourhood $O(q_0)$ of $\qz$,
and $(\Pi_X,\Pi_Y)\neq (0,0)$ everywhere on $O(q_0)$,
then, after shrinking $O(q_0)$ around $q_0$ if necessary,
the curvature $K$ of $(M,g)$ is constant on $V$, and $\beta$ vanishes on $O(q_0)$.
\end{proposition}

\begin{proof}
By the above analysis, vector fields in \eqref{eq:sec_5_3:VFs_on_O}
are tangent to $O(q_0)$,
and hence so are $\lns (\Xtilde_A) \onq$ and $\lns (\Ytilde_A) \onq$.
On the other hand, by Lemma \ref{l2.3} it holds that
$\lns (\Xtilde_A) \onq \hsigma_{(\cdot)}=0$ and $\lns (\Ytilde_A) \onq \hsigma_{(\cdot)} = 0$ on $O(q_0)$,
and since by assumption $K-\hsigma_{(\cdot)}=0$ on $O(q_0)$ as well,
it follows that
\[
0={}&\lns (\Xtilde_A) \onq (K-\hsigma_{(\cdot)})=\tilde{X}_A(K), \\
0={}&\lns (\Ytilde_A) \onq (K-\hsigma_{(\cdot)})=\tilde{Y}_A(K),
\]
and thus $X(K)=0$, $Y(K)=0$ (by \eqref{e2.10}) on the open neighbourhood $V=\pi_{Q,M}(O(q_0))$ of $x_0$ in $M$.
Since by definition $O(q_0)$ is connected, then so is $V$ and because $X,Y$ span the tangent space of $M$ on each point of $V$,
it follow that $K$ is constant on $V$.

For every $\q \in O(q_0)$, the matrix of $\hR|_{\hx}$ with respect
to the orthonormal oriented basis $(\star A \XtildeA,\star A \YtildeA , \star \hZ_A)$ is given by
\begin{align}\label{eq:ss10:hR}
\hR|_{\hx}
=
\left(
\begin{array}{ccc}
-\tilde{\hsigma}_A^1 &  \tilde{\Pi}_{\hZ}  &  r \\
\tilde{\Pi}_{\hZ}  &  -\tilde{\hsigma}_A^2 &  0 \\
r  &  0  &  - K
\end{array}
\right),
\end{align}
where we have also used the assumption that ($\hsigma^3_A=$) $\hsigma_A=K(x)$ for all $\q\in O(q_0)$.

We know that the curvature tensor $\hR|_{\hx}$ depends only on the given point $\hx$
in $\hV=\pi_{Q,\hM}(O(q_0))$,
and thus considered as a linear map $\hR|_{\hx}: \wedge^2 T_{\hx} M\to \wedge^2 T_{\hx} M$,
its characteristic polynomial $f_{\hx}(\tau)$ only depends on $\hx$ in $\hV$.
For $\q\in O(q_0)$,
the above matrix representation of $\hR|_{\hx}$ with respect to the
basis $\star A \XtildeA,\star A \YtildeA , \star \hZ_A$
gives $f_{\hx}(\tau)$ explicitly as
\begin{align}\label{eq:ss10:char_hR}
f_{\hx} (\tau) ={} &  \tau^3 + ( K +\tilde{\hsigma}_A^1 + \tilde{\hsigma}_A^2) \tau^2 + \big(K (\tilde{\hsigma}_A^1 + \tilde{\hsigma}_A^2) + \tilde{\hsigma}_A^1 \tilde{\hsigma}_A^2 - (\tilde{\Pi}_{\hZ})^2-r^2 \big) \tau \nonumber \\
& + K (\tilde{\hsigma}_A^1 \tilde{\hsigma}_A^2 - (\tilde{\Pi}_{\hZ})^2)-r^2 \tilde{\hsigma}_A^2=:\sum_{k=0}^3 C_{k}(q) \tau^k,
\end{align}
which implies that each coefficient $C_k(q)$, $\q\in O(q_0)$, only depends on $\hx\in\hV$,
i.e., there are functions $\hat{c}_k:\hV\to\R$ such that $C_k(q)=\hat{c}_k(\pi_{Q,\hM}(q))$ for all $q\in O(q_0)$.

In particular, concentrating on the coefficient $C_1$ of $\tau^1$, we have
\[
0={}&
\lns (\XtildeA) \onq C_1
=
\lns (\XtildeA) \onq  \big(K (\tilde{\hsigma}_A^1 + \tilde{\hsigma}_A^2) + \tilde{\hsigma}_A^1 \tilde{\hsigma}_A^2 - (\tilde{\Pi}_{\hZ})^2-r^2 \big),
\]
where the first equality is a consequence of $(\pi_{Q,\hM})_*\lns (\XtildeA)=0$
and the above observation that $C_1=\hat{c}_1\circ \pi_{Q,\hM}$.
But since $K$ is constant on the open $V\subset M$, and 
$\lns (\XtildeA)\onq \tilde{\hsigma}_{(\cdot)}^1=0$, $\lns (\XtildeA)\onq \tilde{\hsigma}_{(\cdot)}^2=0$,
$\lns (\XtildeA)\onq \tilde{\Pi}_{\hZ}=0$
by Lemmas \ref{l2.7}, \ref{l2.8} and \ref{l2.9},
the above equation reduces to
\[
0=-\lns (\XtildeA) \onq  (r^2)=-2r(q)\lns (\XtildeA) \onq  r=r(q)\beta(q),\quad \forall q\in O(q_0),
\]
where the last equality is a consequence of Lemma \ref{l2.4}.
Because $r\neq 0$ on $O(q_0)$, we conclude that $\beta=0$ on $O(q_0)$,
which completes the proof.
\end{proof}

\begin{proposition}\label{p2.3}
Under the assumptions of Proposition \ref{p2.2}, we have:
\begin{enumerate}
\item[a)]
$\tilde{\hsigma}_{(\cdot)}^2 $ and $\tilde{\Pi}_{\hZ}$ both vanish on $O(q_0)$;
\item[b)]
$r(\cdot)^2 = K \tilde{\hsigma}_{(\cdot)}^1$ on $O(q_0)$; in particular $K \neq 0$ and $\tilde{\hsigma}_{(\cdot)}^1 \neq 0$ everywhere on $O(q_0)$.
\end{enumerate}
\end{proposition}

\begin{proof}
We start with item $a)$.
Since $\beta = 0$ on $O(q_0)$ (by Proposition \ref{p2.2})
it follows that the vector $F \onq = \lns (\hZ_A) \onq$ (see Eq. \eqref{e2.18})
is a tangent to $O(q_0)$ at every $\q\in O(q_0)$.
According to Lemma \ref{l2.2}
the Lie brackets $[\lr(\Xtilde) , \lns (\hZ)] \onq$ and $[\lr(\Ytilde) , \lns (\hZ)] \onq$ are equal to
\begin{align*}
[\lr(\Xtilde), \lns (\hZ)] \onq
={}& - (\lns (\hZ_A) \onq \phi) \lr (\YtildeA) \onq +\tilde{\hsigma}_A^2 \nu(\theta_{\XtildeA} \otimes \hZ_A) \onq + \tilde{\Pi}_{\hZ} \nu (\theta_{\YtildeA} \otimes \hZ_A) \onq, \\[2mm]
[\lr(\Ytilde) , \lns (\hZ)] \onq
={}&
(\lns (\hZ_A) \onq \phi) \lr (\XtildeA) \onq + \tilde{\Pi}_{\hZ} \nu(\theta_{\XtildeA} \otimes \hZ_A) \onq \\
{}&
+ \tilde{\hsigma}_A^1 \nu (\theta_{\YtildeA} \otimes \hZ_A) \onq - \Pi_{\Xtilde} \nu (A ( X \wedge Y) ) \onq,
\end{align*}
and therefore Lemma \ref{l2.3}
allows one to infer that the derivatives of $\hsigma_{(\cdot)}$ along these brackets are
\begin{align*}
[\lr(\Xtilde) , \lns (\hZ)] \onq \hsigma_{(\cdot)} = {}& - (\lns (\hZ_A) \onq \phi) \lr (\YtildeA) \onq \hsigma_{(\cdot)} + 2 \tilde{\hsigma}_A^2  r ,\\
[\lr(\Ytilde) , \lns (\hZ)] \onq \hsigma_{(\cdot)} = {}& (\lns (\hZ_A) \onq \phi) \lr (\XtildeA) \onq \hsigma_{(\cdot)} + 2 \tilde{\Pi}_{\hZ} r.
\end{align*}
Since $\hsigma_{(\cdot)}=K$ on $O(q_0)$ and $K$ is constant on $V$ by Proposition \ref{p2.2},
while the vectors $[\lr(\Xtilde) , \lns (\hZ)]\onq$, $[\lr(\Ytilde) , \lns (\hZ)] \onq $ and $\lr (\YtildeA)\onq$, $\lr (\XtildeA) \onq$ are tangent to $O(q_0)$,
the above equations imply that $2\tilde{\hsigma}_A^2 r(q)=0$ and $2\tilde{\Pi}_{\hZ}(q) r(q)=0$ for all $q\in O(q_0)$.
Given that $r\neq 0$ everywhere on $O(q_0)$,
we thus find that $\tilde{\hsigma}_A^2\equiv 0$ and $ \tilde{\Pi}_{\hZ} \equiv 0$ on $O(q_0)$
as claimed.

It remains to prove the assertion of case b).
Since $\tilde{\Pi}_{\hZ} = 0$ and $\tilde{\hsigma}_A^2 = 0$ on $O(q_0)$,
and since $\nu (\theta_{\YtildeA} \otimes \hZ_A) \onq $ is tangent to $O(q_0)$ (see \eqref{eq:sec_5_3:VFs_on_O}),
we have by Lemma \ref{l2.9},
\[
0 = \nu (\theta_{\YtildeA} \otimes \hZ_A) \onq \tilde{\Pi}_{\hZ} = \frac{1}{r} ( - \hsigma_A \tilde{\hsigma}_A^1 + r^2).
\]
This means that $\hsigma_A \tilde{\hsigma}_A^1 = r^2$, and hence $\tilde{\hsigma}_A^1 \neq 0$ and $K = \hsigma_A \neq 0$,
because $r\neq 0$.
\end{proof}

For the next corollary, we recall that $\hV=\pi_{Q,\hM}(O(q_0))$ is an open neighbourhood of $\hx_0$ in $\hM$ (see Proposition \ref{pr:ss10:1}).

\begin{corollary}\label{c2.3}
Under the assumptions of Proposition \ref{p2.2},
the following are true:
\begin{itemize}
\item[a)] There exists a function $\hat{r}\in C^\infty(\hat{V})$
such that $r(q)=\hat{r}(\hat{x})$ for all $\q\in O(q_0)$.

\item[b)] For every $\hx\in \hV$,
the curvature $\hR|_{\hat{x}}$ has 0 as a double eigenvalue 
and $-(K^2 + \hat{r}(\hx)^2)/K$ as a simple (non-zero) eigenvalue.
\end{itemize}
\end{corollary}

\begin{proof}
We start with the case b), using $r(q)$ instead of $\hat{r}(\hx)$.
Once this is done, we can easily show that $r(q)$ is represented by a function $\hat{r}$ on $\hV\subset \hM$.

For any $\q\in O(q_0)$, Proposition \ref{p2.3} and Eq. \eqref{eq:ss10:hR} imply that the matrix representation
of the curvature $\hR|_{\hx}$ w.r.t. $(\star A \XtildeA,\star A \YtildeA , \star \hZ_A)$ is given by
\begin{align}\label{eq:ss10:hR:2}
\hR|_{\hx}
=
\qmatrix{
-r(q)^2/K  & 0  &  r(q)\\
0  &  0 &  0\\
r(q)  &  0  &  - K
},
\end{align}
and the characteristic polynomial $f_{\hat{x}}(\tau)$ at $\hat{x}$ of this
matrix is found to be (in comparison to \eqref{eq:ss10:char_hR}),
\[
f_{\hat{x}}(\tau)
=\tau^2\Big(\tau+\frac{r(q)^2}{K} + K\Big).
\]
This shows that, for any $\q\in O(q_0)$,
the eigenvalues of $\hat{R}|_{\hat{x}}$ are $0$ and $-(K^2+r(q)^2)/K$ with $0$ at least a double eigenvalue.

But we know that $r(q)\neq 0$ (and $K\neq 0$),
implying that $-(K^2+r(q)^2)/K\neq 0$,
and therefore $0$ is a double eigenvalue, while
$-(K^2+r(q)^2)/K$ is a simple (non-zero) eigenvalue of $\hat{R}|_{\hx}$.

Finally note that if we start with any point $\hx\in\hV$,
then because $\hV=\pi_{Q,\hM}(O(q_0))$, there exists $q\in O(q_0)$ such that $\pi_{Q,\hM}(q)=\hx$,
allowing us to write $q$ as $\q$,
and hence conclude, by what we have just observed, that
$\hR|_{\hx}$ has $0$ as a double eigenvalue, and $-(K^2+r(q)^2)/K$ as a simple (non-zero) eigenvalue.
This completes the proof of the case b), with $r(q)$ replacing $\hat{r}(\hx)$ until
the assertion of case a) has been verified.

Now the case a) follows almost immediately.
Indeed, we now know that $\hR|_{\hx}$ has a simple eigenvalue for all $\hx\in \hV$,
and therefore it must define a smooth function $\hx\mapsto \hat{\lambda}(\hx)\in C^\infty(\hV)$.
We also know that, for all $\q\in O(q_0)$,
this eigenvalue $\hat{\lambda}(\hx)$ equals to $-(K^2+r(q)^2)/K$ by the above reasoning.
That is $\hat{\lambda}(\hx)=-(K^2+r(q)^2)/K$, from which we find that
$r(q)^2=-K\big(\hat{\lambda}(\hx)+K\big)$ for all $\q\in O(q_0)$.
Since $r(q)>0$, we may take $\hat{r}(\hx):=\sqrt{-K\big(\hat{\lambda}(\hx)+K\big)}$
to finish the proof of the case a).
\end{proof}

\begin{proposition}\label{pr:5.14}
Under the assumptions of Proposition \ref{p2.2},
the rolling neighbourhood $O(q_0)$ of $\qz$ has dimension $7$, and
it has a frame of vector fields consisting of
\begin{align}\label{e2.19:1}
\lr (\XtildeA) \onq,\, \lr (\YtildeA) \onq,\, \lns (A \XtildeA) \onq,\, \lns (A \YtildeA) \onq,\, \lns (\hZ_A) \onq,\, \nu (\theta_{\YtildeA} \otimes \hZ_A) \onq,\, \nu (A (X \wedge Y) ) \onq.
\end{align}
\end{proposition}

\begin{proof}
We already know that vector fields listed in \eqref{eq:sec_5_3:VFs_on_O}
are tangent to $O(q_0)$ (see the discussion preceding Proposition \ref{pr:ss10:1}),
and that $F \onq = \lns (\hZ_A) \onq$
due to the definition \eqref{e2.18} of $F$
and the fact that $\beta = 0$ on $O(q_0)$ as established in Proposition \ref{p2.2}.

In order to show that 7 vector fields \eqref{e2.19:1}, which are clearly linearly independent 
at every point of $O(q_0)$,
actually span the tangent space of $O(q_0)$ at every point of $O(q_0)$,
and hence form a frame on $O(q_0)$,
it is thus sufficient (and necessary) to show that
these vector fields form an involutive system on $O(q_0)$.
Indeed, then the constant rank 7 distribution $D$ spanned by them on $O(q_0)$
is involutive, and since $\dr|_q\subset D|_q$ for all $q\in O(q_0)$,
an integral manifold (which exists e.g. by Frobenius' theorem) of $D$ on $O(q_0)$
will clearly contain $O(q_0)$ itself, which implies that $D|_q=T_q (O(q_0))$ for all $q\in O(q_0)$,
and completes the proof.

First we note that the vector fields \eqref{e2.19:1}
form an involutive system if and only if
their mutual Lie brackets do not contain the term $\nu (\theta_{\XtildeA} \otimes \hZ_A)|_q$
with a non-vanishing coefficient.

The expressions for all the Lie brackets between vector fields \eqref{e2.19:1}
can be read off from Lemma \ref{l2.2},
and they are subject to important simplifications
thanks to the additional identities obtained so far in this section.
One such simplification comes from \eqref{e2.14} combined with $\beta=0$ on $O(q_0)$
(Proposition \ref{p2.2})
so that
$\lr (\Xtilde_A) \onq \phi + g (\Gamma, \XtildeA)=0$
and
$\lr (\Ytilde_A) \onq \phi + g (\Gamma, \YtildeA)=0$
on $O(q_0)$.
The remaining identities one needs to use are
$\tilde{\hsigma}_A^2=0$ and $\tilde{\Pi}_{\hZ}=0$ from Proposition \ref{p2.3},
and the relations
\begin{align}\label{eq:ss10:lns_XYZ_phi}
\lns(A\Xtilde_A)\onq \phi=0,\quad
\lns(A\Ytilde_A)\onq \phi=0,\quad
\lns(\hZ_A)\onq \phi=0
\end{align}
holding on $O(q_0)$ as we will now verify.

Indeed, we have that the vector fields appearing in \eqref{e2.19:1}
are tangent to $O(q_0)$ while $\hsigma_{(\cdot)}-K = 0$ on $O(q_0)$ by assumption,
hence using Lemma \ref{l2.2}, Lemma \ref{l2.3},
and the fact that $\beta=0$ on $O(q_0)$ (Proposition \ref{p2.2}),
we find that
\[
0={}&[\nu (\theta_{\Ytilde} \otimes \hZ) , \lns ((\cdot) \Xtilde) ] \onq (K-\hsigma_{(\cdot)})
= 2r(q)(\lns (A \XtildeA) \onq \phi ) \\
0={}&[\nu (\theta_{\Ytilde} \otimes \hZ) , \lns ((\cdot) \Ytilde) ] \onq (K-\hsigma_{(\cdot)})
= 2r(q)(\lns (A \YtildeA) \onq \phi ),\\
0={}&[\nu (\theta_{\Ytilde} \otimes \hZ) , \lns (\hZ) ] \onq (K-\hsigma_{(\cdot)})
= 2r(q)(\lns (\hZ_A) \onq \phi),
\]
for all $q\in O(q_0)$. But $r\neq 0$ everywhere on $O(q_0)$, hence
$\lns (A \XtildeA) \onq \phi$, $\lns (A \YtildeA) \onq \phi$ and $\lns (\hZ_A) \onq \phi$
all vanish on $O(q_0)$ as was claimed.

Given these observations, one readily sees from Lemma \ref{l2.2}
that the Lie brackets of $\lr (\Xtilde_A)$, $\lr (\Ytilde_A)$, and $\nu(A(X\wedge Y))$
with all the others appearing in \eqref{e2.19:1}
remain $C^\infty(O(q_0))$-linear combinations of the vector fields in \eqref{e2.19:1}.
That leaves us to check that Lie brackets of the vector fields
\[
\lns (A \XtildeA) \onq,\quad \lns (A \YtildeA) \onq,\quad \lns (\hZ_A) \onq,\quad \nu (\theta_{\YtildeA} \otimes \hZ_A) \onq
\]
are $C^\infty(O(q_0))$-linear combinations of the vector fields listed in \eqref{e2.19:1}.

Since $\tilde{\hsigma}_A^2=0$ and $\tilde{\Pi}_{\hZ_A}=0$,
one also sees (Lemma \ref{l2.2}) that the Lie brackets
$[\lns (\hZ) , \lns ((\cdot) \Xtilde)]$
and $[\lns (\hZ) , \lns ((\cdot) \Ytilde)]$,
remain in the span of the vector fields \eqref{e2.19:1}.
And clearly the same is true for $[\lns ((\cdot) \Xtilde) , \lns ((\cdot) \Ytilde)]$.

It thus remains to check the Lie brackets of
$\nu (\theta_{\YtildeA} \otimes \hZ_A)$
with $\lns (A \XtildeA)$, $\lns (A \YtildeA)$ and $\lns (\hZ_A)$.
In these three cases, as one can readily see from Lemma \ref{l2.2},
the coefficient of $\nu (\theta_{\XtildeA} \otimes \hZ_A)$
will be, respectively,
$\lns (A \XtildeA) \onq \phi$, $\lns (A \YtildeA) \onq \phi$
and $\lns (\hZ_A) \onq \phi$,
which all vanish on $O(q_0)$ in view of \eqref{eq:ss10:lns_XYZ_phi}.
This completes the proof.
\end{proof}

\begin{proposition}\label{pr:c2.4}
Under the assumptions of Proposition \ref{p2.2},
there is a smooth oriented orthonormal frame $\hE_1, \hE_2, \hE_3$ of $\hM$ defined on $\hV$
with respect to which the connection table $\hGamma$ of $\hnabla$ takes the form
\begin{align}\label{e2.19}
\hGamma =
\left(
\begin{array}{ccc}
0 & 0 & 0\\
\hGamma_{(3,1)}^1 & \hGamma_{(3,1)}^2 & \hGamma_{(3,1)}^3 \\
0 & 0 & 0
\end{array}
\right).
\end{align}
In addition,
$\star\hE_2|_{\hx}$ is an eigenvector of $\hR|_{\hx}$ corresponding to the eigenvalue
$-(K^2 + \hat{r}(\hx)^2)/K$ at every $\hx\in\hV$.
\end{proposition}

\begin{proof}
Since by Corollary \ref{c2.3},
$\hR|_{\hx}$ has a simple (smooth, non-zero) eigenvalue $\hat{\lambda}(\hx):=-(K^2 + \hat{r}(\hx)^2)/K$
on the open subset $\hV=\pi_{Q,\hM}(O(q_0))$ of $\hM$,
it follows that (after shrinking $O(q_0)$ around $q_0$ if necessary)
one can choose a smooth unit vector field $\hE_2$ on $\hV$ such that
$\star\hE_2$ is an eigenvector of $\hR$
corresponding to $\hat{\lambda}$ at every $\hx\in\hV$.
This establishes the second assertion in the statement of this result.

It then follows that the 2-dimensional eigenspace associated to the eigenvalue 0 of $\hR$,
as indicated in Corollary \ref{c2.3}, is nothing else but $\star (\hE_2^\perp)$.
As $\hE_2^\perp$ is a smooth 2-dimensional distribution on $\hV$,
one can choose (at least after shrinking $O(q_0)$ further around $q_0$)
smooth unit vector fields $\hE_1,\hE_3$ on $\hV$
such that $(\hE_1,\hE_2,\hE_3)$ forms an oriented orthonormal frame of $\hM$ on $\hV$
(in particular $\hE_1,\hE_3\in \hE_2^\perp$).

Proposition \ref{pr:5.14} and the vector fields listed in Eq. \eqref{e2.19:1} 
imply that $\lns(\hE_i)\in \VF(O(q_0))$, $i=1,2,3$,
and hence $L_{ij}:=[\lns(\hE_i),\lns(\hE_j)]\in \VF(O(q_0))$
and $H_{ij}:=\lns([\hE_i, \hE_j])\in\VF(O(q_0))$,
for $i,j=1,2,3$.
It follows that
\[
\nu\big(\hR(\hE_i,\hE_j)(\cdot)\big)=L_{ij}-H_{ij}\in \VF(O(q_0)),\quad i,j=1,2,3,
\]
which in the particular case of $i=3,j=1$ at $\q\in O(q_0)$ yields
\[
\nu\big(\hR(\hE_3,\hE_1)A\big)|_q=\nu\big(\hR((\star \hE_2)A\big)\big|_q = \hat{\lambda}(\hx) \nu((\star \hE_2)A)|_q ,
\]
and since $\hat{\lambda}\neq 0$ on $\hV$,
we conclude that
$\nu((\star \hE_2)(\cdot))\in\VF(O(q_0))$.

We now iterate Lie brackets of vector fields of $O(q_0)$.
First $\lns(\hE_i),\nu((\star \hE_2)(\cdot))\in\VF(O(q_0))$, $i=1,2,3$, implying that
(see Proposition \ref{p1.7}),
\[
\VF(O(q_0))\ni q\mapsto F_1^{(i)}|_q
:={}&
[\lns(\hE_i),\nu((\star\hE_2)(\cdot))|_q
=\nu(\star \lns(\hE_i)|_q\hE_2 A)|_q
=\nu(\star \lns(\hnabla_{\hE_i}\hE_2) A)|_q \\
={}&
\nu\big(\star (-\hGamma^i_{(1,2)}\hE_1+\hGamma^i_{(2,3)}\hE_3)A\big)|_q,\quad i=1,2,3,
\]
where we recalled that $\hGamma^i_{(j,k)}=\hg(\hnabla_{\hE_i} \hE_j, \hE_k)$,
and we used the fact that $\hE_i$, $i=1,2,3$, are unit vector fields.
Now that $F_1^{(i)},\nu((\star \hE_2)(\cdot))\in\VF(O(q_0))$, $i=1,2,3$, we also have
(see Proposition \ref{p1.7}),
\[
\VF(O(q_0))\ni q\mapsto F_2^{(i)}|_q
:={}&
[F_1^{(i)},\nu((\star \hE_2)(\cdot))]|_q
=
\nu\big(\big[\star \hE_2,\star (-\hGamma^i_{(1,2)}\hE_1+\hGamma^i_{(2,3)}\hE_3)\big]_{\sso} A\big)\big|_q \\
={}&
\nu\big( \star(-\hGamma^i_{(1,2)}\hE_3 - \hGamma^i_{(2,3)}\hE_1 )A  \big)|_q,\quad i=1,2,3,
\]
where $[\cdot,\cdot]_{\sso}$ denotes the commutator of the elements of the Lie-algebra $\sso(T_{\hx}\hM)$.

Now let us fix $i\in \{1,2,3\}$.
According to Proposition \ref{pr:5.14} and the vector fields listed in Eq. \eqref{e2.19:1},
the $\pi_Q$-vertical space $V|_q(O(q_0)):=V|_q(Q)\cap T_q (O(q_0))$ of $O(q_0)$ is spanned by the last two (linearly independent) vector fields in \eqref{e2.19:1}, hence $\dim V|_q(O(q_0))=2$.
However, the argument given above 
implies that $\nu((\star \hE_2)A)|_q$, $F_1^{(i)}|_q$ and $F_2^{(i)}|_q$ all belong to $V|_q(O(q_0))$, for $q\in O(q_0)$,
which therefore is only possible if these
three vectors are linearly dependent,
i.e., the equality
\[
0 = \det
\qmatrix{
0 & 1 & 0\\
- \hGamma_{(1,2)}^i & 0 & \hGamma_{(2,3)}^i \\
- \hGamma_{(2,3)}^i & 0 & - \hGamma_{(1,2)}^i
}
= - ((\hGamma_{(1,2)}^i)^2 + (\hGamma_{(2,3)}^i)^2)
\]
must be satisfied at all points $\q\in O(q_0)$.
From this we conclude that
\[
\hGamma_{(1,2)}^i = 0,\quad \hGamma_{(2,3)}^i = 0,\quad \mathrm{on}\ \hV \mathrm{\ for\ } i=1,2,3.
\]
The vanishing of these coefficients shows that the connection table $\hGamma$ of $\hnabla$ indeed takes
the form in \eqref{e2.19},
and therefore we have finished the proof.
\end{proof}

\begin{corollary}\label{cor:ss10:const_alpha}
The smooth function
\begin{align}\label{eq:ss10:alpha}
\alpha:O(q_0)\to \R;\quad \alpha(q):=\hg(\hZ_A,\hE_2|_{\hx}),\quad \q\in O(q_0).
\end{align}
is constant on $O(q_0)$.
\end{corollary}

\begin{proof}
Write the orthonormal frame $(X,Y)$ on $M$ as $(X,Y)=(E_1,E_2)$.
Since $\nabla_{E_i} E_1$ is parallel to $E_2$,
while $\nabla_{E_i} E_2$ is parallel to $E_1$, for $i=1,2$, we have
we have
\[
\lr(E_i)\onq \hZ_{(\cdot)}=\star \big((A\nabla_{E_i} E_1)\wedge AE_2\big)
+\star \big(AE_1\wedge A(\nabla_{E_i} E_2)\big)=0,
\quad i=1,2,
\]
for all $\q\in O(q_0)$.
In other words,
\[
\lr(\xi)\onq \hZ_{(\cdot)}=0,\quad \forall \q\in O(q_0),
\]
for any vector field $\xi$ of $M$
defined on the open subset $V=\pi_{Q,M}(O(q_0))$ of $M$.

On the other hand, it follows from the connection table \eqref{e2.19} that
\[
\hnabla_{\hE_i} \hE_2=0,\quad i=1,2,3,
\]
i.e., $\hnabla_{\hat{\xi}} \hE_2=0$ for any vector field $\hat{\xi}$ of $\hM$
defined on the open subset $\hV=\pi_{Q,\hM}(O(q_0))$ of $\hM$.

Thus for any vector field $\xi$ of $M$ defined on $V$ and any $\q\in O(q_0)$,
\[
\lr(\xi)\onq \alpha=\hg(\lr(\xi)\onq \hZ_{(\cdot)},\hE_2|_{\hx})
+\hg(\hZ_A, \hnabla_{A\xi|_{x}} \hE_2)=0.
\]
This implies that $\alpha$ is constant along any $\dr$-horizontal curve on $O(q_0)$,
and hence that $\alpha$ is constant on $O(q_0)$ as claimed.
\end{proof}

We are now in position to formulate the main theorem of this section.
As a reminder, we are writing $V=\pi_{Q,M}(O(q_0))$, $\hV=\pi_{Q,\hM}(O(q_0))$,
and they are open subsets of $M$ and $\hM$, respectively,
because $\pi_{Q,M}|_{O(q_0)}:O(q_0)\to M$ is submersion by Proposition \ref{pr:prelim:1} and (the) Definition \ref{def:rolN} of $O(q_0)$,
while $\pi_{Q,\hM}|_{O(q_0)}:O(q_0)\to \hM$ is a submersion in the case under consideration
by Proposition \ref{pr:ss10:1}.

\begin{theorem}\label{th:5.17}
If $K(x) = \hsigma_A$ and $(\Pi_X(q),\Pi_Y(q))\neq (0,0)$
for all $\q$ on a rolling neighbourhood $O(q_0)$ of $\qz$,
then, after shrinking $O(q_0)$ around $q_0$ if necessary,
the curvature $K\neq 0$ of $(M,g)$ is constant on $V$ and
$(\hV , \hg|_{\hV})$ is isometric to a Riemannian product $(\hI \times \hN, s_1\oplus \hat{h})$, where
$(\hN,\hat{h})$ is a 2-dimensional Riemannian manifold whose curvature $K^{\hN}$ is constant and non-zero,
and $\hI\subset \mathbb{R}$ is a non-empty open interval equipped with the standard Riemannian metric $s_1$.

Finally, if $\hF:(\hI \times \hN, s_1\oplus \hat{h})\to (\hV,\hg|_{\hV})$ is this isometry,
and $\hF(\hat{r}_0,\hat{y}_0)=\hx_0$,
then $A_0 T_{x_0} M\neq (\hF_{\hat{r}_0})_* T_{\hat{y}_0} \hN$,
where $\hF_{\hat{r}_0}:\hN\to\hV$; $\hF_{\hat{r}_0}(\hat{y})=\hF(\hat{r}_0,\hat{y})$,
and furthermore $\hE_2|_{\hx_0}\notin A_0 T_{x_0} M$.
\end{theorem}

\begin{proof}
Recall first that the constancy of $K$ over the neighbourhood $V\subset M$ of $x_0$ has already been
established in Proposition \ref{p2.2} above,
and $K\neq 0$ by case b) of Proposition \ref{p2.3}.

Next we show that $(\hV,g|_{\hV})$ is a Riemannian product as claimed.
To this end, using \cite{hiepko79} or \cite[Theorem C.14]{ChitourKokkonen1} (or \cite[Theorem D.14]{ChitourKokkonen} -- containing a detailed proof), the result of Proposition \ref{pr:c2.4} implies,
after shrinking $O(q_0)$ around $q_0$ if needed,
that $(\hV,\hg|_{\hV})$ is isometric to a \emph{warped} product $(\hI \times \hN , s_1\oplus_{\hf} \hat{h})$
for some 2-dimensional Riemannian manifold $(\hN,\hat{h})$,
open interval $\hI \subset \mathbb{R}$,
and a warping function $\hf \in C^{\infty} (\hI)$
which satisfies $\frac{\hf'(\hr)}{\hf(\hr)} = 0$ for all $\hr\in \hI$,
because $\hGamma^1_{(1,2)}=0$ on $\hV$ according to Eq. \eqref{e2.19}.
This means that $\hf$ is a constant $\hat{c}\in\R$, which
(whose square $\hat{c}^2$) one can absorb into the metric $\hat{h}$ of $\hN$
and conclude that
$(\hV,\hg|_{\hV})$ is isometric to a \emph{Riemannian} product $(\hI \times \hN , s_1\oplus \hat{h})$.

Our next task is to show that the (Gaussian) curvature $K^{\hN}$ of $(\hN,\hh)$
is constant and non-zero.

By Proposition \ref{pr:c2.4}, $\star \hE_2|_{\hx}$, $\hx\in \hV$, is 
an eigenvector of $\hR|_{\hx}$ corresponding to the eigenvalue
$\hlambda(\hx)=-(K^2 + \hat{r}(\hx)^2)/K$.
The (Gaussian) curvature $K^{\hN}(\hy)$ of $(\hN,\hh)$ at any $\hy\in\hN$ being equal to
$-\hg(\hR|_{\hx}(\star \hE_2|_{\hx}),\star\hE_2|_{\hx})$,
we find
\begin{align}\label{eq:ss10:K_hN}
K^{\hN}(\hy)=-\hg(\hR|_{\hx}(\star \hE_2|_{\hx}),\star\hE_2|_{\hx})=(K^2+\hr(\hx)^2)/K,
\end{align}
because $\hE_2|_{\hx}$ is a unit vector. As $\hr\neq 0$ on $\hV$, we conclude that
$K^{\hN}\neq 0$ everywhere on $\hN$.

From \eqref{eq:ss10:hR:2}, which gives the matrix of $\hR|_{\hx}$ w.r.t
the basis $(\star A \XtildeA,\star A \YtildeA , \star \hZ_A)$ of $\bigwedge^2 T_{\hx} \hM$,
one can see that $\star\hat{W}_A$ is a (non-zero) eigenvector of $\hR|_{\hx}$
corresponding to the eigenvalue $\hlambda(\hx)$, if we define
\[
\hat{W}_A:=-\hat{r}(\hx) A\XtildeA + K\, \hZ_A.
\]
In view of Corollary \ref{c2.3} (item b)),
and Proposition \ref{pr:c2.4},
we thus have two non-zero eigenvectors $\hat{W}_A$ and $\hE_2|_{\hx}$
of $\hR|_{\hx}$ corresponding to its \emph{simple} eigenvalue $\hlambda(\hx)$.
Consequently, $\hE_2|_{\hx}$ and $\hat{W}_A$ are parallel vectors, for every $\q\in O(q_0)$,
say, $\hE_2|_{\hx}=\eta(q)\hat{W}_A$, with $\eta(q)\neq 0$.
Since $\n{\hW_A}_{\hg}^2=\hr(\hx)^2+K^2$, and $\n{\hE_2}_{\hg}=1$, we have $|\eta(q)|=(\hr(\hx)^2+K^2)^{-1/2}$.

But then, for every $\q\in O(q_0)$ one has
\[
|\hg(\hZ_A, \hE_2|_{\hx})|=|\eta(q)||\hg(\hZ_A,-\hat{r}(\hx) A\XtildeA + K\hZ_A)|=(K^2+\hr(\hx)^2)^{-1/2} |K|,
\]
and in view of Eq. \eqref{eq:ss10:K_hN} and the definition of the function $\alpha$ in \eqref{eq:ss10:alpha},
this can be recast into an identity
\begin{align}\label{eq:ss10:alpha:2}
\alpha(q)^2=\frac{K}{K^{\hN}(\hy)},
\end{align}
holding for every $\q\in O(q_0)$ with $(\hr,\hy)=\hF^{-1}(\hx)$.
From this we can deduce that $K^{\hN}$ is constant on $\hN$ by the following argument.
Let $\hy_1,\hy_2$ be two points $\hN$.
Then $(\hr_0,\hy_i)\in \hI\times \hN$, and so there are $\hx_i\in \hV$ such that $\hF(\hr_0,\hy_i)=\hx_i$, $i=1,2$.
Because $\hV=\pi_{Q,\hM}(O(q_0))$, there thus are two points $q_1,q_2\in O(q_0)$ such that
$\pi_{Q,M}(q_i)=\hx_i$.
Corollary \ref{cor:ss10:const_alpha} implies that $\alpha(q_1)=\alpha(q_2)$,
and hence Eq. \eqref{eq:ss10:alpha:2} that
$\frac{K}{K^{\hN}(\hy_1)}=\frac{K}{K^{\hN}(\hy_2)}$.
Thus $K^{\hN}(\hy_1)=K^{\hN}(\hy_2)$, i.e., the curvature $K^{\hN}$ is constant on $\hN$.

It remains to demonstrate the last two assertions in the statement of this theorem.
The result \cite[Theorem C.14]{ChitourKokkonen1} also tells us
that if we denote the canonical vector field on $\hI$ by $\pa{\hr}$,
and if we identify it with the obvious vector field on $\hI\times \hN$,
then
\[
\hF_*\pa{\hr}\big|_{(\hr,\hy)}=\hat{E}_2|_{\hF(\hat{r},\hat{y})},\quad \forall (\hat{r},\hat{y})\in \hI\times \hN,
\]
which then implies that $\hE_2^\perp|_{\hx_0}=(\hF_{\hat{r}_0})_*(T_{\hat{y}_0} \hN)$.
On the other hand, one has (for example)
\[
\hg(A_0\tilde{X}_{A_0},\hE_2|_{\hx_0})
=\eta(q_0)\hg(A_0\tilde{X}_{A_0}, -\hat{r}(\hx_0) A_0\tilde{X}_{A_0} + K\, \hZ_{A_0})
=-\eta(q_0)\hat{r}(\hx_0)\neq 0,
\]
since $\hat{r}\neq 0$ on $\hV$,
showing that $A_0\tilde{X}_{A_0}\notin \hE_2^\perp|_{\hx_0}=(\hF_{\hat{r}_0})_*(T_{\hat{y}_0} \hN)$.
As $A_0 T_{x_0} M=\spn\{A_0\tilde{X}_{A_0},A_0\tilde{Y}_{A_0}\}$,
we can conclude that $A_0 T_{x_0} M\neq (\hF_{\hat{r}_0})_*(T_{\hat{y}_0} \hN)$ as asserted.

Finally, suppose that $\hE_2|_{\hx_0}\in A_0 T_{x_0} M$. Then because $\hg(A_0\tilde{Y}_{A_0},\hE_2|_{\hx_0})=0$,
this would imply that $\hE_2|_{\hx_0}=\pm A_0 \Xtilde_{A_0}$
since $\hE_2|_{\hx_0}$ and $A_0 \Xtilde_{A_0}$ are unit vectors,
and hence
\[
1=|\hg(A_0\tilde{X}_{A_0},\hE_2|_{\hx_0})|=\hat{r}(\hx_0)|\eta(q_0)|=\hr(\hx_0)(\hr(\hx_0)^2+K^2)^{-1/2}.
\]
This would mean that $K=0$, which is a contradiction in view of Proposition \ref{p2.3}.
Thus we conclude that $\hE_2|_{\hx_0}\notin A_0 T_{x_0} M$ as claimed.
\end{proof}

\begin{remark}
\begin{itemize}
\item[(i)] Referring to the statement of Theorem \ref{th:5.17},
since $\hN$ is a totally geodesic submanifold
of the Riemannian product $(\hI\times\hN, s_1\oplus\hat{h})$,
it follows that $\hN':=F_{\hat{r}_0}(\hN)$ is a totally
geodesic submanifold of $(\hM,\hg)$
and $\hat{h}=(F_{\hat{r}_0})^*(\hg|_{\hN'})$.

Thus, in view of Propositions \ref{pr:5.10} and \ref{p1.12},
if it were the case that $A_0 T_{x_0} M=T_{\hx_0} \hN'$,
then the rolling neighbourhood $O(q_0)$ (chosen small enough around $q_0$)
would either have dimension $5$,
or dimension $2$, the latter being the case
when there exists an isometry $H:(M,g)\to (\hN',\hg|_{\hN'})$
such that $H_*|_{x_0}=A_0$.

These observations provide an alternative proof
for the property $A_0 T_{x_0} M\neq (F_{\hat{r}_0})_* T_{\hat{y}_0} \hN$
stated in Theorem \ref{th:5.17}.

\item[(ii)]
Substituting the expression for $\alpha$ from \eqref{eq:ss10:alpha}
into \eqref{eq:ss10:alpha:2} yields
\[
\hg(\hZ_A, \hE_2|_{\hx})^2 K^{\hN}=K,
\]
an identity which holds at every $\q\in O(q_0)$ if we write $(\hr,\hy):=\hF^{-1}(\hx)$.

\item[(iii)]
As mentioned in the proof of Theorem \ref{th:5.17},
one has
\[
\hF_*\pa{\hr}\big|_{(\hr,\hy)}=\hat{E}_2|_{\hF(\hat{r},\hat{y})},\quad \forall (\hat{r},\hat{y})\in \hI\times \hN,
\]
where $\pa{\hr}$ is the canonical vector field on $\hI$,
which we also consider as a vector field on $\hI\times \hN$ in the obvious way.

\item[(iv)]
If $\alpha(q)=\hg(\hZ_A,\hE_2|_{\hx})$ is the function defined in \eqref{eq:ss10:alpha},
then it is clear that $|\alpha(q)|\leq 1$ for all $\q\in O(q_0)$ since $\hZ_A$ and $\hE_2|_{\hx}$
are unit vectors.
Moreover, since $\hZ_{A_0}\perp A_0 T_{x_0} M$ we have
$\alpha(q_0)=0$ if and only if $\hE_2|_{\hx_0}\in A_0 T_{x_0} M$, which is impossible by Theorem \ref{th:5.17}.
Finally, $|\alpha(q_0)|=1$ if and only if $\hE_2|_{\hx_0}=\pm \hZ_{A_0}$,
which is equivalent to having $\hE_2|_{\hx_0}\perp A_0 T_{x_0} M$, i.e., $(\hF_{\hat{r}_0})_* T_{\hat{y}_0} \hN = A_0 T_{x_0} M$,
which is again ruled out by Theorem \ref{th:5.17}.
We thus conclude that $0<|\alpha(q_0)|<1$,
i.e.,
\[
0<|\hg(\hZ_{A_0},\hE_2|_{\hx_0})|<1.
\]

\end{itemize}
\end{remark}

\vspace{2\baselineskip}\subsection{Converse to the Main Theorem of Section \ref{ss10}}\label{sec:conv_ss10}
\ \newline

In this subsection, we shall
consider the case where $(M,g)$ is a connected, oriented $2$-dimensional
Riemannian manifold of \emph{constant non-zero curvature} $K$, i.e., a space form,
and $(\hM,\hg)=(\hI\times\hN, s_1\oplus \hh)$
is a \emph{Riemannian product} of an oriented 2-dimensional Riemannian manifold $(\hN,\hh)$
of \emph{constant non-zero curvature} $K^{\hN}$
and an open non-empty interval $\hI\subset\R$ equipped with the standard Riemannian metric $s_1$.

Our goal will be to show that
for all initial states $\qz$ taken from a certain open dense subset of $Q$
(the set $Q_1$ to be defined below),
the orbit $\odr(q_0)$ of rolling of $(M,g)$ against $(\hM,\hg)$ has dimension $\dim \odr(q_0)=7$.

This provides a converse result for
Proposition \ref{pr:5.14} and Theorem \ref{th:5.17},
in the sense that $\dr$-orbits of dimension $7$ do occur
in the geometric setting described by that theorem.

Let $\pa{\hr}$ be the canonical vector field on $\hI$, which we shall identify with a vector field
on $\hI\times\hN$ in the obvious way.
We fix the orientation on $(M,g)$ in one of the two ways, and take for the orientation of $\hM$
the one for which the frame of vectors $(\hE_1,\pa{\hr},\hE_3)$ of $\hM$ at a given point $\hx_0=(\hr_0,\hy_0)$ is positively oriented on $\hM$,
whenever the frame $(\hE_1,\hE_3)$ of $\hN$ at $\hy_0$ is positively oriented on $\hN$.

For any $\q\in Q$, let $\hZ_A$ be the unit vector in $T_{\hx} \hM$
such that if $(X,Y)$ is any oriented orthonormal basis on $T_x M$,
then $\hZ_A=\star (AX\wedge AY)$,
with the operator $\star$ being the Hodge dual on $(\hM,\hg)$ as usual.
As before, it can be checked that the map $\hZ_{(\cdot)}:Q\to T\hM$; $\q\to \hZ_A$ is smooth.
With the use this mapping, one can then define a smooth function
\begin{align}\label{eq:conv_ss10:alpha}
\alpha:Q\to\R;\quad \alpha(q):=\hg\big(\hZ_A,\pa{\hr}\big|_{\hr}\big),
\end{align}
for $\q\in Q$ and with $(\hr,\hy)=\hx$.
As both $\hZ_A$ and $\pa{\hr}\big|_{\hr}$ are unit vectors on $\hM$,
it is clear that $|\alpha(q)|\leq 1$ for all $q\in Q$.

The first half of the goal at hand will be to show the upper limit
$\dim\odr(q_0)\leq 7$
for orbits passing through certain point $q_0\in Q$.
This will be achieved by proving that:
\begin{itemize}
\item[(i)] $\alpha$ is constant on every rolling orbit $\odr(q_0)$, $q_0\in Q$;
and
\item[(ii)] $\alpha$ restricts to a submersion $Q_1\to\R$
on the open subset $Q_1$ of $Q$ defined as
\begin{align}\label{eq:conv_ss10:Q_1}
Q_1:=\{q\in Q\ |\ |\alpha(q)|<1\}.
\end{align}
\end{itemize}

Indeed, once (i) and (ii) have been established, we can show that $\dim \odr(q_0)\leq 7$ for any $q_0\in Q_1$ as follows:
Consider the set $S_{q_0}:=\{q\in Q_1\ |\ \alpha(q)=\alpha(q_0)\}$.
As $q\in S_{q_0}$ implies that $|\alpha(q)|=|\alpha(q_0)|<1$,
we have $S_{q_0}\subset Q_1$.
By item (ii), the map $\alpha$ has rank 1 on $Q_1$,
and therefore on $S_{q_0}$, which implies that
$S_{q_0}$ is a (closed) smooth embedded submanifold of the open subset $Q_1$ of $Q$ of codimension $1$, that is $\dim S_{q_0}=\dim Q_1-1=8-1=7$.
Finally, item (i) implies that $\odr(q_0)\subset S_{q_0}$,
whence
$\dim \odr(q_0)\leq \dim S_{q_0}=7$.

We shall now demonstrate the assertion (i) i.e., for every $q_0\in Q$, the function $\alpha$ is constant on the orbit $\odr(q_0)$.
 To that end, let $(X,Y)=(E_1,E_2)$ be an arbitrary smooth, oriented local orthonormal frame on $(M,g)$, defined on an open subset $V$ of $M$.
On $V$ we have
\[
{}& (A\nabla_X X)\wedge AY=0,\quad AX\wedge (A\nabla_X Y)=0 \\
{}& (A\nabla_Y X)\wedge AY=0,\quad AX\wedge (A\nabla_Y Y)=0.
\]
Consequently, as $\hZ_A$ is defined by $\hZ_A=\star (AX\wedge AY)$, it follows that
\begin{align}\label{eq:conv_ss10:lr_E_i_hZ}
\lr(E_i)\onq \hZ_{(\cdot)}=0,\quad \forall q\in Q,
\end{align}
and therefore
\[
\lr(E_i)\onq \alpha = \hg\big(\hZ_A, \hnabla_{AE_i} \pa{\hr}\big),\quad \forall \q\in Q,\ i=1,2.
\]

For any given state $\q\in Q$ such that $x\in V$,
write $\hx=(\hr,\hy)$ and let $(\hE_1,\hE_3)$ be a smooth, local orthonormal frame on $(\hN,\hh)$ defined on an open neighbourhood $\hU$ of $\hx$ in $\hM$.
Considering $\hE_1$ and $\hE_2$ as vector fields on $\hM=\hI\times\hN$ in the obvious way, and writing
$\hE_2:=\pa{\hr}$ on $\hV$,
we get a local orthonormal frame $(\hE_1,\hE_2,\hE_3)$ of $\hM$
defined on the open subset $\hI\times \hV$ of $\hM$.

One can thus write $AE_i=\sum_{j=1}^3 \hg(AE_i,\hE_j)\hE_j$, $i=1,2$.
Recalling $\hE_2=\pa{\hr}$ is a geodesic vector field on $\hM=\hI\times\hN$,
and that $\hM$ is a \emph{Riemannian product} of $(\hI,s_1)$ and $(\hN,\hh)$,
we find that $\hnabla_{\hE_2} \hE_2=0$ and $\hnabla_{\hE_j} \hE_2=0$ for $j=1,3$. Hence
\begin{align}\label{eq:conv_ss10:hnabla_AE_i_hE_2}
\hnabla_{AE_i} \pa{\hr}=\sum_{j=1}^3 \hg(AE_i,\hE_j) \hnabla_{\hE_j} \hE_2=0,
\end{align}
and we can conclude that
\begin{align}\label{eq:conv_ss10:lr_E_i_alpha}
\lr(E_i)\onq \alpha = 0,\quad i=1,2.
\end{align}

Because $(X,Y)=(E_1, E_2)$ was an arbitrary oriented local 
orthonormal frame on $(M,g)$ and $\q$ was an arbitrary point in $Q$
such that $x$ lies in the domain of definition of that frame,
we have thus shown that
\[
\lr(Z)\onq \alpha = 0,\quad  \forall \q\in Q,\ \forall Z\in T_x M.
\]
Since the vectors $\lr(Z)\onq$ span $\dr|_q$,
this property of $\alpha$ is equivalent to $\alpha$ being constant on 
every orbit $\dr$ in $Q$ (i.e., on every $\odr(q_0)$, $q_0\in Q$).
This proves that the claim (i) above is true.

Our next task is to prove the claim (ii),
i.e., that the map $\alpha:Q\to\R$ has rank 1
on the open subset $Q_1$ of $Q$.
To achieve that goal, we take an arbitrary $\q\in Q$,
an orthonormal frame $(X,Y)$ of $T_x M$,
and we differentiate $\alpha$ with respect to $\nu(\theta_X\otimes \hZ_A)\onq$ and  $\nu(\theta_Y\otimes \hZ_A)\onq$.
In fact, we have (see also the proof of Lemma \ref{le:6.3})
\[
\nu(\theta_X\otimes \hZ_A)\onq \hZ_{(\cdot)}
={}& \star \big((\theta_X\otimes \hZ_A)X)\wedge AY)+\star \big(AX\wedge (\theta_X\otimes \hZ_A)Y)\big) \\
={}& \star (\hZ_A\wedge AY) + 0
=-AX,
\]
while $\nu(\theta_X\otimes \hZ_A)\pa{\hr}=0$,
because $\hE_2:=\pa{\hr}$ is a vector field on $\hM$,
i.e., considered as a map $Q\to T\hM$,
it is a composition of $\pi_{Q,\hM}:Q\to\hM$ and
$\hE_2:\hM\to T\hM$.
Thus
\[
\nu(\theta_X\otimes \hZ_A)\onq \alpha=\hg\big(-AX,\pa{\hr}\big|_{\hr}\big).
\]
Similarly, we have
\[
\nu(\theta_Y\otimes \hZ_A)\onq \hZ_{(\cdot)}=0+\star (AX\wedge \hZ_A)
=-AY
\]
and hence
\[
\nu(\theta_Y\otimes \hZ_A)\onq \alpha=\hg\big(-AY,\pa{\hr}\big|_{\hr}\big).
\]

Consequently,
$\nu(\theta_X\otimes \hZ_A)\onq \alpha=0$
and $\nu(\theta_Y\otimes \hZ_A)\onq \alpha=0$
if and only if
$AX$ and $AY$ are both orthogonal to $\pa{\hr}\big|_{\hr}$.
But this is equivalent to the vector $\star (AX\wedge AY)=\hZ_A$ being
parallel to $\pa{\hr}\big|_{\hr}$,
and since both are unit vector,
that happens if and only if
$\big|\hg\big(\hZ_A,\pa{\hr}\big|_{\hr}\big)\big|=|\alpha(q)|=1$.
In particular, $\alpha:Q\to \R$ has rank 1 at any point $q$ of $Q_1$.
This concludes the proof of the assertion in item (ii) above.

As already explained, the items (i) and (ii) just established
readily imply that the dimension of the orbit $\odr(q_0)$,
for any $q_0\in Q_1$, is at most $7$.

Before we proceed further, let us observe
the set $Q_1$ is dense in $Q$, in addition to being open.
First note that, for instance,
\[
\nu(\theta_X\otimes AX)\onq \hZ_{(\cdot)}
=\star (AX\wedge AY)=\hZ_A,
\]
and hence
\[
\nu(\theta_X\otimes AX)\onq \alpha=\hg\big(\hZ_A,\pa{\hr}\big|_{\hr}\big)=\alpha(q),
\]
for any $\q\in Q$.
If $\Gamma(t)$ is any smooth curve in $Q$ 
whose tangent vector at $t=0$ equals $\nu(\theta_X\otimes AX)\onq$
(and in particular $\Gamma(0)=q$),
then $\dif{t}\big|_{0} \alpha(\Gamma(t))=\nu(\theta_X\otimes AX)\onq \alpha=\alpha(q)$,
and hence
if $\alpha(q)\neq 0$,
we have $\alpha(\Gamma(t))\neq \alpha(q)$ for all $t\neq 0$ small enough.
In particular, if $q\in Q\setminus Q_1=\{q'\in Q\ |\ |\alpha(q')|=1\}$
it follows that $\alpha(\Gamma(t))\neq 1$, i.e.,
$\Gamma(t)\in Q_1$ for all $t$ small enough.
This shows that $Q_1$ is indeed dense in $Q$ as claimed.

The remainder of this section will go into showing that $\dim \odr(q_0)$ actually equals $7$
for any $q_0\in \{q\in Q\ |\ 0<|\alpha(q)|<1\}=:Q_{0,1}$.
Knowing already that $\odr(q_0)$ is at most 7-dimensional,
because $Q_{0,1}\subset Q_1$,
it suffices for us to show that it is at least 7-dimensional.

Without loss of generality, we may assume that $(X,Y)=(E_1,E_2)$ and $(\hE_1,\hE_3)$ are both global frames on $M$ and $\hN$, respectively.

At every $\q\in Q$, the vector $\hZ_A$ decomposes into a $\hg$-orthogonal sum
$\hZ_A=\alpha(q)\pa{\hr}\big|_{\hr}+\hH_A$, where $\hH_A\in T_{\hy} \hN$
and $(\hr,\hy)=\hx$.

Since for $\q\in Q$ (see \eqref{eq:hsigma_Pi})
\[
\hsigma_A=-\hg(\hR(\star \hZ_A),\star \hZ_A),
\]
and $(\hM,\hg)$ is the \emph{Riemannian product} of $(\hI,s_1)$ and $(\hN,\hh)$,
with the curvature of the former being zero, while that of the latter being $R^{\hN}=-K^{\hN} \mathrm{Id}$,
we have $\hR(\star \hH_A)=0$, $\hR(\star \pa{\hr})=-K^{\hN}\star \pa{\hr}$,
and hence $\hR(\star \hZ_A)=\hR\big(\alpha(q) (\star \pa{\hr})+(\star \hH_A)\big)=-\alpha(q) K^{\hN} (\star \pa{\hr})$.
It follows from this and the definition of $\hsigma_A$ that
\begin{align}\label{eq:conv_ss10:hsigma_A}
\hsigma_A=\alpha(q)^2 K^{\hN},\quad \forall \q\in Q.
\end{align}

Similarly, $\Pi_X$ and $\Pi_Y$ (see \eqref{eq:hsigma_Pi}) are given by 
\[
\Pi_1(q):=\Pi_X(q)={}&\hg(\hR(\star \hZ_A),\star AX)=\hg\big(-K^{\hN}\star \big(\alpha(q)\pa{\hr}\big), \star AX)
=-\alpha(q) K^{\hN} \hg\big(\pa{\hr}, AE_1\big) \\
\Pi_2(q):=\Pi_Y(q)={}&\hg(\hR(\star \hZ_A),\star AY)=\hg\big(-K^{\hN}\star \big(\alpha(q)\pa{\hr}\big), \star AY)
=-\alpha(q) K^{\hN} \hg\big(\pa{\hr}, AE_2\big),
\]
for all $\q\in Q$,
or if we introduce the functions smooth function on $Q$
\begin{align}\label{eq:conv_ss10:tau_i}
\tau_i(q):=\hg\big(AE_i,\pa{\hr}\big),\quad i=1,2,
\end{align}
then
\[
\Pi_i(q)=-\alpha(q) K^{\hN} \tau_i(q),\quad \forall q\in Q, i=1,2.
\]

In order to avoid notational clutter, we introduce the smooth functions $Q\to\R$
\begin{align}\label{eq:conv_ss10:Theta_Psi}
\Theta(q):=K-\alpha(q)^2 K^{\hN},
\quad
\Psi(q):=\alpha(q) K^{\hN},
\end{align}
and notice that, due to the identity \eqref{eq:conv_ss10:lr_E_i_alpha}
and the fact that $K$ and $K^{\hN}$ are constants,
one has
\begin{align}\label{eq:conv_ss10:lr_E_i_Theta_Psi}
\lr(E_i)\onq \Theta=0,\quad \lr(E_i)\onq \Psi=0,\quad \forall q\in Q,\ i=1,2.
\end{align}

The vectors $AE_1, AE_2, \hZ_A$ being orthonormal,
and $\pa{\hr}$ being a unit vector,
the above definitions of $\tau_1$, $\tau_2$ and $\alpha$
directly imply
\begin{align}\label{eq:conv_ss10:resolution_of_identity}
\tau_1(q) AE_1+\tau_2(q) AE_2+\alpha(q)\hZ_A=\pa{\hr},
\end{align}
and
\begin{align}\label{eq:conv_ss10:orthonorm}
\tau_1(q)^2+\tau_2(q)^2+\alpha(q)^2=1
\end{align}
hold at every point $q\in Q$.

The expression for $\Rol_q=\Rol_q(X,Y)$ in \eqref{eq:RolXY}
can now be rewritten into the form
\[
\Rol_q ={}&
-\Theta(q)A(E_1\wedge E_2) + \Psi(q) (-\tau_2(q)\theta_{E_1} + \tau_1(q) \theta_{E_2})\otimes \hZ_A,
\]
at any point $\q\in Q$,
where, we recall, that for a vector field $Z$ on $M$, we have defined $\theta_Z=g(Z,\cdot)$.
Observing that
\[
\widetilde{T}(q):=-\tau_2(q)\theta_{E_1} + \tau_1(q) \theta_{E_2}
={}&-\hg\big(\pa{\hr}, AE_2\big)\hg(AE_1,A(\cdot))
+\hg\big(\pa{\hr}, AE_1\big)\hg(AE_2,A(\cdot)) \\
={}& -\hg\big(\pa{\hr}, (AE_1\wedge AE_2)A(\cdot)\big)
=-\hg\big(\pa{\hr}, (\star \hZ_A)A(\cdot)\big) \\
={}& -\hat{\theta}_{\hE_2}(\star \hZ_A)A,
\]
where for a vector field $\hU$ on $\hM$, we define $\hat{\theta}_{\hU}:=\hg(\hU,\cdot)$.
Consequently, we may represent $\Rol_q$ in the form
\begin{align}\label{eq:conv_ss10:Rol}
\Rol_q
={}&-\Theta(q)A(E_1\wedge E_2)
+\Psi(q) \widetilde{T}(q) \otimes \hZ_A.
\end{align}

Next we compute the derivatives $\lr(E_i)$ of $\tau_j$, $i,j=1,2$,
\[
\lr(E_i)\onq \tau_j=\hg\big(\hnabla_{AE_i} \pa{\hr},AE_j\big)+\hg\big(\pa{\hr}, A\nabla_{E_i} E_j)
=\sum_{k=1}^2 \Gamma^i_{(j,k)}(x) \tau_k(q),\quad \forall \q\in Q,
\]
thanks to \eqref{eq:conv_ss10:hnabla_AE_i_hE_2}, the definition of the connection coefficients $\Gamma^i_{(j,k)}$
(see Section \ref{sec:preliminaries}),
and the definition of the functions $\tau_k$ in \eqref{eq:conv_ss10:tau_i}.
Also, for every $i,j=1,2$ and $q\in Q$,
\[
\lr(E_i)\onq \theta_{E_j}=\theta_{\nabla_{E_i} E_j}=\sum_{k=1}^2 \Gamma^i_{(j,k)}(x) \theta_{E_k}.
\]

Because $\lr(E_i)\onq \hat{\theta}_{\hE_2}=\hat{\theta}_{\hnabla_{A E_i} \hE_2}=0$
as we have already observed above,
and given the identity \eqref{eq:conv_ss10:lr_E_i_hZ},
these derivatives of the operator $\widetilde{T}$ are thus
\[
\lr(E_i)\onq \widetilde{T}
=-\hat{\theta}_{\hnabla_{AE_i} \hE_2} (\star \hZ_A)A
+\hat{\theta}_{\hE_2} \big( \star ( \lr(E_i)\onq \hZ_{(\cdot)})\big)A=0,
\]
for all $\q\in Q$, $i=1,2$.

In addition, as $A(E_1\wedge E_2)=(AE_1\wedge AE_2)A=(\star \hZ_A)A$,
we have by \eqref{eq:conv_ss10:lr_E_i_hZ} that
\[
\lr(E_i)\onq \big((\cdot)(E_1\wedge E_2)\big)=0,\quad \forall q\in Q,\ i=1,2.
\]
In view of \eqref{eq:conv_ss10:Rol},
the last two identities and the ones in Eq. \eqref{eq:conv_ss10:lr_E_i_Theta_Psi}
allow us to conclude that
\begin{align}\label{eq:conv_ss10:lr_E_i_Rol}
\lr(E_i)\onq \Rol=0,\quad \forall q\in Q,\ i=1,2.
\end{align}

Given these identities,
and the facts that $K$, $K^{\hN}$ are constant, while $\lr(E_i) \onq \alpha=0$ by \eqref{eq:conv_ss10:lr_E_i_alpha},
the commutator of the vector fields $\lr(E_i)$ and $\nu(\Rol)$
on $Q$ becomes (see Proposition \ref{p1.7})
\begin{align}\label{eq:conv_ss10:commutator:lr_E_i_nu_Rol}
[\lr(E_i), \nu(\Rol)]\onq
={}& -\lns(\Rol_q E_i)\onq
+\nu\big(\lr(E_i)\onq \Rol\big)\onq \nonumber \\[2mm]
={}& (-1)^{i}\lns\big(\Theta(q) AE_{\sigma(i)}
+\Psi(q) \tau_{\sigma(i)} \hZ_A\big)\onq,
\end{align}
for every $i=1,2$ and $\q\in Q$,
and where $\sigma$ is defined as (the cyclic permutation of $\{1,2\}$ of length 2)
\[
\sigma:\{1,2\}\to \{1,2\};\quad \sigma(1)=2,\ \sigma(2)=1.
\]
At the last equality, we have made use of the identity
\begin{align}\label{eq:conv_ss10:Rol_E_i}
\Rol_q E_i=(-1)^{i}\big(\Theta(q) AE_{\sigma(i)}
+\Psi(q) \tau_{\sigma(i)}(q) \hZ_A\big),
\end{align}
holding at every $\q\in Q$, and
which follows from observing that
\[
(E_1\wedge E_2)E_i={}& (-1)^{i+1} E_{\sigma(i)} \\
\widetilde{T}(q)E_i=-\hat{\theta}_{\hE_2} A(E_1\wedge E_2)E_i
={}& -(-1)^{i+1}\hg(\hE_2,A E_{\sigma(i)})
=-(-1)^{i+1} \tau_{\sigma(i)}(q).
\]

Defining vector fields $L_j$, $j=1,2$, on $Q$ by setting
\begin{align}\label{eq:conv_ss10:L_j}
L_j\onq:=\lns\big(\hat{\xi}_j|_q\big)\onq,\quad q\in Q,\ j=1,2,
\end{align}
where
\begin{align}\label{eq:conv_ss10:hxi}
\hat{\xi}_j|_q:=\Theta(q) AE_{j} + \Psi(q) \tau_{j}(q) \hZ_A,\quad \q\in Q,\ j=1,2,
\end{align}
the commutator identity \eqref{eq:conv_ss10:commutator:lr_E_i_nu_Rol}
implies that these vector fields $L_1$, $L_2$
are tangent to every orbit of $\dr$ in $Q$.

The next task we will undertake is to compute the Lie brackets of the vector fields $\lr(E_i)$ with $L_j$, $i,j=1,2$.
At $\q\in Q$ they are given by (see Proposition \ref{p1.7})
\begin{align}\label{eq:conv_ss10:commutator:lr_E_i_L_j}
[\lr(E_i), L_j]\onq
={}&
\lns\big(\lr(E_i)\onq \hat{\xi}_j\big)\onq - \lr\big(L_j E_i\big)\onq
+\nu\big(-\hR(AE_i, \hat{\xi}_j|_q)A\big)\onq.
\end{align}
In order for this formula to be of practical use for us,
we need to evaluate explicitly the quantities appearing inside $\lns(\cdot)\onq$ and $\nu(\cdot)\onq$.

Starting with the latter, 
we find, due to the fact that $(\hM,\hg)$ is a Riemannian product of $(\hI,s_1)$ and $(\hN,\hh)$,
\begin{align}\label{eq:conv_ss10:hR:1}
\hR(AE_i, \hat{\xi}_j|_q)
={}&
-K^{\hN}\Big(AE_i-\hg\big(AE_i,\pa{\hr}\big)\pa{\hr}\Big)\wedge \Big(\hat{\xi}_j|_q - \hg\big(\hat{\xi}_j|_q,\pa{\hr}\big)\pa{\hr}\Big)
\end{align}
for all $\q\in Q$, $i=1,2$.
In this expression, one has
\[
AE_i-\hg\big(AE_i,\pa{\hr}\big)\pa{\hr}
= {}&
AE_i-\tau_i(q)\pa{\hr} \\
\hat{\xi}_j|_q - \hg\big(\hat{\xi}_j|_q,\pa{\hr}\big)\pa{\hr}
= {}&
\hat{\xi}_j|_q - \tau_j(q)\big(\Theta(q) + \Psi(q) \alpha(q)\big)\pa{\hr} \\
= {}&
\Theta(q)\big(AE_j - \tau_j(q)\pa{\hr}\big)
+\Psi(q) \tau_j(q) \big(\hZ_A - \alpha(q)\pa{\hr}\big).
\]

Recall that given an oriented orthonormal basis $(\hX', \hY', \hZ')$ of $T_{\hx} \hM$, $\hx\in \hM$,
and a unit vector $\hU'\in T_{\hx} \hM$,
then
\begin{align}\label{eq:conv_ss10:cross_prod_identity}
(\hX'-\hg(\hX',\hU')\hU')\wedge (\hY'-\hg(\hY',\hU')\hU')=\hg(\hZ',\hU') \star \hU'.
\end{align}

Setting in this identity $\hX'=AE_1$, $\hY'=AE_2$, $\hZ'=\hZ_A$, and $\hU'=\pa{\hr}$,
and recalling that $\tau_i(q)=\hg(AE_i,\pa{\hr})$, $i=1,2$,
and $\alpha(q)=\hg(\hZ_A,\pa{\hr})$,
we obtain
\[
\big(AE_1-\tau_1(q)\pa{\hr}\big)\wedge \big(AE_2 - \tau_2(q)\pa{\hr}\big)
={}& \alpha(q) \star \pa{\hr},
\]
where \eqref{eq:conv_ss10:orthonorm} has used at the last step
(here the assumption $i\neq j$ was used).

The other outer product we need to compute
is obtained from \eqref{eq:conv_ss10:cross_prod_identity}
by taking $\hX'=AE_i$, $\hY'=(-1)^{i}\hZ_A$, $\hZ'=AE_{\sigma(i)}$ and $\hU'=\pa{\hr}$,
\[
\big(AE_i-\tau_i(q)\pa{\hr}\big)\wedge \big(\hZ_A - \alpha(q)\pa{\hr}\big)
={}& (-1)^i \tau_{\sigma(i)}(q) \star \pa{\hr}.
\]

One can now re-express \eqref{eq:conv_ss10:hR:1},
for $i\neq j$ (hence $\sigma(i)=j$), in the form
\begin{align}\label{eq:conv_ss10:hR:2}
\hR(AE_i, \hat{\xi}_j|_q)
={}&
-K^{\hN}\big((-1)^{i+1}\Theta(q)\alpha(q) + (-1)^i \Psi(q) \tau_j(q) \tau_{\sigma(i)}(q)\big)\star \pa{\hr} \\
={}&
(-1)^{i}K^{\hN}\big(\alpha(q)\Theta(q) - \tau_j(q)^2\Psi(q)\big)\star \pa{\hr},
\quad i\neq j,
\end{align}
and when $i=j$,
\begin{align}\label{eq:conv_ss10:hR:3}
\hR(AE_i, \hat{\xi}_i|_q)
={}& (-1)^{i+1} K^{\hN} \tau_i(q)\tau_{\sigma(i)}(q) \Psi(q) \star \pa{\hr},
\quad i=1,2.
\end{align}

Next we need to evaluate the quantities $\lr(E_i)\onq \hat{\xi}_j$ appearing
inside $\lns(\cdot)\onq$ in \eqref{eq:conv_ss10:commutator:lr_E_i_L_j}.
In fact, using
\eqref{eq:conv_ss10:lr_E_i_hZ}
and
\eqref{eq:conv_ss10:lr_E_i_Theta_Psi}
we find that
\[
\lr(E_i)\onq \hat{\xi}_j
={}& \lr(E_i)\onq \big(\Theta (\cdot)E_{j} + \Psi \tau_j \hZ_{(\cdot)}\big) \\
={}& \Theta(q) (-1)^{j+1} \Gamma^i_{(1,2)}(x) AE_{\sigma(j)} + \Psi(q) (\lr(E_i)\onq \tau_j) \hZ_A,
\]
for $\q\in Q$ and $i,j=1,2$.
By the definition of $\tau_j$ in \eqref{eq:conv_ss10:tau_i},
at every $\q\in Q$ and $i,j=1,2$, one has
\[
\lr(E_i)\onq \tau_j
=\hg\big((-1)^{j+1} \Gamma^i_{(1,2)}(x) AE_{\sigma(j)},\pa{\hr}\big)+\hg\big(AE_j,\hnabla_{AE_i} \pa{\hr}\big)
=(-1)^{j+1} \Gamma^i_{(1,2)}(x) \tau_{\sigma(j)}(q),
\]
where \eqref{eq:conv_ss10:hnabla_AE_i_hE_2} has been used.
Thus 
\begin{align}\label{eq:conv_ss10:lr_E_i_hxi_j}
\lr(E_i)\onq \hat{\xi}_j
={}& (-1)^{j+1}\Gamma^i_{(1,2)}(x) \big(\Theta(q) AE_{\sigma(j)} + \Psi(q) \tau_{\sigma(j)}(q) \hZ_A\big) \nonumber \\
={}&(-1)^{j+1}\Gamma^i_{(1,2)}(x) \hat{\xi}_{\sigma(j)},
\end{align}
so that combining this with \eqref{eq:conv_ss10:hR:2},
the commutator in Eq. \eqref{eq:conv_ss10:commutator:lr_E_i_L_j} becomes
\begin{align}\label{eq:conv_ss10:commutator:lr_E_i_L_j:2}
[\lr(E_i), L_j]\onq =- \lr\big(L_j E_i\big)\onq + H_{i,j}'',
\end{align}
where for $q\in Q$ and $i,j=1,2$,
\begin{align}\label{eq:conv_ss10:F_ij}
H_{i,j}''\onq:={}& (-1)^{j+1}\Gamma^i_{(1,2)}(x) \lns\big(\hat{\xi}_{\sigma(j)}\big)\onq
+ (-1)^{i+1}K^{\hN}\big(\alpha(q)\Theta(q) - \tau_j(q)^2\Psi(q)\big) \nu\big(\star \pa{\hr} A\big)\onq,
\quad i\neq j \nonumber \\
H_{i,i}''\onq:={}&
(-1)^{i+1}\Gamma^i_{(1,2)}(x) \lns\big(\hat{\xi}_{\sigma(i)}\big)\onq
+(-1)^{i} K^{\hN} \tau_i(q)\tau_{\sigma(i)}(q) \Psi(q) \nu\big(\star \pa{\hr} A\big)\onq.
\nonumber \\[2mm]
\end{align}
Notice that by \eqref{eq:conv_ss10:commutator:lr_E_i_L_j:2}, the vector fields $H_{i,j}''$, $i,j=1,2$, are tangent to every orbit of $\dr$ in $Q$.
Explicitly $H_{i,j}''$ look like
\[
H_{1,1}''\onq={}&
\Gamma^1_{(1,2)}(x) \lns\big(\hat{\xi}_{2}\big)\onq
-K^{\hN}\tau_1(q)\tau_2(q)\Psi(q) \nu\big(\star \pa{\hr} A\big)\onq \\
H_{1,2}''\onq={}&
-\Gamma^1_{(1,2)}(x) \lns\big(\hat{\xi}_{1}\big)\onq
+K^{\hN}\big(\alpha(q)\Theta(q) - \tau_2(q)^2\Psi(q)\big) \nu\big(\star \pa{\hr} A\big)\onq \\
H_{2,2}''\onq={}&
-\Gamma^2_{(1,2)}(x) \lns\big(\hat{\xi}_{1}\big)\onq
+ K^{\hN} \tau_{1}(q)\tau_2(q)\Psi(q) \nu\big(\star \pa{\hr} A\big)\onq \\
H_{2,1}''\onq={}&
\Gamma^2_{(1,2)}(x) \lns\big(\hat{\xi}_{2}\big)\onq
- K^{\hN}\big(\alpha(q)\Theta(q) - \tau_1(q)^2\Psi(q)\big) \nu\big(\star \pa{\hr} A\big)\onq.
\]

By \eqref{eq:conv_ss10:commutator:lr_E_i_L_j:2},
the vector fields $H_{i,j}''$ of $Q$, $i,j=1,2$,
are tangent to every $\dr$-orbit in $Q$,
and since we know that the vector fields $L_j=\lns\big(\hat{\xi}_j\big)$, $j=1,2$,
are too (see \eqref{eq:conv_ss10:L_j}),
it follows from
the above expressions for $H_{i,j}''$,
and the fact that $K^{\hN}\neq 0$,
that the vector fields $H'_{i,j}$ of $Q$
defined by
\[
H_{1,1}'\onq:={}&
-\tau_1(q)\tau_2(q)\Psi(q) \nu\big(\star \pa{\hr} A\big)\onq \\
H_{1,2}'\onq:={}&
\big(\alpha(q)\Theta(q) - \tau_2(q)^2\Psi(q)\big) \nu\big(\star \pa{\hr} A\big)\onq \\
H_{2,2}'\onq:={}& \tau_{1}(q)\tau_2(q)\Psi(q) \nu\big(\star \pa{\hr} A\big)\onq \\
H_{2,1}'\onq:={}&
-\big(\alpha(q)\Theta(q) - \tau_1(q)^2\Psi(q)\big) \nu\big(\star \pa{\hr} A\big)\onq,
\]
are all tangent to every $\dr$-orbit in $Q$.

Assuming that $\alpha(q)\neq 0$,
we see that if all the coefficients of $\nu\big(\star \pa{\hr} A\big)$
in vectors $H'_{i,j}\onq$ were to vanish at $q$,
then (see \eqref{eq:conv_ss10:Theta_Psi})
\[
\tau_1(q)\tau_2(q)=0,\quad
K-(\alpha(q)^2 + \tau_{i}(q)^2)K^{\hN}=0,\quad i=1,2,
\]
i.e.,
$\tau_1(q)\tau_2(q)=0$ and $\tau_{i}(q)^2=K/K^{\hN}-\alpha(q)^2$, $i=1,2$.
These would imply that
$0=(\tau_{1}(q)\tau_{1}(q))^2=(K/K^{\hN}-\alpha(q)^2)^2$
and hence that $0=K/K^{\hN}-\alpha(q)^2=\tau_{i}(q)^2$, $i=1,2$.
Since $\tau_1(q)=\tau_2(q)=0$, the relation \eqref{eq:conv_ss10:orthonorm}
would then yield $|\alpha(q)|=1$.

As we have shown, $\alpha$ is constant on every $\dr$-orbit,
and therefore we can conclude that the vector field $H$ on $Q$ defined by
\[
H\onq := \nu\big(\star \pa{\hr} A\big)\onq,\quad \q\in Q,
\]
is tangent to every orbit $\odr(q_0)$
for which $|\alpha(q_0)|<1$.

It follows that
the smooth distribution $\mc{H}$ on $Q$ by
\[
\mc{H}\onq = \spn_{\R} \{\nu(\Rol_q),\ H\onq\}
\]
is tangent to any orbit $\odr(q_0)$ for which $|\alpha(q_0)|<1$.

The remaining Lie brackets we desire to compute
are those between $\lr(E_i)$ and $H$, $i=1,2$.
At $\q\in Q$ and for $i=1,2$ they can be expressed as (see Proposition \ref{p1.7})
\[
[\lr(E_i),H]\onq
={}& -\lns\big(\big(\star \pa{\hr}\big) AE_i\big)\onq
+\nu\big(\star\big(\hnabla_{AE_i} \pa{\hr}\big)A\big)\onq \\
={}& (-1)^{i+1}\lns\big(\tau_{\sigma(i)}(q)\hZ_A-\alpha(q)AE_{\sigma(i)}\big)\onq \\
={}& (-1)^{i+1}  N_i\onq,
\]
where \eqref{eq:conv_ss10:hnabla_AE_i_hE_2}
and \eqref{eq:conv_ss10:resolution_of_identity} have been used,
and we have defined
\[
N_i\onq:=\lns\big(\tau_{\sigma(i)}(q)\hZ_A-\alpha(q)AE_{\sigma(i)}\big),
\]
which are vector fields on $Q$.
Since $[\lr(E_i),H]$ is tangent to any orbit $\odr(q_0)$
for which $\alpha(q_0)\neq 0$,
we conclude that the vector fields $N_i$, $i=1,2$,
are tangent to any such $\dr$-orbit as well.

Expressing the vectors $N_i\onq$ and $L_i\onq=\lns(\hat{\xi}_i)\onq$, $i=1,2$,
with respect to the linearly independent ones $\lns(E_1)\onq,\lns(E_2)\onq,\lns(\hZ_A)\onq$,
at $\q\in Q$,
yields a matrix
\[
\mc{M}(q)=\qmatrix{
0 & -\alpha(q)  & \tau_2(q) \\
-\alpha(q) & 0 & \tau_1(q) \\
\Theta(q) & 0 & \Psi(q)\tau_1(q) \\
0 & \Theta(q) & \Psi(q)\tau_2(q).
}
\]
Eliminating either the 3rd or 4th row yields
a matrix $\mc{M}_3(q)$ or $\mc{M}_4(q)$, whose determinant is
\[
\det\mc{M}_3(q)
={}& -\alpha(q)\tau_2(q)(\alpha(q)\Psi(q)+\Theta(q))
=-K\alpha(q)\tau_2(q)
\\[2mm]
\det \mc{M}_4(q)
={}&-\alpha(q)\tau_1(q)(\alpha(q)\Psi(q)+\Theta(q))
=-K\alpha(q)\tau_1(q),
\]
because $\alpha(q)\Psi(q)+\Theta(q)=K$.

If $0<|\alpha(q)|<1$, then because $K\neq 0$
and because \eqref{eq:conv_ss10:orthonorm} holds,
we deduce that
either $\det\mc{M}_3(q)\neq 0$ or $\det\mc{M}_4(q)\neq 0$,
and hence that
the $\R$-span on $\{L_1\onq,L_2\onq,N_1\onq, N_2\onq\}$
in $T_q Q$ has dimension at least $3$.
But these vectors all belong to
the 3-dimensional space $\lns(T_{\hx} \hM)\onq$ at $\q$,
and hence we have shown that
\[
\lns(T_{\hx} \hM)\onq
=\spn_{\R} \{L_1\onq,L_2\onq,N_1\onq, N_2\onq\}
\subset T_q \odr(q),\quad \textrm{if}\ 0<|\alpha(q)|<1.
\]

It is clear that at each point $q\in Q$
the subspaces $\lr(T_x M)\onq$, $\lns(T_{\hx} \hM)$
and $\mc{H}\onq$ have pairwise trivial ($=\{0\}$) intersection,
and their respective dimensions are $2$, $3$ and $2$.
By what we have shown above, all these three spaces
are tangent to any orbit $\odr(q_0)$
for which $0<|\alpha(q_0)|<1$.
From this we can deduce that $\dim T_q\odr(q_0)\geq 7$ for $q\in \odr(q_0)$,
and finally
$\dim \odr(q_0)\geq 7$.

As we have already shown in the beginning of this section that
$\dim \odr(q_0)\leq 7$,
we conclude that
$\dim \odr(q_0)=7$.
The analysis performed in this section
is summarized in the following theorem.

\begin{theorem}\label{th:conv_ss10:1}
Let $(M,g)$ be a connected, oriented $2$-dimensional
Riemannian manifold of constant $K$, i.e., a space form,
and $(\hM,\hg)=(\hI\times\hN, s_1\oplus \hh)$
is a Riemannian product of an oriented 2-dimensional Riemannian manifold $(\hN,\hh)$
of constant curvature $K^{\hN}$
and an open non-empty interval $\hI\subset\R$ equipped with the standard Riemannian metric $s_1$.
Assuming further that $K\neq 0$ and $K^{\hN}\neq 0$
one has
\[
\dim\odr(q_0)=7,\quad \forall q_0\in Q;\quad 0<|\alpha(q_0)|<1,
\]
where $\alpha$ has been defined in \eqref{eq:conv_ss10:alpha}.
\end{theorem}

\begin{remark}\label{re:conv_ss10:1}
\begin{itemize}
\item[(a)]
We emphasize the fact that  while Theorem \ref{th:5.17}
was derived under the assumption that
$K-\hsigma_{(\cdot)}=0$ holds on the orbit $\odr(q_0)$
(or on a local version $O(q_0)$ of such an orbit),
its converse Theorem \ref{th:conv_ss10:1} makes no,
and needs no such assumption at all.

\item[(b)]
By \eqref{eq:conv_ss10:hsigma_A} and \eqref{eq:conv_ss10:Theta_Psi}
\[
K-\hsigma_A=K-\alpha(q)^2 K^{\hN}=\Theta(q),
\]
and hence $K=\hsigma_A$ $\Leftrightarrow$ $K=\alpha(q)^2 K^{\hN}$
$\Leftrightarrow$ $\Theta(q)=0$ at $\q\in Q$.

In addition, since $K$ and $K^{\hN}$ are constant,
and $\alpha:Q\to\R$ is constant on every $\dr$-orbit of $q$,
it follows that if $K=\hsigma_{A_0}$
for some $\qz\in Q$,
then $K=\hsigma_A$ for all $\q\in \odr(q_0)$.

Since $\alpha$ can take any value in $[-1,1]$
over $Q$,
there exists a point $\qz\in Q_0$
such that $0<|\alpha(q_0)|<1$
and $K=\hsigma_{A_0}$
if and only if
$0<K<K^{\hN}$ or $K^{\hN}<K<0$
(assuming that $K\neq 0$ and $K^{\hN}\neq 0$).

\item[(d)]
The set $Q_{0,1}=\{q\in Q\ |\ 0<|\alpha(q)|<1\}$
used in Theorem \ref{th:conv_ss10:1}
is open and dense in $Q$.
To see that it is dense, we note that
$Q_{0,1}=Q_1\cap\{q\in Q\ |\ \alpha(q)\neq 0\}$,
where $Q_1$ was defined in \eqref{eq:conv_ss10:Q_1}.
Because we have already shown that $Q_1$ is open and dense in $Q$,
it is enough to show that the open set $Q_0:=\{q\in Q\ |\ \alpha(q)\neq 0\}$ of $Q$
is dense in $Q$.

To this end, recall that
$\nu(\theta_X\otimes \hZ_A)\onq \alpha=\hg\big(-AX,\pa{\hr}\big|_{\hr}\big)$
for all $\q\in Q$,
and note that $q\in Q_0$ if and only if $\hZ_A\perp \pa{\hr}\big|_{\hr}$,
which is itself equivalent to the condition that $\pa{\hr}\big|_{\hr}\in AT_x M$.

Thus there exists a non-zero (in fact a unit) vector $Z\in T_x M$ such that $\pa{\hr}\big|_{\hr}=AZ$,
and for that $Z$
we find that $\nu(\theta_Z\otimes \hZ_A)\onq \alpha=\hg\big(-AZ,AZ\big)=-g(Z,Z)\neq 0$.

If $\Gamma(t)$ is any smooth curve in $Q$ 
whose tangent vector at $t=0$ equals $\nu(\theta_Z\otimes \hZ_A)\onq$
(and in particular $\Gamma(0)=q$),
then $\dif{t}\big|_{0} \alpha(\Gamma(t))=\nu(\theta_Z\otimes \hZ_A)\onq \alpha\neq 0$.
Thus for all $t\neq 0$ small enough, $\alpha(\Gamma(t))\neq 0$,
i.e., $\Gamma(t)\notin Q\setminus Q_0$.
We conclude from this analysis that $Q_0$ is dense in $Q$.

\end{itemize}
\end{remark}

\vspace{2\baselineskip}\subsection{Particular Subcase $\hsigma_A \neq K (x)$ on $O(q_0)$}\label{ss11}
\ \newline

We assume in this subsection that $\hsigma_A \neq K (x) $
and $(\Pi_X , \Pi_Y) \neq (0,0)$ on a rolling neighbourhood $O(q_0)$ of $\q\in Q$.
Early on, though, we will restrict ourselves for the remainder of this section
to a particular, special case where several differential-algebraic relations (see \eqref{e2.30})
are assumed to hold on $O(q_0)$ between certain quantities few of which we already know,
and the remaining of which will be introduced below.

The rolling curvature tensor $\Rol$ is then equal to
(see Eq. \eqref{eq:RolXY})
\begin{align*}
\Rol_q (X, Y) &  = - ( K - \hsigma_{A} ) \bigg( A (X \wedge Y)  - \frac{\Pi_Y}{K - \hsigma_{A}} \theta_{X} \otimes \hZ_{A}  + \frac{\Pi_X}{K - \hsigma_{A}} \theta_{Y} \otimes \hZ_{A} \bigg),
\quad q\in O(q_0).
\end{align*}

Define the real valued functions $r,\phi$ as in \eqref{e2.9},
and choose the $TM$-valued vector fields on $O(q_0)$
as in \eqref{e2.10}.
Finally, set
\begin{align}\label{eq:omega}
\omega(q):=\frac{r(q)}{K(x)-\hsigma_A},\quad \q\in O(q_0),
\end{align}
so that on $O(q_0)$ we have
\begin{eqnarray}\label{eq:Pi_XYtilde}
\left\lbrace
  	\begin{array}{ll}
	\frac{\Pi_X}{K- \hsigma_A} & = \omega \cphi,\\[2mm]
	\frac{\Pi_Y}{K- \hsigma_A} & = \omega \sphi,
\end{array}
\right.
\quad \quad \quad
\left\lbrace
\begin{array}{ll}
	\Pi_{\Xtilde}& =\omega (K - \hsigma_A),\\[2mm]
	\Pi_{\Ytilde}& =0.
\end{array}
\right.
\end{eqnarray}
We emphasize that by \eqref{e2.2}, \eqref{eq:hsigma_Pi}, \eqref{e2.9} and \eqref{eq:omega}, the functions
$\phi$, $r$ and $\omega$ are defined and smooth
in some (small enough) open subset $O'$ of $Q$ containing a (small enough) rolling neighbourhood $O(q_0)$ of $q_0$.

Using these and \eqref{e2.10}, $\Rol$ is equal to
\begin{align*}
\Rol_q (X, Y) &  = - ( K - \hsigma_{A} ) \big( A (X \wedge Y) + \omega \theta_{\YtildeA} \otimes \hZ_A \big)\\
& =  - ( K - \hsigma_{A} ) \overline{\Rol}_q (X, Y),
\end{align*}
where in the last line we defined,
\begin{align}\label{eq:Rolbar}
\overline{\Rol}_q (X, Y):=A (X \wedge Y) + \omega(q) \theta_{\YtildeA} \otimes \hZ_A,
\end{align}
for all $\q$ in an appropriate open neighbourhood of $q_0$ in $Q$.

The first order Lie bracket tangent to $O(q_0)$ is
(see Lemmas \ref{le:LR_Xtilde_Ytilde} and \ref{l2.5})
\begin{align}\label{eq:ss11:bracket_Xtilde_Ytilde}
[\lr (\Xtilde) , \lr (\Ytilde)] \onq ={}& (- \lns (A \XtildeA) \onq \phi ) \lr (\XtildeA) \onq + (- \lns (A \YtildeA) \onq \phi ) \lr (\YtildeA) \onq \nonumber\\
{}& - (K - \hsigma_A) \nu (\Rolbar_q) \onq.
\end{align}
The second order of $O(q_0)$-tangent Lie brackets are
(see Lemma \ref{l2.2})
\begin{align}\label{eq:ss11:bracket_Xtilde_Rol}
& [\lr (\Xtilde) , \nu (\Rolbar_{(\cdot)} )] \onq \nonumber\\
= & [\lr (\Xtilde) , \nu (A (X \wedge Y))] \onq + ( \lr (\XtildeA) \onq \omega ) \nu (\theta_{\YtildeA} \otimes \hZ_A) \onq + \omega [\lr (\Xtilde) , \nu (\theta_{\Ytilde} \otimes \hZ)] \onq \nonumber\\
= & ( 1 - \omega \nu (\theta_{\YtildeA} \otimes \hZ_A)\onq \phi ) \lr (\YtildeA) \onq  - \lns (A \YtildeA) \onq \nonumber\\
& + \omega (- \lr (\XtildeA) \onq \phi - g (\Gamma , \XtildeA)) \nu (\theta_{\XtildeA} \otimes \hZ_A) \onq + ( \lr (\XtildeA) \onq \omega ) \nu (\theta_{\YtildeA} \otimes \hZ_A) \onq,
\end{align}
and
\begin{align}\label{eq:ss11:bracket_Ytilde_Rol}
& [\lr (\Ytilde) , \nu (\Rolbar_{(\cdot)} )] \onq \nonumber\\
= & [\lr (\Ytilde) , \nu (A (X \wedge Y))] \onq + ( \lr (\YtildeA) \onq \omega ) \nu (\theta_{\YtildeA} \otimes \hZ_A) \onq + \omega [\lr (\Ytilde) , \nu (\theta_{\Ytilde} \otimes \hZ)] \onq \nonumber\\
= & ( -1 + \omega \nu (\theta_{\YtildeA} \otimes \hZ_A)\onq \phi ) \lr (\XtildeA) \onq  + \lns (A \XtildeA) \onq  - \omega \lns (\hZ_A) \onq \nonumber\\
& + \omega (- \lr (\YtildeA) \onq \phi - g (\Gamma , \YtildeA)) \nu (\theta_{\XtildeA} \otimes \hZ_A) \onq + ( \lr (\YtildeA) \onq \omega ) \nu (\theta_{\YtildeA} \otimes \hZ_A) \onq.
\end{align}

In order to simplify the formulas that appear below, we define the following quantities
\begin{align}\label{eq:G_H}
& G_{\Xtilde} : = \lns (A \XtildeA)\onq \phi = \lr (\XtildeA) \onq \phi + g (\Gamma , \XtildeA),\nonumber\\
& G_{\Ytilde} : = \lns (A \YtildeA)\onq \phi = \lr (\YtildeA) \onq \phi + g (\Gamma , \YtildeA),\nonumber\\
& H_{\Xtilde} : = \lr (\XtildeA) \onq \omega, \nonumber\\
& H_{\Ytilde} : = \lr (\YtildeA) \onq \omega,
\end{align}
where in the first two definitions we have used the identities given in Lemma \ref{l2.5}.
Hence, the two new $O(q_0)$-tangent vector fields
created by the above Lie brackets $[\lr (\Xtilde) , \nu (\Rolbar_{(\cdot)} )]$
and $ [\lr (\Ytilde) , \nu (\Rolbar_{(\cdot)} )] \onq$
are
\begin{align}\label{eq:ss11:F1_F2}
F_1 \onq : = {}& - \lns (A \YtildeA) \onq   - \omega G_{\Xtilde} \nu (\theta_{\XtildeA} \otimes \hZ_A) \onq + H_{\Xtilde}  \nu (\theta_{\YtildeA} \otimes \hZ_A) \onq,\nonumber\\
F_2 \onq : = {}&  \lns (A \XtildeA) \onq  - \omega \lns (\hZ_A) \onq - \omega G_{\Ytilde} \nu (\theta_{\XtildeA} \otimes \hZ_A) \onq + H_{\Ytilde} \nu (\theta_{\YtildeA} \otimes \hZ_A) \onq.
\end{align}

We know so far that the vector fields
\begin{align}\label{eq:ss11:tangent}
\lr (\Xtilde_{(\cdot)}),\quad \lr(\Ytilde_{(\cdot)}),\quad \nu(\Rolbar_{(\cdot)}),\quad F_1,\quad F_2,
\end{align}
are tangent to $O(q_0)$, and hence to $\odr (q_0)$.
They are, moreover, defined on an open neighbourhood $O'$ of $q_0$ in $Q$ such that $O(q_0)\subset O'$,
and they are pointwise linearly independent on $O'$.
Hence, by looking at the coefficients of $\lns (A \XtildeA)\onq$, $\lns (A \YtildeA)\onq$ and $\lns (\hZ_A)\onq$ in the Lie brackets computed in Lemma \ref{lemmaf},
we conclude that if any one of the following conditions is satisfied at any point of $\q\in O(q_0)$,
\begin{eqnarray}\label{e2.29}
\begin{array}{rl}
1) & G_{\Xtilde} \neq 0,\\
2) & \omega G_{\Ytilde} - H_{\Xtilde} \neq 0,\\
3) & H_{\Ytilde} \neq  0,\\
4) & \omega \nu (\theta_{\YtildeA} \otimes \hZ_A) \onq \phi - 1 \neq 0,\\
5) & \nu (\theta_{\YtildeA} \otimes \hZ_A) \onq \omega  \neq 0,
\end{array}
\end{eqnarray}
then $O(q_0)$, and hence $\odr (q_0)$, has dimension at least 6.

For the rest of this section, we will be assuming the none of these conditions 1)-5) hold.
That is, our standing assumption in the sequel, within the current section, will be that
\begin{align}\label{e2.30}
G_{\Xtilde} \equiv 0,\quad
\omega G_{\Ytilde} \equiv H_{\Xtilde}, \quad
H_{\Ytilde} \equiv 0, \quad
\omega \nu (\theta_{\YtildeA} \otimes \hZ_A) \onq \phi \equiv 1,\quad
\nu (\theta_{\YtildeA} \otimes \hZ_A) \onq \omega \equiv 0
\end{align}
everywhere on $O(q_0)$.

\begin{proposition}\label{pf1}
Let $(\Pi_X , \Pi_Y) \neq (0,0)$ and $\hsigma_A \neq K (x)$ on $O(q_0)$.
If equalities \eqref{e2.30} hold on $O(q_0)$, then $\dim O(q_0)=5$,
and the vector fields \eqref{eq:ss11:tangent}
form a frame on $O(q_0)$.
\end{proposition}

\begin{proof}
First we observe that \eqref{eq:ss11:bracket_Xtilde_Rol},
\eqref{eq:ss11:bracket_Ytilde_Rol} and
\eqref{eq:ss11:F1_F2}
are simplified by \eqref{e2.30} to
\[
& [\lr(\Xtilde),\nu(\Rolbar_{(.)})]\onq = F_1,\quad
[\lr(\Ytilde),\nu(\Rolbar_{(.)})]\onq = F_2,\\
& F_1 \onq = - \lns (A \YtildeA) \onq  + H_{\Xtilde}  \nu (\theta_{\YtildeA} \otimes \hZ_A) \onq, \nonumber\\
& F_2 \onq = \lns (A \XtildeA) \onq  - \omega \lns (\hZ_A) \onq - H_{\Xtilde} \nu (\theta_{\XtildeA} \otimes \hZ_A) \onq.
\]

By Eq. \eqref{e3.25} in Appendix, assumption \eqref{e2.30} implies that on $O(q_0)$,
\[
\tilde{\hsigma}_A^2 = K,\quad \tilde{\Pi}_{\hZ} = 0,
\]
and consequently, using \eqref{eq:omega}, \eqref{eq:Rolbar} and relations
in Lemma \ref{l2.9} and
in Eq. \eqref{e2.30},
\begin{align}\label{eq:pf1:1}
{}& (\tilde{\hsigma}_A^1 - K ) + \omega^2 (K - \hsigma_A) = \nu (\Rolbar_q)\onq \tilde{\Pi}_{\hZ} = 0, \nonumber \\
{}& \nu (\Rolbar_q) \onq \phi = 0,\quad \nu (\Rolbar_q) \onq \omega = 0.
\end{align}
Relations
\[
\nu(\Rolbar_q)\onq\tilde{X}_{(\cdot)}=0,\quad \nu(\Rolbar_q)\onq\tilde{Y}_{(\cdot)}=0,
\]
then follow directly from \eqref{e2.10} thanks to $\nu (\Rolbar_q)\onq\phi=0$.

Since $\nu (\theta_{\YtildeA} \otimes \hZ_A) \onq \phi - 1=0$,
we have $[\lr (\Xtilde) , \nu (\Rolbar_{(\cdot)} )]=F_1$
(see \eqref{eq:ss11:F1_F2} and formulas above it)
and hence by $\nu (\Rolbar_q) \onq \phi = 0$,
Eq. \eqref{eq:G_H},
and Eq. \eqref{e2.10} we obtain
\begin{align}\label{eq:F1_phi}
F_1 \onq \phi=[\lr (\Xtilde) , \nu (\Rolbar_{(\cdot)} )]|_q\phi
=\nu (\Rolbar_q)|_q \big(-G_{\Xtilde}+g(\Gamma,\Xtilde_{(\cdot)})\big)
= 0,
\end{align}
for all $q\in O(q_0)$.

Similarly using \eqref{eq:G_H}
and the fact that $[\lr (\Ytilde) , \nu (\Rolbar_{(\cdot)})] =F_2$, we get
\begin{align}\label{eq:F2_omega}
F_2 \onq \omega
=-\nu(\Rolbar_q)|_q(H_{\tilde{Y}})
=0.
\end{align}

Computation of $[\lr (\Xtilde) , \lr (\Ytilde)] \onq \phi$
using \eqref{eq:ss11:bracket_Xtilde_Ytilde},
and directly by definition of Lie bracket and Lemma \ref{le:LR_Xtilde_Ytilde} yield
(below formulas are evaluated at $\q$)
\[
[\lr (\Xtilde) , \lr (\Ytilde)] \onq \phi
={}&(-G_{\Xtilde})(G_{\Xtilde}-g(\Gamma,\Xtilde))
+(-G_{\Ytilde})(G_{\Ytilde}-g(\Gamma,\Ytilde)) \\
={}&\lr(\Xtilde_A)|_q(G_{\Ytilde}-g(\Gamma,\Ytilde))-\lr(\Ytilde_A)|_q(G_{\Xtilde}-g(\Gamma,\Xtilde))
\]
i.e., since $G_{\Xtilde}=0$,
\[
\lr (\XtildeA) \onq G_{\Ytilde}
={}&
(-G_{\Ytilde})(G_{\Ytilde}-g(\Gamma,\Ytilde_{A}))
+
\lr(\Xtilde_A)|_q(g(\Gamma,\Ytilde))-\lr(\Ytilde_A)(g(\Gamma,\Xtilde)).\]
By Eqs. \eqref{e2.2}, \eqref{e2.10} it holds
\[
{}& g(\nabla_{\Xtilde_A}\Gamma,\Ytilde_A)-g(\nabla_{\Ytilde_A}\Gamma,\Xtilde_A) \\
={}&
g(c_\phi \nabla_X \Gamma+s_\phi \nabla_Y \Gamma,-s_\phi X+c_\phi Y)
-
g(-s_\phi \nabla_X \Gamma+c_\phi \nabla_Y \Gamma, c_\phi X+s_\phi Y) \\
={}& g(\nabla_X \Gamma, Y) - g(\nabla_Y \Gamma, X) \\
={}&-K(x),
\]
which coupled with Lemma \ref{le:LR_Xtilde_Ytilde}
allows us to write
\[
{}&
\lr(\Xtilde_A)|_q(g(\Gamma,\Ytilde))-\lr(\Ytilde_A)(g(\Gamma,\Xtilde)) \\
={}&
g(\nabla_{\Xtilde_A}\Gamma,\Ytilde_A)-g(\nabla_{\Ytilde_A}\Gamma,\Xtilde_A)
-G_{\Xtilde}g(\Gamma,\Xtilde_A)
-G_{\Ytilde}g(\Gamma,\Ytilde_A) \\
={}& -K-G_{\Xtilde}g(\Gamma,\Xtilde_A)
-G_{\Ytilde}g(\Gamma,\Ytilde_A).
\]
Recalling $G_{\Xtilde}=0$, it thus follows that
\begin{align}\label{eq:ss11:G_Ytilde:1}
\lr (\XtildeA) \onq G_{\Ytilde}
=-(G_{\Ytilde})^2-K,
\end{align}
in which both sides are evaluated at $\q\in O(q_0)$.

Next computing 
$[\lr (\Xtilde) , \lr (\Ytilde)] \onq \omega$
both using \eqref{eq:ss11:bracket_Xtilde_Ytilde}, as well as
directly via definition of the Lie bracket
while accounting for \eqref{eq:G_H}
and relations in the beginning of the proof,
produces
\[
[\lr (\Xtilde), \lr (\Ytilde)] \onq \omega
={}& -G_{\Xtilde}H_{\Xtilde}-G_{\Ytilde}H_{\Ytilde} \\
={}& \lr(\Xtilde_A)\onq H_{\Ytilde}-\lr(\Ytilde_A)\onq H_{\Xtilde}
\]
that is since $G_{\Xtilde}=0$, $H_{\Ytilde}=0$
by \eqref{e2.30}
\[
\lr (\YtildeA) \onq H_{\Xtilde} = 0,
\]
It then follows from $\omega G_{\Ytilde} = H_{\Xtilde}$
and $\lr (\YtildeA) \onq\omega=H_{\Ytilde}(q)=0$ that
\begin{align}\label{eq:ss11:lr_Ytilde_GYtilde}
0=\lr (\YtildeA) \onq H_{\Xtilde}
=(\lr (\YtildeA) \onq\omega) G_{\Ytilde}(q)
+\omega(q)\lr (\YtildeA) \onq G_{\Ytilde}
=\omega(q)\lr (\YtildeA) \onq G_{\Ytilde}
\end{align}

Proceeding in similar fashion (using $F_1=[\lr (\Xtilde) , \nu (\Rolbar_{(\cdot)})]$, $\nu (\Rolbar_q)\onq\omega=0$, \eqref{eq:G_H} and \eqref{e2.30}), one obtains
\[
-F_1\onq \omega=-[\lr (\Xtilde) , \nu (\Rolbar_{(\cdot)})] \onq \omega=\nu(\Rolbar_q)\onq (H_{\Xtilde})
=\nu(\Rolbar_q)\onq (\omega G_{\Ytilde})
=\omega\nu(\Rolbar_q)\onq (G_{\Ytilde})
\]
as well as (using $F_2=[\lr (\Ytilde) , \nu (\Rolbar_{(\cdot)})]$, $\nu (\Rolbar_q)\onq\phi=0$, \eqref{eq:G_H} and \eqref{e2.30})
\[
-F_2\onq \phi=-[\lr (\Ytilde) , \nu (\Rolbar_{(\cdot)})] \onq \phi
=\nu(\Rolbar_q)\onq (G_{\Ytilde}-g(\Gamma,\Ytilde))
=\nu(\Rolbar_q)\onq (G_{\Ytilde}).
\]
From this, we find that
\begin{align}\label{eq:F1_omega:pre}
F_1\onq \omega=\omega F_2\onq \phi=-\omega\nu(\Rolbar_q)\onq (G_{\Ytilde})=-\nu(\Rolbar_q)\onq (H_{\Xtilde}).
\end{align}

As we have shown above,
$F_1 \phi = 0$ and $\nu (\Rolbar_{(\cdot)}) \phi = 0$
on $O(q_0)$,
implying that
$[\nu (\Rolbar_{(\cdot)}),F_1]\phi=0$.
On the other hand, the bracket
$[\nu (\Rolbar_{(\cdot)}),F_1]$ as written out explicitly
in Lemma \ref{lemmaf} becomes,
in view of \eqref{e2.30} and the various relations derived
earlier in this proof
(e.g. $F_1\phi=0$),
\[
[\nu (\Rolbar_{(\cdot)}), F_1] \onq
={}&
\lns (A \XtildeA) \onq - \omega\lns (\hZ_A) \onq
- \omega G_{\Ytilde} \nu (\theta_{\XtildeA} \otimes \hZ_A) \onq
- ( 2F_1 \onq \omega ) \nu (\theta_{\YtildeA} \otimes \hZ_A) \onq \\
={}&
F_2|_q - ( 2F_1 \onq \omega ) \nu (\theta_{\YtildeA} \otimes \hZ_A)\onq,
\]
so that
\[
0=[\nu (\Rolbar_{(\cdot)}), F_1] \onq \phi
=F_2|_q \phi - ( 2F_1 \onq \omega ) \nu (\theta_{\YtildeA} \otimes \hZ_A)\onq \phi
=F_2|_q \phi - \frac{2}{\omega}(F_1 \onq \omega)
=-\frac{1}{\omega}(F_1 \onq \omega),
\]
which yields $F_1\onq \omega=0$ and hence by \eqref{eq:F1_omega:pre}
\begin{align}\label{eq:F1_omega}
0=F_1\onq \omega=\omega F_2\onq \phi=-\omega\nu(\Rolbar_q)\onq (G_{\Ytilde})=-\nu(\Rolbar_q)\onq (H_{\Xtilde}),
\end{align}
for all $q\in O(q_0)$.

Finally, note that
\begin{align}\label{eq:lrX_HX}
\lr (\XtildeA) \onq H_{\Xtilde}
=
\lr (\XtildeA) \onq ( \omega G_{\Ytilde} )
=
H_{\Xtilde} G_{\Ytilde} + \omega \lr (\XtildeA)\onq G_{\Ytilde}
=-K\omega.
\end{align}

Thanks to all these computations
that have lead to additional relations,
we see that the assumption \eqref{e2.30}
leads to the following simplifications
in Lie brackets listed in Lemma \ref{lemmaf},
\begin{align}\label{pf1:bracket:1}
[\lr (\Xtilde) , F_1] \onq
={}&
- \hsigma_A \nu (\Rolbar_q) \onq,
&\qquad
[\lr (\Ytilde) , F_1] \onq
={}&
G_{\Ytilde} F_2|_q, \\[2mm]
[\lr (\Xtilde) , F_2] \onq
={}&
0,
&\qquad
[\lr (\Ytilde) , F_2] \onq
={}& -G_{\Ytilde} F_1\onq
+\lambda \nu (\Rolbar_q) \onq, \nonumber \\[2mm]
[\nu (\Rolbar_{(\cdot)}), F_1] \onq
={}&
F_2|_q,
&\qquad
[\nu (\Rolbar_{(\cdot)}), F_2] \onq
={}&
-( 1 + \omega^2 )F_1\onq
+ \omega H_{\Xtilde} \nu (\Rolbar_q) \onq, \nonumber
\end{align}
which hold at every $q\in O(q_0)$,
and we have used the function
(recall that $(\tilde{\hsigma}_A^1 - K ) + \omega^2 (K - \hsigma_A) = 0$)
\begin{align}\label{eq:ss11:lambda}
\lambda(q):=- \hsigma_A  + \omega^2(K(x) - \hsigma_A) 
= -\tilde{\hsigma}_A^1 - \hsigma_A + K(x),
\quad \q\in O(q_0).
\end{align}

It remains to examine the bracket $[F_1,F_2]$,
which is presented in Lemma \ref{lemmaf}
in terms of the local (around $q_0$) frame of $Q$ we are already very familiar with.

First, as we have seen, assumptions \eqref{e2.30}
allow one to simplify the expressions \eqref{eq:ss11:F1_F2}
for $F_1,F_2$ on $O(q_0)$ into
\begin{align}\label{eq:ss11:F1_F2:simpl}
F_1 \onq = {}& - \lns (A \YtildeA) \onq + H_{\Xtilde}  \nu (\theta_{\YtildeA} \otimes \hZ_A) \onq,\nonumber\\
F_2 \onq = {}&  \lns (A \XtildeA) \onq  - \omega \lns (\hZ_A) \onq - H_{\Xtilde} \nu (\theta_{\XtildeA} \otimes \hZ_A) \onq.
\end{align}
These expressions combined with $F_1\onq \omega=0$, $F_2\onq \phi=0$, $\tilde{\Pi}_{\hZ}=0$
and relations \eqref{e3.25} yield
\[
\lns(A\Ytilde_A)\onq \omega=0,
\quad
\omega \lns(\hZ_A)\onq\phi=\lns(A\Xtilde_A)\onq \phi=G_{\Xtilde}=0,
\quad \forall q\in O(q_0).
\]

Secondly, using these two relations along with
already known identities
$G_{\Xtilde}=0$, $H_{\Ytilde}=0$,
$\omega G_{\Ytilde}=H_{\Xtilde}$,
$\tilde{\Pi}_{\hZ}=0$ and $(\tilde{\hsigma}^2_A-\hsigma_A)/(K-\hsigma_A)=1$,
one sees that the expression for $[F_1,F_2]$ given
in Lemma \ref{lemmaf} reduces to
\[
[F_1,F_2]\onq
={}&
\omega H_{\Xtilde}\lns(A\Ytilde_A)\onq
+
\big(
\omega F_2\onq G_{\Xtilde}
-\omega F_1\onq G_{\Ytilde}
\big)\nu(\theta_{\Xtilde_A}\otimes \hZ_A)\onq \\
{}&
+
\Big(
\omega(\tilde{\hsigma}^1_A-K+\hsigma_A)
+F_1\onq H_{\Ytilde} - F_2\onq H_{\Xtilde}
\Big)\nu(\theta_{\Ytilde_A}\otimes \hZ_A)\onq \\
{}&
+
\Big(\hsigma_A-\omega^2(K-\hsigma_A)+(H_{\Xtilde})^2
\Big) \nu(A(X\wedge Y))\onq.
\]

In order to simplify this further, we note that
earlier identities imply
\begin{align}\label{eq:F12_HX}
-F_1\onq H_{\Xtilde}
{}&=[\lr (\Xtilde), F_1] \onq \omega
=
-\hsigma_A \nu (\Rolbar_q) \onq \omega=0, \nonumber \\
-F_2\onq H_{\Xtilde}
{}&=[\lr (\Xtilde), F_2] \onq \omega
=0,
\end{align}
which in combination with identities
$F_1\onq \omega=0$, $\omega G_{\Ytilde}=H_{\Xtilde}$
imply
$\omega F_1\onq G_{\Ytilde}=F_1\onq (\omega  G_{\Ytilde})=F_1\onq H_{\Xtilde}=0$.
Moreover, facts that $G_{\Xtilde}=0$, $H_{\Ytilde}=0$ on $O(q_0)$ and $F_1,F_2$ are tangent to $O(q_0)$
readily imply $F_1\onq H_{\Ytilde}=0$ and $F_2\onq G_{\Xtilde}=0$.

These additional observations lead to
the following simplification of $[F_1,F_2]$ above:
\begin{align}\label{pf1:bracket:2}
[F_1,F_2]\onq
={}&
\omega H_{\Xtilde}\lns(A\Ytilde_A)\onq
+
\omega\big(
\tilde{\hsigma}^1_A-K+\hsigma_A
\big)\nu(\theta_{\Ytilde_A}\otimes \hZ_A)\onq \nonumber \\
{}&
+
\big(\hsigma_A-\omega^2(K-\hsigma_A)+(H_{\Xtilde})^2
\big) \nu(A(X\wedge Y))\onq \nonumber \\
={}&
-\omega H_{\Xtilde} F_1\onq
+
\big(
(H_{\Xtilde})^2-\lambda
\big)\nu(\Rolbar_q)\onq,
\end{align}
where we have used again $\lambda$ from \eqref{eq:ss11:lambda}.

Evidently, all seven Lie brackets
appearing in 
\eqref{pf1:bracket:1} and \eqref{pf1:bracket:2}
belong to the $C^\infty(O(q_0))$-span of the
vector fields \eqref{eq:ss11:tangent} on $O(q_0)$,
an observation which allows us to conclude that $O(q_0)$
is 5-dimensional
(see for instance the proof of Proposition \ref{pr:5.14}),
and that vector fields
\eqref{eq:ss11:tangent} form a frame on $O(q_0)$.

Finally, note that
$(\pi_{Q,\hM}|_{O(q_0)})_*:TO(q_0)\to T\hM$
maps the tangent vectors
$\lr(\Xtilde_A)\onq$, $\lr(\Ytilde_A)\onq$ and $F_2\onq$
of $O(q_0)$ at $\q$ to the vectors
$A\Xtilde_A$,
$A\Ytilde_A$,
$A\Xtilde_A-\omega(q)\hZ_A$,
respectively,
which span $T_{\hx} \hM$ since $\omega(q)\neq 0$.
Thus $\pi_{Q,\hM}|_{O(q_0)}$ is a submersion as claimed.
Proof is complete.
\end{proof}

Recall that here $\pi_Q$ is the mapping $Q\to M\times\hM$; $(x,\hx;A)\mapsto (x,\hx)$.
As a consequence of the fact from previous proposition that \eqref{eq:ss11:tangent} forms a frame on $O(q_0)$
we have:

\begin{corollary}\label{cor:pf1}
The map $\pi_{Q,\hM}|_{O(q_0)}:O(q_0)\to \hM$ is a submersion,
and $\pi_Q|_{O(q_0)}:O(q_0)\to M\times\hM$ has constant rank 4,
so that its image $\pi_Q(O(q_0))$ is a 4-dimensional embedded submanifold of $M\times\hM$
(possibly after shrinking $O(q_0)$ around $q_0$).
\end{corollary}

Notice that if $k_1,k_2$ are functions on $V\times\hV$,
saying that relation $k_1(x,\hx)=k_2(x,\hx)$ holds for all $\q\in O(q_0)$,
is equivalent to saying that $k_1(x,\hx)=k_2(x,\hx)$ for all $(x,\hx)\in \pi_Q(O(q_0))$.
We mention explicitly this simple remark because
in what follows, we will be using both ways of writing such relations.

The following proposition holds when only the last
two of the five relations in \eqref{e2.30} hold.

\begin{proposition}\label{pr:ss11:1}
Assume that $(\Pi_X , \Pi_Y) \neq (0,0)$ and $\hsigma_A \neq K (x)$ and that
\[
\omega \nu (\theta_{\YtildeA} \otimes \hZ_A) \onq \phi \equiv 1,\quad
\nu (\theta_{\YtildeA} \otimes \hZ_A) \onq \omega \equiv 0
\]
hold on $O(q_0)$.
Then the following are true:
\begin{itemize}
\item[a)] There is a function $\hat{\lambda}\in C^\infty(\hV')$
defined on some open neighbourhood $\hV'$ of $\hx_0$ in $\hM$
such that $\lambda(q)=\hat{\lambda}(\hx)$
for all $\q\in O(q_0)\cap (\pi_{Q,\hM})^{-1}(\hV')$,
where $\lambda(q)$ is defined in \eqref{eq:ss11:lambda},
with its last equality being true as well.

\item[b)]
For every $\q\in O(q_0)$,
the curvature tensor $\hR|_{\hx}$ has $-K(x)$
as a double eigenvalue, and $\lambda(q)$
as a simple eigenvalue.
In particular, for all
$\q\in O(q_0)\cap (\pi_{Q,\hM})^{-1}(\hV')$
this simple eigenvalue is $\hat{\lambda}(\hx)$.

\item[c)] $\Ytilde_A(K)=0$
i.e., $-s_{\phi}X(K)+c_{\phi} Y(K)=0$
(see \eqref{e2.10})
for every $\q\in O(q_0)\cap (\pi_{Q,\hM})^{-1}(\hV')$.
\end{itemize}

\end{proposition}

\begin{proof}
The relation
$\nu (\theta_{\YtildeA} \otimes \hZ_A) \onq \omega \equiv 0$ on $O(q_0)$
alone implies by \eqref{e3.25} in Appendix that
\[
\tilde{\Pi}_{\hZ}=0\quad \mathrm{on}\ O(q_0),
\]
and hence by \eqref{eq:Rolbar}
and Lemma \ref{l2.9} in Appendix,
\[
0=\nu(\Rolbar_q)\onq \tilde{\Pi}_{\hZ}
=\frac{\omega}{r} \big((\tilde{\hsigma}^1_A - \tilde{\hsigma}^2_A)(\tilde{\hsigma}^2_A-\hsigma_A)+r^2\big)
\]
i.e., by the definition of $\omega$ in \eqref{eq:omega},
\[
\omega^2 (K-\hsigma_A)^2=(\tilde{\hsigma}^2_A-\tilde{\hsigma}^1_A)(\tilde{\hsigma}^2_A-\hsigma_A)\quad \mathrm{on}\ O(q_0).
\]

On the other hand, the relation
$\omega \nu (\theta_{\YtildeA} \otimes \hZ_A) \onq \phi \equiv 1$
on $O(q_0)$ alone
implies by \eqref{e3.25} in Appendix that
\[
\tilde{\hsigma}^2_A=K(x)\quad \mathrm{on}\ O(q_0).
\]

Combining the last two relation above, yields
(using $\hsigma_A\neq K$)
\[
\omega^2 (K-\hsigma_A)=K-\tilde{\hsigma}^1_A\quad \mathrm{on}\ O(q_0)
\]

Thus, in view of the three identities obtained,
namely $\tilde{\hsigma}^2_A=K(x)$,
$\tilde{\Pi}_{\hZ}=0$, $\tilde{\hsigma}^1_A=K-\omega^2 (K-\hsigma_A)$,
in combination with \eqref{eq:Pi_XYtilde},
the matrix of $\hR|_{\hx}$ with respect to
the basis $\star A\Xtilde_A$, $\star A\Ytilde_A$, $\star \hZ_A$
is given by
\begin{align}\label{eq:ss11:hR}
\hR|_{\hx}
=
\qmatrix{
-K+\omega^2 (K-\hsigma_A) & 0 & \omega(K-\hsigma_A) \\
0 & -K & 0 \\
\omega(K-\hsigma_A) & 0 & -K+(K-\hsigma_A),
}
\end{align}
implying that the characteristic polynomial $f_{\hx}(\tau)$ of $\hR|_{\hx}$ is
\[
f_{\hx}(\tau)
={}&
(\tau+K)^2(\tau+\hsigma_A-\omega^2(K-\hsigma_A))
=(\tau+K(x))^2(\tau-\lambda(q))
\]
for every $\q\in O(q_0)$.
This shows that the eigenvalues of the $\hg$-symmetric
linear map $\hR|_{\hx}$
are $-K(x)$ and $\lambda(q)$ for all $\q\in O(q_0)$,
the multiplicity of $-K(x)$ being at least 2.

Next notice that having $\lambda(q)=-K(x)$ would mean that
$- \hsigma_A  + \omega^2(K(x) - \hsigma_A) = -K(x)$
i.e.,
$(1+\omega^2)(K(x)-\hsigma_A)=0$,
which contradicts the assumption that $\hsigma_A\neq K(x)$.

This shows that $-K(x)$ is a double,
and $\lambda(q)$ is a simple eigenvalue of $\hR|_{\hx}$
for $\q\in O(q_0)$,
completing the proof of item b).

Since in particular $\hR|_{\hx_0}$
has a simple eigenvalue, it follows that
$\hR|_{\hx}$ has a simple eigenvalue for all $\hx$
in some open neighbourhood $\hV'$ of $\hx_0$ in $\hM$,
and, consequently,
the map $\hat{\lambda}$ that assigns this eigenvalue to $\hx\in\hV'$
must be a $C^\infty$-function on $\hV'$.
It follows that $\hat{\lambda}(\hx)=\lambda(q)$ for all $\q\in O(q_0)\cap (\pi_{Q,\hM})^{-1}(\hV')$,
and thus we have proved the claim of item a).

Finally, to prove c),
let $F_1'\in \VF(O(q_0))$ be the vector field on $O(q_0)$
defined as
(see \eqref{eq:ss11:F1_F2})
\[
F_1'\onq=\lr(\Ytilde_A)\onq + F_1\onq
=\lns(\Ytilde_A)\onq
-\omega G_{\Xtilde}\nu(\theta_{\Xtilde_A}\otimes \hZ_A)\onq
+ H_{\Xtilde}  \nu (\theta_{\YtildeA} \otimes \hZ_A) \onq,
\]

Fix a point $\q\in O(q_0)\cap (\pi_{Q,\hM})^{-1}(\hV')$,
and let $\Gamma(t)$, $t\in I$, be an integral curve of $F_1'$
passing through $q$ at $t=0$,
with $I$ an open interval such that $0\in I$
and $I$ small enough so that $\Gamma(t)$, $t\in I$,
stays inside the open subset $O(q_0)\cap (\pi_{Q,\hM})^{-1}(\hV')$ of $O(q_0)$.
We may write $\Gamma(t)$ as $\Gamma(t)=(\gamma(t),\hgamma(t);A(t))$.

It is clear that $(\pi_{Q,\hM}|_{O(q_0)})_* F_1'=0$,
and hence
$\dot{\hgamma}(t)=\frac{d}{dt} (\pi_{Q,\hM}|_{O(q_0)}\circ\Gamma)(t)=(\pi_{Q,\hM}|_{O(q_0)})_*\dot{\Gamma}(t)=0$ for all $t\in I$,
as a result of which
$\pi_{Q,\hM}(\Gamma(t))=\pi_{Q,\hM}(\Gamma(0))=\pi_{Q,\hM}(q)=\hx$ for all $t\in I$.
This means that $\Gamma(t)=(\gamma(t),\hx;A(t))$,
and therefore, owing to the result stated in item b),
the curvature tensor $\hR|_{\hx}$ at $\hx$ has a (double) eigenvalue $\hat{\kappa}$
which equals $-K(\gamma(t))$, for all $t\in I$.

In other words, $K(\gamma(t))=-\hat{\kappa}$ is constant in $t\in I$,
and since $\gamma(0)=x$ and $\dot{\gamma}(t)=(\pi_{Q,M}|_{O(q_0)})_* F_1'|_{\Gamma(t)}=\Ytilde_{A(t)}$,
we get
$0=\frac{d}{dt} K(\gamma(t))=\Ytilde_{A(t)}(K)$,
which at $t=0$
implies $\Ytilde_A(K)=0$.
Proof is complete.
\end{proof}

\begin{remark}\label{re:ss11:1}
\begin{itemize}
\item[(i)]
The open set $\hV'$ of $\hM'$
in item a) of Proposition \ref{pr:ss11:1}
can be taken to be $\hV=\pi_{Q,\hM}(O(q_0))$
if the conditions \eqref{e2.30} hold.
This is because, according to Corollary \ref{cor:pf1},
the map $\pi_{Q,\hM}$ is a submersion, hence an open map.

In the more general situation of Proposition \ref{pr:ss11:1}
we do not know \emph{a priori} if $\hV=\pi_{Q,\hM}(O(q_0))$
is open or not.

\item[(ii)]
We make an observation concerning item c) of Proposition \ref{pr:ss11:1}.
If $(X(K),Y(K))\neq 0$ on the open subset $V$ of $M$,
then, if say $X(K)\neq 0$ on $V$, we have
$\tan(\phi(q))=\frac{Y|_x(K)}{X|_x(K)}$ for all $\q\in O(q_0)$,
where the right hand side is a smooth function on $V$.
Hence there exists a function $\phi'\in C^\infty(V)$
such that
$\phi'(x)=\phi(q)$ for every $\q\in O(q_0)$
(after shrinking $O(q_0)$ around $q_0$ if necessary).
The same observation holds if, instead, $Y(K)\neq 0$ on $V$.

It then immediately follows that $\nu(\Rolbar_q)\onq\phi=0$,
$F_1\onq \phi=0$ and $F_2\onq\phi=0$ on $O(q_0)$,
where $F_1,F_2$ are given in \eqref{eq:ss11:F1_F2}.
These three identities for $\phi$ thus hold without assuming any of the first three relations in \eqref{e2.30}
(under the assumption that $(X(K),Y(K))\neq 0$ on $V$).
\end{itemize}
\end{remark}

\begin{remark}\label{re:ss11:2}
In what follows, we will continue using the shorthand notation that has been already employed in the
proofs of the previous results in this section.
Namely, we will regularly identify functions $f$ on $M$ (resp. $\hf$ on $\hM$) such as $K$ and $\Gamma^i_{(j,k)}$
(resp. $\hlambda$, $\hGamma^i_{(j,k)}$)
with the composite functions $f\circ\pi_{O(q_0),M}$
(resp. $\hf\circ\pi_{O(q_0),\hM}$) on $O(q_0)$,
whenever these quantities are needed to be viewed as function on $O(q_0)$.
It is therefore understood from now on that, for instance, $K(x)=(K\circ\pi_{O(q_0),M})(x,\hx;A)$,
$\Gamma^i_{(j,k)}(x)=(\Gamma^i_{(j,k)}\circ\pi_{O(q_0),M})(x,\hx;A)$,
and $\hlambda(\hx)=(\hlambda\circ\pi_{O(q_0),\hM})(x,\hx;A)=\lambda(x,\hx;A)$,
$\hGamma^i_{(j,k)}(\hx)=(\hGamma^i_{(j,k)}\circ\pi_{O(q_0),\hM})(x,\hx;A)$ at points $(x,\hx;A)\in O(q_0)$.
\end{remark}

Like mentioned in Remark \ref{re:ss11:1} case (i) we will from now on assume that $\hV'=\hV$ in Proposition \ref{pr:ss11:1}.

To prepare ourselves for the proof of a proposition that follows below, we shall need the following technical lemma.

\begin{lemma}\label{le:ss11:1}
At every point $\q\in O(q_0)$ we have
\begin{align}\label{eq:le:ss11:1:1}
\big(\omega A\Xtilde_A
+\hZ_A\big)(\hat{\lambda})
=-2(1+\omega^2) H_{\Xtilde}(K-\hsigma_A).
\end{align}
\end{lemma}

\begin{proof}
As we have seen in the proof of Proposition \ref{pf1},
on $O(q_0)$ the vector fields $F_1$ and $F_2$ (defined in \eqref{eq:ss11:F1_F2}) are simplified to \eqref{eq:ss11:F1_F2:simpl},
and all the possible Lie-brackets of the (involutive) frame
of vector fields \eqref{eq:ss11:tangent} of $O(q_0)$
have themselves the expressions given in
\eqref{eq:ss11:bracket_Xtilde_Ytilde},
\eqref{pf1:bracket:1}, \eqref{pf1:bracket:2},
and
\eqref{eq:ss11:bracket_Xtilde_Rol},
\eqref{eq:ss11:bracket_Ytilde_Rol}
which reduce, in view of \eqref{e2.30}, to
\[
[\lr (\Xtilde) , \nu (\Rolbar_{(\cdot)} )] \onq
=F_1\onq,
\quad
[\lr (\Ytilde) , \nu (\Rolbar_{(\cdot)} )] \onq
=F_2\onq.
\]

The Jacobi identity for the three vector fields $\lr (\XtildeA) \onq$, $F_1 \onq$ and  $F_2 \onq$ reads
\begin{align}\label{e2.31:1}
[\lr (\Xtilde) , [F_1 , F_2] ] \onq + [F_2, [\lr (\Xtilde) , F_1]] \onq + [F_1 , [ F_2 , \lr (\Xtilde)]] \onq = 0.
\end{align}
Using \eqref{eq:G_H}, \eqref{e2.30}, \eqref{pf1:bracket:1}, \eqref{pf1:bracket:2} 
and $\lr(\Xtilde_A)\onq H_{\Xtilde}=-K\omega$ on $O(q_0)$
(shown in the proof of Proposition \ref{pf1}),
the first term on the left can be expressed as
\begin{align*}
[\lr (\Xtilde) , [F_1 , F_2] ] \onq = & \big[\lr (\Xtilde) , - \omega H_{\Xtilde} F_1 + ((H_{\Xtilde})^2 - \lambda) \nu (\Rolbar_{(\cdot)})\big] \onq\\
={}&
-(H_{\Xtilde})^2 F_1\onq
+K\omega^2 F_1+\omega H_{\Xtilde}\hsigma_A\nu(\Rolbar_q)\onq \\
{}& +(-2H_{\Xtilde}K\omega-\lr(\Xtilde_A)\onq\lambda )\nu(\Rolbar_q)\onq
+((H_{\Xtilde})^2 - \lambda)F_1\onq \\
={}& (K\omega^2-\lambda)F_1\onq
+\big(\omega H_{\Xtilde}(\hsigma_A-2K)-\lr(\Xtilde_A)\onq\lambda\big)\nu(\Rolbar_q)\onq
\end{align*}
the second one is equal to
\begin{align*}
[F_2, [\lr (\Xtilde) , F_1]] \onq
={}& [F_2 , -\hsigma_{(\cdot)} \nu(\Rolbar) ]\onq \\
={}& -(F_2\onq\hsigma_{(\cdot)}) \nu(\Rolbar_q)\onq-\hsigma_A\big((1+\omega^2)F_1\onq - \omega H_{\Xtilde} \nu(\Rolbar_q)\onq,\big)
\end{align*}

The last term on the left of \eqref{e2.31:1} vanishes
$[F_1 , [ F_2 , \lr (\Xtilde)]] \onq = 0$
by \eqref{pf1:bracket:1}.
Thus \eqref{e2.31:1} yields
\[
0=(K\omega^2-\lambda-(1+\omega^2)\hsigma_A)F_1\onq
+\big(2\omega H_{\Xtilde}(\hsigma_A-K)-\lr(\Xtilde_A)\onq\lambda -F_2\onq\hsigma_{(\cdot)}\big)\nu(\Rolbar_q)\onq,
\]
that is (recall \eqref{eq:ss11:lambda})
\[
\lr(\Xtilde_A)\onq\lambda+F_2\onq\hsigma_{(\cdot)}=-2\omega H_{\Xtilde}(K-\hsigma_A).
\]
But we have $F_2\onq \lambda=-(1+\omega^2)F_2\onq\hsigma_{(\cdot)}$, which is a consequence of the definition \eqref{eq:ss11:lambda} of $\lambda$, \ref{eq:F2_omega} and the relation $F_2\onq K=0$ that holds since $(\pi_{Q,M})_* F_2\onq=0$. Hence
\[
(1+\omega^2)\lr(\Xtilde_A)\onq\lambda-F_2\onq\lambda=-2(1+\omega^2)\omega H_{\Xtilde}(K-\hsigma_A).
\]

Because $\lambda=\hlambda\circ \pi_{Q,\hM}$, we have $\nu(\theta_{\Xtilde_A}\otimes\hZ_A)\onq \lambda=0$,
and consequently we find
\[
\omega\big(\omega A\Xtilde_A
+\hZ_A\big)(\hat{\lambda})
=-2(1+\omega^2)\omega H_{\Xtilde}(K-\hsigma_A).
\]
Dividing this identity with $\omega(q)\neq 0$
yields \eqref{eq:le:ss11:1:1} and thus completes the proof.
\end{proof}

At this point, we will make some preliminary observations and introduce some notations that will be used throughout the rest of this section.

By Proposition \ref{pr:ss11:1}, we can choose a 
smooth unit vector field $\hE_2$ on $\hV$
such that $\star\hE_2$ is an eigenvector field
of $\hR$ corresponding to the simple eigenvalue (function) $\hat{\lambda}$.

Choosing then $\hE_1,\hE_3\in \VF(\hV)$
such that $(\hE_1,\hE_2,\hE_3)$ is a positively oriented orthonormal frame on $\hV$,
it follows that $\star\hE_1|_{\hx},\star\hE_3|_{\hx}$ are eigenvector fields of $\hR|_{\hx}$
with eigenvalue $-K(x)$ for all $\q\in O(q_0)$,
again by Proposition \ref{pr:ss11:1}.

In addition to this orthonormal frame,
we define $T\hM$-valued smooth vector fields on $O(q_0)$
by
\begin{align}\label{eq:hM}
\hM_1\onq := -A\Ytilde_A,\quad
\hM_2\onq := -\omega(q) A\Xtilde_A-\hZ_A,\quad
\hM_3\onq := A\Xtilde_A - \omega(q)\hZ_A,
\end{align}
for $\q\in O(q_0)$.
As is easily checked, for each $\q\in O(q_0)$,
the vectors
$\hM_1\onq, \hM_2\onq, \hM_3\onq$
are mutually $\hg$-orthogonal (but are not all normalized to $1$),
and $\star \hM_1\onq, \star \hM_2\onq, \star \hM_3\onq$
are eigenvectors of $\hR|_{\hx}$
corresponding to its eigenvalues $-K(x)$, $\hat{\lambda}(\hx)$ and $-K(x)$,
respectively.
It follows that $\spn\{\hE_1|_{\hx}, \hE_3|_{\hx}\}=\spn\{\hM_1\onq,\hM_3\onq\}$
and that $\hE_2|_{\hx}$ is parallel to $\hM_2\onq$
for all $\q\in O(q_0)$.

The choice of $\hM_1$ and $\hM_3$ (and hence $\hM_2$)
is motivated by the fact that
in terms of them, Eq. \eqref{eq:ss11:F1_F2:simpl} reads
\begin{align}\label{eq:ss11:F1_F2:simpl:2}
F_1 \onq = {}& \lns (\hM_1\onq) \onq + H_{\Xtilde}  \nu (\theta_{\YtildeA} \otimes \hZ_A) \onq,\nonumber\\
F_2 \onq = {}&  \lns (\hM_3\onq) \onq - H_{\Xtilde} \nu (\theta_{\XtildeA} \otimes \hZ_A) \onq.
\end{align}

As $\hE_2|_{\hx}$ and $\hM_2\onq$ are parallel vectors in $T_{\hx}\hM$ for $\q\in O(q_0)$,
while $\n{\hM_2\onq}_{\hg}=(1+\omega(q)^2)^{1/2}$ and $\n{\hE_2|_{\hx}}_{\hg}=1$,
we may assume w.l.o.g that $\hE_2$ was chosen in such a way that
\begin{align}\label{eq:hE2_to_hM2}
\big(1+\omega(q)^2\big)^{1/2}\hE_2|_{\hx}=\hM_2\onq,
\quad \forall \q\in O(q_0).
\end{align}

Here is one of the key results of this section.

\begin{proposition}\label{pr:ss11:key1}
Under the assumptions of Proposition \ref{pf1},
and after shrinking the rolling neighbourhood $O(q_0)$
around $q_0$ if necessary,
there is a smooth oriented orthonormal frame $\hE_1, \hE_2, \hE_3$ on $\hV=\pi_{Q,\hM}(O(q_0))$
with respect to which the connection table $\hGamma$
of the Levi-Civita connection of $(\hV,\hg|_{\hV})$ has the form
\begin{align}\label{eq:ss11:hGamma}
\hGamma=\qmatrix{
0 & 0 & -\hGamma^1_{(1,2)} \\
\hGamma^1_{(3,1)} & \hGamma^2_{(3,1)} & \hGamma^3_{(3,1)} \\
\hGamma^1_{(1,2)} & 0 & 0
},
\end{align}
where we recall that $\hGamma^i_{(j,k)}=\hg(\hnabla_{\hE_i} \hE_j,\hE_k)$.
In addition,
\begin{align}\label{eq:ss11:hE_perp_hGamma}
\hH(\hGamma^1_{(1,2)})=0,\quad \forall \hH\in \hE_2^\perp
\end{align}
holds on $\hV$ and relations
\begin{gather}
-\hE_2|_{\hx} (\hGamma^1_{(1,2)})+\big(\hGamma^1_{(1,2)}(\hx)\big)^2=-K(x) \label{eq:ss11:E2_hGamma112} \\[2mm]
\hE_2|_{\hx}(\hat{\lambda})
-2\hGamma^1_{(1,2)}(\hx)(K(x)+\hat{\lambda}(\hx))=0 \label{eq:E2_lambda}
\end{gather}
hold at every point $(x,\hx)\in \pi_Q(O(q_0))$.
\end{proposition}

\begin{proof}

Our first step will be to show the following claim:
\begin{itemize}
\item[{\bf (A)}]
\emph{The 2-dimensional distribution $\hE_2^\perp$ on $V$ is involutive.}
\end{itemize}

To that end, observe first that by \eqref{pf1:bracket:1} and \eqref{pf1:bracket:2}
the system of vector fields $\mc{F}':=\{\nu(\Rolbar), F_1, F_2\}$
forms an involutive system on $O(q_0)$.

Fix arbitrary $\hx_1\in \hV$,
and recall that $\hV=\pi_{Q,\hM}(O(q_0))$ which implies the existence of a point $q_1=(x_1,\hx_1;A_1)$ in $O(q_0)$ above $\hx_1$.
Let $O'$ be a connected neighbourhood of $q_1$
in the 3-dimensional orbit $O_{\mc{F}'}(q_1)$ of $\mc{F}'$
passing through $q_1$ (and contained in $O(q_0)$).

The vector fields
in $\mc{F}'$ at a point $\q\in O(q_0)$
are mapped by $(\pi_{Q,\hM}|_{O(q_0)})_*$
(see \eqref{eq:ss11:F1_F2:simpl:2})
to vectors
$0$, $-A\Ytilde_A=\hM_1\onq$ and $A\Xtilde_A-\omega\hZ_A=\hM_3\onq$ in $T_{\hx}\hM$, respectively.
This means that the smooth map $\pi_{Q,\hM}|_{O'}$ has constant rank 2,
and therefore we may choose the connected neighbourhood
$O'$ of $q_1$ in the orbit $O_{\mc{F}'}(q_1)$
to be small enough
to guarantee
that $\pi_{Q,\hM}|_{O'}$ is a submersion 
onto an embedded 2-dimensional submanifold $\hN$ of $\hM$ containing $\hx_1$.
Furthermore, because $O'$ is connected, so is $\hN$.

Since $(\pi_{Q,\hM})_* F_1\onq=\hM_1\onq$,
$(\pi_{Q,\hM})_* F_2\onq=\hM_3\onq$,
for every $\q\in O'$,
and since $\pi_{Q,\hM}(O')=\hN$,
the tangent space of $\hN$ at any $\hx\in \hN$
is spanned by $\hM_1\onq$ and $\hM_3\onq$,
for any point $q$ in $O'$ above $\hx$ (i.e., any $q$ in $(\pi_{Q,\hM}|_{O'})^{-1}(\hx)$).
But $\spn\{\hM_1\onq,\hM_3\onq\}=\spn\{\hE_1|_{\hx}, \hE_3|_{\hx}\}=\hE_2^\perp|_{\hx}$ and hence we have
$T_{\hx} \hN=\hE_2^\perp|_{\hx}$.
This shows that $\hN$ is an integral manifold (of dimension 2) of $\hE_2^\perp$ though the point $\hx_1$.
Since $\hx_1$ was an arbitrarily chosen point in $\hV$,
we conclude that $\hE_2^\perp$ is a smooth involutive distribution on $\hV$.
Claim (A) above is therefore proven.

Here we can write down our first relation between the connection coefficients $\hGamma^i_{(j,k)}$ on $\hM$.
As is easily verified (see item \ref{ss11:item:E2_perp_integrable} in Remark \ref{re:ss11:connection}),
the distribution $\hE_2^\perp$ being involutive, which is the case by claim (A) above,
is equivalent to the following relation between the connection coefficients:
\begin{align}\label{eq:ss11:rel_coeffs:1}
\hat{\Gamma}^3_{(1,2)} = -\hat{\Gamma}^1_{(2,3)}\quad \mathrm{on}\ \hV.
\end{align}

The next claim we will demonstrate below is the following:
\begin{itemize}
\item[{\bf (B)}]
\emph{
For any $\hx_1\in \hV$, any connected integral manifold $\hN$ of $\hE_2^\perp$
passing through $\hx_1$,
and any $q_1=(x_1,\hx_1;A_1)\in O(q_0)$ above $\hx_1$,
it holds
\begin{align}\label{eq:eigen:E2_perp_const}
\hR|_{\hx}(\star \hat{H})=-K(x_1) {\star\hat{H}},\quad \forall \hx\in \hN,\ \forall \hat{H}\in \hE_2^\perp|_{\hx}.
\end{align}
}
\end{itemize}

Note that through any $\hx_1\in \hV$ there passes some connected integral manifold $\hN$ of $\hE_2^\perp$,
thanks to claim (A) above.

We will use the notations from the proof of claim (A) above.
First, projecting $\mc{F}'$ into the 2-dimensional manifold $M$,
we observe that $(\pi_{Q,M})_*\mc{F}'\onq=0$,
and consequently $\pi_{Q,M}(O')=\{x_1\}$ because $O'$ is a connected integral manifold of $\mc{F}'$.

Taking an arbitrary point $\hx\in \hN$,
there is a smooth path $\hgamma:[0,1]\to \hN$
in $\hN$ from $\hx_1$ to $\hx$ (because $\hN$ is connected),
and it can be lifted to a smooth path
$\Gamma:[0,1]\to O'$ from $q_1$ to some $q=(x,\hx;A)\in O'$.
Consequently, $\pi_{Q,M}(\Gamma(t))=x_1$ for all $t\in [0,1]$,
implying in particular that $x=\pi_{Q,M}(q)=\pi_{Q,M}(\Gamma(1))=x_1$.
For any $\hat{H}\in \hE_2^\perp|_{\hx}$ we therefore have
$\hR|_{\hx}(\star \hat{H})=-K(x)\star\hat{H}=-K(x_1)\star\hat{H}$,
which completes the proof of claim (B) since $\hx\in \hN$ was arbitrary.

To derive our second set of relations between connection coefficient,
along with one additional differential relation involving the eigenvalues of $\hR$,
we shall next make use of Eq. \eqref{eq:eigen:E2_perp_const},
the eigen-equation $\hR({\star \hE_2})=\hat{\lambda}\, {\star\hE_2}$
and the second Bianchi identity.

Fix $\hx_1$, $q_1=(x_1,\hx_1;A_1)$ and $\hN$ as in claim (B) above.
Choosing $\hat{H}=\hE_1$ (resp. $\hat{H}=\hE_3$)
in \eqref{eq:eigen:E2_perp_const},
then applying covariant derivative
$\hnabla_{\hE_1}$ (resp. $\hnabla_{\hE_3}$)
onto it, one gets
\[
{}&
(\hnabla_{\hE_1} \hR)(\star \hE_1)+\hR(\star (\hGamma^1_{(1,2)} \hE_2-\hGamma^1_{(3,1)} \hE_3))
=-K(x_1)\star (\hGamma^1_{(1,2)} \hE_2-\hGamma^1_{(3,1)} \hE_3)
\\
{}&
(\hnabla_{\hE_3} \hR)(\star \hE_3)+\hR(\star (\hGamma^3_{(3,1)} \hE_1-\hGamma^3_{(2,3)} \hE_2))
=-K(x_1)\star (\hGamma^3_{(3,1)} \hE_1-\hGamma^3_{(2,3)} \hE_2).
\]
An important point here is that these identities do not involve any derivatives of $K$.

Similarly, applying $\hnabla_{\hE_2}$
to the eigen-equation
$\hR({\star \hE_2})=\hat{\lambda}\, {\star\hE_2}$
gives
\[
{}&
(\hnabla_{\hE_2} \hR)(\star \hE_2)+\hR(\star (-\hGamma^2_{(1,2)} \hE_1+\hGamma^2_{(2,3)} \hE_3))
=\hat{\lambda}\star (-\hGamma^2_{(1,2)} \hE_1+\hGamma^2_{(2,3)} \hE_3)
+\hE_2(\hat{\lambda}) {\star \hE_2}
\]

Using \eqref{eq:eigen:E2_perp_const} and the eigen-equation $\hR({\star \hE_2})=\hat{\lambda}\, {\star\hE_2}$
again,
these three identities can be simplified into 
\[
{}&
(\hnabla_{\hE_1} \hR)(\star \hE_1)
=-\hGamma^1_{(1,2)}(K(x_1)+\hlambda) {\star \hE_2} \\
{}&
(\hnabla_{\hE_3} \hR)(\star \hE_3)
=\hGamma^3_{(2,3)}(K(x_1)+\hlambda) {\star \hE_2} \\
{}&
(\hnabla_{\hE_2} \hR)(\star \hE_2)
=-\hGamma^2_{(1,2)} (K(x_1)+\hlambda) {\star \hE_1}
+ \hGamma^2_{(2,3)} (K(x_1)+\hlambda) {\star \hE_3}
+\hE_2(\hat{\lambda}) {\star \hE_2}.
\]
Plugging these into the second Bianchi identity
$\sum_{i=1}^3 (\hnabla_{\hE_i} \hR)(\star \hE_i)=0$
allows us to deduce from the resulting $\star \hE_1$ and $\star \hE_3$ components,
while keeping in mind that $\hat{\lambda}(\hx)\neq -K(x_1)$ for $\hx\in \hN$,
the identities $\hGamma^2_{(1,2)}(\hx)=0$, $\hGamma^2_{(2,3)}(\hx)=0$ for all $\hx\in \hN$.

Since $\hN$ was an arbitrary (local) integral manifold of $\hE_2^\perp$,
as can be understood from the above,
on the open subset $\hV$ of $\hM$,
these relations hold on all of $\hV$, i.e.,
\begin{align}\label{eq:ss11:rel_coeffs:2}
\hGamma^2_{(1,2)}(\hx)=0,\quad \hGamma^2_{(2,3)}(\hx)=0,
\quad \forall \hx\in \hV.
\end{align}

Similarly, $\star \hE_2$ component of the second Bianchi identity just discussed yields
\[
\hE_2|_{\hx}(\hat{\lambda})
+
\big(-\hGamma^1_{(1,2)}(\hx)+\hGamma^3_{(2,3)}(\hx)\big)(K(x_1)+\hat{\lambda}(\hx))
=0,
\quad \forall \hx\in\hN,
\]
holding for any $\hx_1\in \hV$, any connected integral manifold $\hN$ of $\hE_2^\perp$
passing through $\hx_1$,
and any $q_1=(x_1,\hx_1;A_1)\in O(q_0)$ above $\hx_1$.
In particular
\begin{align}\label{eq:E2_lambda:special}
\hE_2|_{\hx}(\hat{\lambda})
+
\big(-\hGamma^1_{(1,2)}(\hx)+\hGamma^3_{(2,3)}(\hx)\big)(K(x)+\hat{\lambda}(\hx))=0
\end{align}
holds at every point $(x,\hx)\in \pi_Q(O(q_0))$.

Next we shall focus on proving the following claim:
\begin{itemize}
\item[{\bf (C)}]
\emph{
Below every point $\q\in O(q_0)$ in $M\times\hM$ it holds 
\begin{align}\label{eq:hE_hGamma}
{}& \hE_1|_{\hx}\big(-\hGamma^1_{(1,2)}+\hGamma^3_{(2,3)}\big)
=0, \nonumber\\
{}& \hE_2|_{\hx}\big(-\hGamma^1_{(1,2)}+\hGamma^3_{(2,3)}\big)
+ (\hGamma^1_{(1,2)})^2
+ 2(\hGamma^1_{(2,3)})^2
+ (\hGamma^3_{(2,3)})^2
=-2K(x), \\
{}& \hE_3|_{\hx}\big(-\hGamma^1_{(1,2)}+\hGamma^3_{(2,3)}\big)
=0, \nonumber
\end{align}
where the $\hGamma^i_{(j,k)}$ without $\hE_m$ derivatives involved are shorthand notations for $\hGamma^i_{(j,k)}(\hx)$.
}
\end{itemize}

Taking into account that
$\star \hE_1=\qmatrix{1\\ 0\\ 0}$, $\star \hE_3=\qmatrix{0\\ 0\\ 1}$
and $-K(x)\star E_i|_{\hx} = \hR(\star \hE_i|_{\hx})$, $i=1,3$,
for any $x\in M$ and $\hx\in \hM$ such that $\pi_Q(q)=(x,\hx)$ for some $q\in O$,
it follows that the components $c_1$ and $a_2$
of the curvature formulas \eqref{eq:app:curvature} in the Appendix (section \ref{s2.1})
are equal to $-K(x)$.
Hence making use of \eqref{eq:ss11:rel_coeffs:2}
the lines of $c_1$ and $a_2$ in \eqref{eq:app:curvature} yield
\begin{align}\label{eq:ss11:minus_K}
-K(x)=-\hE_2(\hGamma^1_{(1,2)}) + (\hGamma^1_{(1,2)})^2 + (-\hGamma^1_{(2,3)} - \hGamma^2_{(3,1)})\hGamma^3_{(1,2)}+\hGamma^1_{(2,3)}\hGamma^2_{(3,1)}
\nonumber \\ 
-K(x)=\hE_2(\hGamma^3_{(2,3)}) + (-\hGamma^2_{(3,1)} - \hGamma^3_{(1,2)})\hGamma^1_{(2,3)} + \hGamma^2_{(3,1)}\hGamma^3_{(1,2)} + (\hGamma^3_{(2,3)})^2.
\end{align}
Summing these equations up and using \eqref{eq:ss11:rel_coeffs:1} we obtain
the second relation in \eqref{eq:hE_hGamma}.

It thus remains to show that the first and the third relations in \eqref{eq:hE_hGamma} hold,
in order to complete the proof of claim (C) above.

Using the expression for $\hM_2$ given in \eqref{eq:hM}
in combination with \eqref{eq:hE2_to_hM2} in Eq. \eqref{eq:le:ss11:1:1} of Lemma \ref{le:ss11:1},
the latter can be rewritten into the form
\[
\hE_2|_{\hx}(\hat{\lambda})
=2(1+\omega^2)^{1/2} H_{\Xtilde}(K-\hsigma_A).
\]
Comparison with \eqref{eq:E2_lambda:special} then implies that
\[
2(1+\omega(q)^2)^{1/2} H_{\Xtilde}(q)(K(x)-\hsigma_A)
+\big(-\hGamma^1_{(1,2)}(\hx)+\hGamma^3_{(2,3)}(\hx)\big)(K(x)+\hat{\lambda}(\hx))
=0,
\]
holds for all $\q\in O(q_0)$.
Plugging in here $\hat{\lambda}=\lambda$ from \eqref{eq:ss11:lambda}
then results in 
\[
2(1+\omega^2)^{1/2} H_{\Xtilde}(K-\hsigma_A)
+(1+\omega^2)\big(-\hGamma^1_{(1,2)}+\hGamma^3_{(2,3)}\big)( K-\hsigma_A)
=0,
\]
i.e., since $K-\hsigma_A\neq 0$,
\begin{align}\label{eq:HX_DeltaGamma}
2H_{\Xtilde}
+(1+\omega^2)^{1/2}\big(-\hGamma^1_{(1,2)}+\hGamma^3_{(2,3)}\big)
=0\quad \textrm{on\ } O(q_0).
\end{align}

Using $\lr(\Xtilde_A)\onq H_{\Xtilde}=-K\omega$,
which is the relation \eqref{eq:lrX_HX} derived in the proof of Proposition \ref{pf1},
and using the defining relation $\lr(\Xtilde_A)\onq \omega=H_{\Xtilde}$ (see \eqref{eq:G_H}),
one gets
\begin{align}\label{eq:AX_DeltaGamma}
-2K\omega + (1+\omega^2)^{-1/2}\omega H_{\Xtilde} \big(-\hGamma^1_{(1,2)}+\hGamma^3_{(2,3)}\big)
+
(1+\omega^2)^{1/2}(A\Xtilde_A)\big(-\hGamma^1_{(1,2)}+\hGamma^3_{(2,3)}\big)
=0.
\end{align}

Applying $F_1\onq$ to \eqref{eq:HX_DeltaGamma},
using $F_1\onq \omega=0$,
$F_1\onq H_{\Xtilde}=0$
which are \eqref{eq:F1_omega} and \eqref{eq:F12_HX}
in the proof of Proposition \ref{pf1},
and using the expression of $F_1$ in \eqref{eq:ss11:F1_F2:simpl:2}, we get
\[
(1+\omega^2)^{1/2}\hM_1\onq\big(-\hGamma^1_{(1,2)}+\hGamma^3_{(2,3)}\big)
=0\quad \textrm{on\ } O(q_0).
\]
i.e.,
\[
\hM_1\onq\big(-\hGamma^1_{(1,2)}+\hGamma^3_{(2,3)}\big)
=0\quad \textrm{on\ } O(q_0).
\]

Likewise, applying $F_2\onq$ to \eqref{eq:HX_DeltaGamma},
using $F_2\onq \omega=0$,
$F_2\onq H_{\Xtilde}=0$
which are \eqref{eq:F2_omega} and \eqref{eq:F12_HX}
in the proof of Proposition \ref{pf1},
and using the expression of $F_2$ in \eqref{eq:ss11:F1_F2:simpl:2}, we get
\[
(1+\omega^2)^{1/2}\hM_3\onq\big(-\hGamma^1_{(1,2)}+\hGamma^3_{(2,3)}\big)
=0\quad \textrm{on\ } O(q_0).
\]
i.e.,
\begin{align}\label{eq:ss11:M3_DeltaGamma}
\hM_3\onq\big(-\hGamma^1_{(1,2)}+\hGamma^3_{(2,3)}\big)
=0\quad \textrm{on\ } O(q_0).
\end{align}

Since $\spn\{\hM_1\onq, \hM_3\onq\}=(\hM_2\onq)^\perp=(\hE_2|_{\hx})^\perp$
we can conclude that
\[
\hH\big(-\hGamma^1_{(1,2)}+\hGamma^3_{(2,3)}\big)
=0,
\quad \forall \hx\in \hV,
\quad \forall \hH\in \hE_2^{\perp}|_{\hx},
\]
which yields the first and the third relation in \eqref{eq:hE_hGamma},
completing the proof of claim (C).

Our next and final separate claim in this proof is as follows:
\begin{itemize}
\item[{\bf (D)}]
\emph{
Below every point $\q\in O(q_0)$ in $M\times\hM$ it holds
\begin{align}\label{eq:ss11:E2_DeltaGamma}
\hE_2|_{\hx} \big(-\hGamma^1_{(1,2)}+\hGamma^3_{(2,3)}\big)
=-2K(x) - \frac{1}{2} \big(-\hGamma^1_{(1,2)}+\hGamma^3_{(2,3)})^2
\end{align}
}
\end{itemize}

To make the formulas appearing below a bit less busy,
let us write
\[
\Delta\hGamma:=-\hGamma^1_{(1,2)}+\hGamma^3_{(2,3)}.
\]
Then Eqs. \eqref{eq:HX_DeltaGamma} and \eqref{eq:AX_DeltaGamma},
that hold for all $\q\in O(q_0)$, take the form
\[
2H_{\Xtilde}
+(1+\omega^2)^{1/2}\Delta\hGamma
=0
\]
and
\[
-2K\omega + (1+\omega^2)^{-1/2}\omega H_{\Xtilde} \Delta\hGamma
+
(1+\omega^2)^{1/2}(A\Xtilde_A)\Delta\hGamma
=0.
\]
Using the former relation in the latter yields,
\[
-2K\omega - \frac{1}{2} \omega (\Delta\hGamma)^2
+
(1+\omega^2)^{1/2}(A\Xtilde_A)\Delta\hGamma
=0
\]
i.e.,
\begin{align}\label{eq:AX_DeltahGamma}
(A\Xtilde_A)\Delta\hGamma
=(1+\omega^2)^{-1/2}\omega\big(2K + \frac{1}{2} (\Delta\hGamma)^2\big)
\end{align}

Next recall that $\hM_3\onq = A\Xtilde_A - \omega(q)\hZ_A$
(Eq. \eqref{eq:hM}) and that  $\hM_3\onq \Delta\hGamma=0$ by \eqref{eq:ss11:M3_DeltaGamma},
implying, at every $\q\in O(q_0)$,
\[
\omega \hZ_A\Delta\hGamma
=(A\Xtilde_A)\Delta\hGamma
=(1+\omega^2)^{-1/2}\omega\big(2K + \frac{1}{2} (\Delta\hGamma)^2\big)
\]
i.e.,
\begin{align}\label{eq:omegahZ_DeltahGamma}
\hZ_A\Delta\hGamma
=(1+\omega^2)^{-1/2}\big(2K + \frac{1}{2} (\Delta\hGamma)^2\big)
\end{align}

Finally, combining \eqref{eq:AX_DeltahGamma} and \eqref{eq:omegahZ_DeltahGamma}
to form $-(\omega(q) A\Xtilde_A+\hZ_A)\Delta\hGamma=\hM_2\onq \Delta\hGamma$ (see \eqref{eq:hM}) on the left hand side,
and recalling \eqref{eq:hE2_to_hM2} yields
\[
(1+\omega^2)^{1/2}\hE_2|_{\hx} (\Delta\hGamma)
=-(1+\omega^2)^{1/2}\big(2K + \frac{1}{2} (\Delta\hGamma)^2\big).
\]
This relation readily yields \eqref{eq:ss11:E2_DeltaGamma},
completing the proof of claim (D).

We are now in position to complete the proof of the proposition at hand.

Substituting \eqref{eq:ss11:E2_DeltaGamma} into the second relation in \eqref{eq:hE_hGamma}
and cancelling the $-2K(x)$ term we find that
\[
(\hGamma^1_{(1,2)})^2
+ 2(\hGamma^1_{(2,3)})^2
+ (\hGamma^3_{(2,3)})^2
-
\frac{1}{2} \big(-\hGamma^1_{(1,2)}+\hGamma^3_{(2,3)}\big)^2
=0,
\]
which after using $2(a^2+b^2)-(-a+b)^2=a^2+b^2+2ab=(a+b)^2$ becomes
\[
(\hGamma^1_{(1,2)}+\hGamma^3_{(2,3)})^2+4(\hGamma^1_{(2,3)})^2=0.
\]
From this we conclude that
\[
\hGamma^3_{(2,3)}=-\hGamma^1_{(1,2)},\quad \hGamma^1_{(2,3)}=0,
\]
on the open subset $\hV=\pi_{Q,\hM}(O(q_0))$ of $\hM$.
Whence \eqref{eq:ss11:rel_coeffs:1} yields $\hGamma^3_{(1,2)}=-\hGamma^1_{(2,3)}=0$ on $\hV$.
These last relations, along with \eqref{eq:ss11:rel_coeffs:2},
show that the connection table $\hGamma$ on $(\hM,\hg)$ has the form \eqref{eq:ss11:hGamma} as claimed.

At last, substituting $\hGamma^3_{(2,3)}=-\hGamma^1_{(1,2)}$ and $\hGamma^1_{(2,3)}=0$
into \eqref{eq:E2_lambda:special} as well as into the three relations in \eqref{eq:hE_hGamma}
and recalling that $\hE_2^\perp=\spn\{\hE_1,\hE_3\}$,
we find \eqref{eq:ss11:hE_perp_hGamma}, \eqref{eq:ss11:E2_hGamma112}
and \eqref{eq:E2_lambda}.
This completes the proof.

\end{proof}

Next result is a rather direct corollary of \eqref{eq:ss11:hGamma} and \eqref{eq:ss11:E2_hGamma112}.

\begin{corollary}\label{cor:ss11:warped_hM}
The space $(\hV,\hg|_{\hV})\subset (\hM,\hg)$ is isometric to a warped product $(\hI\times \hN, s_1\oplus_{\hf} \hh)$
via some isometry $\hF:\hI\times \hN\to \hV$,
where $(\hN,\hh)$ is a 2-dimensional Riemannian manifold
$\hI$ is a non-empty open interval of $\R$ equipped with the standard metric $s_1$,
the warping function $\hf\in C^\infty(\hI)$ satisfies
\begin{align}\label{eq:ss11:hM_warping_function:1}
\frac{\hf'(\hr)}{\hf(\hr)}= -\hGamma^1_{(1,2)}(\hF(\hr,\hy)), \quad (\hr,\hy)\in \hI\times\hN,
\end{align}
and the canonical unit vector field $\pa{\hr}$ on $\hI$ is related to $\hE_2$ on $\hV$ by
\begin{align}\label{eq:ss11:hM_warp_VFs:1}
\hF_*\pa{\hr}\Big|_{(\hr,\hy)}=\hE_2|_{\hF(\hr,\hy)},\quad (\hr,\hy)\in \hI\times\hN.
\end{align}
Furthermore, $\hf$ obeys the relation
\begin{align}\label{eq:ss11:hM_warping_function:2}
\quad
\frac{\hf''(\hr)}{\hf(\hr)}=-K(x),
\end{align}
at every $x\in M$ and $\hr\in\hI$ such that $(x,\hF(\hr,\hy))\in \pi_Q(O(q_0))$ for some $\hy\in\hN$.
\end{corollary}

\begin{proof}
The identities \eqref{eq:ss11:hGamma} and \eqref{eq:ss11:hE_perp_hGamma}
fulfill the assumptions of the result characterizing
warped products given in \cite{hiepko79} or \cite[Theorem C.14]{ChitourKokkonen1} (or \cite[Theorem D.14]{ChitourKokkonen}).
According to it,
after possibly shrinking the rolling neighbourhood $O(q_0)$ around $q_0$ and hence $\hV$ around $\hx_0$,
the space $(\hV,\hg|_{\hV})$
isometric to a warped product $(\hI \times \hN , s_1\oplus_{\hf} \hat{h})$
where $(\hN,\hat{h})$ is a 2-dimensional Riemannian manifold,
$\hI\subset \R$ is an open non-empty interval,
and the warping function $\hf \in C^{\infty} (\hI)$
satisfies \eqref{eq:ss11:hM_warping_function:1},
and \eqref{eq:ss11:hM_warp_VFs:1} holds.

Let $(x,\hx)\in \pi_Q(O(q_0))$, $\hx=\hF(\hr,\hy)$.
Applying \eqref{eq:ss11:hM_warping_function:1} and \eqref{eq:ss11:hM_warp_VFs:1} to Eq. \eqref{eq:ss11:E2_hGamma112} one finds
\[
-K(x)=-\pa{\hr}\big(-\frac{\hf'(\hr)}{\hf(\hr)}\big)+\big(-\frac{\hf'(\hr)}{\hf(\hr)}\big)^2
\]
which yields \eqref{eq:ss11:hM_warping_function:2} after elementary calculus.
\end{proof}

\begin{remark}\label{re:ss11:connection}
Let us address the geometrical meaning of the
various algebraic relations for connection coefficients
$\hGamma^i_{(j,k)}$ of $\hnabla$ derived
in the course of the proof of Proposition \ref{pr:ss11:key1}.
We let here $(\hM,\hg)$ be some Riemannian manifold 
of dimension 3, equipped with a local orthonormal frame $\hE_1,\hE_2,\hE_3$ defined on some open non-empty set $\hV\subset\hM$.
As before, we define $\hGamma^i_{(j,k)}=\hg(\hnabla_{\hE_i} \hE_j, \hE_k)$.
\begin{enumerate}
\item \label{ss11:item:E2_perp_integrable}
$\hat{\Gamma}^3_{(1,2)} = -\hat{\Gamma}^1_{(2,3)}$ on $\hV$
if and only if
the distribution $\hE_2^\perp$ is involutive on $\hV$.
Reason:
\[
\hat{g}([\hE_3,\hE_1], \hE_2)
=\hat{g}(\hat{\nabla}_{\hE_3} \hE_1-\hat{\nabla}_{\hE_1} \hE_3,\hE_2)
=\hat{\Gamma}^3_{(1,2)} + \hat{\Gamma}^1_{(2,3)}.
\]

\item
$\hGamma^2_{(1,2)}=0$, $\hGamma^2_{(2,3)}=0$ on $\hV$
if and only if
$\hE_2$ is a (unit) geodesic vector field on $\hV$, if and only if 
$\hE_2^\perp$ is parallel along $\hE_2$
(i.e., $\hnabla_{\hE_2}(\hE_2^\perp)\subset \hE_2^\perp$).
Reason:
\[
\hg(\hnabla_{\hE_2} \hE_2,\hE_2)
=\frac{1}{2}\hE_2(\hg(\hE_2,\hE_2))=0,
\quad
\hg(\hnabla_{\hE_2}\hE_2,\hE_1)=-\hGamma^2_{(1,2)},
\quad \hg(\hnabla_{\hE_2}\hE_2,\hE_3)=\hGamma^2_{(2,3)},
\]
and if $\hat{\zeta}$ is a smooth vector field on $\hV$
with values in $\hE_2^\perp$, then
\[
0=\hE_2(\hg(\hat{\zeta},\hE_2))
=\hg(\hnabla_{\hE_2} \hat{\zeta},\hE_2)
+\hg(\hat{\zeta},\hnabla_{\hE_2} \hE_2).
\]

\item
$\hGamma^1_{(1,2)}=-\hat{\Gamma}^3_{(2,3)},$ $\hGamma^1_{(2,3)}=0$, $\hGamma^3_{(1,2)}=0$
hold on $\hV$
if and only if there is a function $\hat{\eta}\in C^\infty(\hV)$
such that
$\hnabla_{\hat{U}} \hE_2=-\hat{\eta}\hat{U}$
for all $\hat{U}\in \hE_2^\perp$
(and in that case $\hat{\eta}=\hGamma^1_{(1,2)}$).
Reason: If $\hat{U}=a\hE_1+b\hE_3$ for some $a,b\in C^\infty(\hV)$, then
\[
\hnabla_{\hat{U}} \hE_2
=a(-\hGamma^1_{(1,2)}\hE_1
+\hGamma^1_{(2,3)}\hE_3)
+b(-\hGamma^3_{(1,2)}\hE_1
+\hGamma^3_{(2,3)}\hE_3).
\]

In these circumstances,
$\hat{E}_2^\perp$ is an integrable distribution (by case (1) above),
and the shape tensor $\mathrm{II}(\hat{U},\hat{W})
=-\hg(\hnabla_{\hat{U}} \hE_2, \hat{W})\hE_2$
of its integral manifolds satisfies
\[
\mathrm{II}(\hat{U},\hat{W})
=\hat{\eta}\, \hat{g}(\hat{U},\hat{W})\hat{E}_2,
\quad \forall \hat{U},\hat{W}\in \hE_2^\perp.
\]
When this condition holds, one says that the integral manifolds of $\hE_2^\perp$ are \emph{totally umbilic}
(see \cite{oneill83}, Definition 4.15 and the paragraphs right after it).

\end{enumerate}
\end{remark}

\begin{lemma}\label{le:ss11:1a}
Using the notations of Corollary \ref{cor:ss11:warped_hM},
and letting $K^{\hN}$ be the (intrinsic) Gaussian curvature of the 2-dimensional space $(\hN,\hh)$,
then the simple eigenvalue $\hlambda$ of $\hR$ (see Proposition \ref{pr:ss11:1}) can be expressed as
\begin{align}\label{eq:ss11:hlambda_warped}
\hat{\lambda}(\hF(\hr,\hy))
=-K^{\hN}(\hy) +  \Big(\frac{\hf'(\hr)}{\hf(\hr)}\Big)^2,
\quad \forall (\hr,\hy)\in I\times\hN.
\end{align}
\end{lemma}

\begin{proof}
By Proposition 7.42 (p. 210) in \cite{oneill83}
\[
\hR(\hV,\hW)\hU
={}&
R^{\hN}(\hV,\hW)\hU+\frac{\hg(\grad_{\hg} \hf,\grad_{\hg} \hf)}{\hf^2} (\hg(\hV,\hU)\hW - \hg(\hW,\hU)\hV) \\
={}&
R^{\hN}(\hV,\hW)\hU + \Big(\frac{\hf'}{\hf}\Big)^2 (\hg(\hU,\hV)\hW - \hg(\hW,\hU)\hV),
\]
for every $\hU,\hV,\hW\in \hE_2^\perp$
(here $\hV$ is temporarily a vector field, not the open subset $\hV=\pi_{Q,\hM}(O(q_0))$ of $\hM$).
Here $R^{\hN}$ is the curvature tensor of the 2-dimensional Riemannian manifold $(\hN,\hh)$.
Choosing in this relation $\hV=\hE_1$, $\hW=\hE_3$, $\hU=\hE_1$
and taking the $\hg$-inner product of it with respect to $\hE_3$
yields,
\[
\hat{\lambda}
={}&
\hg(\hR(\star \hE_2), \star\hE_2)
=\hg(\hR(\hE_1,\hE_3)\hE_1,\hE_3) \\
={}&
\hg(R^{\hN}(\hE_1,\hE_3)\hE_1,\hE_3)+ (\hf'/\hf)^2 \hg\big((\hg(\hE_1,\hE_1)\hE_3 - \hg(\hE_3,\hE_1)\hE_1), \hE_3\big) \\
={}&
\hh(R^{\hN}(\hE_1,\hE_3)\hE_1,\hE_3) + (\hf'/\hf)^2 \\
={}&
-K^{\hN} +  (\hf'/\hf)^2,
\]
where in the first step we used the fact that
$\star\hE_2=-\hE_1\wedge \hE_3$ is the eigenvector corresponding to
the eigenvalue $\hat{\lambda}$ of $\hR$,
in the second-to-last step the fact that
$\hg$ restricted to $\hN$ is the metric $\hh$,
and in the final step the definition
of the sectional (i.e., Gaussian) curvature $K^{\hN}$ of $(\hN,\hh)$.
\end{proof}

\begin{proposition}\label{pr:ss11:hN_hh_flat}
If $\hf'\neq 0$ everywhere on $\hI$, then $(\hN,\hh)$ is flat.
Moreover, in this case, for all $(\hr,\hy)\in \hF^{-1}(\hV)$ it holds
\begin{align}\label{eq:ss11:hlambda_warped:special}
\hat{\lambda}(\hF(\hr,\hy))=\Big(\frac{\hf'(\hr)}{\hf(\hr)}\Big)^2.
\end{align}
\end{proposition}

Letting $\hF(\hr_0,\hy_0)=\hx_0$, we remark that
if it holds $\hf'(\hr_0)\neq 0$, we can always shrink $O(q_0)$ enough
around $q_0$, and hence $\hI$ around $\hr_0$,
so as to guarantee that $\hf'\neq 0$ on all of $\hI$.

\begin{proof}
Taking into account that $-\hGamma^1_{(1,2)}=\hf'/\hf$
and that $\hE_2=\pa{\hr}$
via the isometry $\hF:I\times\hN\to\hM$
(see Corollary \ref{cor:ss11:warped_hM}),
the relation \eqref{eq:E2_lambda} takes the form
\begin{align}\label{eq:ss11:pa_r_hlambda}
\pa{\hr}\hat{\lambda}(F(\hr,\hy))
+
2\frac{\hf'(\hr)}{\hf(\hr)}\big(K(x)+\hat{\lambda}(\hf(\hr,\hy))\big)
=0,
\end{align}
at every $(x,\hx)\in \pi_Q(O(q_0))$, writing $\hF(\hr,\hy)=\hx$.
Obviously $\pa{\hr}(K^{\hN}(\hy))=0$,
and therefore applying $\pa{\hr}$ onto
\eqref{eq:ss11:hlambda_warped}
yields
\[
\pa{\hr}\hat{\lambda}(\hF(\hr,\hy))=0 + \pa{\hr}\Big(\frac{\hf'(\hr)}{\hf(\hr)}\Big)^2
=2\frac{\hf'(\hr)}{\hf(\hr)}\frac{\hf''(\hr)\hf(\hr)-(\hf'(\hr))^2}{\hf(\hr)^2}.
\]
These two relations, the assumption $\hf'(\hr)\neq 0$ for all $\hr\in I$,
and \eqref{eq:ss11:hlambda_warped} then imply
\[
\frac{\hf''(\hr)\hf(\hr)-(\hf'(\hr))^2}{\hf(\hr)^2}
+K(x)-K^{\hN}(\hy) +  \frac{(\hf'(\hr))^2}{\hf(\hr)^2}=0
\]
i.e.,
\[
\frac{\hf''(\hr)}{\hf(\hr)}+K(x)-K^{\hN}(\hy)=0,
\]
holding at every $(x,\hx)\in \pi_Q(O(q_0))$ and writing $\hF(\hr,\hy)=\hx$.
Combining this with \eqref{eq:ss11:hM_warping_function:2}
we thus find that $K^{\hN}(\hy)=0$ for any such $\hy$.

We want to show that all such points $\hy$ comprise $\hN$.
Indeed, if $\hy\in\hN$ is arbitrary, we can take any $\hr\in \hI$
and thus $\hx:=\hF(\hr,\hy)\in\hV$.
As $\hV=\pi_{Q,\hM}(O(q_0))$, there is $q\in O(q_0)$ such that $\pi_{Q,\hM}(q)=\hx$. Letting $x=\pi_{Q,M}(q)$, we thus find that
$(x,\hx)\in \pi_Q(O(q_0))$ with $\hx=\hF(\hr,\hy)$
and hence $K^{\hN}(\hy)=0$.
As $\hy\in\hN$ was arbitrary, we can conclude that
$K^{\hN}(\hy)=0$ for all $\hy\in\hN$, which means that $(\hN,\hh)$ is flat as claimed.

Finally, \eqref{eq:ss11:hlambda_warped:special} is a direct consequence of $K^{\hN}=0$ in view of \eqref{eq:ss11:hlambda_warped}.
\end{proof}

A partial complement of the above result is the following.

\begin{proposition}\label{pr:ss11:V_g_flat}
If $\hf'$ is constant on $\hI$, then $(V,g|_V)$ is flat.
\end{proposition}

\begin{proof}
Of course $\hf'$ is constant on the interval $\hI$ if and only if $\hf''=0$ on $\hI$.
Let $x\in V$. Because $V=\pi_{Q,M}(O(q_0))$, there is $q\in O(q_0)$ such that $\pi_{Q,M}(q)=x$. The point $\hx=\pi_{Q,\hM}(q)$
thus belongs to $\pi_{Q,\hM}(O(q_0))=\hV$, and therefore if $(\hr,\hy)\in\hI\times \hN$ is such that $\hF(\hr,\hy)=\hx$,
then $\hf''(\hr)=0$ which by \eqref{eq:ss11:hM_warping_function:2} implies that $K(x)=0$.
From this we conclude that $K=0$ on $V$, i.e., $(V,g|_V)$ is flat, as claimed.
\end{proof}

Now that we know pretty well what the Riemannian geometry of the 3-dimensional space $(\hM,\hg)$ is,
it is time to start searching for information about the Riemannian geometry of the 2-dimensional space $(M,g)$.

First we formulate an important corollary of the proof of Proposition \ref{pf1} which allows us to identify the $TM$-valued sections $\q\mapsto\Xtilde_A,\Ytilde_A$ with the vector fields $X,Y$ on $M$,
and therefore lets us work subsequently with authentic vector fields on $M$
instead of the $\Xtilde_{(\cdot)},\Ytilde_{(\cdot)}$.

\begin{proposition}\label{pr:phi}
With the rolling neighbourhood $O(q_0)$ chosen small enough around $q_0$,
there is a smooth function $\ol{\phi}:V\to\R$
such that $\ol{\phi}(x)=\phi(q)$ for all $\q\in O(q_0)$.

Consequently, $X',Y'$ defined by $X':=\cos(\ol{\phi})X+\sin(\ol{\phi})Y$, $Y'=-\sin(\ol{\phi})X+\cos(\ol{\phi})Y$
is an orthonormal system of smooth vector fields on $V\subset M$
and $\Xtilde_A=X'|_x$, $\Ytilde_A=Y'|_x$ for all $\q\in O(q_0)$.
\end{proposition}

\begin{proof}
Relations in \eqref{eq:pf1:1}, \eqref{eq:F1_phi} and \eqref{eq:F1_omega}
show that $\nu(\Rolbar_q)\onq \phi=0$, $F_1\onq \phi=0$, $F_2\onq\phi=0$
for each $\q\in O(q_0)$.

The system of vector fields $\mc{F}':=\{\nu(\Rolbar), F_1, F_2\}$
on $O(q_0)$ spans an involutive 3-dimensional distribution on $O(q_0)$ 
in view of the Lie-bracket relations \eqref{pf1:bracket:1} and \eqref{pf1:bracket:2}.

By Proposition \ref{pf1} the system of vector fields
$\mc{F}=\{\lr(\Xtilde_{(\cdot)}), \lr(\Ytilde_{(\cdot)}),\nu(\Rolbar), F_1, F_2\}$ on $O(q_0)$ spans a 5-dimensional involutive distribution
whose orbit $O_{\mc{F}}(q_0)$ through $q_0$ is $O(q_0)$.
The map $(t_1,t_2,t_3,t_4,t_5)\mapsto \big((\Phi_{F_1})_{t_1}\circ (\Phi_{F_2})_{t_2}\circ (\Phi_{\nu(\Rolbar)})_{t_3}\circ (\Phi_{\lr(\Xtilde_{(\cdot)})})_{t_4}\circ (\Phi_{\lr(\Ytilde_{(\cdot)})})_{t_5}\big)(q_0)$
from a small enough connected open neighbourhood $]a_1,b_1[\times\cdots\times ]a_5,b_5[$ of the origin in $\R^5$
is a diffeomorphism onto a 5-dimensional submanifold of $O(q_0)$,
which we can assume to be the whole $O(q_0)$.
Here $(\Phi_{H})_t$ is the flow at time $t$ of a vector field $H$ on $O(q_0)$.

As a consequence of that observation
and the fact that $(\pi_{Q,M})_*\mc{F}'\onq=\{0\}$ for all $\q\in O(q_0)$,
we have that for any $\q\in O(q_0)$ the fiber
$(\pi_{Q,M}|_{O(q_0)})^{-1}(x)$ is 
the image of the connected open neighbourhood $]a_1,b_1[\times ]a_2,b_2[\times ]a_3,b_3[$ of origin in $\R^3$ by the map
$(t_1,t_2,t_3)\mapsto \big((\Phi_{F_1})_{t_1}\circ (\Phi_{F_2})_{t_2}\circ (\Phi_{\nu(\Rolbar)})_{t_3}\big)(q)$.
As we observed above, $\mc{F'}\onq \phi=\{0\}$ for all $\q\in O(q_0)$,
hence it follows that $\phi$ is constant on the connected fiber
$(\pi_{Q,M}|_{O(q_0)})^{-1}(x)$ for any $x\in V=\pi_{Q,M}(O(q_0))$

This means that for every $x\in V$,
there is a unique number $\ol{\phi}(x)$ in $\R$
such that $\phi(q)=\ol{\phi}(x)$ for all $q\in O(q_0)$ such that $\pi_{Q,M}(q)=x$.
That is $\ol{\phi}$ is a function $V\to\R$ such that $\ol{\phi}\circ (\pi_{Q,M}|_{O(q_0)})=\phi$.
The map $\pi_{Q,M}|_{O(q_0)}$ being a submersion $O(q_0)\to V$,
we deduce that $\ol{\phi}:V\to\R$ is smooth.

Finally, if the vector fields $X',Y'$
on $V=\pi_{Q,M}(O(q_0))$ are defined as in the statement of this
proposition, then it is clear that $X'|_x=\Xtilde_A$, $Y'|_x=\Ytilde_A$
for all $\q\in O(q_0)$ by the fact that $\ol{\phi}(x)=\phi(q)$
and by \eqref{e2.10}.

\end{proof}

\begin{remark}
One observes that by \eqref{eq:pf1:1}, \eqref{eq:F2_omega} and \eqref{eq:F1_omega}
we have $\nu(\Rolbar_q)\onq \omega=0$, $F_1\onq \omega=0$, $F_2\onq\omega=0$
for each $\q\in O(q_0)$
which by the argument given in the proof of Proposition \ref{pr:phi}
implies that there is a smooth function $\ol{\omega}:V\to\R$
defined on the open subset $V=\pi_{Q,M}(O(q_0))$ of $M$ 
such that $\ol{\omega}\circ (\pi_{Q,M}|_{O(q_0)})=\omega$.
However, we will not have a need to use this fact in the subsequent arguments.
\end{remark}

It will from now on be useful to start working with the orthonormal vector
fields $X',Y'$ on $M$ instead of $\Xtilde_{(\cdot)},\Ytilde_{(\cdot)}$ defined on $O(q_0)$, or $X,Y$ defined on $V=\pi_{Q,M}(O(q_0)) $.
However, because the choice of the orthonormal vector fields $X,Y$ on $V$
was arbitrary, we may from now on simply assume that we chose them to be $X',Y'$ respectively.
We then have $X|_x=\Xtilde_A$, $Y|_x=\Ytilde_A$ (by Proposition \ref{pr:phi}) for each $\q\in O(q_0)$,
which by \eqref{e2.10} amounts to us having $\phi=0$ on $O(q_0)$ and hence $\ol{\phi}=0$ on $V$.
Let us highlight these observations by writing them down into a separate  equation
\begin{align}\label{eq:X_Xtilde}
\phi(q)=0,\quad \Xtilde_A=X|_x,\quad \Ytilde_A=Y|_x,
\quad \forall \q\in O(q_0).
\end{align}

Obviously $\phi$ is now constant, and therefore the definitions of $G_{\Xtilde}, G_{\Ytilde}$ in \eqref{eq:G_H},
the standing assumption $G_{\Xtilde}=0$ in \eqref{e2.30}
and the definition of $\Gamma$ in \eqref{eq:Gamma} by which $g(\Gamma,X)=\Gamma^1_{(1,2)}$, $g(\Gamma,Y)=\Gamma^2_{(1,2)}$ yield
\begin{align}\label{eq:ss11:G:special}
0=G_{\Xtilde}(q)={}& g(\Gamma,\Xtilde_A)=g(\Gamma,X|_x)=\Gamma^1_{(1,2)}(x), \\
G_{\Ytilde}(q)={}& g(\Gamma,\Ytilde_A)=g(\Gamma,Y|_x)=\Gamma^2_{(1,2)}(x), \nonumber
\end{align}
for all $\q\in O(q_0)$.

We are now ready to formulate our second key result of this section.

\begin{proposition}\label{pr:ss11:key2}
The connection coefficients $\Gamma^i_{(j,k)}$ (see \eqref{eq:Gamma})
and the curvature $K$ of $(V,g|_V)$
obey the following relations on $V=\pi_{Q,M}(O(q_0))$,
\begin{align}\label{eq:ss11:warped_M_relations:1}
\Gamma^1_{(1,2)}=0,
\quad
Y(\Gamma^2_{(1,2)})=0, 
\quad
X(\Gamma^2_{(1,2)}) + (\Gamma^2_{(1,2)})^2 = -K,
\quad
Y(K)=0.
\end{align}
\end{proposition}

\begin{proof}
The first relation $\Gamma^1_{(1,2)}=0$
in \eqref{eq:ss11:warped_M_relations:1} has already appeared
in \eqref{eq:ss11:G:special}.
As $\Ytilde_A=Y|_x$, the fourth relation in \eqref{eq:ss11:warped_M_relations:1}
is a direct consequence of case c) of Proposition \ref{pr:ss11:1}, $0=\Ytilde_A(K)=Y|_x(K)$ for all $\q\in O(q_0)$
and hence $Y(K)=0$ on $V$.

The third relation in \eqref{eq:ss11:warped_M_relations:1}
is a consequence of $\Gamma^1_{(1,2)}=0$ and the definition of the Gaussian (sectional) curvature $K$,
\[
K={}&
g(R(X,Y)Y,X)=g(\nabla_X (-\Gamma^2_{(1,2)} X) - \nabla_Y (0) - \nabla_{-\Gamma^2_{(1,2)} Y} Y,X) \\
={}&
g(-X(\Gamma^2_{(1,2)}) X + \Gamma^2_{(1,2)} (-\Gamma^2_{(1,2)})X, X)
=
-X(\Gamma^2_{(1,2)}) - (\Gamma^2_{(1,2)})^2,
\]
where we used the identities
$\nabla_X Y = -\Gamma^1_{(1,2)} X=0$, $\nabla_Y X = \Gamma^2_{(1,2)} Y$,
$\nabla_X X = \Gamma^1_{(1,2)} Y=0$, $\nabla_Y Y = -\Gamma^2_{(1,2)} X$,
$[X,Y]=\nabla_X Y-\nabla_Y X=-\Gamma^2_{(1,2)} Y$.

Lastly, we obtain the second relation in \eqref{eq:ss11:warped_M_relations:1}
by recalling from \eqref{eq:ss11:lr_Ytilde_GYtilde} that $\lr(\Ytilde_A)\onq G_{\Ytilde}=0$, $\forall \q\in O(q_0)$,
that $Y=\Ytilde_{(\cdot)}$ and using the second line in \eqref{eq:ss11:G:special}
to obtain for $\q\in O(q_0)$,
\[
0=\lr(\Ytilde_A)\onq G_{\Ytilde}=\Ytilde_A(\Gamma^2_{(1,2)})=Y|_x(\Gamma^2_{(1,2)}).
\]
\end{proof}

\begin{corollary}\label{cor:ss11:warped_M}
The space $(V,g|_{V})\subset (M,g)$ is isometric to a warped product $(I\times N, s_1\oplus_{f} h)$
via some isometry $F:I\times N\to V$,
where $(N,h)$ is a 1-dimensional Riemannian manifold, $I$ is a non-empty open interval of $\R$ equipped with the standard metric $s_1$,
the warping function $f\in C^\infty(I)$ satisfies
\begin{align}\label{eq:ss11:M_warping_function:1}
\frac{f'(r)}{f(r)}=\Gamma^2_{(1,2)}(F(r,y)), \quad (r,y)\in I\times N,
\end{align}
and the canonical unit vector field $\pa{r}$ on $I$ is related to $X$ on $V$ by
\begin{align}\label{eq:ss11:M_warp_VFs:1}
F_*\pa{r}\Big|_{(r,y)}=X|_{F(r,y)},\quad (r,y)\in I\times N.
\end{align}
Furthermore, $f$ obeys the relation
\begin{align}\label{eq:ss11:M_warping_function:2}
\quad
\frac{f''(r)}{f(r)}=-K(F(r,y)), \quad (r,y)\in I\times N.
\end{align}
\end{corollary}

\begin{proof}
Let $\mc{S}=\spn\{Y\}$ be a 1-dimensional
distribution on $V\subset M$,
whose orthogonal complement is $\mc{S}^\perp=\spn\{X\}$.
Let $\mc{P}$ and $\mc{Q}$ be, respectively,
the orthogonal projection operators from $TM$ over $V$
onto $\mc{S}$ and $\mc{S}^\perp$ respectively.
Moreover, define a section $\eta$ of $\mc{S}^\perp$ over $V$
by $\eta:=-\Gamma^2_{(1,2)} X$.

Recall that since $\Gamma^1_{(1,2)}=0$ by \eqref{eq:ss11:warped_M_relations:1},
we have $\nabla_X Y=0$, $\nabla_Y X = \Gamma^2_{(1,2)} Y$,
$\nabla_X X=0$, $\nabla_Y Y = -\Gamma^2_{(1,2)} X$.

By using these connection relations,
the identity $Y(\Gamma^2_{(1,2)})=0$ from \eqref{eq:ss11:warped_M_relations:1}
and the fact that $g(Y,Y)=1$,
we find
\[
\mc{Q}\nabla_Y Y={}& -\Gamma^2_{(1,2)} \mc{Q}X
=-\Gamma^2_{(1,2)} X
=g(Y,Y)\eta \\
\mc{Q}\nabla_Y \eta
={}& \mc{Q}(-Y(\Gamma^2_{(1,2)})X-\Gamma^2_{(1,2)}\nabla_Y X)
=\mc{Q}(0-(\Gamma^2_{(1,2)})^2 Y)=0 \\
\mc{P}\nabla_X X={}& \mc{P}0=0.
\]

These three relations correspond precisely to the conditions (1)-(3) 
of the main theorem of \cite{hiepko79} (\emph{Generalvoraussetzung.}, p. 210), a result characterizing warped product Riemannian manifolds.
Consequently, by that theorem,
the space $(V,g|_V)$ is (after possibly shrinking $O(q_0)$ around $q_0$ and hence $V$ around $x_0$)
is isometric to a warped product $(I\times N, s_1\oplus_f h)$,
where $I,N\subset\R$ are both non-empty open intervals,
$s_1$, $h:=s_1$ are the canonical metrics on them,
and $f:I\to\R$ is a smooth strictly positive warping function.

Let this isometry be $F:(I\times N, s_1\oplus_f h)\to (V,g|_V)$.
Taking into account that $F$ is denoted by $\phi$ in \cite{hiepko79},
one sees from (17) in \cite{hiepko79} that
$F(I\times \{y\})$, $y\in N$ (resp. $F(\{r\}\times N)$, $r\in I$) are integral manifolds of
$\mc{S}^\perp=\spn\{X\}$ (resp. $\mc{S}=\spn\{Y\}$).
Writing $\pa{r}$ the canonical vector field on $I$,
this implies that $F_*\pa{r}\in\mc{S}^\perp=\spn\{X\}$,
and then that $F_*\pa{r}=\pm X$,
because $g(F_*\pa{r},F_*\pa{r})=s_1(\pa{r},\pa{r})=1$ and $g(X,X)=1$.
In the case the minus sign occurred, one could always replace $F$
by the isometry $(r,y)\mapsto F(-r,y)$,
which would map $\pa{r}$ to $+X$.
Hence we can assume that $F_*\pa{r}=+X$ as claimed in \eqref{eq:ss11:M_warp_VFs:1}.

Using then (3) of Proposition 7.35 in \cite{oneill83},
while observing in our case $I$ (resp. $N$) is the base manifold (resp. fiber),
we see that (identifying $f:I\to\R$ and $f\circ F^{-1}:V\to\R$)
\[
\Gamma^2_{(1,2)}(F(r,y))
=&{} g(\nabla_Y X, Y)
=g((X(f)/f)Y,Y)=X(f)/f \\
={}& \frac{(F_*\pa{r})(f)}{f}=\frac{f'(r)}{f(r)},
\]
which proves \eqref{eq:ss11:M_warping_function:1}.

Finally, \eqref{eq:ss11:M_warping_function:2} holds
generally on a warped product of type $(I\times N, s_1\oplus_f h)$,
and we can directly get it as a consequence of \eqref{eq:ss11:M_warping_function:1}
and the third identity
in \eqref{eq:ss11:warped_M_relations:1}
\[
-K(F(r,y))=X(\Gamma^2_{(1,2)})+(\Gamma^2_{(1,2)})^2
=\pa{r}\Big(\frac{f'(r)}{f(r)}\Big)+\Big(\frac{f'(r)}{f(r)}\Big)^2
=\frac{f''(r)}{f(r)}.
\]
\end{proof}

Here and later on below $f$ and $\hf$ are also used as a shorthand notation for
the functions $f\circ F^{-1}:V\to\R$ and $\hf\circ \hF^{-1}:\hV\to\R$.
Likewise, the vector fields $\pa{r}$ and $\pa{\hr}$ are often identified with $X$ and $\hE_2$
via $F_*$ and $\hF_*$ without further mention.
We also point out that since $X=\pa{r}$, $Y\in X^\perp$
and $\hE_2=\pa{\hr}$, $\spn\{\hM_1,\hM_3\}=\spn\{\hE_1,\hE_3\}=\hE_2^\perp$,
one finds that $Y(f)=0$ and $\hM_i(\hf)=\hE_i(\hf)=0$ for $i=1,3$.

Next we formulate two technical lemmas.

\begin{lemma}\label{le:ss11:2.1}
At every $\q\in O(q_0)$ we have
\begin{align}\label{eq:ss11:relations:1}
\hGamma^1_{(1,2)}={}& \omega (1+\omega^2)^{-1/2}\Gamma^2_{(1,2)}, \nonumber \\
(AX)(\hf)={}& -\omega (1+\omega^2)^{-1/2}\hf', \\
(AY)(\hf)={}& 0, \nonumber \\
\hZ_A(\hf)={}& -(1+\omega^2)^{-1/2}\hf', \nonumber 
\end{align}
and, in particular
\begin{align}\label{eq:ss11:relations:2}
\Gamma^2_{(1,2)} (AX)(\hf)
=\omega \Gamma^2_{(1,2)} \hZ_A(\hf)
={}& \frac{(\hf')^2}{\hf}.
\end{align}
\end{lemma}

\begin{proof}
Using $H_{\Xtilde}=\omega G_{\Ytilde}$ from the assumptions \eqref{e2.30}
and $\hGamma^3_{(2,3)}=-\hGamma^1_{(1,2)}$ from Proposition \ref{pr:ss11:key1},
the identity \eqref{eq:HX_DeltaGamma}
becomes
\[
2\omega G_{\Ytilde}
-2(1+\omega^2)^{1/2}\hGamma^1_{(1,2)}
=0
\]
which is the first line of \eqref{eq:ss11:relations:1}
because $G_{\Ytilde}=\Gamma^2_{(1,2)}$ by \eqref{eq:ss11:G:special}

In view of \eqref{eq:hM}, \eqref{eq:hE2_to_hM2}, \eqref{eq:X_Xtilde},
we also have
\begin{align}\label{eq:ss11:hE2_hM3_AX}
\omega(1+\omega^2)^{1/2}\hE_2 - \hM_3
=-(1+\omega^2) AX,
\end{align}
which applied to $\hat{f}$, considering that $\hM_3(\hat{f})=0$ and $\hE_2(\hat{f})=\pa{\hr}f(\hr)=\hat{f}'$,
yields
\[
\omega(1+\omega^2)^{1/2}\hf'=-(1+\omega^2) (AX)(\hf),
\]
which is equivalent to the second line of \eqref{eq:ss11:relations:1}.

Because $AY=A\Ytilde_A=-\hM_1\onq$
and $\hM_1(\hat{f})=0$,
we have $(AY)(\hat{f})=0$,
which is the third relation in \eqref{eq:ss11:relations:1}.

Finally, applying the relation (see \eqref{eq:hM}, \eqref{eq:hE2_to_hM2})
\[
(1+\omega^2)^{1/2}\hE_2+\omega\hM_3=-(1+\omega^2)\hZ_A
\]
to $\hat{f}$ yields, since $\hM_3(\hat{f})=0$ and $\hE_2(\hat{f})=\hat{f}'$,
\[
(1+\omega^2)^{1/2}\hat{f}'=-(1+\omega^2)\hZ_A(\hat{f})
\]
which gives the fourth line in \eqref{eq:ss11:relations:1},
and therefore completes the proof.
\end{proof}

\begin{lemma}\label{le:ss11:transversality}
We have
\[
0<\hg(A_0X|_{x_0},\hE_2|_{\hx_0})^2<1
\]
and hence
\[
A_0 T_{x_0} M\neq {}& (\hF_{\hr_0})_*(T_{\hy_0} \hN), \\
\hE_2|_{\hx_0}\notin {}& A_0 T_{x_0} M,
\]
where $\hF$ and $\hN$ are as in Corollary \eqref{cor:ss11:warped_hM},
$(\hr_0,\hy_0)=\hF^{-1}(\hx_0)$
and $\hF_{\hr_0}(\hy)=\hF(\hr_0,\hy)$ for $\hy\in \hN$.
\end{lemma}

\begin{proof}
Applying $\hg(\hE_2,\cdot)$ to \eqref{eq:ss11:hE2_hM3_AX}
and recalling that $\hg(\hE_2,\hE_2)=1$ while $\hE_2|_{\hx}\perp \hM_3|_q$ for $\q\in O(q_0)$, implies
$\omega(1+\omega^2)^{1/2}=-(1+\omega^2)\hg(AX,\hE_2)$,
and hence $\hg(A_0X|_{x_0},\hE_2|_{\hx_0})^2=\omega(q_0)^2/(1+\omega(q_0)^2)$.
But one has $\omega(q_0)\neq 0$, and therefore $0<\hg(A_0X|_{x_0},\hE_2|_{\hx_0})^2<1$.

By Corollary \eqref{cor:ss11:warped_hM},
the vector $\pa{\hr}\big|_{\hr_0}$ being normal to $\hN$ in $\hI\times \hN$,
its image $(\hF^{\hy_0})_*\pa{\hr}\big|_{\hr_0}=\hE_2|_{\hx_0}$ is normal to $\hF_{\hr_0}(\hN)$,
where $\hF^{\hy_0}(r)=\hF(r,\hy_0)$ for $r\in\hI$.
In other words,
$\hE_2|_{\hx_0}\perp (\hF_{\hr_0})_*(T_{\hy_0} \hN)$.
However, $A_0X|_{x_0}\in A_0 T_{x_0} M$ and we have already shown that
$A_0X|_{x_0}$ is not orthogonal to $\hE_2|_{\hx_0}$,
hence $A_0X|_{x_0}\notin (\hF_{\hr_0})_*(T_{\hy_0} \hN)$.
This shows that $A_0 T_{x_0} M\neq (\hF_{\hr_0})_*(T_{\hy_0} \hN)$ as claimed.

Lastly, by \eqref{eq:hM}, \eqref{eq:hE2_to_hM2} and \eqref{eq:X_Xtilde}
one has $\hg(A_0 Y|_{x_0}, \hE_2|_{\hx_0})=0$,
so if $\hE_2|_{\hx_0}$ belonged to the set $A_0 T_{x_0} M$,
we would have $\hE_2|_{\hx_0}=\pm A_0 X|_{x_0}$ since $\hE_2|_{\hx_0}$ and $A_0 X|_{x_0}$
are unit vectors. But then it would happen that $|\hg(A_0X|_{x_0},\hE_2|_{\hx_0})|=1$,
which is in contradiction with what we have already shown.
This shows that $\hE_2|_{\hx_0}\notin A_0 T_{x_0} M$ as claimed.
\end{proof}

In the lemmas that follow next, we will be using certain geodesics on $M$ and $\hM$
to further understand the relationship between their Riemannian geometries.

First we will write a relationship between the warping function and the $r$-component of an arbitrary geodesic on a Riemannian manifold
of dimension $n\geq 2$
which is a warped product of an $(n-1)$-dimensional Riemannian manifold and a $1$-dimensional interval
in the same way as $(V,g|_V)$ and $(\hV,\hg|_{\hV})$ are
(see Corollaries \ref{cor:ss11:warped_M} and \ref{cor:ss11:warped_hM}).

In what follows, we will denote the derivatives of curves
on $M$, $\hM$, $Q$ such as $x(t)$, $\hx(t)$, $q(t)$ w.r.t $t$ with a dot on top,
while we will use a prime, or explicit $\dif{t}$, when other quantities are involved.

\begin{lemma}\label{le:ss11:2.3}
Assume that $\ol{F}:(\ol{I}\times \ol{N},s_1\oplus_{\ol{f}} \ol{h})\to (\ol{M},\ol{g})$
is an isometry,
where $\ol{I}\subset\R$ is an open interval equipped with standard 
Riemannian metric $s_1$, $(\ol{N},\ol{h})$ an $(n-1)$-dimensional Riemannian 
manifold and $(\ol{M},\ol{g})$ an $n$-dimensional Riemannian manifold,
$n\geq 2$.

Let $\ol{x}_0=\ol{F}(\ol{r}_0,\ol{y}_0)\in \ol{M}$,
$\ol{W}\in T|_{\ol{x}_0} \ol{M}$ be a unit vector and
let $\ol{x}_{\ol{W}}(t)$, $t\in ]-a,a[$, be the geodesic on $\ol{M}$
starting from $\ol{x}_0$ with initial velocity $\ol{W}$,
where $a=a(\ol{W})>0$.
Writing $\ol{x}_{\ol{W}}(t)=\ol{F}(\ol{r}_{\ol{W}}(t),\ol{y}_{\ol{W}}(t))$,
we have for all $t\in ]-a,a[$ we have
\begin{gather}\label{eq:ss11:relations:2.1}
(\ol{r}_{\ol{W}}'(t))^2+C_0(\ol{W})^2\frac{\ol{f}(\ol{r}_0)^2}{\ol{f}(\ol{r}_{\ol{W}}(t))^2}=1,
\quad \ol{r}_{\ol{W}}(0)=\ol{r}_0,
\quad
\ol{r}_{\ol{W}}'(0)=\ol{g}(\ol{W},\ol{F}_*\pa{\ol{r}}\big|_{\ol{r}_0}),
\end{gather}
where
$C_0(\ol{W}):=\big(1-\ol{g}(\ol{W},\ol{F}_*\pa{\ol{r}}\big|_{\ol{r}_0})^2\big)^{1/2}$
is a constant, $C_0(\ol{W})\in [0,1]$,
and $\pa{\ol{r}}$ is the canonical unit vector field on $\ol{I}$.

Moreover, (i) if $C_0(\ol{W})\neq 0$ then $|\ol{r}_{\ol{W}}'(t)|<1$ for all $t\in ]-a,a[$;
(ii) if $C_0(\ol{W})\neq 1$ then $\ol{r}_{\ol{W}}'(0)\neq 0$.
\end{lemma}

We remark that the condition
$\ol{r}_{\ol{W}}'(0)=\ol{g}(\ol{W},\ol{F}_*\pa{\ol{r}}\big|_{\ol{r}(0)})$
only serves to fix the correct sign for $\ol{r}_{\ol{W}}'(0)$,
since the form of the above differential equation for $\ol{r}_{\ol{W}}(t)$
only implies that $|\ol{r}_{\ol{W}}'(0)|=|\ol{g}(\ol{W},\ol{F}_*\pa{\ol{r}}\big|_{\ol{r}(0)})|$.

\begin{proof}
Since the initial velocity $\ol{W}$ is understood fixed, we suppress in the 
following the subscript $\ol{W}$ from quantities $\ol{x}(t)$, $\ol{r}(t)$ and $\ol{y}(t)$.

The space $\ol{M}$ being a warped product $(\ol{I}\times \ol{N}, s_1\oplus_{\ol{f}} \ol{h})$ with $\ol{I}$ as the \emph{base} manifold and $\ol{N}$ as the \emph{fiber},
we can use Proposition 7.38 in \cite{oneill83}
to write the geodesic equations for $\ol{x}(t)$ as
\[
\ol{r}''(t) = {}& \ol{h}(\ol{y}'(t),\ol{y}'(t)) \ol{f}(\ol{r}(t)) \ol{f}'(\ol{r}(t)) \\
\nabla^{\ol{h}}_{\ol{y}'(t)} \ol{y}'(t) ={}& \frac{-2}{\ol{f}(\ol{r}(t))} \dif{t} \big(\ol{f}(\ol{r}(t))\big) \ol{y}'(t),
\]
where $\nabla^{\ol{h}}$ is the Levi-Civita connection of $(\ol{N},\ol{h})$.
In addition, these differential equations are subject to the initial conditions $\ol{F}(\ol{r}(0),\ol{y}(0))=\ol{x}_0$
and
$\ol{F}_*(\ol{r}'(0)\pa{\ol{r}}\big|_{\ol{r}(0)},\ol{y}'(0))=\ol{W}$.
In particular, $\ol{r}'(0)=\ol{g}(\ol{W},\ol{F}_*\pa{\ol{r}}\big|_{\ol{r}(0)})$.

As is easily verified by a direct computation, the above differential equation for $\ol{y}(t)$ is equivalent to
\begin{align}\label{eq:ss11:dot_hy}
\ol{y}'(t)=\frac{\ol{f}(\ol{r}(0))^2}{\ol{f}(\ol{r}(t))^2}(P^{\ol{N}})_0^t \ol{y}'(0),
\end{align}
where $(P^{\ol{N}})_0^t$ is the parallel transport on $(\ol{N},\ol{h})$ along $\ol{y}(t)$.

The fact that $\ol{W}$ is a unit vector on $\ol{M}$,
implies that $\ol{x}(t)$ is a unit speed geodesic on $\ol{M}$,
due to which
\[
(\ol{r}'(t))^2+\ol{f}(\ol{r}(t))^2 \ol{h}(\ol{y}'(t),\ol{y}'(t))
={}&
\n{(\ol{r}'(t),\ol{y}'(t))}_{s_1\oplus_{\ol{f}} \ol{h}}^2
=\n{\ol{F}_*(\ol{r}'(t),\ol{y}'(t))}_{\ol{g}}^2
=\n{\dot{\ol{x}}(t)}_{\ol{g}}^2=1.
\]
Using Eq. \eqref{eq:ss11:dot_hy}
and the fact that $(P^{\ol{N}})_0^t$ is an isometry between tangent spaces,
one finds
\[
\ol{h}(\ol{y}'(t),\ol{y}'(t))=\frac{\ol{f}(\ol{r}(0))^4}{\ol{f}(\ol{r}(t))^4}\ol{h}(\ol{y}'(0),\ol{y}'(0))
=\frac{\ol{f}(\ol{r}_0)^2}{\ol{f}(\ol{r}(t))^4}C_0(\ol{W})^2,
\]
which substituted into the previous equation 
yields the relation \eqref{eq:ss11:relations:2.1}
with
$C_0\geq 0$ defined by $C_0(\ol{W})^2=\ol{f}(\ol{r}_0)^2\ol{h}(\ol{y}'(0),\ol{y}'(0))$.

In order to obtain the claimed expression for $C_0$,
first note that
\begin{multline*}
\ol{F}_*(\ol{r}'(0)\pa{\ol{r}}\big|_{\ol{r}(0)},\ol{y}'(0))
=\ol{W} \\
=\ol{g}(\ol{W},\ol{F}_*\pa{\ol{r}}\big|_{\ol{r}_0})\ol{F}_*\pa{\ol{r}}\big|_{\ol{r}_0}+\big(\ol{W}-\ol{g}(\ol{W},\ol{F}_*\pa{\ol{r}}\big|_{\ol{r}_0})\ol{F}_*\pa{\ol{r}}\big|_{\ol{r}_0}\big)
\in T|_{\ol{r}(0)} \ol{I}\oplus T_{\ol{y}(0)} \ol{N},
\end{multline*}
we have $\ol{F}_*\ol{y}'(0)=\ol{W}-\ol{g}(\ol{W},\ol{F}_*\pa{\ol{r}}\big|_{\ol{r}_0})\ol{F}_*\pa{\ol{r}}\big|_{\ol{r}_0}$
and hence, recalling that $\ol{g}(\ol{W},\ol{W})=1$,
\[
\ol{f}(\ol{r}_0)^2\ol{h}(\ol{y}'(0),\ol{y}'(0))
=
\ol{g}(\ol{F}_*\ol{y}'(0),\ol{F}_*\ol{y}'(0))
=
1-\ol{g}(\ol{W},\ol{F}_*\pa{\ol{r}}\big|_{\ol{r}_0})^2,
\]
i.e., $C_0(\ol{W})^2=\ol{f}(\ol{r}_0)^2 \ol{h}(\ol{y}'(0),\ol{y}'(0))=1-\ol{g}(\ol{W},\ol{F}_*\pa{\ol{r}}\big|_{\ol{r}_0})^2$.

Finally, (i) if $C_0(\ol{W})\neq 0$ then 
$C_0(\ol{W})^2\ol{f}(\ol{r}_0)^2/\ol{f}(\ol{r}_{\ol{W}}(t))^2>0$
because $\ol{f}>0$ on $\ol{I}$,
and hence \eqref{eq:ss11:relations:2.1} immediately
yields $|\ol{r}_{\ol{W}}'(t)|<1$ for all $t$.
On the other hand, (ii) if $C_0(\ol{W})\neq 1$ i.e., if $\ol{g}(\ol{W},\ol{F}_*\pa{\ol{r}}\big|_{\ol{r}_0}) \neq 0$,
the relation \eqref{eq:ss11:relations:2.1} at $t=0$
yields $(\ol{r}_{\ol{W}}'(0))^2=1-C_0(\ol{W})^2\neq 0$.
This completes the proof.
\end{proof}

\begin{lemma}\label{le:ss11:2.4}
Let $x(t)$, $t\in ]-a,a[$, $a>0$, be the geodesic on $M$
starting from $x_0$ with initial velocity $X|_{x_0}$,
let $q(t)=(x(t),\hx(t);A(t))$, $t\in ]-a,a[$, be its $\dr$-lift onto $O(q_0)$
and write $x(t)=F(r(t),y(t))$, $\hx(t)=\hF(\hr(t),\hy(t))$.
Assuming that $\hf'\neq 0$ everywhere on $\hI$,
the warping functions $f$ and $\hf$ obey
\begin{align}\label{eq:ss11:G_Ytilde:2}
\frac{f'(r(t))}{f(r(t))}\hr'(t)=\frac{\hf'(\hat{r}(t))}{\hat{f}(\hat{r}(t))},
\quad \forall t\in ]-a,a[.
\end{align}
In particular, $\hr'(t)\neq 0$ for all $t\in ]-a,a[$.
\end{lemma}

\begin{proof}
Note that since $x(t)$ is a geodesic on $M$ with $\dot{x}(0)=X|_{x_0}$,
the curve $\hx(t)$ is a geodesic on $\hM$ starting with $\dot{\hx}(0)=A_0X|_{x_0}$.
It is also understood that we \emph{assume} $x(t)$ and its $\dr$-lift to be defined on the same open interval $]-a,a[$
by taking $a>0$ small enough.

Observe that $\Gamma^1_{(1,2)}=0$ on $V\subset M$
means that $\nabla_X X=\Gamma^1_{(1,2)}Y=0$,
that is $X$ is a unit geodesic vector field on $M$.
From this and the initial condition $\dot{x}(0)=X|_{x_0}$
we therefore conclude that $\dot{x}(t)=X|_{x(t)}$ for all $t$,
and consequently
\begin{align}\label{eq:ss11:dot_hx}
\dot{\hx}(t)=A(t)\dot{x}(t)=A(t)X|_{x(t)},\quad \forall t\in ]-a,a[.
\end{align}

Using \eqref{eq:ss11:M_warping_function:1}, \eqref{eq:ss11:dot_hx} and  \eqref{eq:ss11:relations:2},
in that order, one may compute
\[
\frac{f'(r(t))}{f(r(t))}\hf'(\hr(t))\hr'(t)
={}& \Gamma^2_{(1,2)}(x(t))\frac{d}{dt} (\hf(\hr(t)))
=\Gamma^2_{(1,2)}(x(t))\frac{d}{dt} (\hf\circ\hF^{-1})(\hx(t)) \\
={}& \Gamma^2_{(1,2)}(x(t))(A(t)X|_{x(t)})(\hat{f}\circ\hF^{-1})
=\frac{(\hf'(\hat{r}(t)))^2}{\hat{f}(\hat{r}(t))},
\]
which is a relation equivalent to \eqref{eq:ss11:G_Ytilde:2}
because $\hf'\neq 0$ on $\hI$ by assumption.

Lastly, \eqref{eq:ss11:G_Ytilde:2} and the assumption $\hf'\neq 0$ on $\hI$
clearly imply that $\hr'(t)\neq 0$ for all $t\in ]-a,a[$.
This completes the proof.
\end{proof}

\begin{lemma}\label{le:ss11:2.5}
Let $x(t)$, $t\in ]-a,a[$, $a>0$, be the geodesic on $M$
starting from $x_0$ with initial velocity $X|_{x_0}$,
let $q(t)=(x(t),\hx(t);A(t))$ be its $\dr$-lift onto $O(q_0)$
such that $q(0)=\qz$,
and write $x(t)=F(r(t),y(t))$, $\hx(t)=\hF(\hr(t),\hy(t))$.
Assuming that $\hf'\neq 0$ everywhere on $\hI$,
then $r(t)=t+r_0$ and the following relation holds
\begin{align}\label{eq:ss11:f_to_hf}
\frac{\hf(\hr(t))^2}{\hf(\hr_0)^2}-1=\hg(A_0X|_{x_0},\hE_2|_{\hx_0})^2\big(\frac{f(t+r_0)^2}{f(r_0)^2}-1\big),\quad t\in ]-a,a[,
\end{align}
where
$x_0=F(r_0,y_0)$, $\hx_0=\hF(\hr_0,\hy_0)$.
\end{lemma}

\begin{proof}
First let us note again that since $x(t)$ is a geodesic on $M$ with $\dot{x}(0)=X|_{x_0}$,
the curve $\hx(t)$ is a geodesic on $\hM$ starting with $\dot{\hx}(0)=A_0X|_{x_0}$.

Multiply both sides of \eqref{eq:ss11:G_Ytilde:2} by $\hr'(t)$ and use \eqref{eq:ss11:relations:2.1}
with $(\ol{M},\ol{g})=(\hM,\hg)$, $\ol{f}=\hf$, $\ol{W}=A_0X|_{x_0}$ and so on, so that $\ol{r}_{\ol{W}}(t)=\hr(t)$, to obtain the relation
\[
\frac{f'(r(t))}{f(r(t))}\Big(1-C_0^2\frac{\hf(\hr_0)^2}{\hf(\hr(t))^2}\Big)=\frac{\hf'(\hat{r}(t))}{\hat{f}(\hat{r}(t))}\hr'(t),
\]
where $C_0:=C_0(A_0X|_{x_0})=\big(1-\hg(A_0X|_{x_0},\hF_*\pa{\hr}\big|_{\hr_0})^2\big)^{1/2}$ and
$\hF_*\pa{\hr}\big|_{\hr_0}=\hE_2|_{\hx_0}$,
or equivalently
\[
\frac{f'(r(t))}{f(r(t))}
={}&\frac{\hf(\hr(t))}{\hf(\hr(t))^2-C_0^2\hf(\hr_0)^2}\dif{t} \big(\hf(\hat{r}(t))\big).
\]
We can readily integrate this ODE over the interval $[0,t]$ to yield
\[
\int_0^t \frac{f'(r(s))}{f(r(s))} ds
=\frac{1}{2}\ln \Big(\frac{\hf(\hr(t))^2-C_0^2\hf(\hr_0)^2}{(1-C_0^2)\hf(\hr_0)^2}\Big),
\]
where we used that the expression on the right inside the logarithm function is strictly positive, because by Lemma \ref{le:ss11:2.3}
we have $\hf(\hr(t))^2=C_0^2\hf(\hr_0)^{2}(1-\hr'(t)^2)^{-1}>C_0^2\hf(\hr_0)^2$ while $C_0^2=C_0(A_0X|_{x_0})^2<1$,
since $\hg(A_0X|_{x_0},\hE_2|_{\hx_0})^2=\omega(q_0)^2/(1+\omega(q_0)^2)>0$.

In order to deal with the integral on the left, use\eqref{eq:ss11:relations:2.1}
with $(\ol{M},\ol{g})=(M,g)$, $\ol{f}=f$, $\ol{W}=X|_{x_0}$ and so on, so that $\ol{r}_{\ol{W}}(t)=r(t)$, $C_0(\ol{W})=0$ (because $X|_{x_0}=F_*\pa{r}$) to conclude that $r'(t)=1$ i.e., $r(t)=t+r_0$,
and thus
\[
\int_0^t \frac{f'(r(s))}{f(r(s))} ds
=\int_0^t \frac{f'(r(s))}{f(r(s))}r'(s) ds
=\ln \frac{f(r(s))}{f(r_0)}.
\]

Combining the last two equations and exponentiating both sides yields
\[
\frac{f(r(t))^2}{f(r_0)^2}=\frac{\hf(\hr(t))^2-C_0^2\hf(\hr_0)^2}{(1-C_0^2)\hf(\hr_0)^2},
\]
where $r(t)=t+r_0$.
Since $C_0^2=1-\hg(A_0X|_{x_0},\hE_2|_{\hx_0})^2$,
this identity is clearly equivalent to \eqref{eq:ss11:f_to_hf}.
The proof is therefore complete.
\end{proof}

Combining Corollaries \ref{cor:ss11:warped_hM} and \ref{cor:ss11:warped_M},
Propositions \ref{pr:ss11:hN_hh_flat} and \ref{pr:ss11:V_g_flat} and Lemmas \ref{le:ss11:transversality} and \ref{le:ss11:2.5},
we arrive at the main theorem of this section.

\begin{theorem}\label{th:ss11:main}
Assume that $(\Pi_X(q) , \Pi_Y(q)) \neq (0,0)$ and $\hsigma_A \neq K (x)$ for all $\q\in O(q_0)$,
and that the equalities \eqref{e2.30} hold on $O(q_0)$.
Then, after shrinking $O(q_0)$ around $\qz$ if necessary,
there are open non-empty intervals $I,\hI\subset\R$ equipped with the standard Riemannian metric $s_1$,
strictly positive smooth functions $f:I\to\R$, $\hf:\hI\to\R$,
Riemannian manifolds $(N,h)$, $(\hN,\hh)$ with $\dim N=1$, $\dim \hN=2$
and isometries
$F:(I\times N, s_1\oplus_{f} h)\to (V,g)$,
$\hF:(\hI\times \hN, s_1\oplus_{\hf} \hh)\to (\hV,\hg)$ from warped products onto
the open neighbourhoods $V=\pi_{Q,M}(O(q_0))$ of $x_0$ and $\hV=\pi_{Q,M}(O(q_0))$ of $\hx_0$, respectively.
The warping functions $f$, $\hf$ are implicitly related by
\[
\quad
\frac{\hf''(\hr)}{\hf(\hr)}=\frac{f''(r)}{f(r)},
\]
holding at every $r\in I$, $\hr\in\hI$ for which there are $y\in N$ and $\hy\in\hN$
such that $(F(r,y),\hF(\hr,\hy))\in \pi_Q(O(q_0))$.

In addition,
\begin{itemize}
\item[(i)] if $\hf'$ is constant on all of $\hI$, then $(V,g)$ is flat;

\item[(ii)] if $\hf'\neq 0$ on all of $\hI$, then $(\hN,\hh)$ is flat and
for $a>0$ small enough, the warping function $f$ is determined from the system
of relations
\[
{}& \frac{\hf(\hr(t))^2}{\hf(\hr_0)^2}-1=P_0^2\Big(\frac{f(t+r_0)^2}{f(r_0)^2}-1\Big), \quad t\in ]-a,a[, \\
{}& (\hr'(t))^2+(1-P_0^2)\frac{\hf(\hr_0)^2}{\hf(\hr(t))^2}=1,\quad \hr(0)=\hr_0,\quad \hr'(0)=P_0,
\]
where
$P_0:=\hg(A_0\pa{r}\big|_{r_0},\pa{\hr}\big|_{\hr_0})$,
$x_0=F(r_0,y_0)$, $\hx_0=\hF(\hr_0,\hy_0)$,
and  $\pa{r}$, $\pa{\hr}$ are the canonical vector fields on $I$, $\hI$, respectively.
Moreover, $\hr'(t)\neq 0$ for all $t\in ]-a,a[$.
\end{itemize}

Finally, with $P_0$ defined as above, we have $0<P_0^2<1$.
In particular, $A_0 T_{x_0} M\neq (\hF_{\hr_0})_* T_{\hy_0} \hN$ and $\hF_*\pa{\hr}\big|_{\hr_0}\notin A_0 T_{x_0} M$,
where $\hF_{\hr_0}:\hN\to\hV$; $\hF_{\hr_0}(\hy)=\hF(\hr_0,\hy)$.
\end{theorem}

\begin{remark}\label{re:ss11:invariants}
For motivating the choice of functions $I$ and $U$ in
Section \ref{ss12} below (see \eqref{eq:ss12:U_I}), we point out that:
\begin{itemize}
\item[(i)]
$AY|_x\perp \hE_2|_{\hx}=\hF_*(\pa{\hr}|_{\hr})$ for all $\q\in O(q_0)$, $(\hr,\hy):=\hF^{-1}(\hx)$
because $\hM_1\onq=-AY|_x$, $\hM_1\onq \perp \hM_2\onq$ and $\big(1+\omega(q)^2\big)^{1/2}\hE_2|_{\hx}=\hM_2\onq$
(see \eqref{eq:hM}, \eqref{eq:X_Xtilde} and \eqref{eq:hE2_to_hM2}),
i.e.,
\begin{align}\label{eq:ss11:invariant_I}
\hg\big(AY|_x, \hF_*(\pa{\hr}|_{\hr})\big)=0,
\end{align}
for every $\q\in O(q_0)$, $(\hr,\hy):=\hF^{-1}(\hx)$

\item[(ii)]
Applying $\hg(\cdot,\hE_2|_{\hx})$ to \eqref{eq:ss11:hE2_hM3_AX} and using $\hE_2|_{\hx}=\hF_*(\pa{\hr}\big|_{\hr})$, $\hM_3 |_{\hx}\perp \hE_2|_{\hx}$,
$X|_x=F_*(\pa{r}\big|_r)$ and $\hg(\hE_2,\hE_2)=1$,
we find
\[
\hg\big(A\pa{r}\big|_r,\pa{\hr}\big|_{\hr}\big)=-\omega(q)(1+\omega(q)^2)^{-1/2},
\]
which yields, after using the first identity in \eqref{eq:ss11:relations:1},
\[
\hGamma^1_{(1,2)}(\hx)=-\hg\big(A\pa{r}\big|_r,\pa{\hr}\big|_{\hr}\big)\Gamma^2_{(1,2)}(x),
\]
and finally by Corollaries \ref{cor:ss11:warped_hM} and \ref{cor:ss11:warped_M} that
\begin{align}\label{eq:ss11:invariant_U}
\frac{\hf'(\hr)}{\hf(\hr)}-\frac{f'(r)}{f(r)}\hg\big(A\pa{r}\big|_r,\pa{\hr}\big|_{\hr}\big)=0,
\end{align}
for every $\q\in O(q_0)$, $(r,y):=F^{-1}(x)$, $(\hr,\hy):=\hF^{-1}(\hx)$.

\end{itemize}
\end{remark}

Similarly to the comment made right after the statement of Lemma \eqref{le:ss11:2.3},
we point out that the condition $\hr'(0)=P_0$ in case (ii) above
only serves to fix the right sign for $\hr'(0)$.
The differential equation for $\hr(t)$ written in squared form for $\hr'(t)$
would only allow us to conclude that $|\hr'(0)|=|P_0|$.

\vspace{2\baselineskip}\subsection{Partial Converse to the Main Theorem of Section \ref{ss11}}\label{ss12}
\ \newline

In this last section we aim to produce a \emph{partial} converse result to Theorem \ref{th:ss11:main} case (ii) of Section \ref{ss11}.
The objective is to show if that $(M,g)$ and $(\hM,\hg)$ are warped product Riemannian spaces of the kind described in that theorem, except that $(\hN,\hh)$ need not be flat,
then the rolling orbit $\odr(q_0)$ starting from certain points $\qz\in Q$ has dimension \emph{at most} $6$.
In particular, the rolling problem for such spaces is never completely controllable, i.e., $\odr(q)\neq Q$ for every $q\in Q$.

Let $(M,g)=(I\times N, s_1\oplus_f h)$ be a 2-dimensional
warped product with a smooth warping function $f:I\to\R$,
where $I,N\subset\R$ are non-empty open intervals
and $s_1,h$ are the standard Riemannian metrics on those intervals, respectively.

Furthermore, let $(\hN,\hh)$ be a connected 2-dimensional Riemannian manifold,
and let $(\hM,\hh)=(\hI\times \hN, s_1\oplus_{\hf} \hh)$
be a 3-dimensional warped product,
with a smooth warping function $\hf:\hI\to\R$ such that
\begin{align}\label{eq:ss12:dot_hf_neq_0}
\hf'(\hr)\neq 0,\quad \forall \hr\in \hI,
\end{align}
where $\hI\subset\R$ is a non-empty open interval and $s_1$ is the standard Riemannian metric on $\hI$.

Let $\qz\in Q_0$ be given, write $x_0=(r_0,y_0)$, $\hx_0=(\hr_0,\hy_0)$, and 
denote by $\pa{r}$, $\pa{\hr}$ the canonical unit vector fields on $I$ and $\hI$, respectively.
Assume further that the warping function $f$ of $M$ and a smooth function $\hr:I\to \hI$
are defined through the relations
\begin{align}
{}& \frac{\hf(\hr(t))^2}{\hf(\hr_0)^2}-1=P_0^2\Big(\frac{f(t+r_0)^2}{f(r_0)^2}-1\Big),\quad t\in I \label{eq:ss12:warping_functions} \\
{}& (\hr'(t))^2+(1-P_0^2)\frac{\hf(\hr_0)^2}{\hf(\hr(t))^2}=1,\quad \hr(0)=\hr_0,\quad \hr'(0)=P_0,
\label{eq:ss12:dot_hr}
\end{align}
where $P_0$ is the constant
\begin{align}\label{eq:ss12:P0}
P_0:=\hg(A_0 \pa{r}\big|_{r_0}, \pa{\hr}\big|_{\hr_0}),
\end{align}
which we assume to satisfy
\begin{align}\label{eq:ss12:P0_assumption}
0<|P_0|<1.
\end{align}
We also assume that $\qz$ is chosen so that the conditions
\begin{align}
{}& \hg(A_0 Y_0, \pa{\hr}\big|_{\hr_0})=0,\quad \forall Y_0\in T|_{x_0} M\ \mathrm{s.t.}\ Y_0 \perp \pa{r}\big|_{r_0} \label{eq:ss12:IC_1} \\
{}& \frac{\hf'(\hr_0)}{\hf(\hr_0)}-\frac{f'(r_0)}{f(r_0)}P_0=0 \label{eq:ss12:IC_2}
\end{align}
are satisfied.
The space $\{Y_0\in T|_{x_0} M\ |\ Y_0 \perp \pa{r}\big|_{r_0}\}$ is of course $1$-dimensional.
Notice that without the additional condition $\hr'(0)=P_0$, we could only conclude that $|\hr'(0)|=|P_0|$
given the form of the differential equation for $\hr(t)$ above.
Finally, since $\hr'(0)=P_0\neq 0$,
we will make it our last assumption (see case (ii) in Theorem \ref{th:ss11:main}) that
\begin{align}\label{eq:ss12:dot_hr_neq_0}
\hr'(t)\neq 0,\quad \forall t\in I.
\end{align}
In order to ensure that \eqref{eq:ss12:dot_hr} and \eqref{eq:ss12:dot_hr_neq_0}
can be fulfilled, for a given $\hf:\hI\to\R$, we can always replace the original open interval $\hI$
by a sufficiently small one that contains $\hr_0$.

\begin{remark}
We could also view the system \eqref{eq:ss12:warping_functions}-\eqref{eq:ss12:dot_hr}
as defining $\hf_1(t):=\hf(\hr(t))$ in terms of the function $f(t)$ by \eqref{eq:ss12:warping_functions},
and then $\hr(t)$ would be defined in terms of $\hf_1(t)$ via the relation \eqref{eq:ss12:dot_hr}.
\end{remark}

\begin{remark}
It should be reiterated that, unlike in the situation described in Theorem \ref{th:ss11:main} case (ii),
we will \emph{not} be assuming $(\hN,\hh)$ to be flat. 
\end{remark}

Let $(X,Y)=(E_1,E_2)$ be an oriented orthonormal
frame on $M$ such that $X=E_1=\pa{r}$ is the canonical unit vector field of $I$. Hence $E_2$ is tangent to the fibers $N$ of $M=I\times N$.
Similarly, let $(\hE_1,\hE_2,\hE_3)$ be an oriented orthonormal frame on $\hM$ such that $\hE_2=\pa{\hr}$ is the canonical vector field on $\hI$.
It follows that $\hE_1,\hE_3$ are tangent to the fibers $\hN$ of $\hM=\hI\times\hN$.

We will further let $\ol{Y}$ be the canonical vector field on $N$,
which we identify as a vector field on $M$.
Then $Y$ and $\ol{Y}$ point to the same direction, and
$1=g(Y,Y)=g(a\ol{Y},a\ol{Y})=f^2 a^2$ implies $Y=(1/f)\ol{Y}$.

Similarly, let $\ol{\hE}_1,\ol{\hE}_3$ be an $\hh$-orthonormal frame on $\hN$,
such that $\hE_i$ and $\ol{\hE}_i$ point to the same direction for $i=1,3$.
By $\hg$-orthonormality of $\hE_1,\hE_3$,
it follows that as above, that $\hE_i=(1/\hf)\ol{\hE}_i$, $i=1,3$.

The connection relations relevant for us, and which one can readily read out from Proposition 7.35 (p. 206) in \cite{oneill83},
are
\begin{align}\label{eq:ss12:connection_M}
{}&
\nabla_{\ol{Y}} X=\nabla_X \ol{Y}=\frac{f'}{f}\ol{Y}, \nonumber \\
{}&
 \nabla_X Y=\nabla_X \big(\frac{1}{f}\ol{Y}\big)
=-\frac{f'}{f^2}\ol{Y}+\frac{1}{f}\nabla_X \ol{Y}
=0,
\quad
\nabla_Y X=\frac{1}{f}(\frac{f'}{f}\ol{Y})=\frac{f'}{f}Y, \nonumber \\
{}&
\nabla_{X} X=\nabla_{\pa{r}}\pa{r}=0,\quad (\mathrm{trivially\ since\ this\ holds\ on}\ I) \nonumber \\
{}&
\nabla_{\ol{Y}} \ol{Y}=\mathrm{nor} \nabla_{\ol{Y}} \ol{Y}+\mathrm{tan} \nabla_{\ol{Y}} \ol{Y}
=-g(\ol{Y},\ol{Y})\frac{f'}{f}\pa{r} + 0=-f f' \pa{r}, \nonumber \\
{}&
\nabla_Y Y=\frac{1}{f} \nabla_{\ol{Y}}\big(\frac{1}{f}\ol{Y}\big)
=\frac{1}{f^2} \nabla_{\ol{Y}}\ol{Y}=-\frac{f'}{f} \pa{r},
\end{align}
where on the third line $\mathrm{tan} \nabla_{\ol{Y}} \ol{Y}=\nabla^{h}_{\ol{Y}} \ol{Y}=0$
due to the fact that $N$ is 1-dimensional
and hence $\nabla^{h}_{\ol{Y}} \ol{Y}=\alpha \ol{Y}$
where 
$\alpha=h(\nabla^{h}_{\ol{Y}} \ol{Y},\ol{Y})=\frac{1}{2}\ol{Y}(h(\ol{Y},\ol{Y}))=0$,
because $h(\ol{Y},\ol{Y})=1$.
On the last line, we observed that $\ol{Y}(f)=0$.

\begin{align}\label{eq:ss12:connection_hM}
{}& \hnabla_{\hE_2} \hE_i
=\hnabla_{\hE_2} \big(\frac{1}{\hf}\ol{\hE_i}\big)
=-\frac{\hf'}{\hf^2} \ol{\hE_i}
+\frac{1}{\hf} \hnabla_{\hE_2} \ol{\hE_i}
=-\frac{\hf'}{\hf^2} \ol{\hE_i}
+\frac{1}{\hf} \frac{\hf'}{\hf} \ol{\hE_i}
=0
\quad i=1,3, \nonumber \\
{}&
\hnabla_{\hE_2}\hE_2 = \hnabla_{\pa{\hr}} \pa{\hr}=0, \quad (\mathrm{trivially\ since\ this\ holds\ on}\ \hI) \nonumber \\
{}&
\hnabla_{\hE_i} \hE_2
=\hnabla_{\hE_i} \pa{\hr}
=\frac{1}{\hf} \hnabla_{\ol{\hE_i}} \pa{\hr}
=\frac{1}{\hf} \hnabla_{\pa{\hr}} \ol{\hE_i}
=\frac{\hf'}{\hf^2} \ol{\hE_i}=\frac{\hf'}{\hf} \hE_i,\quad i=1,3.
\end{align}

\begin{remark}
Note that the relation $\hnabla_{\hE_2} \hE_i=0$, $i=1,3$,
implies not only that $\hGamma^2_{(1,2)}=0$ and $\hGamma^2_{(2,3)}=0$
but also that $\hGamma^2_{(3,1)}=0$ which we did not derive in section \ref{ss11}.
\end{remark}

According to the above connection relations,
\[
\Gamma^2_{(1,2)}=g(\nabla_Y X,Y)=g(\nabla_Y \pa{r},Y)
=g(\frac{f'}{f}Y,Y)
=\frac{f'}{f}
\]
and
\[
\hGamma^1_{(1,2)}=-\hg(\hnabla_{\hE_1} \hE_2, \hE_1)
=-\hg(\hnabla_{\hE_1} \pa{\hr}, \hE_1)
=-\hg(\frac{\hf'}{\hf} \hE_1, \hE_1)
=-\frac{\hf'}{\hf}.
\]

Because $X=\pa{r}$, $\Gamma^2_{(1,2)}(x)=\frac{f'(r)}{f(r)}$
and $Y$ is tangent to the fiber $N$ of $M=I\times N$, we have
\begin{align}\label{eq:ss12:XY_Gamma}
X(\Gamma^2_{(1,2)})={}& \pa{r} \frac{f'(r)}{f(r)}
=\frac{f''(r)}{f(r)}-\big(\frac{f'(r)}{f(r)}\big)^2 \nonumber \\
Y(\Gamma^2_{(1,2)})={}& 0 
\end{align}

On the other hand,
$\hGamma^1_{(1,2)}(\hx)=-\frac{\hf'(\hr)}{\hf(\hr)}$, $\hx=\hF(\hr,\hy)$,
so after decomposing $AX=\hg(AX,\hE_1)\hE_1+\hg(AX,\pa{\hr})\pa{\hr}+\hg(AX,\hE_3)\hE_3$,
and recalling that
$\hE_1,\hE_3$ are tangent to the fiber $\hN$ of $\hM=\hI\times\hN$,
we obtain
\begin{align}\label{eq:ss12:AX_hGamma}
(AX)(\hGamma^1_{(1,2)})=-\hg(AX,\pa{\hr})\pa{\hr}\big(\frac{\hf'(\hr)}{\hf(\hr)}\big)
=-\hg(A\pa{r},\pa{\hr})\Big(\frac{\hf''(\hr)}{f(\hr)}-\big(\frac{\hf'(\hr)}{\hf(\hr)}\big)^2\Big)
\end{align}
For the same reasons, after decomposing
$AY=\hg(AY,\hE_1)\hE_1+\hg(AY,\pa{\hr})\pa{\hr}+\hg(AY,\hE_3)\hE_3$
we find
\begin{align}\label{eq:ss12:AY_hGamma}
(AY)(\hGamma^1_{(1,2)})
=-\hg(AY,\pa{\hr})\pa{\hr}\big(\frac{\hf'(\hr)}{\hf(\hr)}\big)
=-\hg(AY,\pa{\hr})\Big(\frac{\hf''(\hr)}{\hf(\hr)}-\big(\frac{\hf'(\hr)}{\hf(\hr)}\big)^2\Big).
\end{align}

In view of \eqref{eq:ss12:connection_M},
both vector fields $X$ and $Y$ are parallel along $X$,
and therefore
any vector field $W(t)=a_1(t)X+a_2(t)Y$ that is (defined and) parallel along the flow $t\mapsto (\Phi_X)_t$ of $X$
satisfies $W(t)=a_1(0)X|_{(\Phi_X)_t}+a_2(0)Y|_{(\Phi_X)_t}$.
Likewise, according to \eqref{eq:ss12:connection_hM} all three vector fields $\hE_1,\hE_2,\hE_3$
are parallel along $\hE_2=\pa{\hr}$,
and therefore
any vector field $\hat{W}(t)=a_1(t)\hE_1+a_2(t)\hE_2+a_3(t)\hE_3$ that is (defined and) parallel along the flow $t\mapsto (\Phi_{\hE_2})_t$ of $\hE_2$
satisfies $\hat{W}(t)=a_1(0)\hE_1|_{(\Phi_{\hE_2})_t}+a_2(0)\hE_2|_{(\Phi_{\hE_2})_t}+a_3(0)\hE_3|_{(\Phi_{\hE_2})_t}$.

The two smooth functions $U$ and $I$ on $Q$ that will play an important role in the following are defined by
\begin{align}\label{eq:ss12:U_I}
U(q):=\frac{\hf'(\hr)}{\hf(\hr)}-\frac{f'(r)}{f(r)}\hg(A \pa{r}\big|_{r}, \pa{\hr}\big|_{\hr}\big),
\qquad
I(q):=\hg(AY|_x,\pa{\hr}|_{\hr}),
\end{align}
for $\q\in Q$ and $x=(r,y)$, $\hx=(\hr,\hy)$.
The motivation behind these
definitions can be found in Remark \ref{re:ss11:invariants},
and in fact our goal is to show that both $U$ and $I$ vanish
identically on any orbit $\odr(q_0)$
for which the conditions \eqref{eq:ss12:IC_1} and \eqref{eq:ss12:IC_2} are satisfied at $\qz$.

Notice that the symbol $I$ is now used to denote both an open interval in $\R$ and a function $I:Q\to\R$.
This ambiguity in notations should, however, not create any confusion in the following since both
uses serve sharply different purposes.

Next we shall compute the derivatives of these two, and some associated functions
with respect to $\lr(X)$ and $\lr(Y)$ vector fields on $Q$.
For that purpose, defined first an auxiliary function $P$ on $Q$ by
\begin{align}\label{eq:ss12:P}
P(q):=\hg(A \pa{r}\big|_{r}, \pa{\hr}\big|_{\hr}\big),\quad q\in Q,
\end{align}
and notice that $P(q_0)=P_0$ (see \eqref{eq:ss12:P0})
as well as that
\[
U(q)=\frac{\hf'(\hr)}{\hf(\hr)}-\frac{f'(r)}{f(r)}P(q).
\]
Since $X=\pa{r}$, we find
\begin{align}\label{eq:ss12:lrX_P}
\lr(X)\onq P(\cdot)
={}&
\hg(A \nabla_{\pa{r}} \pa{r}, \pa{\hr}\big)
+
\hg(A \pa{r}, \hnabla_{A\pa{r}}\pa{\hr}\big) \nonumber \\
={}&
0
+
\hg\big(A \pa{r}, \big(\hg(A\pa{r},\pa{\hr})\hnabla_{\pa{\hr}}+\hg(A\pa{r},\hE_1)\hnabla_{\hE_1}+\hg(A\pa{r},\hE_3)\hnabla_{\hE_3}\big)\pa{\hr}\big) \nonumber \\
={}&
\hg\big(A \pa{r}, 0+\hg(A\pa{r},\hE_1)\frac{\hf'}{\hf}\hE_1+\hg(A\pa{r},\hE_3)\frac{\hf'}{\hf}\hE_3) \nonumber \\
={}&
\frac{\hf'}{\hf}\big(\hg(A\pa{r},\hE_1)^2+\hg(A\pa{r},\hE_3)^2\big) \nonumber \\
={}&
\frac{\hf'(\hr)}{\hf(\hr)}\big(1-P(q)^2\big),
\end{align}
where we have used the relations \eqref{eq:ss12:connection_M} and \eqref{eq:ss12:connection_hM}.
By a similar reasoning,
\begin{align}\label{eq:ss12:lrY_P}
\lr(Y)\onq P(\cdot)
={}&
\hg(A \nabla_{Y} \pa{r}, \pa{\hr}\big)
+
\hg(A \pa{r}, \hnabla_{AY}\pa{\hr}\big) \nonumber \\
={}&
\hg(\frac{f'}{f} AY, \pa{\hr}\big)
+
\hg\big(A\pa{r}, \big(\hg(AY,\pa{\hr})\hnabla_{\pa{\hr}}+\hg(AY,\hE_1)\hnabla_{\hE_1}+\hg(AY,\hE_3)\hnabla_{\hE_3}\big)\pa{\hr}\big) \nonumber \\
={}&
\hg(\frac{f'}{f} AY, \pa{\hr}\big)
+
\hg\big(A\pa{r}, 0+\hg(AY,\hE_1)\frac{\hf'}{\hf}\hE_1
+\hg(AY,\hE_3)\frac{\hf'}{\hf}\hE_3) \nonumber \\
={}& 
\frac{f'}{f}I(q)+\frac{\hf'}{\hf}\big(\hg(A\pa{r},\hE_1)\hg(AY,\hE_1)
+\hg(A\pa{r},\hE_3)\hg(AY,\hE_3)\big) \nonumber \\
={}&
\frac{f'}{f}I(q)+\frac{\hf'}{\hf}\big(\hg(A\pa{r},AY)-\hg(A\pa{r},\pa{\hr})\hg(AY,\pa{\hr})\big) \nonumber \\
={}&
\frac{f'}{f}I(q)+\frac{\hf'}{\hf}\big(0-P(q)\hg(AY,\pa{\hr})\big) \nonumber \\
={}&
\big(\frac{f'(r)}{f(r)}-\frac{\hf'(\hr)}{\hf(\hr)}P(q)\big)I(q).
\end{align}

These identities and relations \eqref{eq:ss12:XY_Gamma}-\eqref{eq:ss12:AY_hGamma}
then further yield
\begin{align}\label{eq:ss12:lrX_U}
\lr(X)\onq U(\cdot)
={}&
\lr(X)\onq \big(\frac{\hf'}{\hf}-\frac{f'}{f}P(\cdot)\big)
=
AX(\frac{\hf'}{\hf})-\big(X(\frac{f'}{f})\big)P(q)
-\frac{f'}{f}\lr(X)\onq P(\cdot) \nonumber \\
={}&
-AX(\hGamma^1_{(1,2)})-\big(X(\Gamma^2_{(1,2)})\big)P(q)
-\frac{f'}{f}\lr(X)\onq P(\cdot) \nonumber \\
={}&
\hg(AX,\pa{\hr})\Big(\frac{\hf''}{\hf}-\big(\frac{\hf'}{\hf}\big)^2\Big)
-\Big(\frac{f''}{f}-\big(\frac{f'}{f}\big)^2\Big)P(q)
-\frac{f'}{f}\frac{\hf'}{\hf}\big(1-P(q)^2\big) \nonumber \\
={}&
\Big(\frac{\hf''}{\hf}-\frac{f''}{f}\Big)P(q)
+\Big(-\big(\frac{\hf'}{\hf}\big)^2
+\big(\frac{f'}{f}\big)^2\Big)P(q)
-\frac{f'}{f}\frac{\hf'}{\hf}\big(1-P(q)^2\big) \nonumber \\
={}&
\Big(\frac{\hf''(\hr)}{\hf(\hr)}-\frac{f''(r)}{f(r)}\Big)P(q)
+\Big(-\frac{f'(r)}{f(r)} -\frac{\hf'(\hr)}{\hf(\hr)} P(q) \Big)U(q)
\end{align}
and
\begin{align}\label{eq:ss12:lrY_U}
\lr(Y)\onq U(\cdot)
={}&\lr(Y)\onq \big(\frac{\hf'}{\hf}-\frac{f'}{f}P(\cdot)\big)
=\lr(Y)\onq \big(-\hGamma^1_{(1,2)}-\Gamma^2_{(1,2)} P(\cdot)\big) \nonumber \\
={}& -(AY)(\hGamma^1_{(1,2)})-Y(\Gamma^2_{(1,2)})P(\cdot)
-\Gamma^2_{(1,2)} \lr(Y)\onq P(\cdot) \nonumber \\
={}&
\hg(AY,\pa{\hr})\Big(\frac{\hf''}{\hf}-\big(\frac{\hf'}{\hf}\big)^2\Big)
-0
-\Gamma^2_{(1,2)} \big(\frac{f'}{f}-\frac{\hf'}{\hf}P(q)\big)I(q) \nonumber \\
={}&
\Big(\frac{\hf''}{\hf}-\big(\frac{\hf'}{\hf}\big)^2\Big)I(q)
-\frac{f'}{f} \big(\frac{f'}{f}-\frac{\hf'}{\hf}P(q)\big)I(q) \nonumber \\
={}&
\Big(\frac{\hf''}{\hf}-\big(\frac{\hf'}{\hf}\big)^2
-\big(\frac{f'}{f}\big)^2+\frac{f'}{f}\frac{\hf'}{\hf}P(q)\Big)I(q) \nonumber \\
={}&
\Big(\frac{\hf''(\hr)}{\hf(\hr)}-\big(\frac{f'(r)}{f(r)}\big)^2
-\frac{\hf'(\hr)}{\hf(\hr)}U(q)\Big)I(q).
\end{align}

As for the derivatives of the function $I$, we have using \eqref{eq:ss12:connection_M} and \eqref{eq:ss12:connection_hM},
\begin{align}\label{eq:ss12:lrX_I}
\lr(X)\onq I(\cdot)
={}&
\hg(A \nabla_{\pa{r}} Y, \pa{\hr}\big)
+
\hg(A Y, \hnabla_{A\pa{r}}\pa{\hr}\big) \nonumber \\
={}&
0+
\hg\big(AY, \big(\hg(A\pa{r},\pa{\hr})\hnabla_{\pa{\hr}}+\hg(A\pa{r},\hE_1)\hnabla_{\hE_1}+\hg(A\pa{r},\hE_3)\hnabla_{\hE_3}\big)\pa{\hr}\big) \nonumber \\
={}&
\hg\big(AY, 0+\hg(A\pa{r},\hE_1)\frac{\hf'}{\hf}\hE_1+\hg(A\pa{r},\hE_3)\frac{\hf'}{\hf}\hE_3) \nonumber \\
={}&
\frac{\hf'}{\hf}\big(\hg(A\pa{r},\hE_1)\hg(AY,\hE_1)+\hg(A\pa{r},\hE_3)\hg(AY,\hE_3)\big) \nonumber \\
={}&
\frac{\hf'}{\hf}\big(\hg(A\pa{r},AY)-\hg(A\pa{r},\pa{\hr})\hg(AY,\pa{\hr})\big) \nonumber \\
={}&
\frac{\hf'}{\hf}\big(0-P(q)I(q))\big) \nonumber \\
={}&
-\frac{\hf'(\hr)}{\hf(\hr)}P(q)I(q).
\end{align}
Again by similar computations, while consulting \eqref{eq:ss12:connection_M} and \eqref{eq:ss12:connection_hM}
for the connection identities,
\begin{align}\label{eq:ss12:lrY_I}
\lr(Y)\onq I(\cdot)
={}&
\hg(A \nabla_{Y} Y, \pa{\hr}\big)
+
\hg(A Y, \hnabla_{AY}\pa{\hr}\big) \nonumber \\
={}&
-\frac{f'}{f}\hg\big(A\pa{r}, \pa{\hr}\big)
+
\hg\big(AY, \big(\hg(AY,\pa{\hr})\hnabla_{\pa{\hr}}+\hg(AY,\hE_1)\hnabla_{\hE_1}+\hg(AY,\hE_3)\hnabla_{\hE_3}\big)\pa{\hr}\big) \nonumber \\
={}&
-\frac{f'}{f}P(q)
+
\hg\big(AY, 0+\hg(AY,\hE_1)\frac{\hf'}{\hf}\hE_1
+\hg(AY,\hE_3)\frac{\hf'}{\hf}\hE_3) \nonumber \\
={}& 
-\frac{f'}{f}P(q)
+
\frac{\hf'}{\hf}\big(\hg(AY,\hE_1)^2+\hg(AY,\hE_3)^2\big) \nonumber \\
={}&
-\frac{f'}{f}P(q)
+\frac{\hf'}{\hf}\Big(\hg(AY,AY)-\hg(AY,\pa{\hr})^2\Big) \nonumber \\
={}&
-\frac{f'}{f}P(q)
+\frac{\hf'}{\hf}\big(1-I(q)^2\big) \nonumber \\
={}& U(q)-\frac{\hf'(\hr)}{\hf(\hr)}I(q)^2.
\end{align}

Define $C$ to be the set of those points of $Q$ where $I$ and $U$ both vanish,
i.e.,
\begin{align}\label{eq:ss12:set_C}
C=\{q\in Q\ |\ U(q)=0,\ I(q)=0\},
\end{align}
and notice that the (initial condition) assumptions \eqref{eq:ss12:IC_1} and \eqref{eq:ss12:IC_2}
guarantee that $C$ is non-empty since it contains $q_0$, i.e.,
\[
q_0\in C.
\]

Given $\ol{q}\in C$, the differential relations \eqref{eq:ss12:lrY_U} and \eqref{eq:ss12:lrY_I}
imply that $U=0$ and $I=0$ along the integral curve $q_Y(t,\ol{q})$ of $\lr(Y)$ as well,
that is
\begin{align}\label{eq:ss12:C_invariant_along_Y}
q_Y(t,\ol{q})\in C,\quad \forall \ol{q}\in C,\ t\in D_{\lr(Y)}(\ol{q}).
\end{align}
In addition, \eqref{eq:ss12:lrY_P} implies that $P$ remains constant along $q_Y(t,\ol{q})$ if $\ol{q}\in C$,
that is
\begin{align}\label{eq:ss12:P_invariant_along_Y}
P(q_Y(t,\ol{q}))=P(\ol{q}),\quad \forall \ol{q}\in C,\ t\in D_{\lr(Y)}(\ol{q}).
\end{align}

At this point, we will derive a second order differential relation between the warping functions $f$ and $\hf$.
Differentiating \eqref{eq:ss12:warping_functions} with respect to $t$ we get
\begin{align}\label{eq:ss12:dot_1}
\frac{\hf(\hr(t))}{\hf(\hr_0)^2} \hf'(\hr(t)) \hr'(t) = P_0^2 \frac{f(t+r_0)}{f(r_0)^2} f'(t+r_0).
\end{align}
Multiply both sides $\hr'(t)$ and making use of \eqref{eq:ss12:dot_hr} and then \eqref{eq:ss12:warping_functions} gives
\[
P_0^2 \frac{f(t+r_0)}{f(r_0)^2} f'(t+r_0) \hr'(t)
={}& \frac{\hf(\hr(t))}{\hf(\hr_0)^2} \hf'(\hr(t)) \big(1-(1-P_0^2)\frac{\hf(\hr_0)^2}{\hf(\hr(t))^2}\big) \\
={}& \frac{\hf'(\hr(t))}{\hf(\hr(t))} \big(\frac{\hf(\hr(t))^2}{\hf(\hr_0)^2}-1+P_0^2\big)
=\frac{\hf'(\hr(t))}{\hf(\hr(t))} P_0^2 \frac{f(t+r_0)^2}{f(r_0)^2},
\]
that is
\begin{align}\label{eq:ss12:warping_functions:2}
\frac{f'(t+r_0)}{f(t+r_0)} \hr'(t)
=\frac{\hf'(\hr(t))}{\hf(\hr(t))},\quad \forall t\in I.
\end{align}
Notice that this is precisely the relation we have obtained in Lemma \ref{le:ss11:2.4},
knowing also that $r(t)=t+r_0$ there according to Lemma \ref{le:ss11:2.5}.

Differentiating \eqref{eq:ss12:dot_hr} w.r.t. $t$ and using the assumption \eqref{eq:ss12:dot_hr_neq_0}, we obtain
a formula for the second derivative of $\hr(t)$,
\[
\hr''(t)={}&(1-P_0^2)\frac{\hf(\hr_0)^2}{\hf(\hr(t))^2}\frac{\hf'(\hr(t))}{\hf(\hr(t))},
\]
On the other hand, the derivative w.r.t $t$ of \eqref{eq:ss12:warping_functions:2} is given by
\[
\Big(\frac{f''(t+r_0)}{f(t+r_0)}
-\big(\frac{f'(t+r_0)}{f(t+r_0)}\big)^2\Big)\hr'(t)
+\frac{f'(t+r_0)}{f(t+r_0)} \hr''(t)
=\Big(\frac{\hf''(\hr(t))}{\hf(\hr(t))}
-\big(\frac{\hf'(\hr(t))}{\hf(\hr(t))}\big)^2\Big)\hr'(t).
\]
Multiplying this with $\hr'(t)$, using the previous expression for $\hr''(t)$ and re-arraging some terms, we then find
\begin{multline*}
\Big(-\frac{f''(t+r_0)}{f(t+r_0)}
+\big(\frac{f'(t+r_0)}{f(t+r_0)}\big)^2
+\frac{\hf''(\hr(t))}{\hf(\hr(t))}
-\big(\frac{\hf'(\hr(t))}{\hf(\hr(t))}\big)^2\Big)(\hr'(t))^2 \\
=(1-P_0^2)\frac{f'(t+r_0)}{f(t+r_0)} \frac{\hf(\hr_0)^2}{\hf(\hr(t))^2}\frac{\hf'(\hr(t))}{\hf(\hr(t))}\hr'(t),
\end{multline*}
which, after application of \eqref{eq:ss12:dot_hr} on the right hand side
and \eqref{eq:ss12:warping_functions:2} on both sides, becomes
\[
\Big(-\frac{f''(t+r_0)}{f(t+r_0)}
+\frac{\hf''(\hr(t))}{\hf(\hr(t))}\Big)(\hr'(t))^2
+\big(1-(\hr'(t))^2\big)\big(\frac{\hf'(\hr(t))}{\hf(\hr(t))}\big)^2
=\big(\frac{\hf'(\hr(t))}{\hf(\hr(t))}\big)^2 \big(1-(\hr'(t))^2\big).
\]
Cancelling the common term from both sides and dividing by the non-zero $(\hr'(t))^2$ (see \eqref{eq:ss12:dot_hr_neq_0}),
we finally arrive at an important relation
\begin{align}\label{eq:ss12:warping_functions:3}
\frac{\hf''(\hr(t))}{\hf(\hr(t))}=\frac{f''(t+r_0)}{f(t+r_0)},\quad \forall t\in I.
\end{align}

If $W$ is a smooth vector field on $M$,
we shall denote by $t\mapsto q_W(t,\ol{q})$ the integral curve  of $\lr(W)$ passing
through a given point $\ol{q}=(\ol{x},\ol{\hx};\ol{A})$ of $Q$ at $t=0$,
and write $q_W(t,\ol{q})=(x_W(t,\ol{q}),\hx_W(t,\ol{q});A_W(t,\ol{q}))$,
$x_W(t,\ol{q})=(r_X(t,\ol{q}),y_W(t,\ol{q}))$ and $\hx_W(t,\ol{q})=(\hr_W(t,\ol{q}),\hy_W(t,\ol{q}))$
as well as
$(\ol{r},\ol{y})=\ol{x}$, $(\ol{\hr},\ol{\hy})=\ol{\hx}$ at the initial points on $M$ and $\hM$.
The domain of definition of $t\mapsto q_W(t,\ol{q})$ will be written as $D_{\lr(W)}(\ol{q})$.
Note that, one should not confuse the function $\hr_W(t,\ol{q})$ with $\hr(t)$ appearing in \eqref{eq:ss12:warping_functions}, \eqref{eq:ss12:dot_hr},
although we will actually demonstrate just below that $\hr_X(t,q_0)=\hr(t)$ for all $t\in D_{\lr(X)}(q_0)$.
In addition, we will only be using this notation in cases where $W$ is either $X$ or $Y$.

Because $X=\pa{r}$ is a geodesic vector field on $M$ (see \eqref{eq:ss12:connection_M}),
the curve $x_X(t,\ol{q})$ is a geodesic on $M$,
and it follows from Proposition \ref{pr:prelim:1}
that $\hx_X(t,\ol{q})$ is a geodesic on $\hM$.
Lemma \ref{le:ss11:2.3} implies that $r_X(t,\ol{q})$ is given by
(taking $X|_{\ol{x}}$ for $\ol{W}$ so that $g(X,\pa{r})=1$, hence $C_0(X|_x)=0$ and $\p{r}{t}(0,\ol{q})=1$)
\begin{align}\label{eq:ss12:r_flow}
r_X(t,\ol{q})=t+\ol{r},\quad t\in D_{\lr(X)}(\ol{q}),\quad \ol{q}\in Q,
\end{align}
while $\hr_X(t,\ol{q})$ satisfies the differential equation
(taking $A_{\ol{q}}X|_{\ol{x}}$ for $\ol{W}$)
\begin{align}\label{eq:ss12:geodesic}
{}& \big(\p{\hr_X}{t}(t,\ol{q})\big)^2+(1-P(\ol{q})^2) \frac{\hf(\ol{\hr})^2}{\hf(\hr_X(t,\ol{q}))^2}=1, \\
{}& \hr_X(0,\ol{q})=\ol{\hr},
\quad \p{\hr_X}{t}(0,\ol{q})=P(\ol{q}),
\quad t\in D_{\lr(X)}(\ol{q}),\quad \ol{q}\in Q,
\end{align}
where $P(\ol{q})=\hg(\ol{A}X|_{\ol{x}}, \pa{\hr}\big|_{\ol{\hr}}\big)$,
i.e., the same quantity as given in \eqref{eq:ss12:P}
because $X|_{\ol{x}}=\pa{r}\big|_{\ol{r}}$.
In addition, $D_{\lr(X)}(\ol{q})$ is the open interval of definition of the flow of $\lr(X)$ starting at $\ol{q}$,
which contains $0$.

In particular, taking $\ol{q}=q_0$ we find that
\begin{align}\label{eq:ss12:r_flow_q0}
r_X(t,q_0)=t+r_0,\quad t\in D_{\lr(X)}(q_0)
\end{align}
and, with $\hr(t)$ as in \eqref{eq:ss12:dot_hr},
\begin{align}\label{eq:ss12:hr_flow_q0}
\hr_X(t,q_0)=\hr(t),\quad t\in D_{\lr(X)}(q_0),
\end{align}
because $P(q_0)=P_0$ and hence $\hr_X(t,q_0)$ and $\hr(t)$ solve the same initial value problem.

Let $\rho:Q\to\R$ and $\hat{\rho}:Q\to\R$ be the functions
$\rho(q)=r$, $\hat{\rho}(q)=\hr$,
when $\q\in Q$ with $x=(r,y)$, $\hx=(\hr,\hy)$.
One notes that
$\lr(X)\onq \rho=\hg(X,\pa{r})\p{\rho}{r}=1$,
$\lr(X)\onq \hat{\rho}=\hg(AX,\pa{\hr})\p{\hat{\rho}}{\hr}=P(q)$
and
$\lr(Y)\onq \rho=\hg(Y,\pa{r})\p{\rho}{r}=0$,
$\lr(Y)\onq \hat{\rho}=\hg(AY,\pa{\hr})\p{\hat{\rho}}{\hr}=I(q)$ hold, and
thus
\[
{}& \pa{t}\rho(q_Y(t,\ol{q}))=\lr(Y)|_{q_Y(t,\ol{q})} \rho=0 \\
{}& \pa{t}\hat{\rho}(q_Y(t,\ol{q}))=\lr(Y)|_{q_Y(t,\ol{q})} \hat{\rho}=I(q_Y(t,\ol{q})).
\]
Because $I(q_Y(t,\ol{q}))=0$ if $\ol{q}\in C$ (see \eqref{eq:ss12:C_invariant_along_Y})
we thus see that
\begin{align}\label{eq:ss12:rho_hrho_invariant_along_Y}
\rho(q_Y(t,\ol{q}))=\rho(\ol{q}),
\quad \hat{\rho}(q_Y(t,\ol{q}))=\hat{\rho}(\ol{q}),
\quad\forall \ol{q}\in C.
\end{align}

Write \eqref{eq:ss12:r_flow} and \eqref{eq:ss12:geodesic} as
\[
{}& \rho(q_X(t,\ol{q}))=t+\rho(\ol{q}),\quad \ol{q}\in Q, \\
{}& \big(\pa{t}\hat{\rho}(q_X(t,\ol{q}))\big)^2+(1-P(\ol{q})^2) \frac{\hf(\hat{\rho}(\ol{q}))^2}{\hf\big(\hat{\rho}(q_X(t,\ol{q}))\big)^2}=1,
\quad \hat{\rho}(q_X(0,\ol{q}))=\hat{\rho}(\ol{q})
\]

Then let $\ol{q}\in C$ (see \eqref{eq:ss12:set_C}).
Substitute $q_Y(s,\ol{q})$ for $\ol{q}$ in the previous equation,
and notice by \eqref{eq:ss12:rho_hrho_invariant_along_Y} that $\rho(q_Y(s,\ol{q}))=\rho(\ol{q})$,
$\hat{\rho}(q_Y(s,\ol{q}))=\hat{\rho}(\ol{q})$,
to find
\[
{}& \rho(q_X(t,q_Y(s,\ol{q})))=t+\rho(\ol{q}),\quad \ol{q}\in C, \\
{}& \big(\pa{t}\hat{\rho}(q_X(t,q_Y(s,\ol{q})))\big)^2+(1-P(q_Y(s,\ol{q}))^2) \frac{\hf(\hat{\rho}(q_Y(s,\ol{q})))^2}{\hf\big(\hat{\rho}(q_X(t,\ol{q}))\big)^2}=1,
\quad \hat{\rho}(q_X(0,q_Y(s,\ol{q})))=\hat{\rho}(\ol{q}).
\]

The first relations of the previous two equations immediately yield
\begin{align}\label{eq:ss12:rho_X_Y}
\rho(q_X(t,q_Y(s,\ol{q})))=\rho(q_X(t,\ol{q})),
\end{align}
for all $t,s$ for which the left hand side is defined, and for all $\ol{q}\in C$.

On the other hand, since $\ol{q}\in C$, the curve $q_Y(s,\ol{q})$ remains in $C$ for all $s$,
according to \eqref{eq:ss12:C_invariant_along_Y},
while the function $P$ remains constant along $q_Y(s,\ol{q})$ by \eqref{eq:ss12:P_invariant_along_Y},
that is $P(q_Y(s,\ol{q}))=P(\ol{q})$.
One thus sees that $t\mapsto \hat{\rho}(q_X(t,q_Y(s,\ol{q})))$ and $t\mapsto \hat{\rho}(q_X(t,\ol{q}))$
solve, for any fixed $s$, the same initial value problem,
and hence
\begin{align}\label{eq:ss12:hrho_X_Y}
\hat{\rho}(q_X(t,q_Y(s,\ol{q})))=\hat{\rho}(q_X(t,\ol{q})),
\end{align}
for all $t,s$ for which the left hand side is defined, and for all $\ol{q}\in C$.

Given
\eqref{eq:ss12:rho_X_Y} and \eqref{eq:ss12:hrho_X_Y},
the group property $q_X(t,q_X(s,\ol{q}))=q_X(s+t,\ol{q})$,
and observing that $\rho(q_X(t,\ol{q}))=r_X(t,\ol{q})$,
$\hat{\rho}(q_X(t,\ol{q}))=\hr_X(t,\ol{q})$,
one finds that
\begin{align}\label{eq:ss12:r_hr_additive}
{}& r_X(t_n,q_Y(s_n,q_X(t_{n-1},q_Y(s_{n-1},\dots,q_X(t_1,q_Y(s_1, \ol{q}))\dots)
=r_X(t_n+t_{n-1}+\dots+t_1,\ol{q}), \nonumber \\
{}& \hr_X(t_n,q_Y(s_n,q_X(t_{n-1},q_Y(s_{n-1},\dots,q_X(t_1,q_Y(s_1, \ol{q}))\dots)
=\hr_X(t_n+t_{n-1}+\dots+t_1,\ol{q}),
\end{align}
for any $\ol{q}\in C$ and any $t_1,\dots,t_n,s_1,\dots,s_n\in\R$ for which the left hand sides are defined.
In particular, these hold for $\ol{q}=q_0$ since $q_0\in C$.

Any point $q$ in the orbit $\odr(q_0)$ can be reached from $q_0$
by following the flows of $\lr(X)$ and $\lr(Y)$ starting from $q_0$,
that is there are
$t_1,\dots,t_n,s_1,\dots,s_n\in\R$ such that
\begin{align}\label{eq:ss12:q_in_orbit}
q=q_X(t_n,q_Y(s_n,q_X(t_{n-1},q_Y(s_{n-1},\dots,q_X(t_1,q_Y(s_1,q_0))\dots).
\end{align}
From \eqref{eq:ss12:warping_functions:3}, \eqref{eq:ss12:r_flow_q0}, \eqref{eq:ss12:hr_flow_q0}
and \eqref{eq:ss12:r_hr_additive} (with $\ol{q}=q_0\in C$)
it thus follows that
\[
\frac{\hf''(\pi_{Q,\hM}(q))}{\hf(\pi_{Q,\hM}(q))}
={}&\frac{\hf''(\hr_X(t_n+t_{n-1}+\dots+t_1,q_0))}{\hf(\hr_X(t_n+t_{n-1}+\dots+t_1,q_0))}
=\frac{\hf''(\hr(t_n+t_{n-1}+\dots+t_1))}{\hf(\hr(t_n+t_{n-1}+\dots+t_1))} \\
={}& \frac{f''(t_n+t_{n-1}+\dots+t_1+r_0)}{f(t_n+t_{n-1}+\dots+t_1+r_0)}
=\frac{f''(r_X(t_n+t_{n-1}+\dots+t_1,q_0))}{f(r_X(t_n+t_{n-1}+\dots+t_1,q_0))} \\
={}&\frac{f''(\pi_{Q,M}(q))}{f(\pi_{Q,M}(q))},
\]
where, as usual, we have identified $f$, $\hf$ as functions $M\to\R$, $\hM\to\R$, respectively,
that is as $f\circ\rho$, $\hf\circ\hat{\rho}$.
This being a key result that we will need right below, let us display it separately and a bit differently
\begin{align}\label{eq:ss12:warping_functions:4}
\frac{\hf''(\hr)}{\hf(\hr)}
=
\frac{f''(r)}{f(r)},\quad \forall \q\in \odr(q_0);\ x=(r,y),\ \hx=(\hr,\hy).
\end{align}

Knowing \eqref{eq:ss12:warping_functions:4},
the differential relation \eqref{eq:ss12:lrX_U} simplifies, at points $q$ of the orbit $\odr(q_0)$, to
\begin{align}\label{eq:ss12:lrX_U:2}
\lr(X)\onq U(\cdot)
=\Big(-\frac{f'(r)}{f(r)} - \frac{\hf'(\hr)}{\hf(\hr)} P(q) \Big)U(q),
\quad \forall q\in \odr(q_0).
\end{align}

Now of $\ol{q}\in \odr(q_0)\cap C$, the relations \eqref{eq:ss12:lrX_I} and \eqref{eq:ss12:lrX_U:2}
imply that $U=0$ and $I=0$ along the integral curve $q_X(t,\ol{q})$ of $\lr(X)$ as well,
that is
\begin{align}\label{eq:ss12:C_invariant_along_X}
q_X(t,\ol{q})\in C,\quad \forall \ol{q}\in \odr(q_0)\cap C,\ t\in D_{\lr(X)}(\ol{q}).
\end{align}

By definition of the orbit $\odr(q_0)$,
the integral curves $q_X(t,\ol{q})$ and $q_Y(t,\ol{q})$
of $\lr(X)$ and $\lr(Y)$, respectively, passing through a point $\ol{q}$ of the orbit $\odr(q_0)$
will remain in the orbit $\odr(q_0)$.
Therefore, from \eqref{eq:ss12:C_invariant_along_Y}
and \eqref{eq:ss12:C_invariant_along_X} we deduce that
\[
{}& q_X(t,\ol{q})\in C\cap \odr(q_0), \quad \forall \ol{q}\in C\cap \odr(q_0),\ t\in D_{\lr(X)}(\ol{q}) \\
{}& q_Y(t,\ol{q})\in C\cap \odr(q_0), \quad \forall \ol{q}\in C\cap \odr(q_0),\ t\in D_{\lr(Y)}(\ol{q}).
\]
An arbitrary point $q$ of the orbit $\odr(q_0)$
can be represented in the form \eqref{eq:ss12:q_in_orbit},
and since $q_0\in C\cap \odr(q_0)$,
we infer that
$q_Y(s_1,q_0)\in C\cap \odr(q_0)$, hence $q_X(t_1,q_Y(s_1,q_0))\in C\cap \odr(q_0)$
and so on, by induction,
$q\in C\cap \odr(q_0)$.
This proves that
\begin{align}\label{eq:ss12:main:1}
\odr(q_0)\subset C,
\end{align}
that is (see \eqref{eq:ss12:set_C}) the orbit $\odr(q_0)$ is contained inside the $(0,0)$-level set of the function $Q\to\R^2$; $q\mapsto (U(q),I(q))$.

To conclude this section, it remains to show that after cutting out a certain part of $C$,
the set that remains (the set $C_1$ below) contains $\odr(q_0)$ and it is a smooth submanifold of $Q$ of dimension $6$.
This then implies that $\odr(q_0)$ as a submanifold of $C$ has dimension at most $6$,
which is the claim this entire section had been dedicated to prove.

At any $\q\in Q$ with $x=(r,y)$, $\hx=(\hr,\hy)$, we have the following $\pi_Q$-vertical derivatives of $U$ and $I$
(recalling that $X=\pa{r}$)
\begin{align}\label{eq:ss12:piQ_vert_U_I}
\nu(\theta_X\otimes\hZ_A)\onq U(\cdot)
={}& -\frac{f'(r)}{f(r)}\hg\big((\theta_X\otimes\hZ_A)\pa{r},\pa{\hr}\big)
=-\frac{f'(r)}{f(r)}\hg\big(\hZ_A,\pa{\hr}\big) \nonumber \\
\nu(\theta_Y\otimes\hZ_A)\onq U(\cdot)
={}& -\frac{f'(r)}{f(r)}\hg\big((\theta_Y\otimes\hZ_A)\pa{r},\pa{\hr}\big)
=0 \nonumber \\
\nu(\theta_X\otimes\hZ_A)\onq I(\cdot)
={}& \hg\big((\theta_X\otimes\hZ_A)Y,\pa{\hr}\big)
=0 \nonumber \\
\nu(\theta_Y\otimes\hZ_A)\onq I(\cdot)
={}& \hg\big((\theta_Y\otimes\hZ_A)Y,\pa{\hr}\big)
=\hg\big(\hZ_A,\pa{\hr}\big).
\end{align}

The assumption \eqref{eq:ss12:dot_hf_neq_0},
the result \eqref{eq:ss12:main:1}
and the definitions \eqref{eq:ss12:U_I}, \eqref{eq:ss12:set_C}
together imply that
$\frac{f'(r)}{f(r)}\hg(A \pa{r}\big|_{r}, \pa{\hr}\big|_{\hr}\big)=\frac{\hf'(\hr)}{\hf(\hr)}\neq 0$
and hence $\frac{f'(r)}{f(r)}\neq 0$
at every $\q\in C$ with $x=(r,y)$, $\hx=(\hr,\hy)$.
Moreover, to obtain a useful expression (up to the sign) for the
inner product $\hg\big(\hZ_A,\pa{\hr}\big)$,
note that $0=I(q)=\hg\big(AY|_x,\pa{\hr}|_{\hr}\big)$ at $\q\in C$,
and
\[
\pa{\hr}={}& \hg\big(\pa{\hr},AX\big)AX + \hg\big(\pa{\hr},AY\big)AY + \hg\big(\pa{\hr},\hZ_A\big)\hZ_A \\
={}& P(q) AX + \hg\big(\pa{\hr},\hZ_A\big)\hZ_A,
\]
so that by taking a $\hg$-inner product of this relation with $\pa{\hr}$
we find $1=P(q)^2+\hg\big(\pa{\hr},\hZ_A\big)^2$,
or equivalently,
\[
\hg\big(\hZ_A,\pa{\hr}\big)=\pm (1-P(q)^2)^{1/2},
\]
at every point $\q\in C$.

%
The Cauchy-Schwarz inequality implies that
$|P(q)|\leq \n{A\pa{r}}_{\hg} \n{\pa{\hr}}_{\hg}=1$
for all $q\in Q$.
Therefore having $\hg\big(\hZ_A,\pa{\hr}\big)\neq 0$ at $\q\in Q$
is equivalent to having $|P(q)|<1$ at $q\in Q$.

Define $C_0=\{q\in C\ |\ P(q)^2=1\}$.
If $\ol{q}\in C_0$ then it follows from \eqref{eq:ss12:lrX_P}
that $P(q_X(t,\ol{q}))^2=1$ for all $t$,
and from \eqref{eq:ss12:P_invariant_along_Y}
that $P(q_Y(t,\ol{q}))^2=1$ for all $t$.
In other words, if $\ol{q}\in C_0$, then $q_X(t,\ol{q})\in C_0$ and $q_Y(t,\ol{q})\in C_0$ for all $t$.
Consequently, if $\ol{q}\in C_0$, then $\odr(\ol{q})\subset C_0$.

The set
\[
Q_1:=\{q\in Q\ |\ P(q)^2<1\}
\]
is an open subset of $Q$, disjoint with $C_0$,
and it contains the point $q_0$
because $|P(q_0)|=|P_0|<1$ by assumption \eqref{eq:ss12:P0_assumption}.
Moreover, since $q_0\notin C_0$ it follows by what we have just shown that $\odr(q_0)\cap C_0=\emptyset$,
because otherwise $\ol{q}\in \odr(q_0)\cap C_0$ would imply that $q_0\in \odr(q_0)=\odr(\ol{q})\subset C_0$,
which is a contradiction.
Thus we have
\[
\odr(q_0)\subset C\setminus C_0=\{q\in Q\ |\ U(q)=0,\ I(q)=0,\ P(q)^2<1\}
\]
i.e.,
\[
\odr(q_0)\subset \{q\in Q_1\ |\ U(q)=0,\ I(q)=0\}=:C_1.
\]

At every point of $\q\in C_1$ we have $|P(q)|<1$ and hence $\hg(\hZ_A,\pa{\hr})\neq 0$.
The differentials of $U$ and $I$ in \eqref{eq:ss12:piQ_vert_U_I} are therefore linearly independent on $C_1$,
implying that the map $Q_1\to\R^2$; $q\mapsto (U(q),I(q))$ has rank $2$ at every point $q\in C_1$.
Thus $C_1$ is a (closed) submanifold of $Q_1$ of dimension $\dim Q_1-2=6$.
Finally, because $\odr(q_0)\subset C_1$,
we can conclude that $\odr(q_0)$ has dimension $\leq 6$.

\vspace{2\baselineskip}\section{Appendix}\label{s2.1}\ \newline

If $E_1,\dots, E_n$ is a local $g$-orthonormal frame on an $n$-dimensional Riemannian manifold $(M ,g)$,
and $\nabla$ is the Levi-Civita connection of $g$,
one defines the connection coefficients of $\nabla$ as
\[
\Gamma_{(j,k)}^i : = g (\nabla_{E_i} E_j , E_k),
\]
and notices that $\Gamma_{(j,k)}^i = - \Gamma_{(k,j)}^i$ for all $i,j,k$.
We call connection table $\Gamma$ the array whose elements are $\Gamma_{(j,k)}^i$, for all $i,j,k$ with $j\neq k$.

When the dimension $n=2$, the connection table has the form
\begin{eqnarray*}\label{eq10}
\left(
\begin{array}{cc}
\Gamma^1_{(1,2)} & \Gamma^2_{(1,2)}
\end{array}
\right).
\end{eqnarray*}
Similarly, if we have a 3-dimensional Riemannian manifold $\hM$, its connection table with respect to a local 
orthonormal frame $\hE_1, \hE_2, \hE_3$ is given by
\[
\hGamma =
\left(
\begin{array}{ccc}
\hGamma_{(2,3)}^1 & \hGamma_{(2,3)}^2 & \hGamma_{(2,3)}^3\\
\hGamma_{(3,1)}^1 & \hGamma_{(3,1)}^2 & \hGamma_{(3,1)}^3 \\
\hGamma_{(1,2)}^1 & \hGamma_{(1,2)}^2 & \hGamma_{(1,2)}^3
\end{array}
\right).
\]

If this 3-dimensional $\hM$ has Riemannian curvature tensor $\hR$,
then one can write the components of $\hR$ in terms of the connection coefficients $\hGamma$ of $\hM$
in the following way: If
\[
\hR (\hE_1 , \hE_2)=
\left(
\begin{array}{cc}
a_1\\
b_1\\
c_1
\end{array}
\right),\quad
\hR (\hE_2 , \hE_3)=
\left(
\begin{array}{cc}
a_2\\
b_2\\
c_2
\end{array}
\right),\quad
\hR (\hE_3 , \hE_1)=
\left(
\begin{array}{cc}
a_3\\
b_3\\
c_3
\end{array}
\right),
\]
with respect to an orthonormal frame $\hE_1, \hE_2, \hE_3$, then
{\footnotesize
\begin{align}\label{eq:app:curvature}
		a_1 ={}& \hE_1 (\hGamma_{(2,3)}^2) - \hE_2 (\hGamma_{(2,3)}^1) +  (\hGamma_{(3,1)}^1 + \hGamma_{(2,3)}^2)  \hGamma_{(1,2)}^2 + (- \hGamma_{(3,1)}^2 - \hGamma_{(2,3)}^1) \hGamma_{(2,3)}^3 + (\hGamma_{(2,3)}^1 - \hGamma_{(3,1)}^2) \hGamma_{(1,2)}^1\nonumber\\[2mm]
		b_1 ={}& \hE_1 (\hGamma_{(3,1)}^2) - \hE_2 (\hGamma_{(3,1)}^1) +  (\hGamma_{(3,1)}^1 + \hGamma_{(2,3)}^2)  \hGamma_{(1,2)}^1 + (- \hGamma_{(3,1)}^2 - \hGamma_{(2,3)}^1) \hGamma_{(3,1)}^3 + (\hGamma_{(3,1)}^2 - \hGamma_{(2,3)}^1)  \hGamma_{(1,2)}^2\nonumber\\[2mm]
		c_1 ={}& \hE_1 (\hGamma_{(1,2)}^2) - \hE_2 (\hGamma_{(1,2)}^1) + (\hGamma_{(1,2)}^1)^2 + (- \hGamma_{(2,3)}^1 - \hGamma_{(3,1)}^2) \hGamma_{(1,2)}^3 + \hGamma_{(2,3)}^1 \hGamma_{(3,1)}^2 - \hGamma_{(3,1)}^1 \hGamma_{(2,3)}^2 + (\hGamma_{(1,2)}^2)^2\nonumber\\[4mm]
        a_2 ={}& \hE_2 (\hGamma_{(2,3)}^3) - \hE_3 (\hGamma_{(2,3)}^2) + (\hGamma_{(2,3)}^2)^2 + (- \hGamma_{(3,1)}^2 - \hGamma_{(1,2)}^3) \hGamma_{(2,3)}^1 + \hGamma_{(3,1)}^2 \hGamma_{(1,2)}^3 - \hGamma_{(1,2)}^2 \hGamma_{(3,1)}^3 + (\hGamma_{(2,3)}^3)^2\nonumber\\[2mm]
        b_2 ={}& \hE_2 (\hGamma_{(3,1)}^3) - \hE_3 (\hGamma_{(3,1)}^2) +  (\hGamma_{(3,1)}^3 + \hGamma_{(1,2)}^2) \hGamma_{(2,3)}^3 + (- \hGamma_{(3,1)}^2 - \hGamma_{(1,2)}^3) \hGamma_{(3,1)}^1 + (\hGamma_{(3,1)}^2 - \hGamma_{(1,2)}^3)  \hGamma_{(2,3)}^2\\[2mm]
        c_2 ={}& \hE_2 (\hGamma_{(1,2)}^3) - \hE_3 (\hGamma_{(1,2)}^2) +  (\hGamma_{(3,1)}^3 + \hGamma_{(1,2)}^2)  \hGamma_{(2,3)}^2  + (- \hGamma_{(3,1)}^2 - \hGamma_{(1,2)}^3) \hGamma_{(1,2)}^1 + (\hGamma_{(1,2)}^3 - \hGamma_{(3,1)}^2) \hGamma_{(2,3)}^3\nonumber\\[4mm]
		a_3 ={}& \hE_3 (\hGamma_{(2,3)}^1) - \hE_1 (\hGamma_{(2,3)}^3) + (\hGamma_{(2,3)}^3 + \hGamma_{(1,2)}^1)  \hGamma_{(3,1)}^3 + (- \hGamma_{(1,2)}^3 - \hGamma_{(2,3)}^1) \hGamma_{(2,3)}^2 +  (\hGamma_{(2,3)}^1 - \hGamma_{(1,2)}^3) \hGamma_{(3,1)}^1\nonumber\\[2mm]
		b_3 ={}& \hE_3 (\hGamma_{(3,1)}^1) - \hE_1 (\hGamma_{(3,1)}^3) + (\hGamma_{(3,1)}^1)^2 + (- \hGamma_{(1,2)}^3 - \hGamma_{(2,3)}^1) \hGamma_{(3,1)}^2 + \hGamma_{(2,3)}^1 \hGamma_{(1,2)}^3 - \hGamma_{(2,3)}^3 \hGamma_{(1,2)}^1 + (\hGamma_{(3,1)}^3)^2\nonumber\\[2mm]
		c_3 ={}& \hE_3 (\hGamma_{(1,2)}^1) - \hE_1 (\hGamma_{(1,2)}^3) +  (\hGamma_{(1,2)}^1 + \hGamma_{(2,3)}^3)  \hGamma_{(3,1)}^1+ (- \hGamma_{(1,2)}^3 - \hGamma_{(2,3)}^1) \hGamma_{(1,2)}^2 + (\hGamma_{(1,2)}^3 - \hGamma_{(2,3)}^1) \hGamma_{(3,1)}^3 \nonumber
\end{align}
}

The following lemma lists formulas of Proposition \ref{p1.7} worked out in several special cases
that will be used throughout this paper.

\begin{lemma}\label{l2.1}
On points $\q$ of $Q$ under which the orthonormal frame $X,Y$ of $M$ is defined,
and with $\hZ_A$, $\theta_X$, $\theta_Y$ as well as other quantities appearing as defined in Section \ref{sec:preliminaries},
we have
	\begin{align*}
	&[\lr(X),\lr(Y)]\onq = \lr([X,Y])\onq + \nu(\Rol_q)\onq \\
	&[\lr (X) , \nu ((\cdot) (X \wedge Y))] \onq  =  - \lns (A Y) \onq,\\
	&[\lr (X) , \nu (\theta_{X} \otimes \hZ_{(\cdot)})] \onq  =  - \lns (\hZ_A) \onq + \Gamma_{(1,2)}^1 \nu (\theta_Y \otimes \hZ_A) \onq,\\
	&[\lr (X) , \nu (\theta_{Y} \otimes \hZ_{(\cdot)})] \onq  = -  \Gamma_{(1,2)}^1 \nu (\theta_X \otimes \hZ_A) \onq,\\
	&[\lr (X) , \lns ((\cdot ) X)] \onq  =  \Gamma_{(1,2)}^1 \lns (AY) \onq,\\
	&[\lr (X) , \lns ((\cdot ) Y)] \onq  = - \Gamma_{(1,2)}^1 \lns (AX) \onq + \nu \big( \hsigma_A A (X \wedge Y) + \Pi_Y \theta_X \otimes \hZ_A - \Pi_X \theta_Y \otimes \hZ_A\big)\onq,\\
	&[\lr (X) , \lns (\hZ_A)] \onq  = \nu \big(\Pi_Y A (X \wedge Y) + \hsigma_A^2 \theta_X \otimes \hZ_A + \Pi_Z \theta_Y \otimes \hZ_A\big)\onq,
	\end{align*}
	\begin{align*}
	&[\lr (Y) , \nu ((\cdot) (X \wedge Y))] \onq  =   \lns (A X) \onq,\\
	&[\lr (Y) , \nu (\theta_{X} \otimes \hZ_{(\cdot)})] \onq  =  \Gamma_{(1,2)}^2 \nu (\theta_Y \otimes \hZ_A) \onq,\\
	&[\lr (Y) , \nu (\theta_{Y} \otimes \hZ_{(\cdot)})] \onq  =  - \lns (\hZ_A) \onq - \Gamma_{(1,2)}^2 \nu (\theta_X \otimes \hZ_A) \onq,\\
	&[\lr (Y) , \lns ((\cdot ) X)] \onq  =  \Gamma_{(1,2)}^2 \lns (AY) \onq - \nu \big(  \hsigma_A A (X \wedge Y) + \Pi_Y \theta_X \otimes \hZ_A - \Pi_X \theta_Y \otimes \hZ_A\big)\onq,\\
	&[\lr (Y) , \lns ((\cdot ) Y)] \onq  =  - \Gamma_{(1,2)}^2 \lns (AX) \onq,\\
	&[\lr (Y) , \lns (\hZ_A)] \onq  = \nu \big(- \Pi_X A (X \wedge Y) + \Pi_Z \theta_X \otimes \hZ_A + \hsigma_A^1 \theta_Y \otimes \hZ_A\big)\onq,
	\end{align*}
	\begin{align*}
	&[\lns (\hZ_{(\cdot)}), \lns ((\cdot) X)] \onq = - \nu \big(\Pi_Y A ( X \wedge Y) + \hsigma_A^2 \theta_X \otimes \hZ_A + \Pi_Z \theta_Y \otimes \hZ_A \big) \onq,\\
	&[\lns (\hZ_{(\cdot)}), \lns ((\cdot) Y)] \onq = - \nu \big( - \Pi_X A ( X \wedge Y) + \Pi_Z \theta_X \otimes \hZ_A + \hsigma_A^1 \theta_Y \otimes \hZ_A \big) \onq,\\
	&[\lns (\hZ_{(\cdot)}) , \nu ((\cdot) X \wedge Y)] \onq = 0,\\
	&[\lns (\hZ_{(\cdot)}) , \nu (\theta_X \otimes \hZ_{(\cdot)})] \onq = \lns (A X) \onq,\\
	&[\lns (\hZ_{(\cdot)}) , \nu (\theta_Y \otimes \hZ_{(\cdot)})] \onq = \lns (A Y) \onq,
	\end{align*}
	\begin{align*}
	&[\lns ((\cdot)X) , \lns ((\cdot)Y)] \onq = \nu \big(\hsigma_A A (X \wedge Y) + \Pi_Y \theta_X \otimes \hZ_A - \Pi_X \theta_Y \otimes \hZ_A\big) \onq,\\
	&[\lns ((\cdot) X) , \nu ((\cdot) X \wedge Y)] \onq = - \lns (AY) \onq,\\
	&[\lns ((\cdot) X) , \nu (\theta_X \otimes \hZ_{(\cdot)})] \onq = - \lns (\hZ_A) \onq,\\
	&[\lns ((\cdot) X) , \nu (\theta_Y \otimes \hZ_{(\cdot)})] \onq = 0,
	\end{align*}
      \begin{align*}
	\begin{array}{ll}
	&[\lns ((\cdot) Y) , \nu ((\cdot) X \wedge Y)] \onq =  \lns (AX) \onq,\\
	&[\lns ((\cdot) Y) , \nu (\theta_X \otimes \hZ_{(\cdot)})] \onq = 0,\\
	&[\lns ((\cdot) Y) , \nu (\theta_Y \otimes \hZ_{(\cdot)})] \onq = - \lns (\hZ_A) \onq,
	\end{array}\qquad 
	\begin{array}{ll}
	&[\nu ((\cdot) X \wedge Y) , \nu (\theta_X \otimes \hZ_{(\cdot)})] \onq = \nu (\theta_Y \otimes \hZ_A) \onq,\\
	&[\nu ((\cdot) X \wedge Y) , \nu (\theta_Y \otimes \hZ_{(\cdot)})] \onq = - \nu (\theta_X \otimes \hZ_A) \onq,\\
	&[\nu (\theta_X \otimes \hZ_{(\cdot)}) , \nu (\theta_Y \otimes \hZ_{(\cdot)})] \onq = \nu (A(X \wedge Y)) \onq.
	\end{array}
	\end{align*}
\end{lemma}

Next lemma records the derivatives of different curvature related quantities (see Section \ref{sec:preliminaries})
with respect to a basis of $\pi_Q$-vertical vector fields of $Q$.

\begin{lemma}\label{le:6.3}
For any $\q$ in $Q$, we have
\[
\begin{array}{lll}
	\nu (A ( X \wedge Y)) \onq \hsigma_{(\cdot)}^1 = - 2 \Pi_Z, & \nu (\theta_X \otimes \hZ_A) \onq \hsigma_{(\cdot)}^1 = - 2 \Pi_X , & \nu (\theta_Y \otimes \hZ_A) \onq \hsigma_{(\cdot)}^1 = 0,\\
	\nu (A ( X \wedge Y)) \onq \hsigma_{(\cdot)}^2 = 2 \Pi_Z, & \nu (\theta_X \otimes \hZ_A) \onq \hsigma_{(\cdot)}^2 = 0 , & \nu (\theta_Y \otimes \hZ_A) \onq \hsigma_{(\cdot)}^2 = - 2 \Pi_Y,\\
	\nu (A ( X \wedge Y)) \onq \hsigma_{(\cdot)} = 0, & \nu (\theta_X \otimes \hZ_A) \onq \hsigma_{(\cdot)} = 2 \Pi_X , & \nu (\theta_Y \otimes \hZ_A) \onq \hsigma_{(\cdot)} = 2 \Pi_Y,\\
	\nu (A ( X \wedge Y)) \onq \Pi_X = \Pi_Y, & \nu (\theta_X \otimes \hZ_A) \onq \Pi_X = \hsigma_A^1 - \hsigma_A , & \nu (\theta_Y \otimes \hZ_A) \onq \Pi_X = - \Pi_Z,\\
	\nu (A ( X \wedge Y)) \onq \Pi_Y = - \Pi_X , & \nu (\theta_X \otimes \hZ_A) \onq \Pi_Y = - \Pi_Z , & \nu (\theta_Y \otimes \hZ_A) \onq \Pi_Y = \hsigma_A^2 - \hsigma_A,\\
	\nu (A ( X \wedge Y)) \onq \Pi_Z = \hsigma_A^1 - \hsigma_A^2 , & \nu (\theta_X \otimes \hZ_A) \onq \Pi_Z = \Pi_Y , & \nu (\theta_Y \otimes \hZ_A) \onq \Pi_Z = \Pi_X.\\
	\end{array}
\]
\end{lemma}

\begin{proof}
These relations are immediate consequences of the definitions of the curvature quantities $\Pi_X,\Pi_Y,\hsigma^1_A$ etc. (see section \ref{sec:preliminaries}),
and the formulas for the derivatives of $AX$, $AY$ and $\hZ_A$ with respect to
the vertical vector fields appearing. The latter are given by
\[
\begin{array}{lll}
\nu (A ( X \wedge Y)) \onq \big((\cdot) X\big) = AY, & \nu (\theta_X \otimes \hZ_A) \onq \big((\cdot) X\big) = \hZ_A , & \nu (\theta_Y \otimes \hZ_A) \onq \big((\cdot) X\big) = 0, \\
\nu (A ( X \wedge Y)) \onq \big((\cdot) Y\big) = - AX, & \nu (\theta_X \otimes \hZ_A) \onq \big((\cdot) Y\big) = 0 , & \nu (\theta_Y \otimes \hZ_A) \onq \big((\cdot) Y\big) = \hZ_A, \\
\nu (A ( X \wedge Y)) \onq \hZ_{(\cdot)} = 0, & \nu (\theta_X \otimes \hZ_A) \onq \hZ_{(\cdot)} = - AX , & \nu (\theta_Y \otimes \hZ_A) \onq \hZ_{(\cdot)} = - AY.
\end{array}
\]
\end{proof}

Concerning some of the derivatives of the $T\hM$-valued vector fields $\Xtilde_A,\Ytilde_A$ on $Q$ as
defined in \eqref{e2.10},
we have the following.

\begin{lemma}\label{le:LR_Xtilde_Ytilde}
The following relations hold on an open neighbourhood of $\qz$ in $Q$:
\[
\lr(\Xtilde_A)|_q\Xtilde
={}& (\lr(\Xtilde_A)|_q\phi + g(\Gamma,\Xtilde_A))\Ytilde_A
\
\\
\lr(\Xtilde_A)|_q\Ytilde
={}&-(\lr(\Xtilde_A)|_q\phi + g(\Gamma,\Xtilde_A))\Xtilde_A \\
\lr(\Ytilde_A)|_q\Xtilde
={}&
(\lr(\Ytilde_A)|_q\phi+g(\Gamma,\Ytilde_A)) \Ytilde_A \\
\lr(\Ytilde_A)|_q\Ytilde
={}&
-(\lr(\Ytilde_A)|_q\phi+g(\Gamma,\Ytilde_A)) \Xtilde_A.
\]
\end{lemma}

\begin{proof}
It is enough to show one of these relations, since the others follow from similar computations.
We have by \eqref{e2.10}, the orthonormality of $X,Y$
and the definitions of $\Gamma^i_{(j,k)}$ and $\Gamma$ in section \ref{sec:preliminaries}
(writing $c_\phi=\cos \phi$, $s_\phi=\sin\phi$)
\[
\lr(\Xtilde_A)|_q\Xtilde
={}&(\lr(\Xtilde_A)\onq \phi)(-s_\phi X+c_\phi Y)
+c_\phi \nabla_{\Xtilde_A} X+s_\phi \nabla_{\Xtilde} Y \\
={}&
(\lr(\Xtilde_A)\onq \phi) \Ytilde_A
+c_\phi(c_\phi \nabla_X X+s_\phi \nabla_Y X)
+s_\phi(c_\phi \nabla_X Y+s_\phi \nabla_Y Y) \\
={}&
(\lr(\Xtilde_A)\onq \phi) \Ytilde_A
+c_\phi(c_\phi \Gamma^1_{(1,2)} Y+s_\phi \Gamma^2_{(1,2)} Y)
+s_\phi(-c_\phi \Gamma^1_{(1,2)} X-s_\phi \Gamma^2_{(1,2)} X) \\
={}&
(\lr(\Xtilde_A)\onq \phi) \Ytilde_A
+\Gamma^1_{(1,2)} c_\phi (-s_\phi X + c_\phi Y)
+\Gamma^2_{(1,2)} s_\phi (-s_\phi X+c_\phi Y) \\
={}&
(\lr(\Xtilde_A)\onq \phi) \Ytilde_A
+(\Gamma^1_{(1,2)} c_\phi +\Gamma^2_{(1,2)} s_\phi)\Ytilde_A.
\]
The proof is therefore complete once we notice that
$g(\Gamma,\Xtilde_A)=g(\Gamma^1_{(1,2)}X+\Gamma^2_{(1,2)} Y,\Xtilde_A)=\Gamma^1_{(1,2)} c_\phi + \Gamma^2_{(1,2)} s_\phi$.
\end{proof}

Similarly to Lemma \ref{l2.1}, but with $\Xtilde_A,\Ytilde_A$ playing the roles of $X,Y$, we have
the following catalogue of Lie brackets that are being used at several points.

\begin{lemma}\label{l2.2}
On points $\q$ of $Q$ under which the orthonormal frame $X,Y$ of $M$ is defined,
and with the various quantities as defined in Sections \ref{sec:preliminaries} and \ref{ss1},
we have
\begin{align*}
&[\lr(\Xtilde),\lr(\Ytilde)]\onq
=-\big(\lr(\Xtilde_A)|_q\phi + g(\Gamma,\Xtilde_A)\big)\lr(\Xtilde_A)\onq - \big(\lr(\Ytilde_A)|_q\phi+g(\Gamma,\Ytilde_A)\big)\lr(\Ytilde_A)\onq\\
& \qquad \qquad \qquad \qquad \quad +\nu(\Rol_q)\onq,\\
&[\lr (\Xtilde) , \nu \big((\cdot) (X \wedge Y)\big)] \onq = \lr (\YtildeA) \onq - \lns (A \YtildeA) \onq,\\
&[\lr (\Xtilde) , \nu (\theta_{\Xtilde} \otimes \hZ)] \onq =  - (\nu (\theta_{\XtildeA} \otimes \hZ_A) \onq \phi) \lr (\YtildeA) \onq - \lns (\hZ_A) \onq\\
& \qquad \qquad \qquad \qquad \qquad +  (\lr (\XtildeA) \onq \phi + g (\Gamma, \XtildeA)) \nu (\theta_{\YtildeA} \otimes \hZ_A) \onq,\\
&[\lr (\Xtilde) , \nu (\theta_{\Ytilde} \otimes \hZ)] \onq =  - (\nu (\theta_{\YtildeA} \otimes \hZ_A) \onq \phi) \lr (\YtildeA) \onq - (\lr (\XtildeA) \onq \phi + g (\Gamma, \XtildeA)) \nu (\theta_{\XtildeA} \otimes \hZ_A) \onq,\\
&[\lr (\Xtilde) , \lns ((\cdot) \Xtilde)] \onq =  - ( \lns (A \XtildeA) \onq \phi ) \lr (\YtildeA) \onq + (\lr (\XtildeA) \onq \phi + g (\Gamma , \XtildeA)) \lns (A \YtildeA) \onq,\\
&[\lr (\Xtilde) , \lns ((\cdot) \Ytilde)] \onq =  - ( \lns (A \YtildeA) \onq \phi ) \lr (\YtildeA) \onq - (\lr (\XtildeA) \onq \phi + g (\Gamma , \XtildeA)) \lns (A \XtildeA) \onq\\
& \quad \quad \quad \quad \quad \quad \quad \quad \quad \quad \; \; + \nu (\Rol_q) \onq + K \nu (A(X \wedge Y) )\onq,\\
&[\lr (\Xtilde) , \lns (\hZ)] \onq = - ( \lns (\hZ_A) \onq \phi ) \lr (\YtildeA) \onq + \tilde{\hsigma}_A^2 \nu (\theta_{\XtildeA} \otimes \hZ_A) \onq + \tilde{\Pi}_{\hZ} \nu (\theta_{\YtildeA} \otimes \hZ_A) \onq,
\end{align*}

\begin{align*}
&[\lr (\Ytilde) , \nu ((\cdot) (X \wedge Y))] \onq = - \lr (\XtildeA) \onq + \lns (A \XtildeA) \onq,\\
&[\lr (\Ytilde) , \nu (\theta_{\Xtilde} \otimes \hZ)] \onq =  (\lr (\YtildeA) \onq \phi + g (\Gamma, \YtildeA)) \nu (\theta_{\YtildeA} \otimes \hZ_A) \onq + (\nu (\theta_{\XtildeA} \otimes \hZ_A) \onq \phi) \lr (\XtildeA) \onq,\\
&[\lr (\Ytilde) , \nu (\theta_{\Ytilde} \otimes \hZ)] \onq =  (\nu (\theta_{\YtildeA} \otimes \hZ_A) \onq \phi) \lr (\XtildeA) \onq  - \lns (\hZ_A) \onq - (\lr (\YtildeA) \onq \phi + g (\Gamma, \YtildeA)) \nu (\theta_{\XtildeA} \otimes \hZ_A) \onq,\\
&[\lr (\Ytilde) , \lns ((\cdot) \Xtilde)] \onq =( \lns (A \XtildeA) \onq \phi ) \lr (\XtildeA) \onq + (\lr (\YtildeA) \onq \phi + g (\Gamma , \YtildeA)) \lns (A \YtildeA) \onq\\
& \quad \quad \quad \quad \quad \quad \quad \quad \quad \quad \; \; - \nu (\Rol_q) \onq- K \nu (A(X \wedge Y) )\onq,\\
&[\lr (\Ytilde) , \lns ((\cdot) \Ytilde)] \onq =( \lns (A \YtildeA) \onq \phi ) \lr (\XtildeA) \onq - (\lr (\YtildeA) \onq \phi + g (\Gamma , \YtildeA)) \lns (A \XtildeA) \onq,\\
& [\lr (\Ytilde) , \lns (\hZ)] \onq= ( \lns (\hZ_A) \onq \phi ) \lr (\XtildeA) \onq + \tilde{\Pi}_{\hZ} \nu (\theta_{\XtildeA} \otimes \hZ_A) \onq + \tilde{\hsigma}_A^1 \nu (\theta_{\YtildeA} \otimes \hZ_A) \onq\\
&  \qquad \qquad \qquad \qquad \quad - \Pi_{\Xtilde} \nu (A(X \wedge Y)) \onq,
\end{align*}

\begin{align*}
& [\nu \big((\cdot) (X \wedge Y)\big) , \nu (\theta_{\Ytilde} \otimes \hZ)] \onq = 0,\quad  [\nu ((\cdot) X \wedge Y) , \nu (\theta_{\Xtilde} \otimes \hZ)] \onq = 0,\\
&[\nu \big((\cdot) (X \wedge Y)\big), \lns ((\cdot) \Xtilde)] \onq = 0, \quad [\nu ((\cdot) X \wedge Y), \lns ((\cdot) \Ytilde)] \onq  =  0,\quad [\nu ((\cdot) X \wedge Y), \lns (\hZ) ] \onq  =  0,\\
&[\nu (\theta_{\Xtilde} \otimes \hZ) , \nu (\theta_{\Ytilde} \otimes \hZ)] = - (\nu (\theta_{\XtildeA} \otimes \hZ_A) \onq \phi ) \nu (\theta_{\XtildeA} \otimes \hZ_A) \onq - (\nu (\theta_{\YtildeA} \otimes \hZ_A) \onq \phi ) \nu (\theta_{\YtildeA} \otimes \hZ_A) \onq\\
& \qquad \qquad \qquad \qquad \qquad \quad  + \nu (A(X \wedge Y)) \onq,\\
&[\nu (\theta_{\Xtilde} \otimes \hZ) , \lns ((\cdot) \Xtilde) ] \onq
= (\nu (\theta_{\XtildeA} \otimes \hZ_A) \onq \phi) \lns (A \YtildeA) \onq - (\lns (A \XtildeA) \onq \phi ) \nu (\theta_{\YtildeA} \otimes \hZ_A) \onq + \lns (\hZ_A) \onq,\\
& [\nu (\theta_{\Xtilde} \otimes \hZ) , \lns ((\cdot) \Ytilde) ] \onq =  - (\nu (\theta_{\XtildeA} \otimes \hZ_A) \onq \phi) \lns (A \XtildeA) \onq - (\lns (A \YtildeA) \onq \phi ) \nu (\theta_{\YtildeA} \otimes \hZ_A) \onq,\\
&[\nu (\theta_{\Xtilde} \otimes \hZ) , \lns (\hZ) ] \onq =  - (\lns (\hZ_A) \onq \phi) \nu (\theta_{\YtildeA} \otimes \hZ_A) \onq - \lns (A \XtildeA) \onq,\\
&[\nu (\theta_{\Ytilde} \otimes \hZ) , \lns ((\cdot) \Xtilde) ] \onq =  (\nu (\theta_{\YtildeA} \otimes \hZ_A) \onq \phi) \lns (A \YtildeA) \onq + (\lns (A \XtildeA) \onq \phi ) \nu (\theta_{\XtildeA} \otimes \hZ_A) \onq,\\
&[\nu (\theta_{\Ytilde} \otimes \hZ) , \lns ((\cdot) \Ytilde) ] \onq =  - (\nu (\theta_{\YtildeA} \otimes \hZ_A) \onq \phi) \lns (A \XtildeA) \onq + (\lns (A \YtildeA) \onq \phi ) \nu (\theta_{\XtildeA} \otimes \hZ_A) \onq + \lns (\hZ_A) \onq,\\
&[\nu (\theta_{\Ytilde} \otimes \hZ) , \lns (\hZ) ] \onq =  - \lns (A \YtildeA) \onq + (\lns (\hZ_A) \onq \phi) \nu (\theta_{\XtildeA} \otimes \hZ_A) \onq,
\end{align*}

\begin{align*}
& [\lns (\hZ) , \lns ((\cdot) \Xtilde)] \onq =  (\lns (\hZ_A) \onq \phi ) \lns (A \YtildeA) \onq  - \tilde{\hsigma}_A^2 \nu (\theta_{\XtildeA} \otimes \hZ_A) \onq - \tilde{\Pi}_{\hZ} \nu (\theta_{\YtildeA} \otimes \hZ_A) \onq,\\
&[\lns (\hZ) , \lns ((\cdot) \Ytilde)] \onq =  - (\lns (\hZ_A) \onq \phi ) \lns (A \XtildeA) \onq + \Pi_{\Xtilde} \nu (A(X \wedge Y) ) \onq\\
& \quad \quad \quad \quad \quad \quad \quad \quad \quad \quad \quad - \tilde{\Pi}_{\hZ} \nu (\theta_{\XtildeA} \otimes \hZ_A) \onq - \tilde{\hsigma}_A^1 \nu (\theta_{\YtildeA} \otimes \hZ_A) \onq,\\
&[\lns ( \Xtilde) , \lns ( \Ytilde) ] = - K \nu (A (X \wedge Y)) \onq,\\
&[\lns ((\cdot) \Xtilde) , \lns ((\cdot) \Ytilde) ] =  - (\lns (A \XtildeA) \onq \phi ) \lns (A \XtildeA) \onq - (\lns (A \YtildeA) \onq \phi ) \lns (A \YtildeA) \onq\\
& \quad \quad \quad \quad \quad \quad \quad \quad \quad \quad \quad \; \; + \nu (\Rol_q) \onq + K \nu (A(X \wedge Y) )\onq.
\end{align*}
\end{lemma}

Next lemma encapsulates a particular $\pi_Q$-vertical derivative of the radius $r$ and the angle $\phi$ (see \eqref{e2.9}) that will be used frequently.
It is valid in the case that on a small enough rolling neighbourhood $O(q_0)$ of the point $\qz$, both components of the pair $(\Pi_X,\Pi_Y)$ 
never vanish,
since this is the underlying assumption allowing us to define the functions $r,\phi:O(q_0)\to\R$ in \eqref{e2.9}, with $r>0$ on $O(q_0)$. 
It is also to be understood that we take $O(q_0)$ to be small enough (around $q_0$) so that the local orthonormal frame $X,Y$ of $M$ remains 
defined on the open subset $V=\pi_{Q,M}(O(q_0))$ of $M$.
 
\begin{lemma}\label{le:nuAXAYphi}
Assume that $(\Pi_X,\Pi_Y)\neq (0,0)$ on a rolling neighbourhood $O(q_0)$ of $\q\in Q$.
Then at every $\q\in O(q_0)$ we have
\[
\nu(A(X\wedge Y))|_q r=0,\quad
\nu(A(X\wedge Y))|_q\phi=-1.
\]
\end{lemma}

\begin{proof}
Clearly,
\[
\nu(A(X\wedge Y))|_q X=0,\quad \nu(A(X\wedge Y))|_q Y=0
\]
and therefore
\[
\nu(A(X\wedge Y))|_q \hZ_{(\cdot)}
={}& \star\big(A\big(\nu(A(X\wedge Y))|_q X\big)\wedge (AY)\big)
+\star\big(\big((AX)\wedge \big(A\nu(A(X\wedge Y))|_q Y\big)\big) \\
{}& +\star \big(\big(A(X\wedge Y)X\big)\wedge (AY)\big)\big)
+\star \big((AX)\wedge \big(A(X\wedge Y)Y\big) \\
={}&
0+0+\star \big((AY)\wedge (AY)\big)+\star \big((AX)\wedge (-AX)\big)
=0.
\]
Using these relations and \eqref{eq:hsigma_Pi}, we then find
\[
\nu(A(X\wedge Y))|_q \Pi_X
={}& \hg(\hR(\star \nu(A(X\wedge Y))|_q \hZ_{(\cdot)}), \star AX)
+\hg(\hR(\star \hZ_{(\cdot)}), \star A\nu(A(X\wedge Y))|_q X) \\
{}&+\hg(\hR(\star \hZ_{(\cdot)}), \star A(X\wedge Y)X)
= \Pi_Y \\
\nu(A(X\wedge Y))|_q \Pi_Y
={}& \hg(\hR(\star \nu(A(X\wedge Y))|_q \hZ_{(\cdot)}), \star AY)
+\hg(\hR(\star \hZ_{(\cdot)}), \star A\nu(A(X\wedge Y))|_q Y) \\
{}&+\hg(\hR(\star \hZ_{(\cdot)}), \star A(X\wedge Y)Y)
=-\Pi_X.
\]
Next, the definition of $r$ and $\phi$ in \eqref{e2.9}
yields $\Pi_X^2+\Pi_Y^2=r^2$,
which differentiated along $\nu(A(X\wedge Y))|_q$
yields, in view of the previous identities,
\[
2r \nu(A(X\wedge Y))|_q r= \nu(A(X\wedge Y))|_q (\Pi_X^2+\Pi_Y^2)
=2\Pi_X \Pi_Y+2\Pi_Y(-\Pi_X)=0,
\]
and since $r\neq 0$ on $O(q_0)$, we find $\nu(A(X\wedge Y))|_q r=0$.

Given the last three identities and applying $\nu(A(X\wedge Y))|_q$ to \eqref{e2.9},
we find
\[
\Pi_Y={}& (\nu(A(X\wedge Y))|_q r)\cos \phi
-r\sin\phi (\nu(A(X\wedge Y))|_q \phi)=-r\sin\phi (\nu(A(X\wedge Y))|_q \phi) \\
-\Pi_X={}& (\nu(A(X\wedge Y))|_q r)\sin \phi
+r\cos\phi (\nu(A(X\wedge Y))|_q \phi)=+r\cos\phi (\nu(A(X\wedge Y))|_q \phi),
\]
which in view of \eqref{e2.9}
implies $\nu(A(X\wedge Y))|_q \phi=-1$ and completes the proof.

\end{proof}

In the next Lemmas \ref{l2.3}-\ref{l2.9} below,
we let $\qz\in Q$, and assume (like in section \ref{ss10}) that on a small enough rolling neighbourhood
$O(q_0)$ around $q_0$ (see Definition \ref{def:rolN})
the function $K - \hsigma_{(\cdot)}$ vanishes
while $(\Pi_X,\Pi_Y)\neq (0,0)$ at every point of $O(q_0)$.
It is also part of this assumption that the local orthonormal frame $X,Y$ of $M$ is defined on (at least) the open subset $V=\pi_{Q,M}(O(q_0))$ of $M$.
Also, recall the definitions of $r$, $\phi$ and $\beta$ given in \eqref{e2.9} and \eqref{eq:def:beta},
and those of $\tilde{\hsigma}_{(\cdot)}^1$, $\tilde{\hsigma}_{(\cdot)}^2$, $\tilde{\Pi}_{\hZ}$
in \eqref{eq:def:sigma_Pi_tilde}.

\begin{lemma}\label{l2.3}
Assuming that $K (x) - \hsigma_A=0$ and $(\Pi_X(q),\Pi_Y(q))\neq (0,0)$ for all $\q\in O(q_0)$,
then the derivatives of $\hsigma_{(\cdot)}$ obey the following identities at every $\q\in O(q_0)$:
	\begin{align*}
	&\lns (\Xtilde_A) \onq \hsigma_{(\cdot)} = 0, &  & \lns (\Ytilde_A) \onq \hsigma_{(\cdot)} = 0,& &
	\nu(\theta_{\YtildeA} \otimes \hZ_A) \onq \hsigma_{(\cdot)} = 0,\\
	 & \nu (A (X \wedge Y) ) \onq \hsigma_{(\cdot)} = 0,& &
	\lns (\hZ_A) \onq \hsigma_{(\cdot)} = - \beta (q),   &   &  \nu(\theta_{\XtildeA} \otimes \hZ_A) \onq \hsigma_{(\cdot)} =2 r.
	\end{align*}
\end{lemma}

\begin{lemma}\label{l2.4}
Assuming that $K (x) - \hsigma_A=0$ and $(\Pi_X(q),\Pi_Y(q))\neq (0,0)$ for all $\q\in O(q_0)$,
then the derivatives of $r$ obey the following identities at every $\q\in O(q_0)$:
	\begin{align*}
	&\lr (\Xtilde_A) \onq r = - \frac{3}{2} \beta (q),\quad \lns (A \Xtilde_A) \onq r = - \beta (q),\quad 
	\lns (\Xtilde_A) \onq r = - \frac{\beta(q)}{2},\\
&  \lns (\Ytilde_A) \onq r = 0,\quad \nu(\theta_{\YtildeA} \otimes \hZ_A) \onq r = - \tilde{\Pi}_{\hZ},\quad \nu(A (X \wedge Y)) \onq r = 0,\quad \nu(\theta_{\XtildeA} \otimes \hZ_A) \onq r = \tilde{\hsigma}_A^1 - \hsigma_A.
	\end{align*}
\end{lemma}
\begin{lemma}\label{l2.5}
Assuming that $K (x) - \hsigma_A=0$ and $(\Pi_X(q),\Pi_Y(q))\neq (0,0)$ for all $\q\in O(q_0)$,
then the derivatives of $\phi$ obey the following identities at every $\q\in O(q_0)$:
\begin{align*}
&{} \lns (\Xtilde_A) \onq \phi = - g (\Gamma, \XtildeA),
\quad  \lns (A \Xtilde_A) \onq \phi = 0, \\
&{} \lns (\Ytilde_A) \onq \phi = - g (\Gamma, \YtildeA),
\quad  \lns (A \Ytilde_A) \onq \phi = \frac{\beta (q)}{2r}, \\
&{} \nu(\theta_{\YtildeA} \otimes \hZ_A) \onq \phi = \frac{1}{r}(\tilde{\hsigma}^2_A - \hsigma_A),
\quad
\nu(\theta_{\XtildeA} \otimes \hZ_A) \onq \phi = - \frac{1}{r}\tilde{\Pi}_{\hZ}.
	\end{align*}
\end{lemma}

\begin{lemma}\label{l2.6}
Assuming that $K (x) - \hsigma_A=0$ and $(\Pi_X(q),\Pi_Y(q))\neq (0,0)$ for all $\q\in O(q_0)$,
then the derivatives of $\frac{\beta (\cdot)}{2r}$ obey the following identities at every $\q\in O(q_0)$:
\begin{align*}
&\lr (\Xtilde_A) \onq \frac{\beta (\cdot)}{2r} = - \big( (\frac{\beta (q)}{2r})^2 + \tilde{\hsigma}_A^2 \big),
\quad  \lns (\Ytilde_A) \onq \frac{\beta (\cdot)}{2r} = 0,\\
& \lns (A \Xtilde_A) \onq \frac{\beta (\cdot)}{2r} = - \big( (\frac{\beta (q)}{2r})^2 + \tilde{\hsigma}_A^2 \big),
\quad \nu(A (X \wedge Y)) \onq \frac{\beta (\cdot)}{2r} = 0.
\end{align*}
\end{lemma}

\begin{lemma}\label{l2.7}
Assuming that $K (x) - \hsigma_A=0$ and $(\Pi_X(q),\Pi_Y(q))\neq (0,0)$ for all $\q\in O(q_0)$,
then the derivatives of $\tilde{\hsigma}_{(\cdot)}^1$ obey the following identities at every $\q\in O(q_0)$:
	\begin{align*}
	&\lns (\XtildeA) \onq \tilde{\hsigma}_{(\cdot)}^1 = 0 ,\quad \lns (\YtildeA) \onq \tilde{\hsigma}_{(\cdot)}^1 = 0,\quad \nu (\theta_{\YtildeA} \otimes \hZ_A) \onq \tilde{\hsigma}_{(\cdot)}^1 = - \frac{2}{r} \tilde{\Pi}_{\hZ} (\tilde{\hsigma}^2_A - \hsigma_A),\\  &\nu (A ( X \wedge Y)) \onq \tilde{\hsigma}_{(\cdot)}^1 = 0,\quad \nu (\theta_{\XtildeA} \otimes \hZ_A) \onq \tilde{\hsigma}_{(\cdot)}^1 = + \frac{2}{r} (\tilde{\Pi}_{\hZ})^2 - 2 r.
	\end{align*}
\end{lemma}
\begin{lemma}\label{l2.8}
Assuming that $K (x) - \hsigma_A=0$ and $(\Pi_X(q),\Pi_Y(q))\neq (0,0)$ for all $\q\in O(q_0)$,
then the derivatives of $\tilde{\hsigma}_{(\cdot)}^2$ obey the following identities at every $\q\in O(q_0)$:
	\begin{align*}
	&\lns (\XtildeA) \onq \tilde{\hsigma}_{(\cdot)}^2 = 0,\quad \lns (\YtildeA) \onq \tilde{\hsigma}_{(\cdot)}^2 = 0,\quad \nu (\theta_{\YtildeA} \otimes \hZ_A) \onq \tilde{\hsigma}_{(\cdot)}^2 = \frac{2}{r} \tilde{\Pi}_{\hZ} (\tilde{\hsigma}^2_A - \hsigma_A),\\& \nu (A ( X \wedge Y)) \onq \tilde{\hsigma}_{(\cdot)}^2 = 0,\quad \nu (\theta_{\XtildeA} \otimes \hZ_A) \onq \tilde{\hsigma}_{(\cdot)}^2 = - \frac{2}{r} (\tilde{\Pi}_{\hZ})^2.
	\end{align*}
\end{lemma}
\begin{lemma}\label{l2.9}
Assuming that $K (x) - \hsigma_A=0$ and $(\Pi_X(q),\Pi_Y(q))\neq (0,0)$ for all $\q\in O(q_0)$,
then the derivatives of $\tilde{\Pi}_{\hZ}$ obey the following identities at every $\q\in O(q_0)$:
	\begin{align*}
	&\lns (\XtildeA) \onq \tilde{\Pi}_{\hZ} = 0,\quad \lns (\YtildeA) \onq \tilde{\Pi}_{\hZ} = 0,\quad \nu (\theta_{\YtildeA} \otimes \hZ_A) \onq \tilde{\Pi}_{\hZ} = \frac{1}{r} (\tilde{\hsigma}_A^1 - \tilde{\hsigma}_A^2) (\tilde{\hsigma}_A^2 - \hsigma_A ) + r,\\&
	\nu (A ( X \wedge Y)) \onq \tilde{\Pi}_{\hZ} = 0,
	\quad \nu (\theta_{\XtildeA} \otimes \hZ_A) \onq \tilde{\Pi}_{\hZ} = - \frac{1}{r} \tilde{\Pi}_{\hZ} (\tilde{\hsigma}_A^1 - \tilde{\hsigma}_A^2).
	\end{align*}
\end{lemma}

\begin{proof}
Checking in detail all the identities in Lemmas
\eqref{l2.3} through \eqref{l2.9} would be a straightforward but lengthy task,
and for that reason it is omitted here.
However, in order to gain confidence on the validity
of those relations,
and to exemplify the computations used in their proofs, we content ourselves at showing just four of them,
namely those for
$\nu(\theta_{\YtildeA} \otimes \hZ_A) \onq \hsigma_{(\cdot)}$, $\lr (\XtildeA) \onq r $, $\nu (\theta_{\YtildeA} \otimes \hZ_A) \onq \phi$ and $\lr (\XtildeA) \onq (\frac{\beta (\cdot)}{2r})$.

The first of these derivatives is simply
\[
\nu(\theta_{\YtildeA} \otimes \hZ_A) \onq \hsigma_{(\cdot)} = \nu(\theta_{\YtildeA} \otimes \hZ_A) \onq \hg (\hR (\star \hZ_{(\cdot)}), \star \hZ_{(\cdot)}) = 2 \hg (\hR (\star \YtildeA) , \star \hZ_A) = 2 \Pi_{\Ytilde} = 0,
\]
where at the last step we have used \eqref{e2.11}.

The next derivative $\lr (\XtildeA) \onq r $ is computed
by using Lemma \ref{le:LR_Xtilde_Ytilde},
the fact that $\Pi_{\Ytilde}=0$ from \eqref{e2.11},
the identities (whose last steps follow from Lemma \ref{le:LR_Xtilde_Ytilde})
\[
\lr(\Xtilde_A)\onq \hZ_{(\cdot)}
={}& \star \big((A\lr(\Xtilde_A)\onq \Xtilde)\wedge A\Ytilde_A\big)
+\star \big(A\Xtilde\wedge A(\lr(\Xtilde_A)\onq \Ytilde)\big)
=0 \\
\lr(\Ytilde_A)\onq \hZ_{(\cdot)}
={}& \star \big((A\lr(\Ytilde_A)\onq \Xtilde)\wedge A\Ytilde_A\big)
+\star \big(A\Xtilde\wedge A(\lr(\Ytilde_A)\onq \Ytilde)\big)
=0,
\]
and the second Bianchi identity projected onto $\star \hZ_A$,
\[
\lr (\XtildeA) \onq r & = \lr (\XtildeA) \onq \Pi_{\Xtilde} =  \lr (\XtildeA) \onq (\hg (\hR ( \star (\cdot) \Xtilde) , \star \hZ_{(\cdot)}) )\\[2mm]
={}& \hg ( (\hnabla_{A \XtildeA} \hR) ( \star A \XtildeA) , \star \hZ_A) + \hg (\hR (\star A \lr (\XtildeA) \onq \Xtilde ), \star \hZ_A  )\\[2mm]
{}& + \hg (\hR (\star A \XtildeA ) , \star \lr (\XtildeA) \onq \hZ_{(\cdot)})\\[2mm]
={}& - \hg ((\hnabla_{\hZ_A} \hR) (\star \hZ_A), \star \hZ_A) - \hg ((\hnabla_{A \Ytilde_A} \hR) (\star A \Ytilde_A), \star \hZ_A)\\[2mm]
={}& - \beta(q) - \lr (\YtildeA) \onq (\hg (\hR ( \star (\cdot) \Ytilde) , \star \hZ_{(\cdot)}) ) + \hg (\hR (\star A \lr (\YtildeA) \onq \Ytilde ), \star \hZ_A  ) \\[2mm]
{}&
+\hg(\hR(\star A\Ytilde_A), \star \lr (\YtildeA) \onq \hZ_{(\cdot)}) \\[2mm]
={}& - \beta(q) - \lr (\YtildeA) \onq \Pi_{\Ytilde} - \frac{\beta(q)}{2r} \Pi_{\Xtilde},
\]
where, moreover, at the last step
we have used Lemma \ref{le:LR_Xtilde_Ytilde}
in combination with the identities for $\lns(\Ytilde_A)\onq \phi$ and $\lns(A\Ytilde_A)\onq \phi$ in Lemma \ref{l2.5},
in order to observe that
\[
\lr (\YtildeA) \onq \Ytilde
={}& -\big(\lns(\Ytilde_A)\onq \phi+\lns(A\Ytilde_A)\onq \phi+g(\Gamma,\Ytilde_A)\big)\Xtilde_A \\
={}& -\big(-g(\Gamma,\Ytilde_A)+\frac{\beta (q)}{2r}+g(\Gamma,\Ytilde_A)\big)\Xtilde_A
=-\frac{\beta(q)}{2r}\Xtilde_A.
\]
Using again \eqref{e2.11}
we hence conclude that $\lr (\XtildeA) \onq r  = - \frac{3}{2} \beta (q)$ as claimed.

To obtain the third derivative formula $\nu (\theta_{\YtildeA} \otimes \hZ_A) \onq\phi$,
we proceed first with the computations
\begin{multline}\label{e2.16}
- r \sphi \nu (\theta_{\YtildeA} \otimes \hZ_A) \onq \phi + \cphi \nu (\theta_{\YtildeA} \otimes \hZ_A) \onq r = \nu (\theta_{\YtildeA} \otimes \hZ_A) \onq \Pi_X \\
= - \sphi \nu (\theta_{X} \otimes \hZ_A) \onq \Pi_X + \cphi \nu (\theta_{Y} \otimes \hZ_A) \onq \Pi_X  =  - \sphi (\hsigma^1_A - \hsigma_A) - \cphi \Pi_Z,
\end{multline}
and
\begin{multline}\label{e2.17}
r \cphi \nu (\theta_{\YtildeA} \otimes \hZ_A) \onq \phi + \sphi \nu (\theta_{\YtildeA} \otimes \hZ_A) \onq r = \nu (\theta_{\YtildeA} \otimes \hZ_A) \onq \Pi_Y \\
= - \sphi \nu (\theta_{X} \otimes \hZ_A) \onq \Pi_Y + \cphi \nu (\theta_{Y} \otimes \hZ_A) \onq \Pi_Y =  \sphi \Pi_Z  + \cphi (\hsigma^2_A - \hsigma_A),
\end{multline}
where at the first and second equalities one uses \eqref{e2.9} and \eqref{e2.10},
while the last steps follow from the formulas listed in Lemma \ref{le:6.3}.
Thus, if one multiplies \eqref{e2.16} by $-\sphi$ and \eqref{e2.17} by $\cphi$, and then add the results, we get
\begin{align*}
\nu (\theta_{\YtildeA} \otimes \hZ_A) \onq \phi = \frac{1}{r} \big( \sphi^2 (\hsigma^1_A - \hsigma_A) + 2 \cphi \sphi \Pi_Z   + \cphi^2 (\hsigma^2_A - \hsigma_A) \big)
= \frac{1}{r} (\tilde{\hsigma}_A^2 - \hsigma_A),
\end{align*}
where \eqref{eq:def:sigma_Pi_tilde} was used at the last step.
This proves the claimed form of this derivative.

We shall then demonstrate the indicated formula for $\lr (\XtildeA) \onq (\frac{\beta (\cdot)}{2r})$,
which is the last case we will cover in this proof.
By Lemma \ref{l2.5}, one has
$\lr(\Xtilde_A)\onq \phi=-g(\Gamma,\Xtilde_A)$, $\lr(\Ytilde_A)\onq \phi=-g(\Gamma,\Ytilde_A)+\frac{\beta(q)}{2r(q)}$ and hence by Lemma \ref{le:LR_Xtilde_Ytilde}
it follows that
$\lr(\Xtilde_A)|_q\Ytilde=0$,
$\lr(\Ytilde_A)|_q\Xtilde=\frac{\beta(q)}{2r}\Ytilde_A$.
Using these formulas, one can compute
\[
[\lr (\Xtilde) , \lr (\Ytilde) ] \onq \phi
={}& \lr (\XtildeA) \onq (\lr (\Ytilde) \onq \phi) - \lr (\YtildeA) \onq (\lr (\Xtilde) \onq \phi)\\[2mm]
={}& \lr (\XtildeA) \onq \big(-g(\Gamma,\Ytilde)+\frac{\beta}{2r}\big) - \lr (\YtildeA) \onq \big(-g(\Gamma,\Xtilde_A)\big)\\[2mm]
={}& - g (\nabla_{\XtildeA} \Gamma , \YtildeA ) + g (\nabla_{\YtildeA} \Gamma , \XtildeA ) +  \frac{\beta (q)}{2r} g (\Gamma , \YtildeA) +  \lr (\XtildeA) \onq \big(\frac{\beta}{2r}\big) \\[2mm]
={}&   K (x) + \frac{\beta (q)}{2 r} g (\Gamma , \YtildeA) +  \lr (\XtildeA) \onq \big(\frac{\beta}{2r}\big),
\]
where at the last step we have used
the formula \eqref{e2.2}, which clearly still holds when $X$,$Y$ are replaced by $\Xtilde_A$, $\Ytilde_A$.

On the other hand, using the first formula in Lemma \ref{l2.2},
the expression for $\Rol_q$ in \eqref{eq:Rol:reduced},
the above formulas for $\lr(\Xtilde_A)\onq \phi$ and $\lr(\Ytilde_A)\onq \phi$
and the formula for $\nu (\theta_{\YtildeA} \otimes \hZ_A) \onq \phi$ in Lemma \ref{l2.5},
we find
\[
[\lr (\Xtilde) , \lr (\Ytilde) ] \onq \phi = &
-\frac{\beta(q)}{2r}\lr(\Ytilde_A) \onq \phi  - r \nu (\theta_{\YtildeA} \otimes \hZ_A) \onq \phi \\
={}& - \frac{ \beta (q)}{2r} \big(-g(\Gamma,\Ytilde_A)+\frac{\beta(q)}{2r}\big) - (\tilde{\hsigma}_A^2 - \hsigma_A).
\]

Combining the above two expressions for $[\lr (\Xtilde) , \lr (\Ytilde) ] \onq \phi$
and solving for $\lr (\XtildeA) \onq \big(\frac{\beta}{2r}\big)$
yields
$\lr (\XtildeA) \onq \big(\frac{\beta}{2r}\big)=-K(x)-\big(\frac{\beta(q)}{2r}\big)^2-\tilde{\hsigma}_A^2 + \hsigma_A$.
Finally, the assumption $K(x)=\hsigma_A$ on $O(q_0)$ reduces this relation down to
$\lr (\XtildeA) \onq \big(\frac{\beta}{2r}\big)=-\big(\frac{\beta(q)}{2r}\big)^2-\tilde{\hsigma}_A^2$,
which is what we wanted to prove.

\end{proof}

In the remaining Lemmas \ref{le:app:3.1}-\ref{lemmaf} below,
we will be assuming (like in section \ref{ss11}),
that $K-\hsigma_{(\cdot)}\neq 0$ and $(\Pi_X , \Pi_Y) \neq (0,0)$ on a small enough rolling neighbourhood $O(q_0)$ of $q_0$.
It is also part of this assumption that the local orthonormal frame $X,Y$ of $M$ is defined on (at least) the open subset $V=\pi_{Q,M}(O(q_0))$ of $M$.
Also, recall the definitions of $\phi$, $\omega$, $G_{\Xtilde}$, $G_{\Ytilde}$, $H_{\Xtilde}$ and $H_{\Ytilde}$
given in \eqref{e2.9}, \eqref{eq:omega} and \eqref{eq:G_H},
and those of $\tilde{\hsigma}_{(\cdot)}^1$, $\tilde{\hsigma}_{(\cdot)}^2$, $\tilde{\Pi}_{\hZ}$
in \eqref{eq:def:sigma_Pi_tilde}.

\begin{lemma}\label{le:app:3.1}
Assuming that $K (x) - \hsigma_A\neq 0$ and $(\Pi_X(q),\Pi_Y(q))\neq (0,0)$ for all $\q\in O(q_0)$,
then the $\pi_Q$-vertical derivatives of $\phi$ and $\omega$
are given, at all points $q\in O(q_0)$, by the following expressions:
	\begin{eqnarray}\label{e3.25}
	\begin{array}{ll}
	\quad
	\nu (A (X \wedge Y)) \onq \phi = -1,
	&
	\nu (A (X \wedge Y)) \onq \omega = 0,\\[2mm]
	\quad	
	\nu (\theta_{\YtildeA} \otimes \hZ_A) \onq \phi  = \frac{1}{\omega (K - \hsigma_A)}(\tilde{\hsigma}_A^2 - \hsigma_A),
	&
	\nu (\theta_{\YtildeA} \otimes \hZ_A) \onq \omega  = \frac{- \tilde{\Pi}_{\hZ}}{K - \hsigma_A},\\[2mm]
	\quad	
	\nu (\theta_{\XtildeA} \otimes \hZ_A) \onq \phi  = \frac{- \tilde{\Pi}_{\hZ}}{\omega (K - \hsigma_A)},
	&
	\nu (\theta_{\XtildeA} \otimes \hZ_A) \onq \omega  = \frac{1}{K - \hsigma_A}(\tilde{\hsigma}_A^1 - \hsigma_A) + 2 \omega^2.
	\end{array}
	\end{eqnarray}
\end{lemma}

\begin{proof}
We will content ourselves at showing only the relation $\nu (A (X \wedge Y)) \onq \omega = 0$.
Note also that $\nu (A (X \wedge Y)) \onq \phi = -1$ is already covered by Lemma \ref{le:nuAXAYphi}.

On one hand we have
\begin{align*}
\nu (A (X \wedge Y) )\onq (\omega \cphi) ={}& - \omega \sphi \nu (A (X \wedge Y)) \onq \phi + \cphi \nu (A (X \wedge Y)) \onq \omega,\\[2mm]
\nu (A (X \wedge Y) )\onq (\omega \sphi) ={}& \omega \cphi \nu (A (X \wedge Y)) \onq \phi + \sphi \nu (A (X \wedge Y)) \onq \omega,
\end{align*}
while on the other hand, by \eqref{eq:Pi_XYtilde} and the formulas in Lemma \ref{le:6.3},
\begin{align*}
\nu (A (X \wedge Y) )\onq (\omega \cphi) ={}& \nu (A (X \wedge Y) )\onq (\frac{\Pi_X}{K - \hsigma_{(\cdot)}}) = \frac{\Pi_Y}{K - \hsigma_{A}} = \omega \sphi,\\
\nu (A (X \wedge Y) )\onq (\omega \sphi) ={}& \nu (A (X \wedge Y) )\onq (\frac{\Pi_Y}{K - \hsigma_{(\cdot)}}) = \frac{- \Pi_X}{K - \hsigma_{A}} = - \omega \cphi.
\end{align*}
It is easy to see that these relations imply $\nu (A (X \wedge Y)) \onq \omega = 0$
(and in fact $\nu (A (X \wedge Y)) \onq \phi = -1$ as well).
\end{proof}

The relations listed in the next lemma are derived by a straightforward application of Lemma \ref{l2.2}
along with the definitions of $\omega$, $\Rolbar_q$, $G_{\Xtilde}$, $G_{\Ytilde}$, $H_{\Xtilde}$ and $H_{\Ytilde}$
given in \eqref{eq:omega}, \eqref{eq:Rolbar} and \eqref{eq:G_H}.
We will omit the details.

\begin{lemma}\label{lemmaf}
Assume that $K (x) - \hsigma_A\neq 0$ and $(\Pi_X(q),\Pi_Y(q))\neq (0,0)$ for all $\q\in O(q_0)$.
Using the notations established in \eqref{eq:G_H} and \eqref{eq:ss11:F1_F2},
the following Lie bracket identities hold
on an open neighbourhood of $q_0$ in $Q$,
\begin{align*}
\begin{split}
[\lr (\Xtilde) , F_1] \onq
={} & - (F_1 \onq \phi ) \lr (\YtildeA) \onq + G_{\Xtilde} \lns (A \XtildeA) \onq + \omega G_{\Xtilde} \lns (\hZ_A) \onq  - \hsigma_A \nu (\Rolbar_q) \onq \\
& - ( 2 H_{\Xtilde} G_{\Xtilde} + \omega \lr (\XtildeA) \onq (G_{\Xtilde}))  \nu (\theta_{\XtildeA} \otimes \hZ_A) \onq \\
& + (K \omega - \omega(G_{\Xtilde})^2  + \lr (\XtildeA) \onq ( H_{\Xtilde} ) ) \nu (\theta_{\YtildeA} \otimes \hZ_A) \onq,
\end{split} \\[4mm]
\begin{split}
[\lr (\Ytilde) , F_1] \onq
={} & (F_1 \onq \phi ) \lr (\XtildeA) \onq  + G_{\Ytilde} \lns (A \XtildeA) \onq - H_{\Xtilde} \lns (\hZ_A) \onq \\
& - (G_{\Xtilde} H_{\Ytilde}  +   G_{\Ytilde} H_{\Xtilde} + \omega \lr (\YtildeA) \onq (G_{\Xtilde}) ) \nu (\theta_{\XtildeA} \otimes \hZ_A) \onq \\
& - (\omega G_{\Xtilde}  G_{\Ytilde}  - \lr (\YtildeA) \onq (H_{\Xtilde} )) \nu (\theta_{\YtildeA} \otimes \hZ_A) \onq, 
\end{split}
\end{align*}

\begin{align*}
\begin{split}
[\lr (\Xtilde) , F_2] \onq
={} & - (F_2 \onq \phi ) \lr (\YtildeA) \onq  + G_{\Xtilde}  \lns (A \YtildeA) \onq + (\omega G_{\Ytilde}  - H_{\Xtilde} ) \lns (\hZ_A) \onq \\
& - ( H_{\Xtilde} G_{\Ytilde} + H_{\Ytilde} G_{\Xtilde} + \omega \tilde{\hsigma}_A^2  + \omega \lr (\XtildeA) \onq (G_{\Ytilde} ) )  \nu (\theta_{\XtildeA} \otimes \hZ_A) \onq \\
& - (\omega G_{\Xtilde} G_{\Ytilde} + \omega \tilde{\Pi}_{\hZ} - \lr (\XtildeA) \onq (H_{\Ytilde})) \nu (\theta_{\YtildeA} \otimes \hZ_A) \onq, 
\end{split} \\[4mm]
\begin{split}
[\lr (\Ytilde) , F_2] \onq
={} & (F_2 \onq \phi ) \lr (\XtildeA) \onq  + G_{\Ytilde} \lns (A \YtildeA) \onq - 2 H_{\Ytilde}  \lns (\hZ_A) \onq  \\
& - (\hsigma_A  - (K - \hsigma_A) \omega^2 ) \nu (\Rolbar_q) \onq \\
& - (2 G_{\Ytilde} H_{\Ytilde} + \omega \tilde{\Pi}_{\hZ} + \omega \lr (\YtildeA) \onq ( G_{\Ytilde} )) \nu (\theta_{\XtildeA} \otimes \hZ_A) \onq \\
&- (\omega ( G_{\Ytilde} )^2  +  \omega ( \tilde{\hsigma}_A^1 - K ) + \omega^3 (K - \hsigma_A) + \lr (\YtildeA) \onq (H_{\Ytilde} ) ) \nu (\theta_{\YtildeA} \otimes \hZ_A) \onq,
\end{split}
\end{align*}

\begin{align*}
\begin{split}
[\nu (\Rolbar_{(\cdot)}), F_1] \onq 
={}& \omega (\nu (\theta_{\YtildeA} \otimes \hZ_A) \onq \phi ) \lns (A \XtildeA) \onq - \omega \lns (\hZ_A) \onq  + \omega^2 G_{\Xtilde}  \nu (\Rolbar_q) \onq \\
&
\begin{aligned}
\hspace{1mm}
- \omega\, \big(
{}&
- F_1 \onq \phi + H_{\Xtilde} \nu(\theta_{\YtildeA} \otimes \hZ_A) \onq \phi  \\
{}&
+  G_{\Xtilde} \nu (\theta_{\YtildeA} \otimes \hZ_A) \onq \omega + \nu (\Rolbar_q) \onq ( G_{\Xtilde} ) \big) \nu (\theta_{\XtildeA} \otimes \hZ_A) \onq
\end{aligned} \\
& - \big( \omega^3 G_{\Xtilde} +  F_1 \onq \omega  + \omega^2 G_{\Xtilde} \nu (\theta_{\YtildeA} \otimes \hZ_A) \onq \phi - \nu (\Rolbar_q)\onq (H_{\Xtilde} ) \big) \nu (\theta_{\YtildeA} \otimes \hZ_A) \onq,
\end{split} \\[4mm]
\begin{split}
[\nu (\Rolbar_{(\cdot)}), F_2] \onq
={} & ( \omega^2 + \omega \nu (\theta_{\YtildeA} \otimes \hZ_A) \onq \phi) \lns (A \YtildeA) \onq -  \omega ( \nu (\theta_{\YtildeA} \otimes \hZ_A)\onq \omega ) \lns (\hZ_A) \onq  + \omega^2 G_{\Ytilde} \nu (\Rolbar_q) \onq \\
&
\begin{aligned}
\hspace{1mm}
- \omega\, \big(
{}&
- F_2 \onq \phi + H_{\Ytilde} \nu (\theta_{\YtildeA} \otimes \hZ_A) \onq \phi \\
{}&
+ G_{\Ytilde} \nu (\theta_{\YtildeA} \otimes \hZ_A)\onq \omega + \nu (\Rolbar_q) \onq ( G_{\Ytilde})   \big) \nu (\theta_{\XtildeA} \otimes \hZ_A) \onq
\end{aligned} \\
& + \big(- \omega^3 G_{\Ytilde} - F_2 \onq \omega  - \omega^2 G_{\Ytilde} \nu (\theta_{\YtildeA} \otimes \hZ_A) \onq \phi + \nu (\Rolbar_q)\onq (H_{\Ytilde})  \big) \nu (\theta_{\YtildeA} \otimes \hZ_A) \onq.
\end{split}
\end{align*}

\begin{align*}
[F_1,F_2]\onq
={}&
\Big(-(1+\omega^2)G_{\Xtilde}
+\omega\lns(\hZ_A)\onq\phi
-\frac{G_{\Ytilde}\tilde{\Pi}_{\hZ}}{K-\hsigma_A}
-\frac{H_{\Ytilde}(\tilde{\hsigma}^2_A-\hsigma_A)}{\omega(K-\hsigma_A)}\Big)\lns(A\Xtilde_A)\onq \\
{}& +
\Big(-G_{\Ytilde}+\omega H_{\Xtilde}+\frac{G_{\Xtilde}\tilde{\Pi}_{\hZ}}{K-\hsigma_A}
+\frac{H_{\Xtilde}(\tilde{\hsigma}^2_A-\hsigma_A)}{\omega(K-\hsigma_A)}\Big)\lns(A\Ytilde_A)\onq \\
{}& +
\Big(H_{\Ytilde}+\omega G_{\Xtilde}\big(\frac{\tilde{\hsigma}^1_A-\hsigma_A}{K-\hsigma_A}+2\omega^2-1\big)
+\frac{H_{\Xtilde}\tilde{\Pi}_{\hZ}}{K-\hsigma_A}
+\lns(A\Ytilde_A)\onq\omega\Big)\lns(\hZ_A)\onq \\
{}&
\begin{aligned}
\hspace{1mm}
+
\,\Big(
{} &
G_{\Xtilde}H_{\Xtilde}+G_{\Ytilde}H_{\Ytilde}+\omega\tilde{\Pi}_{\hZ}
-2(G_{\Xtilde}H_{\Ytilde}-G_{\Ytilde}H_{\Xtilde})\frac{\tilde{\Pi}_{\hZ}}{K-\hsigma_A} \\
{}&
+G_{\Xtilde}\lns(A\Xtilde_A)\onq\omega
+G_{\Ytilde} \lns(A\Ytilde_A)\onq\omega
+\omega F_2\onq G_{\Xtilde}
-\omega F_1\onq G_{\Ytilde}
\\
{}&
-G_{\Xtilde}\lns(\hZ_A)\onq \omega
-\omega H_{\Xtilde}\lns(\hZ_A)\onq \phi
\Big)\nu(\theta_{\Xtilde_A}\otimes \hZ_A)\onq
\end{aligned}
\\
{}&
\begin{aligned}
\hspace{1mm}
+
\,\Big(
{}&
\omega(\tilde{\hsigma}^1_A-K+\hsigma_A)+\omega G_{\Xtilde}^2
+\omega G_{\Ytilde}^2
+F_1\onq H_{\Ytilde} - F_2\onq H_{\Xtilde} \\
{}&
-\omega^2 G_{\Xtilde} \lns(\hZ_A)\onq\phi
+(G_{\Xtilde}H_{\Ytilde}-G_{\Ytilde} H_{\Xtilde})\frac{\tilde{\hsigma}^2_A-\hsigma_A}{K-\hsigma_A}
\Big)\nu(\theta_{\Ytilde_A}\otimes \hZ_A)\onq
\end{aligned}
\\
{}& +
\Big(\hsigma_A-\omega^2(K-\hsigma_A)-\omega (G_{\Xtilde} H_{\Ytilde}-G_{\Ytilde} H_{\Xtilde}
\Big) \nu(A(X\wedge Y))\onq.
\end{align*}
\end{lemma}

\bibliographystyle{abbrv}
\bibliography{biblio-2D3D}

\begin{thebibliography}{10}

\bibitem{AgrachevSachkov}
A.~A. Agrachev and Y.~L. Sachkov.
\newblock {\em Control theory from the geometric viewpoint}, volume~87 of {\em
  Encyclopaedia of Mathematical Sciences}.
\newblock Springer-Verlag, Berlin, 2004.
\newblock Control Theory and Optimization, II.

\bibitem{ACL10}
F.~{Alouges}, Y.~{Chitour}, and R.~{Long}.
\newblock A motion-planning algorithm for the rolling-body problem.
\newblock {\em IEEE Transactions on Robotics}, 26(5):827--836, Oct 2010.

\bibitem{CC03}
A.~Chelouah and Y.~Chitour.
\newblock On the motion planning of rolling surfaces.
\newblock {\em Forum Math.}, 15(5):727--758, 2003.

\bibitem{ChitourGodoyMolinaKokkonen1}
Y.~Chitour, M.~Godoy~Molina, and P.~Kokkonen.
\newblock Symmetries of the rolling model.
\newblock {\em Mathematische Zeitschrift}, 281:783--805, 2015.

\bibitem{ChitourGodoyMolinaKokkonen2}
Y.~Chitour, M.~Godoy~Molina, P.~Kokkonen, and I.~Markina.
\newblock Rolling against a sphere: the non-transitive case.
\newblock {\em J. Geom. Anal.}, 26(4):2542--2562, 2016.

\bibitem{CGJK19}
Y.~Chitour, E.~Grong, F.~Jean, and P.~Kokkonen.
\newblock Horizontal holonomy and foliated manifolds.
\newblock {\em Ann. Inst. Fourier (Grenoble)}, 69(3):1047--1086, 2019.

\bibitem{ChitourKokkonen}
Y.~Chitour and P.~Kokkonen.
\newblock Rolling manifolds: Intrinsic formulation and controllability.
\newblock Arxiv preprint 1011.2925v2, 2011.

\bibitem{ChitourKokkonen2}
Y.~Chitour and P.~Kokkonen.
\newblock Rolling manifolds on space forms.
\newblock {\em Annales de l'Institut Henri Poincare (C) Non Linear Analysis},
  29(6):927 -- 954, 2012.

\bibitem{ChitourKokkonen1}
Y.~Chitour and P.~Kokkonen.
\newblock Rolling of manifolds and controllability in dimension three.
\newblock {\em M\'{e}m. Soc. Math. Fr. (N.S.)}, (147):iv+162, 2016.

\bibitem{ChitourGodoyMolinaKokkonen}
Y.~Chitour, M.~G. Molina, and P.~Kokkonen.
\newblock {\em The rolling problem: overview and challenges}, pages 103--122.
\newblock Springer International Publishing, Cham, 2014.

\bibitem{HafassaMortadaChitourKokkonen}
B.~Hafassa, A.~Mortada, Y.~Chitour, and P.~Kokkonen.
\newblock Horizontal holonomy for affine manifolds.
\newblock {\em Journal of Dynamical and Control Systems}, 22:413--440, 2016.

\bibitem{hiepko79}
S.~Hiepko.
\newblock Eine innere {K}ennzeichnung der verzerrten {P}rodukte.
\newblock {\em Math. Ann.}, 241(3):209--215, 1979.

\bibitem{Kokkonen}
P.~Kokkonen.
\newblock A characterization of isometries between {R}iemannian manifolds by
  using development along geodesic triangles.
\newblock {\em Arch. Math. (Brno)}, 48(3):207--231, 2012.

\bibitem{Kokkonen2}
P.~Kokkonen.
\newblock Rolling of manifolds without spinning.
\newblock {\em J. Dyn. Control Syst.}, 19(1):123--156, 2013.

\bibitem{MortadaKokkonenChitour}
A.~Mortada, P.~Kokkonen, and Y.~Chitour.
\newblock Rolling manifolds of different dimensions.
\newblock {\em Acta Appl. Math.}, 139:105--131, 2015.

\bibitem{oneill83}
B.~O'neill.
\newblock {\em Semi-Riemannian geometry with applications to relativity}.
\newblock Academic press, 1983.

\bibitem{petersen06}
P.~Petersen.
\newblock {\em Riemannian geometry}, volume 171 of {\em Graduate Texts in
  Mathematics}.
\newblock Springer, Cham, third edition, 2016.

\bibitem{sussmann73}
H.~J. Sussmann.
\newblock Orbits of families of vector fields and integrability of
  distributions.
\newblock {\em Trans. Amer. Math. Soc.}, 180:171--188, 1973.

\end{thebibliography}
	 
\end{document}